\documentclass[11pt, english, reqno]{amsart}
\usepackage{fullpage}
\usepackage[utf8]{inputenc}
\usepackage{amsmath,amssymb,stmaryrd,array,multirow,dsfont,mathrsfs}
\usepackage{enumerate}
\usepackage{xcolor}
\usepackage{color}
\usepackage{hyperref}
\usepackage{indentfirst}
\usepackage[T1]{fontenc}
\usepackage{microtype}
\usepackage{graphics}
 
\graphicspath{{Images/}}

\xdefinecolor{darkgreen}{rgb}{0,0.4,0}

\usepackage[left=2.16cm, right=2.16cm, lines=45, top=0.8in, bottom=1.0in]{geometry}

\usepackage{hyperref}
\usepackage{indentfirst}
\usepackage{cite}
\usepackage[utf8]{inputenc}
\usepackage[T1]{fontenc}
\usepackage{microtype}

\usepackage{graphics}
 
 \usepackage{amsmath,amssymb,stmaryrd,array,multirow,dsfont,mathrsfs}
\usepackage{amsthm}
\usepackage{xcolor}
\usepackage{color}
\usepackage{hyperref}
\usepackage{tikz}

\graphicspath{{Images/}}

\xdefinecolor{darkgreen}{rgb}{0,0.4,0}

\newcommand{\beq}{\begin{equation}}
\newcommand{\eeq}{\end{equation}}
\newcommand{\beqs}{\begin{equation*}}
\newcommand{\eeqs}{\end{equation*}}
\newcommand{\ben}{\begin{eqnarray}}
\newcommand{\een}{\end{eqnarray}}
\newcommand{\beno}{\begin{eqnarray*}}
\newcommand{\eeno}{\end{eqnarray*}}
\renewcommand{\Re}{{\rm Re}\,}
\renewcommand{\Im}{{\rm Im}\,}

\newcommand{\Id}{{\rm Id}}
\newcommand{\Supp}{{\rm Supp}\,}

\newcommand{\Rmnum}[1]{\uppercase\expandafter{\romannumeral #1} }
 \numberwithin{equation}{section}
\usepackage{mathtools}

\DeclarePairedDelimiterX{\inp}[2]{\langle}{\rangle}{#1, #2}

\newtheorem{thm}{Theorem}[section]
\newtheorem{lem}[thm]{Lemma}
\newtheorem{prop}[thm]{Proposition}
\newtheorem{rmk}[thm]{Remark}
\newtheorem{cor}[thm]{Corollary}

\def\op {\mathrm{Op}}
\def\curl{\mathop{\rm curl}\nolimits}

\def \d {\mathrm {d}}
\def\cA{{\mathcal A}}

\def\cC{{\mathcal C}}

\def\cE{{\mathcal E}}
\def\cF{{\mathcal F}}

\def\cH{{\mathcal H}}

\def\cJ{{\mathcal J}}
\def\cK{{\mathcal K}}
\def\cL{{\mathcal L}}
\def\cM{{\mathcal M}}
\def\cN{{\mathcal N}}
\def\cO{{\mathcal O}}
\def\cP{{\mathcal P}}
\def\cQ{{\mathcal Q}}
\def\cR{{\mathcal R}}
\def\cS{{\mathcal S}}
\def\cT{{\mathcal T}}

\def\cY{{\mathcal Y}}

\let\f=\frac
\def \p {\partial}

\def\mR {\mathbb{R}}

\def\mZ {\mathbb{Z}}

\def\ep{\epsilon}

\def \pt {\partial_{t}}
\def \vr {\varrho}

\def \div {\,\mathrm{div}}
\def \vp {\varphi}
\def \vr {\varrho}
\def \cX {\mathcal{X}} 
\def \bI {\mathbb{I}}
\def \sech {\mathrm{sech}}
\def \kdv {\textnormal{\scalebox{0.8}{KdV}}}

\def \kp {
\textnormal{\scalebox{0.8}{KP}}} 

\def \diag {\textnormal{diag}}

\def \na{\nabla}

\title{Transverse asymptotic stability of line solitary \\ waves   for  the  Ionic Euler-Poisson system}

\begin{document}

\author{Fr\'ed\'eric Rousset}
\address{Universit\'e Paris-Saclay,  CNRS, Laboratoire de Math\'ematiques d'Orsay (UMR 8628),  91405 Orsay Cedex, France}
\email{frederic.rousset@universite-paris-saclay.fr }

\author{Changzhen Sun}
\address{Université Marie-et-Louis-Pasteur, Laboratoire de Mathématiques de Besançon, UMR CNRS 6623, 25000 Besançon, France}
\email{changzhen.sun@univ-fcomte.fr}

\maketitle
\begin{abstract}
    We prove the linear and nonlinear asymptotic stability of small amplitude one-dimensional solitary waves submitted to small localized irrotational perturbations  in the  three dimensional Euler-Poisson system describing the dynamics of ions. In particular, in this regime,  we obtain  the existence of global smooth solutions
    and describe their asymptotic behavior.
    \end{abstract}

\section{Introduction}
We consider the following three dimensional ionic Euler-Poisson system 
 \beqs
 \left\{
\begin{array}{l}
\displaystyle \pt \rho^i +\div_{\mathrm{x}}\big( \rho^i u^i\big)=0,\\[3pt]
\displaystyle \rho^i\pt  u^i+\rho^i u^i\cdot \na_{\mathrm{x}} u^i+
{\nabla_{\mathrm{x}} P(\rho^i)}+\nabla_{\mathrm{x}} \phi^i=0 ,  \\[2pt]
\displaystyle \Delta_{\mathrm{x}} \phi^i =\rho^e-\rho^i,\\[2pt]
\displaystyle u^i|_{t=0} =u_{0}^i,\, \rho^i|_{t=0}=\rho^i_{0}\,,
\end{array}
\right.
\eeqs
where the space variable $\mathrm{x}=(x, y_1, y_2)^t$ is in the whole space $\mathbb{R}^3.$
The unknowns $\rho^i\in \mR^{+}, u^i\in \mR^3, -\na_{\mathrm{x}}\phi^i$ stand for  the ions density, the ions velocity and the self-consistent electric field respectively. The pressure $P=P(\rho^i)$ is  assumed to be smooth and strictly increasing  with respect to $\rho^i.$ The above system, where  magnetic effects are  neglected,  is  a hydrodynamical model of plasma, describing the motion of ions coupled with a background  electron density. The electrons are assumed to be in
 thermodynamical equilibrium, satisfying the Maxwell-Boltzmann relation $\rho^e=\overline{\rho} \exp (\phi^i),$ where
$\overline{\rho}$ is a positive constant. For a thorough justification of this model from the more general two-fluid Euler-Poisson system, we refer to \cite{GGPS-derivationIEP}.
Without loss of generality, we assume that $\overline{\rho}=1$ and study the following system for $(n^i=\rho^i-1, u)^t:$  
\beq\label{EPION3d}
 \left\{
\begin{array}{l}
\displaystyle \pt n^i +\div_{\mathrm{x}}\big( (1+n^i) u\big)=0,\\[3pt]
\displaystyle \pt  u^i+u^i\cdot \na_{\mathrm{x}} u^i+
\f{\nabla_{\mathrm{x}} P(1+n^i)}{1+n^i}+\nabla_{\mathrm{x}} \phi=0 ,  \\[2pt]
\displaystyle \Delta_{\mathrm{x}} \phi =e^{\phi}-1-n^i,\\[2pt]
\displaystyle u^i|_{t=0} =u_{0},\,\, n^i|_{t=0}=n_{0}\,.
\end{array}
\right.
\eeq

Despite the quasilinear-hyperbolic nature of the higher order part of the  system, 
Guo and Pausader  in \cite{G-P-global} proved the global existence of  smooth irrotational small and localized  solutions for the above system, by using the dispersion of the linearized system combined  with the technique of normal form transformation (or more generally the  `space-time resonance' method \cite{G-N-T-normalform,GMS-annals,Germain-JEDP,WW-Acta}). This result can be seen as the  proof of  the asymptotic  stability of  the constant equilibrium
when submitted to small localized perturbations.  The Euler-Poisson  system also supports  many  physically  interesting nontrivial patterns  such as solitary  waves which are spatially localized traveling waves \cite{sagdeev1966cooperative,BK-existencesolitary-EP,degond} or periodic traveling waves \cite{degond,NRS-EEP}. To justify  their observation   in nature, the  study of the dynamics around these non-trivial  patterns is also of much importance. 

Regarding  the 
stability of solitary waves in the  one-dimensional Euler-Poisson system (1d EP), that is to say \eqref{EPION3d} restricted to functions depending only on $x \in \mathbb{R}$,  Haragus and Scheel proved  in \cite{Mariana-scheel-1dEP} the linear asymptotic stability in the context of  the pressureless (1d EP), where the pressure $P$ is taken as zero.
More recently, Bae and Kwon  studied  in \cite{Bae-Kwon-linearEP}  the case of  general pressure laws.  Both studies are performed  in the framework of exponentially weighted spaces in the spirit of the earlier works of Pego and Weinstein for the KdV equation and related models \cite{Pego-Weinstein-kdv,Pego-Weinstein-Boussinesq} 
and rely crucially on the 
KdV approximation in the long-wave regime \cite{Guo-Pu-kdvlimit}. The main reason for the use of this  approach is that the energy-momentum
approach to the orbital stability cannot be used even formally  since the energy momentum functional has no sign for high frequencies. 
The nonlinear stability of one-dimensional solitary waves in (1d EP) remains an open problem.
Note that an important  difficulty in order to get nonlinear  stability results  for (1d EP), even in the study of the stability of  the constant equilibrium 
is  to rule out  singularity formations in the considered regime.  It is proved in  \cite{BCK-nonlinearity}   that smooth solutions to  (1d EP)  can lose their
 $C^1$ regularity in finite time when the gradients of the Riemann invariants are initially large.  
 The structure of the resulting (cusp-type) singularities near the blow-up point has been established more  recently in  \cite{bae2024structure}.
The corresponding electron case, where the Poisson equation is replaced by 
$ \Delta  \phi= n$  is  more favorable, the stability of constants under small localized perturbations  has been proved in \cite{GHZ-EEP,Li-Wu-2dEEP,Ionescu-pausader-2dEEP,Guo-3dEEP}.
Note that the electron model does not support one-dimensional  solitary waves.

The study of the transverse stability of one-dimensional solitary waves  has a long history in dispersive PDE.
In the early 1970's,  by using the theory of integrable systems Zakharov \cite{Zakharov-K} obtained the transverse instability of the soliton of the KdV equation considered as a one-dimensional solution of the (two-dimensional) KP-I equation.
A general criterion allowing to get linear transverse instability of solitary waves has now  been obtained in \cite{RT-MRL}. This criterion applies to a large class of systems like the KP-I equation and the water-waves system with strong surface tension
(Bond number bigger than $1/3$).
 One can then obtain  nonlinear instability, see  \cite{RT-JMPA,RT-AIHP}  and in particular \cite{R-T-Transinstability} for the transverse
 nonlinear instability of solitary water-waves with strong surface tension. 
 Note that when studying the transverse stability of   kinks or topological solitons, 
  the criterion of \cite{RT-MRL} is not matched and one usually expects stability, see for example \cite{Cuccagna}.
 When the algebraic criterion of \cite{RT-MRL} is not matched for solitary waves,  there are various model situations.
  For the hyperbolic Schr\"odinger equation (where the usual  Laplacian is replaced by $\partial_{x}^2 - \Delta_{y}$), as established
  in \cite{Zakharov-R}, there is still transverse linear instability of the one-dimensional solitary wave whereas for the KP-II
  equation, there is stability \cite{Alexander-Pego-Sachs} and even   nonlinear stability as  proven in an important work by Mizumachi \cite{Mizumachi-KP-nonlinear}.
  One can then conjecture that for small amplitude solitary waves, when the weakly transverse long-wave regime is given by 
  the KP-II equation, which is the case for the three dimensional Euler-Poisson system (see Appendix A.1 for the  justification) one should get transverse stability.  This heuristics was recently justified by Mizumachi \cite{Mizumachi-BL-linear} for the Benney-Luke equation which is another
  asymptotic model that can be obtained from the water-waves system while keeping the KP-II equation as
  a long-wave weakly transverse model.  Note that  for  these semilinear equations,  which are globally well-posed in the energy space, 
   nonlinear asymptotic stability has been obtained  \cite{Mizumachi-KP-nonlinear,Mizumachi-BL-nonlinear}. 
  In these results, which can be seen as the extension to higher dimensions of the classical asymptotic stability results for
   the KdV equation \cite{Pego-Weinstein-kdv, Martel-Merle,Mizu-Nikolay},  the dynamics of the perturbation is not described  everywhere in space, 
   the decay is proven only in front of the solitary wave i.e  in a region $\{x-ct \geq R\}$ where $\mathrm{x}= (x, y)$.
   The proof of the nonlinear transverse stability for   
  quasilinear equations,  which  are not well posed in the energy space and are expected to face singularity formations at least  in some regimes, like the Euler-Poisson system here as mentioned above,   thus requires new ingredients.
  Indeed, in order to exclude singularity formations and thus prove that a high order Sobolev norm will remain finite, we need to control the decay of the perturbation everywhere in space.
  The goal  of this  work is to address this general question on the example of the  three-dimensional Euler-Poisson system. 
   Other interesting physical systems share the same type of structure, for example,  the gravity water-waves system, 
 the asymptotic linear stability has been obtained recently in \cite{RS-WW}, 
  the approach  developed in this paper will be  useful to attack them. 
  Note that from a broader perspective,  the   stability of non-constant equilibria in quasilinear PDEs is still  largely unexplored, 
  we can refer for example to the recent works   \cite{LOS-catenoid}, or \cite{Kerr1}, \cite{Kerr2}, \cite{Kerr3} in very different settings.

\subsection{Existence of line solitary waves}
We shall consider in this work only  irrotational flows 
 and thus assume  that  the velocity can  be written as $u^i=\na\psi^i$ so that  the Euler-Poisson system \eqref{EPION3d} can be rewritten as
\beq\label{EP3d-potential}
 \left\{
\begin{array}{l}
\displaystyle \pt n^i+\div_{\mathrm{x}}\big( (1+n^i) \na_{\mathrm{x}} \psi^i\big)=0,\\[3pt]
\displaystyle \pt  \psi^i+\f12{|\na_{\mathrm{x}} \psi^i|^2} +
h(1+n^i)+\nabla_{\mathrm{x}} \phi=0 ,  \\
\displaystyle \Delta_{\mathrm{x}} \phi =e^{\phi}-1-n^i,
\end{array}
\right.
\eeq
where $\mathrm{x}=(x, y_1, y_2)^t\in \mR^3$ and
$$h(z):= \int_1^z \f{P'(z')}{z'} \, \d z'.$$
 Note that \eqref{EPION3d} and \eqref{EP3d-potential} 
 are equivalent up to an arbitrary function of time in the right hand side of the $\psi^{i}$ equation. However, since we will work within a spatially weighted space, this function must necessarily be zero.
To study solitary waves with speed $c>0,$ we begin by performing the change of variable $\mathrm{x}\rightarrow (x-ct,y_1,y_2)^t,$ which transforms equation \eqref{EP3d-potential} into: 
\beq\label{EP3d-potential-mov}
 \left\{
\begin{array}{l}
\displaystyle \pt n^i-c\p_x n^i +\div_{\mathrm{x}}\big( (1+n^i) \na_{\mathrm{x}} \psi^i\big)=0,\\[3pt]
\displaystyle \pt  \psi^i-c\p_x \psi^i+\f12{|\na_{\mathrm{x}} \psi^i|^2} +
h(1+n^i)+\nabla_{\mathrm{x}} \phi=0 ,  \\
\displaystyle \Delta_{\mathrm{x}} \phi =e^{\phi}-1-n^i\,.
\end{array}
\right.
\eeq

The existence of a family of \textit{one}-dimensional small amplitude solitary waves which are stationary solutions to \eqref{EP3d-potential-mov} depending only on the one-dimensional  variable $x$ is established in \cite{BK-existencesolitary-EP, degond}. For the convenience of the reader, we formulate the result of \cite{BK-existencesolitary-EP} in the following theorem:

\begin{thm}[Restatement of Theorem 1.1 \& 1.2, \cite{BK-existencesolitary-EP}]\label{thm-existence}
Let  $V=\sqrt{1+P'(1)}$ and assume that $c=V+\ep^2.$  There exists $\bar{\ep}_0$ such that for every $\ep\in (0, \bar{\ep}_0),$  the system \eqref{EP3d-potential-mov} has 
 a stationary solution 
  $(n_{c}(x), \psi_{c}(x), \phi_c(x))$ (independent of $y_1, y_2$), given by 
  \beq\label{solitarywave-def}
n_{c}(x)=\ep^2\Theta_{\ep}(\ep x) , \qquad \psi_{c}(x)=\ep\Psi_{\ep}(\ep x) , \quad \phi_c(x)=\ep^2\Phi_{\ep}(\ep x) ,
  \eeq
 where $\Theta_{\ep}, \Psi_{\ep},$ $\Phi_{\ep}$ satisfy the following: \\[2pt]
  (1)
  For some  $d>0$ and  for any integers $k\geq 0 , \, \ell\geq 1, $ there exists $C_k, \, C_{\ell}$ which are independent of $\ep\in(0,\ep_0],$ such that: 
  \beqs 
|\p_{\hat{x}}^k (\Theta_{\ep}-\Theta_{0}, \Phi_{\ep}-\Phi_{0}) (\hat{x})|\leq C_k \, \ep^2 \, e^{-d\,|\hat{x}|} , \qquad |\p_{\hat{x}}^{\ell} (\Psi_{\ep}-\Psi_{0}) (\hat{x})|\leq C_{\ell}\, \ep^2\, e^{-d\,|\hat{x}|} .
  \eeqs
 (2) It holds that when $\ep=0,$
  \beq \label{solitarywave-0}
\Theta_0 (\hat{x})=\Phi_0(\hat{x})=V^{-1}\p_{\hat{x}} \Psi_0 (\hat{x})=\Psi_{\kdv}(\hat{x})
  \eeq
  where $\Psi_{\kdv}(\hat{x})=\f{3}{C(V)}\sech^2\big(\sqrt{{V}/{2}}\, \hat{x}\big) $
is the steady solution to the KdV equation in the moving frame 
\beqs
-\p_{\hat{x}} \Psi +C(V)\Psi\p_{\hat{x}} \Psi+\f{1}{2V}\p_{\hat{x}}^3 \Psi=0\, . \, \qquad 
\eeqs
Here the constant $C(V)$ is defined as 
\begin{align}
  C(V)=V+\f{P''(1)}{2V}\,.
\end{align}
\end{thm}

\subsection{Main result}
Before stating the main results, we need to introduce some functional spaces. 
For any $a\in \mR,$ we shall use the weighted $L^2$ space
 $L_a^2 (\mR^3):= L^2 (\mR^3; e^{2ax}\d x\d y\,) \,$  endowed with the norm
\beqs 
\|f\|_{L_a^2 (\mR^3) }:= \bigg(\int |f(x)|^2  e^{2ax} \d x\d y\bigg)^{\f{1}{2}}\, , \quad  
\eeqs
where hereafter, we use the notation  $y=(y_1, y_2)$ and $\d y=\d y_1 \d y_2\,.$
 We also use weighted norms with higher regularity,  
for any $k>0,$ we set 
$$H_{a}^k(\mR^3):= \big\{f\,\big| \partial^\alpha f \in L_a^2(\mR^3), \, | \alpha| \leq k \big\}.$$

Our main result is  the transverse nonlinear asymptotic stability of the  one-dimensional small amplitude waves :
\begin{thm}\label{thm-nonlinear}
Let $(n_{c_0}, \psi_{c_0})^t(x)$ be a small-amplitude solitary wave 
given by Theorem \ref{thm-existence} with $c_0=V+\ep^2.$ We consider initial data  under the form
$$(n_0^i(\mathrm{x}), \na{\psi}_0^i(\mathrm{x}))^t=(n_{c_0}(x), \psi'_{c_0}(x),0)^t+(n_0,\nabla \psi_0)^t(\mathrm{x})$$ and introduce the quantity
\begin{align*}
\cM(0)
:=
\|(1+|y|^2)^{\f12}(n_0,\nabla \psi_0)\|_{H_a^M}+\|e^{ax}(n_0,\na\psi_0)\|_{L_y^1L_x^2}+
\|(n_0,\nabla \psi_0)\|_{H^{M}\cap \dot{H}^{-1}}+\|(n_0,\nabla \psi_0)\|_{W^{\f{29}{4},1}},
\end{align*}
with $M=12,\,a=\f{\ep}{4}.$
There exists $\ep_0>0, \delta_0>0$ such that for every 
$0<\ep\leq \ep_0, 0<\delta\leq \delta_0,$ if the initial perturbation $(n_0, \psi_0)$ is such that  $\cM(0)\leq \delta,$
then there exists a unique  global smooth  solution  of  \eqref{EP3d-potential}  and there exist two functions $c(t,y), \gamma(t,y)\in C^{M-1}([0,T]\times \mR^2)$ such that
the following decay estimates hold:
\beq\label{decayes-thm}
\begin{aligned}
 \sup_{t\in[0,+\infty)} (1+t)^{1+\iota} \bigg\|  \left(\begin{array}{c}
         n^i  \\
         \na\psi^i
    \end{array}\right)(t, \mathrm{x})- \left(\begin{array}{c}
         n_{c(t,y)}\big(x-c_0t-\gamma(t,y)\big)  \\[4pt]
         \na\psi_{c(t,y)}\big(x-c_0t-\gamma(t,y)\big)
    \end{array}\right)\bigg\|_{L^{\infty}(
    \mR^3)}&\lesssim \delta, \\
    \sup_{t\in[0,+\infty)} \bigg( (1+t)^{\f12}\|\gamma(t,\cdot)\|_{L_y^{\infty}}+(1+t)\|({c}-c_0)(t,\cdot)\|_{L_y^{\infty}} \bigg)&\lesssim  \delta,
\end{aligned}
\eeq
for some small but positive $\iota.$
\end{thm}

In the above statement, 
we have  only stressed  the pointwise decay estimates that we have obtained. As we shall see below, we also control the boundedness of  a high Sobolev norm of the perturbation and the decay in an exponentially weighted space $H^k_{a}$ of the perturbation.
Note that we are assuming exponential decay of the perturbation for negative $x$. We could think of relaxing this assumption by requiring
only an algebraic decay by using Mizumachi's argument \cite{Mizumachi-cmp}, see also \cite{Germain-Pusateri-Rousset}.
 This  would nevertheless require the obtention of  high-order virial estimates for the Euler–Poisson system,
interesting by themselves. In order to avoid the multiplication of difficulties  of different nature, this is left for further investigation.

The strategy of the proof and its main steps will be explained later in this introduction. An important  first step toward the proof of  the  transverse nonlinear stability will be  to 
establish the transverse linear asymptotic stability in the exponentially weighted space. 
\subsection{Transverse linear asymptotic stability}
We linearize the system \eqref{EP3d-potential-mov} about a solitary wave $(n_c,\psi_c, \phi_c)$ to get the following linearized system for $V=(n,\psi)^t:$
\beq\label{Linear-1}
\pt {V}=
\tilde{L}_{c} {V}+\left( \begin{array}{c} 0 \\ -\phi 
\end{array}\right)\, , \qquad \qquad \tilde{L}_c= \left( \begin{array}{cc}
   \p_x(d_c\cdot)  &  -\div_{\mathrm{x}}(\rho_c \na_{\mathrm{x}})   \\[5pt]
 - h'(\rho_c) 
 & d_c\,\p_x
\end{array}\right)\, ,
\eeq
where  $\phi$ solves 
\beq\label{elliptic}
(e^{\phi_c}-\Delta_{\mathrm{x}})\, \phi=n
\eeq
and we have set 
\beq\label{defdcphoc}
d_c=c-\p_x\psi_c\,, \qquad \rho_c=1+n_c\, .
\eeq
We now introduce the functional framework in which we study the asymptotic behavior of the linearized system \eqref{Linear-1}.  
Besides the weighted spaces $H^k_{a}$ defined above, we shall also use 
 the weighted space $\dot{H}_{a}^{1}(\mR^3)$  defined as the closure of smooth compactly supported functions endowed
 with the norm 
\begin{align*}
 \|f\|_{ \dot{H}_{a}^{1}(\mR^3)}:= \bigg(\int \big|(\nabla_{\mathrm{x}} f)(x,y)\big|^2  e^{2ax} \d x\d y\bigg)^{\f{1}{2}}.
\end{align*}
Let us point out that $\|\cdot\|_{\dot{H}_{a}^{1}(\mR^3)}$ is indeed a norm thanks to the 
inequality $$\|\p_x f\|_{L_a^2(\mR^3)}\geq |a| \|f\|_{L_a^2(\mR^3)}\,.$$
Thanks to the Lax-Milgram theorem, for any $a\in [0,1/4],$ and $\ep$ small enough such that $$\sup_{x\in \mR}\,e^{\phi_c}(x)\in [1/2,3/2],$$
we have, for any $n\in L_a^2(\mR^3)$ that  the elliptic problem \eqref{elliptic}
has a unique solution.
 Denoting by  $(e^{\phi_c}-\Delta)^{-1}$  the corresponding solution map, we can  rewrite the system \ref{Linear-1} as
\beq\label{Linear-sys}
\pt {V}=
{L_c} {V} \, , \qquad \qquad {L_c}= \left( \begin{array}{cc}
   \p_x(d_c\cdot)  &  -\div_{\mathrm{x}}(\rho_c \na_{\mathrm{x}})   \\[5pt]
 -\big[ \, h'(\rho_c) +(e^{\phi_c}-\Delta)^{-1}\big]
 & d_c\,\p_x
\end{array}\right)\, .
\eeq
We shall investigate the asymptotic stability of the above linearized system in the weighted space $X_a=L_a^2(\mR^3)\times \dot{H}_a^1(\mR^3)\,.$ 
The main linear result  is a time decay estimate  of  the semigroup generated by  $L_c$ in an infinite dimensional spectral subspace of $X_a,$ denoted as $\mathbb{Q}_c(\eta_0)X_a,$ which is 
supplementary to a spectral subspace 
associated to the spectral curves close to zero $0.$ We refer to \eqref{def-bQ} for the definition of the projection 
$\mathbb{Q}_c(\eta_0).$ To state the result,  we use the notation $\| \cdot \|_{B(X)} $ for the operator norm for linear operators on $X$.
\begin{thm}\label{thm-resol-L}
Let $0<a=\hat{a}\ep < \f{\sqrt{3}}{4}\ep.$ 
 There exist $\ep_1>0, \beta_0=\beta_0(\hat{a})>0, \hat{\eta}_0=\hat{\eta}_0(\hat{a}),$ 
 such that for any $0<\ep\leq \ep_1,  \,0<\beta\leq\beta_0,$ there is $C=C(\hat{a}, \hat{\eta}_0, \beta_0, \ep_1)$ such that
  \beq\label{semigroup}
\|e^{tL_c}\mathbb{Q}_c(\ep^2\hat{\eta}_0)\|_{B(X_a)}\leq C e^{-\beta\ep^3 t}.
 \eeq
\end{thm}

The main reason for which  we get   the exponential decay of the semigroup \eqref{semigroup} 
is that, when considering the linear operator on $\mathbb{Q}_c(\ep^2\hat{\eta}_0)X_a$, 
the resolvent set is shifted slightly into the left half-plane by an amount of order $\cO(\ep^3).$ That is, for any $0<\ep\leq \ep_1, \, 0<\beta\leq\beta_0,\, $ it holds that
 \beqs 
 \Omega_{\beta,\,\ep}:= \big\{ \lambda\in \mathbb{C}\,\big|\, \Re \lambda > -\beta \ep^3\big\} \subset \rho (L_c; \mathbb{Q}_{c}(\ep^2\hat{\eta}_0)X_a)
 \eeqs 
 where we use the notation $\rho(\cdot\,; X)$ for the resolvent set of an operator on $X.$

The semigroup estimate \eqref{semigroup} is then a consequence of the  Gearhart-Prüss theorem \cite{Gearhart,Pruss,Proof-GP} and the following uniform resolvent estimate: 
there exists $C_0=C_0(\hat{a}, \hat{\eta}_0, \beta_0, \ep_1)>0,$ such that for any $\lambda \in \Omega_{\beta, \ep},$
\beq\label{uni-resol}
\|(\lambda-L_c)^{-1} \|_{B(\mathbb{Q}_c(\eta_0)X_a)}\leq C_0\,. 
\eeq
As the proof of the above result is independent of the nonlinear stability analysis, 
we postpone its proof to the final section, Section \ref{sec-prooflinear}.

\subsection{Main difficulties and outline of the proofs}

\subsubsection{Nonlinear stability}
To prove the  stability of a solitary wave, we have as usual to  take into account the translational invariance of the system and the  smooth dependence of the wave on the wave speed $c.$ Therefore, assuming that the initial data is a small  disturbance to the solitary wave $Q_{c_0}(x)=\big(n_{c_0}, \psi_{c_0}\big)^t(x),$  we expect that the solution to \eqref{EP3d-potential} stays close (in a suitable space) to a modulated solitary wave $Q_{c(t,y)}\big(x-c_0t-\gamma(t,y)\big),$ where $c(t,y), \gamma(t,y)$ are the modulated wave speed and translation parameter respectively. The derivation of the modulation parameters is achieved by 
imposing suitable orthogonality conditions 
 that will allow us to use the linear asymptotic stability obtained in Theorem \ref{thm-resol-L}. 
 
In order to prove our  main result, we can use a bootstrap argument to control the perturbation of the modulated solitary wave. The rough idea is that we have  to put simultaneously into the bootstrap the boundedness of a high order Sobolev norm, the decay at the same order of regularity of  an exponentially weighted
Sobolev norm and the decay of the Lipshitz norm together with the decay of the modulation functions.
In this general strategy, we make a crucial  use of  the fact that due to the decay of the solitary wave,  the terms in the perturbation equation which contains
the solitary waves will enjoy time  decay since they are in the weighted space, in particular, this will allow us to get dispersive estimates by using only the constant
coefficients part of the linearized equation and thus we do not need to develop  a distorded Fourier transform theory or a related approach 
 as it was done recently for example  in \cite{Germain-Pusateri,Luhrmann-Schlag,Luhrmann-Schlag2,Delort-Masmoudi,Chen-Luhrmann,Collot-Germain},   for
the non-self-adjoint, nonlocal operator arising from the linearization about the solitary wave.

The main difficulties  in the proof can be summarized as follows:
\begin{itemize}
\item {\bf Slow decay of the modulation functions.} Here the modulation functions solve a system of partial differential equations for which the linear 
part has the low frequency behavior of the heat type equation  $\partial_{t} -\Delta_{y}$ in the transverse direction.
 The decay is thus imposed by this type of equation. Note that the situation is very different in dimension one where
 the modulation coefficients solve ordinary differential equations, their decay is then imposed by the source terms, in particular
 they may enjoy fast decay when dealing with very localized solutions, see for example \cite{Germain-Pusateri-Rousset}.
 \item {\bf Weighted estimates.} In Theorem \ref{thm-resol-L}, we have stated  a low order  estimate  in the energy space for the semigroup
 of the linearized equation. For semilinear equations,  as in  \cite{Mizumachi-BL-nonlinear} for example, the  estimate in the weighted space for the nonlinear equations  follows from Duhamel formula by considering  the nonlinear terms as a  source
 term. This approach is also successful for  equations which contain derivatives in the nonlinearity 
 when their linear part  enjoys a   strong local smoothing property 
 like KP-II \cite{Mizumachi-KP-nonlinear} as it was originally used for KdV by Pego and Weinstein \cite{Pego-Weinstein-kdv}. Nevertheless, here,  there is no local smoothing for the linearized equation even in the weighted space 
 since the group velocity is linear at infinity, we thus have to combine the semigroup estimate with a high order energy estimate in order
 to avoid a loss of derivative. The main property of the system that allows us  to conclude  is the following one.  For small amplitude waves
  of amplitude $\epsilon$  given 
 by  Theorem \ref{thm-existence}, 
 the rate of decay of the semigroup that we obtain in \eqref{thm-resol-L} is roughly  $e^{- \epsilon^3 t}$, whereas the amplitude
 of the Lipschitz norm of the solitary wave is $\epsilon^2$ and the real parts of the eigenvalues of the constant coefficient part of 
 the operator $e^{a \cdot}L_{c} e^{-a \cdot}$ for frequencies bigger than one
   are of order $\epsilon$.  The damping for  high frequencies  is thus much stronger than the semigroup decay and bigger than the Lipschitz norm
   of the solitary wave.
   We can thus combine a low order  semigroup estimate  on the Duhamel formula  with an energy estimate localized to high frequencies in order to close the argument. This combination of arguments  is in spirit  close to what is called  high frequency damping  in the study of the stability  of travelling waves
   in partially dissipative perturbations of hyperbolic systems \cite{Mascia-Zumbrum,Miguel-zumbrum,Miguel-Faye}.
   \item {\bf Correction to the modulated solitary wave}. In our framework of transverse stability we have to be careful  about the localization 
    properties of  the perturbation
   of the modulated solitary wave for the velocity potential $\psi^i$. The anzatz
   $$ \psi^i(t, \mathrm{x}) = \psi_{c(t,y)}(x- c_{0}t- \gamma(t,y)) + \psi(t,\mathrm{x})$$ is not adapted to perform energy estimates in the usual Sobolev spaces.
  Indeed, in the equation for $\psi$, this generates a source term $\partial_{t} c \partial_{c} \psi_{c}$ which is not in $\dot{H}^1$ since $\partial_{c} \psi_{c}$ is not localized.
  We thus need to correct further   the modulated wave before studying the perturbation. The correction made in \cite{Mizumachi-BL-nonlinear} which is 
  mostly one-dimensional gives rise in our framework to a source term which is in $\dot{H}^1$ but which has too slow decay  in $L^2$ based spaces, 
  like $(1 + t)^{-{1 \over 2}}$
  due to the slow decay of the modulated coefficients. We use here a different three-dimensional correction which gives a source term
  with critical decay   $(1 + t)^{-1}$ in $L^2$   in the equation for the perturbation at the velocity level (i.e at the level of $\nabla \psi^{i}$)
  while keeping the property that the final velocity perturbation that we have to study is curl free.
  \item{\bf Energy estimates.} We can then study  an appropriate perturbation of the modulated solitary wave plus corrector written
   as $( \rho, v)$ by using the velocity equation. Since the source term has now a critical decay in $L^2$, we are only able
   to get a logarithmic growth for the norm $\|(\rho, v) \|_{H^M}$. Nevertheless, we are able to obtain that
   $ \| \nabla_{x,y} (\rho, v) \|_{H^{M-1}}$ remains bounded. The main idea is 
   that the critical source term
   has  better decay when we apply the derivatives $\nabla_{y}$ and $\partial_{t} + c_{0} \partial_{x}$, we can thus perform
   energy estimates by using these vector fields and finally recover the $\partial_{x}$ derivatives by using the system itself.
    The crucial property  that allows to perform this last step is the fact that the speed $c_{0}$ of the solitary wave
    is strictly bigger than the sound speed as given by Theorem \ref{thm-existence}.
    \item{\bf Dispersive estimates.} The last part is to use dispersive estimates to obtain integrable decay of roughly
     the Lipschitz norm $\| \nabla (\rho, v) \|_{L^\infty}$. As mentioned before, by using decay in the weighted space, we can consider
     the terms involving  the solitary wave as source terms.  There are here two main difficulties.  The first one is again the slow decay
     of the source term involving only the modulated solitary wave:  as for the energy estimates, we shall get first decay for 
     $\nabla_{y}(\rho, v)$ and  $(\partial_{t} + c_{0} \partial_{x})(\rho, v)$ and then use the equation to recover the decay for
      $\partial_{x}(\rho, v)$.  The second difficulty is related to the presence of quadratic nonlinearities.  In dimension $3$, they cannot
      be handled by using only the Duhamel formula and dispersive estimates. This difficulty already shows up in the study of
      the stability of the constant states in \cite{G-P-global}, a normal form transformation 
      (which is done through time integration by parts in the Duhamel formula for the profiles due to the absence of time resonances)  can be performed to transform
     the   quadratic terms into cubic ones. Nevertheless this normal form is singular in two ways, there are singularities in the low frequencies
     and the resulting product law loses in the $L^p$ scale compared to the usual Holder inequality.   In \cite{G-P-global}, this can be handled
      by controlling the $\dot{H}^{-1}$ norm of the solution. Here, due to the slow decay of the source term involving the modulated solitary wave,
       we cannot control the 
      boundedness of the  $\dot{H}^{-1}$ norm. There is  some flexibility in the argument of \cite{G-P-global}, 
      nevertheless, here,   the most  negative Sobolev norm which can be proven to be bounded  is at the scale of   $\dot{H}^{-{1 \over 2}}$
       which is too far.
      We thus need to push  further the analysis of the normal form.
       In particular,  we shall establish a better  behavior for the new unknown which takes into account  the quadratic time boundary term. 
       This allows us to perform a two-tier estimate in order to close the argument.
\end{itemize}

We now provide more concrete details for the above steps.

$\bullet$ {\bf Norms to propagate  regularity and decay.}
To propagate regularity and decay, we work with a quantity that combines several key estimates: the decay estimates for the modulation parameters, defined in \eqref{def-decays-modu} and denoted by
$\cN_{\tilde{c},\gamma}(T)$; the estimates  in the weighted Sobolev space for the perturbation  $U=(n^i,\psi^i)-(n_c,\psi_c)(\cdot-\gamma),$  defined in  \eqref{def-weighted-U} and denoted by $\cN_{U}(T);$ and various  energy and decay estimates for the corrected remainder $(\vr, v)$ which will be defined shortly, given in \eqref{def-norms-vr-v} and denoted by $\cN_{(\vr, v)}(T).$  
We denote this combined quantity by
\begin{align}\label{def-cNT}
    \cN(T):=\cN_{\tilde{c},\gamma}(T)+\cN_{U}(T)+ \cN_{(\vr, v)}(T)\,.
\end{align}

$\bullet$ {\bf Decay of modulation parameters.} 
The modulation functions $(c, \gamma)(t,y)$ are 
found  by
imposing the orthogonal condition \eqref{secularcondtion}. As in \cite{Mizumachi-BL-nonlinear}, it is found that $(\tilde{c}=c-c_0, \gamma)$ solves a dissipative wave system \eqref{modueq-1} and enjoy  similar decay properties as the heat equation with quadratic nonlinearities: 
\beqs\label{modueq-intro}
\begin{aligned}
\bigg(\pt-\tilde{A}(c_0, D_y)\bigg)\left(\begin{array}{c}
 |\na_y|\gamma \\[3pt]
\tilde{c}
\end{array}\right)= \cQ(\na_y \tilde{c}, \na_y \gamma)+ \cR(U) 
\end{aligned}
\eeqs
where $\cQ$ are quadratic nonlinear terms,  $\cR(U) $ depends only on the perturbation and 
the linear semigroup associated to $\tilde{A}(c_0, D_y)$ reads
\begin{align*}
      e^{t \tilde{A}_0(c_0, \zeta)}=e^{-\lambda_{2,c_0}|\zeta|^2 t} B(c_0,\zeta, t)
      \end{align*}
      here $B(c_0,\zeta, t)$  is an  oscillatory factor, which is uniformly bounded in $\zeta,t.$ We thus expect that the time decay of the modulation parameters $(\tilde{c},\gamma)$ is imposed by the one of the  heat equation with quadratic nonlinearity. Specifically, we shall get that 
\begin{align*}
    \|y^{\ell}\p_y^k (\tilde{c}, \na_y \gamma)(t) \|_{L_y^2}\lesssim (1+t)^{-\f{k-\ell+1}{2}}, \qquad  \ell=0,1; k-\ell=-1,0,1,2.
\end{align*}
We refer to \eqref{def-decays-modu} for more detailed estimates.

$\bullet$ {\bf Weighted estimate.} 
Let us set  $ z=x-c_0t.$ The remainder defined by 
\beqs
 U(t, z, y)=
 (n^i, \psi^i)^t (t, \mathrm{x})- Q_{c(t,y)}\big(z-\gamma(t,y)\big)
 \eeqs
solves  the system 
 \beq\label{eq-U-intro}
\begin{aligned}
  \pt U-L_{c_0} U = (L_{c,\gamma}-L_{c_0}-\tilde{c}\p_z)U+ R  + N( U, \tilde{c}, c_y, \gamma_y), 
\end{aligned}
\eeq
where $L_{c,\gamma}, R, N$ are the linear term, source term and  nonlinear term respectively and are 
defined in \eqref{def-Lcgamma}-\eqref{def-NLterm}.
We shall  explain here  how to get the weighted estimates for $(n, u)$ (or equivalently $U$ in $X_a^M=\langle \na\rangle^{-M}X_a$). On the one hand, when focusing on  lower order weighted norms (say $\|U\|_{X_a^{k}}, k\leq M-1$), we can make use of the semigroup estimate \eqref{semigroup}:
\begin{align*}
   \| \mathbb{Q}_{c_0} U(t)\|_{X_a^{k}}\lesssim e^{-c_{*}t} \, \| \mathbb{Q}_{c_0} U_0\|_{X_a^{k}}+\int_0^t e^{-c_{*}(t-s)} 
  \big( \|(L_{c,\gamma}-L_{c_0}-\tilde{c}\p_z)U(s)\|_{X_a^{k}}+\| R(s)\|_{X_a^{k}}  +\|N(s)\|_{X_a^{k}}\big) \, \d s
\end{align*}
to find that it decays like $(1+t)^{-1}$ (since the source  term $R$ decays at best at the  rate $(1+t)^{-1}$ in the  weighted norm $X_a$). However, when focusing on the highest order weighted norm $\|U\|_{X_a^M},$ application of the semigroup estimates will inevitably lose derivatives due to the presence of derivatives in  the nonlinear terms, for example $u_{c}\cdot\na U.$  
Our strategy to resolve the problem is, on the one  hand, to obtain the usual $X_a^0$ estimate by applying the semigroup estimates as illustrated above but on the other hand, to perform  energy estimates in  the weighted space when focusing on  high frequencies. In the latter step, we detect the damping mechanism from the linear part of the system in the high frequency region, and control the term involving the background waves by the damping. 
More precisely, let $\tilde{\chi}_{_K}: \mR^3\ni (\xi,\zeta)\rightarrow \tilde{\chi}_{_K}(\xi,\zeta)\in  \mR$ be a cut-off function that vanishes on $B_K(0)$ and equals to $1$ on $B_{2K}^c(0).$
We shall prove  by designing appropriate energy functionals that 
\beq \label{energy-weighted-intro}
\begin{aligned}
\pt \|\tilde{\chi}_{_K}(D)U\|_{X_a^M}^2 
+ a \kappa_0  \|\tilde{\chi}_{_K}(D) U\|_{X_a^M}^2
&\lesssim a^{-1}\big(\|U\|_{X_a^0}^2
+\|R\|_{X_a^M}^2\big)\\
&\,+\big(\|
Q_c(z-\gamma)\|_{W^{M+1,\infty}}+\|\na (n, u)\|_{W^{1,\infty}}\big)\|\tilde{\chi}_{_K}(D) U\|_{X_a^M}^2.
\end{aligned}
\eeq
Since $\|
Q_c(z-\gamma)\|_{W^{M+1,\infty}}\lesssim \ep^2$ and $a=\f{\ep}{4},$ the term involving 
the background waves in the right hand side of 
\eqref{energy-weighted-intro} can thus be absorbed by the damping in the left hand side by choosing $\ep$ small enough.
Based on this, we can can conclude from the Gr\"onwall inequality that
$$ \|U(t)\|_{X_a}^2 \lesssim e^{-a\kappa_0t}\|U(0)\|_{X_a}^2 +\int_0^t e^{-a\kappa_0(t-s)}
\big(\|U(s)\|_{X_a^0}^2
+\|R(s)\|_{X_a^M}^2\big)\, \d s\lesssim (1+t)^{-2}.$$

 $\bullet$ {\bf Correction to the modulated solitary wave}. 
 The source term $R$ in the equation \eqref{eq-U-intro} for $U$ includes a term  $\p_t c \, \p_c Q_c,$ which is not bounded in the usual Sobolev spaces (indeed,  $\psi_c(\cdot)$ does not vanish at $-\infty$ and thus $\partial_{c} \psi_{c}$ is not localized), this prevents $U$ from  being bounded in the usual Sobolev spaces. 
To resolve the problem, a natural ideal is to modify the velocity potential by $\psi_c+\check{\psi}_c,$ where $\check{\psi_c}$ is an auxiliary function such that 
$\lim_{x\rightarrow -\infty}\check{\psi}=-(\psi_c-\psi_{c_0})(-\infty),$ and thus $\pt c\, \p_c (\psi_c+\check{\psi}_c)$ is bounded in $\dot{H}^1(\mR^3).$
This is essentially what is used  in \cite{Mizumachi-BL-nonlinear}  in the study of 
the semilinear  Benney-Luke equation. 
However, in our current study, such a correction is not enough. 
Indeed, the introduction of $\check{\psi}_c$ gives rise to a source term like $(c-c_0) \check{\vp}(x)$ where $\check{\vp}(x)\in L^2(\mR).$ However, this term decays at the slow rate  $(1+t)^{-\f12}$ in $L^2(\mR^3),$ which turns out to be too slow to close the Sobolev energy estimates.
Instead, we shall  make a different choice and  construct a suitable correction term  that produces a source term with a  critical decay. We 
use the decomposition
\beq\label{def-tvp-intro}
\begin{aligned}
\left(\begin{array}{c}
          n^i \\
          \psi^i
 \end{array}\right)(t,x, y)= \left(\begin{array}{c}
          n_{c(t,y)} \big(z-\gamma(t,y)\big)  \\[3pt]
          \underline{\psi_c}(t, z, y)
 \end{array}\right)+\left(\begin{array}{c}
          n\\
          \tilde{\vp}
 \end{array}\right)(t,z,y), \quad (z=x-c_0t),
\end{aligned}
\eeq
where $ \underline{\psi_c}(t, z, y):= \psi_{c_0}\big(z-\gamma(t,y)\big)+\widetilde{\psi}(t,z,y),$ $\widetilde{\psi}$ being defined in the following way.
Consider a divergence-free vector function 
\beq \label{def-w-intro}
w=-\curl (-\Delta)^{-1}\curl \bigg( \psi_{c}'(\cdot-\gamma)\left(\begin{array}{c}
          1 \\
    -\na_y\gamma
 \end{array}\right)\bigg)
\eeq 
which depends only on the solitary  wave through  $\psi_c'$ and the modulation parameters $\gamma, c.$ Assuming the 
existence of $(\gamma, c-c_0)$ in $C([0,T], H^2(\mR^2)),$
we can verify that $ (\psi_{c}'-\psi_{c_0}')(\cdot-\gamma)(1,-\na_y\gamma)^t+w$ is curl-free and  belongs to $C([0,T],L^2(\mR^3)).$ 
We then define $\widetilde{\psi}\in C([0,T],\dot{H}^1(\mR^3))$ such that  
\begin{align*}
\na \widetilde{\psi}(t, z,y) = (\psi_c'-\psi_{c_0}')(z-\gamma) \left(\begin{array}{c}
          1 \\
    -\na_y\gamma
 \end{array}\right)+w(t,z,y).
\end{align*}
Let us remark that the auxiliary function $w$ is introduced to replace the contribution of $(0,-c_y\,\p_c\psi_c)^t$ when taking the gradient of the velocity potential $\psi_c$ while keeping the same curl, so that 
$\psi_c'(z-\gamma) 
        (  1 ,
    -\na_y\gamma)^t+w(t,z,y)$ is curl-free. 
By using this decomposition, 
the source terms in the equation of $(n, \na\tilde{\vp})^t$ are bounded in the usual Sobolev spaces and have a critical  time decay,  see Corollary \ref{cor-sourceterm}.

 $\bullet$ {\bf Energy estimates for the corrected perturbation}. 
The corrected perturbation $(n, \tilde{v}=\nabla\tilde{\varphi})^t$ solves the system
\beq
\label{introsys}
\left\{ 
\begin{array}{l}
  \pt n-c_0\p_z n+ {\bf \div\,} \big(
  {\bf n\, \underline{v_c}(z) }
  \big) = F_1 ,    \\[3pt]
 \pt \tilde{v}-c_0\p_z \tilde{v}+
  {\bf \na (\tilde{v}\cdot \na \underline{v_c}(z)) +\pt w}
 =F_2,\\[3pt]
 \Delta \phi=e^{\phi}-\rho,
\end{array}
\right.
\eeq
where $
\underline{v}_c=  \na \underline{\psi_c}(t, z, y)=\psi_c'(z-\gamma) ( 1 , \,  -\na_y\gamma)^t+w(t,z,y),$ $w$ is defined in \eqref{def-w-intro} and  $F_1,F_2$ are the terms that are irrelevant to the current discussion. 
 The terms in bold are the ones we shall first  discuss here, in particular, 
they contain the terms which are linear in the perturbation.
As mentioned already, we have to be careful with these linear terms that contain the background solitary wave since they
can yield exponential growth when performing energy estimate. The initial plan was to use that they are in the weighted space due to 
the multiplication by the solitary wave which is localized  and to use the decay of the perturbation in the weighted space. 
Nevertheless, we had to include the corrector $w$ and study a different perturbation.
 Due to the elliptic equation that it solves, $w$ is not in the weighted space and therefore $\tilde v$ is also not in the weighted space, 
 this is   due to the source term $\partial_{t}w$ in the system \eqref{introsys}.
Nevertheless, we  can observe from the definition that 
\beq\label{identiy-import-intro}
\tilde{v}+w=u+(0,\,\p_c\psi_c(z_1)\na_y c)^t, \quad \text{ with } u= \na \big(\psi^i(t,z+c_0t,y)-\psi_c(z_1)\big).
\eeq
The energy estimate can thus be performed by using the favorable time decay of $w$ in the usual  Sobolev spaces (See Proposition \ref{prop-w}) and the critical time decay of $U$ in the weighted space. This will be enough to prove a bound for $(n,\tilde{v})(t,\cdot)$ in $L^2(\mR^3)$ with a logarithmic growth. Next, since  $(\p_t,\na_y)U$ has  integrable time decay in the weighted  space, we can further show that 
 $(\p_t,\na_y)(n,\tilde{v})$ is uniformly bounded in the usual Sobolev space  $H^{M-1}.$ Finally, we can recover a  similar estimate for $\p_z (n,\tilde{v})$
by using the equations  
\beqs
-c_0 \p_z n+\p_z \tilde{v}_1=H_1, \quad -c_0 \p_z \tilde{v}_1+\p_z \big(h'(1)+(\Id-\p_z^2)^{-1}\big)n =H_2,
\eeqs
where $H_1=\pt n+\div_y \tilde{v}_y+\cdots, \, H_2=\pt \tilde{v}+\cdots.$ We prove in Lemma \ref{lem-recoverpx} that as long as $c_0>\sqrt{h'(1)+1},$ which is assumed, since it is  necessary in Theorem \ref{thm-existence} to ensure the existence of the solitary waves, we have 
\beqs 
\|\p_z (n, \tilde{v}_1)\|_{L^p}\lesssim \|(H_1, H_2)\|_{L^p}.
\eeqs



$\bullet$ \textbf{Decay of Lipschitz norms.}
To prove the time decay estimates for $(n, \tilde{v}),$ we return to the original reference frame. Let $(\vr, v)(t,x,y)=\colon (n, \tilde{v})(t, z, y),$ and consider the symmetrized form of the equation solved  by $(\vr, v)(t,x,y).$ 
The problem is thus reduced to studying the following equation for  $\alpha=\vr+i \f{\div}{P(D)}v:$
$$\pt \alpha- i P(D)\alpha=\mathfrak{N}_1+\mathfrak{N}_2+\mathfrak{N}_3+\mathfrak{N}_4, $$
where $P(D)=\sqrt{\Delta \big(h'(1)+(1-\Delta)^{-1}\big)}, $ and  $\mathfrak{N}_j=\mathfrak{F}_j^0+i \f{\div}{P(D)}\mathfrak{F}_j'.$ We refer to \eqref{def-fF} for the definition of $\mathfrak{F}_j=(\mathfrak{F}_j^0, \mathfrak{F}_j').$  
 The source term $\mathfrak{N}_1$ is the one involving only the modulation parameters and the background wave:
 \begin{align*}
     \mathfrak{N}_1=\pt c\, (\p_c n_c )(x-c_0t)+(\pt \gamma-\tilde{c}) n_c'(x-c_0t)+\cdots
 \end{align*}
The $\mathfrak{N}_1$ term  is the main cause of numerous additional difficulties.
 Indeed, it  follows from  Proposition \ref{prop-modulation}   that $\mathfrak{N}_1$ is decaying like  $(1+t)^{-(\f{1}{2}+\f1p)}$ in $L^{p'}.$ We may apply the dispersive estimate \eqref{disper-general} to obtain that for any $p\in(8,+\infty)$
\begin{align*}
    \big\|\int_0^t e^{i(t-s)P(D)}\mathfrak{N}_1(s)\, \d s\big\|_{L^p}\lesssim \int_0^t (1+t-s)^{-\f43(1-\f2p)} (1+s)^{-(\f{1}{2}+\f1p)}\, \d s\lesssim (1+t)^{-(\f{1}{2}+\f1p)}
\end{align*}
which is far from what we need  to close the Sobolev  estimates. 
As we used  in the energy estimates, we can notice  again that the source term $\mathfrak{N}_1$ behaves better once we apply one $Z$ derivative, where $Z\in \{\na_y, \pt+c_0\p_x\},$ in the sense that $\|Z\mathfrak{N}_1\|_{L^{p'}}\lesssim (1+t)^{-\big(1+\f1p\big)}.$ This motivates us to first establish the desired decay estimates for $Z(\vr, v)$ and 
then recover the corresponding estimate for $\p_z(\vr, v_1)$ or $\p_z(n, \tilde{v}_1),$ as we explained above.

 To obtain favorable time decay estimates for $Z\alpha,$ the next step is to apply the normal form transformation
 used in \cite{G-P-global}
  to switch the quadratic terms like $\mathfrak{N}_3$ into cubic ones. After this transformation, we find that 
\begin{align*}
\int_0^t e^{i(t-s) P(D)} Z \mathfrak{N}_3\, \d s&\approx
\sum_{\mu,\nu\in\{\pm\}} \bigg(e^{itP(D)}B^{\mu\nu}\big(Z\alpha^{\mu}(0),\alpha^{\nu}(0)\big)- B^{\mu\nu}\big(Z\alpha^{\mu}(t),\alpha^{\nu}(t)\big)+\text{symmetric terms}\\
&\qquad + \sum_{j=1}^4 \int_0^t e^{i(t-s)P(D)} \big(B^{\mu\nu}\big(Z\mathfrak{N}_j^{\mu}(s), \alpha^{\nu}(s)\big)+ B^{\mu\nu}\big(Z\alpha^{\mu}(s), \mathfrak{N}_j^{\nu}(s)\big)\big)\,\d s\bigg)
\end{align*}
where $B^{\mu\nu}(f,g)=\cF^{-1}\bigg(\int_{\mR^3}\f{|\varsigma|} {i\,\phi^{\mu\nu}(\varsigma, \varsigma')} \,f(\varsigma-\varsigma') \,g (\varsigma')\, \d \varsigma'\bigg), \, \text{ with } \phi^{\mu\nu}(\varsigma, \varsigma')=P(\varsigma)-\mu P(\varsigma-\varsigma')-\nu P(\varsigma'),$ 
and $\mu,\nu\in\{\pm\}, \alpha^{+}=\alpha, \alpha^{-}=\bar{\alpha}.$
The phase $\phi^{\mu\nu}$ has some degeneracies in low frequencies, leading to the unfavorable bilinear estimates of $B^{\mu\nu}(f,g)$ when both $f$ and $g$  have  low frequencies. For instance, by \eqref{bilinear-worst}, it holds that for any $r\in [2,3],$
\begin{align}\label{BL-LL-intro}
    \|B^{\mu\nu}(P_{\leq -5}f, P_{\leq -5}g)\|_{L^p}\lesssim \big\|P_{\leq -5}\f{f}{|\na|}\big\|_{L^q} \|P_{\leq -5}\f{g}{|\nabla|}\|_{L^r}, \quad 
     \big(\f{1}{q}+\f{1}{r}=
     \f{1}{3}+\f{1}{2p}+\ep, \, \, 0< \ep\ll 1 \big).
\end{align}
 Note that not only we have some singularities in low frequencies, but also we lose some integrability.
These give rise to technical difficulties, especially when controlling the boundary term 
 $B^{\mu\nu}\big(Z\alpha^{\mu}(t), \alpha^{\nu}(t)\big).$ 
 In view of \eqref{BL-LL-intro}
 we need some control of  of $\alpha$ in $\dot{W}^{-1,r}$  which turns out to be unbounded due to the source term $\mathfrak{N}_1.$ Indeed, for $k\ll j\leq 0,$ 
 we may expect to control the boundary term in the following way: 
 \beq\label{problem-intro}
 \begin{aligned}
  \|  B^{\mu\nu}\big(P_j Z\alpha^{\nu}(t), P_k\alpha^{\mu}(t))\|_{L^{p}}&\lesssim \||\na|  B^{\mu\nu}\big(P_j Z\alpha^{\mu}(t), P_k\alpha^{\nu}(t))\|_{L^{\f{1}{1/3+1/p}}}\\
 &\lesssim \|P_j Z\alpha^{\mu}(t)\|_{L^{p}}\|P_k \alpha^{\nu}(t)\|_{\dot{W}^{-1,\f{1}{\ep+1/2-1/p}}}.
 \end{aligned}
 \eeq
On the one hand, by the dispersive estimate \eqref{disper-general}, to ensure the integrable time decay of $Z\alpha$ in $\dot{W}^{1,p},$ we need to  take $p>8.$ On the other hand, we can expect at best that  the $\dot{W}^{-1,r}$ norm of $\alpha(t)$ behaves like $(1+t)^{-(\f34-\f2r)}$ (see Proposition \ref{prop-decayes}). These two facts together prevent us from closing the decay estimates. 
Our strategy to resolve the problem is to do two-tier estimates. We write 
$$\alpha=\tilde{\alpha}+\sum_{\mu,\nu\in\{\pm\}}B_{LL}^{\mu\nu}(Z\alpha^{\mu},\alpha^{\nu}), \qquad \text{ with }\quad B_{LL}^{\mu\nu}  (f, g):= \sum_{ k+5\leq j\leq -5} B^{\mu\nu}(P_j f, P_{k} \,g).$$
 We prove, on the one hand that  
$
    \|\tilde{\alpha}\|_{W^{1+(\f{3}{p})^{-},p}}\lesssim (1+t)^{-\f{4}{3}(1-\f2p)}\, ,
$
and on the other hand, by using  the  estimate for $\tilde{\alpha},$ that
$ \||\na|^{(\f{3}{p})^{-}} B_{LL}^{\mu\nu}\|_{L^p}\lesssim (1+t)^{-1^{+}}.$ In the process, we find that $p$ should be taken sufficiently close to $8.$
We refer to Section \ref{sec-decayes} for the detailed analysis. 


\subsubsection{Transverse linear asymptotic stability}
The proof of the linear asymptotic stability result in Theorem \ref{thm-resol-L} relies on
the uniform resolvent estimate \eqref{uni-resol}. 
The main difficulties  for the uniform resolvent estimates arise again  from the quasilinear structure of the system,   the part of the system involving coefficients
depending on the solitary wave cannot be seen as a lower order small perturbation
 as in the study of the linearization of semilinear problems, 
see for instance \cite{Mizumachi-BL-linear}. 
To address this issue, we  cut frequencies and use 
 pseudo-differential calculus
and energy estimates in the high-frequency regions. 

More precisely, we divide the frequency space $(\xi, \zeta_1, \zeta_2)\subset \mathbb{R}^3$ into three regions, corresponding to (uniform) high, intermediate, and low \textbf{transverse} frequencies:
\begin{align}
\label{defregions}
    R_{\zeta}^{UH}=\big\{ (\xi,\zeta_1, \zeta_2)\big|\, |\zeta|\geq 2 \big\}, \,\,
    R_{\zeta}^I=\big\{(\xi,\zeta_1, \zeta_2)\big|\, A\ep^2\leq |\zeta|\leq 2 \big\}, \,\,  R_{\zeta}^L=\big\{ (\xi,\zeta_1, \zeta_2)\big|\,|\zeta|\leq  A\ep^2 \big\}, 
\end{align}
where $A$ is a sufficiently large constant satisfying $A^2\ep \leq 1.$ 
To show the uniform resolvent estimates, we use  different arguments depending on the size of the transverse frequencies. In the (uniform) high transverse frequencies region $R_{\zeta}^{UH}$,  we use reductions based on pseudodifferential calculus. For intermediate one $R_{\zeta}^I$, we use 
 an energy-based approach relying on  the design  of various appropriate energy functionals for different regimes of  longitudinal 
 frequencies.  As for low transverse frequencies $R_{\zeta}^I$, we use the energy estimates  and  KP-II approximation in the high  and low longitudinal  frequencies.
This approach has recently been employed in our study of the transverse linear stability of solitary gravity water waves \cite{RS-WW}. However, the forms of the linearized operators in these two studies differ significantly, leading to substantial differences in the associated energy functionals.

\subsubsection{Organization of the paper}
In Section 2, we  introduce  some spectral preliminaries for the linearized operator in  \eqref{Linear-sys}, in particular, we  define the spectral projections used in the statement of the linear asymptotic stability result, Theorem \ref{thm-resol-L}.
We then present the proof of the nonlinear stability result stated in Theorem \ref{thm-nonlinear}, which is carried out in Sections 3–7.
Section 3 establishes  existence and decay estimates for  the modulation parameters $(\tilde{c},\gamma).$ Section 4 focuses on estimates in the  weighted space, while Section 5 provides the usual   Sobolev estimates.
In Section 6, we prove the  time decay estimates 
in $L^{8_{\kappa}}, (8_{\kappa}=\f{8}{1-\kappa}, \kappa>0$ small).
We then combine in Section 7 the estimates obtained in Section 3-6 to finish the proof of Theorem \ref{thm-nonlinear}.
Finally, in Section 8, we prove the linear asymptotic stability result stated in Theorem \ref{thm-resol-L}. While this result is used in Section 4,
the reader does not need  its proof  to go through the one of the  nonlinear stability result.
In Appendix A, we establish the existence of continuous resonant modes, as stated in Theorem \ref{thm-resolmodes}. The appendix begins with a formal derivation of the KP-II equation from the Euler–Poisson system in the long-wave regime. Appendices B–E present the statements and proofs of several auxiliary results that support the main arguments in  the proofs of Theorems \ref{thm-nonlinear}–\ref{thm-resol-L}.

\textbf{Notations:} For any $j\in \mZ,$ we use   denote $P_j$ the Littlewood-Paley multipliers which are defined by
\begin{align*}
    (P_j f)(\mathrm{x})=\cF_{\varsigma\rightarrow \mathrm{x}}\bigg(\varphi \big(\f{\varsigma}{2^j}\big) (\cF f)(\varsigma) \bigg)
\end{align*}
where $\varphi\in C_c^{\infty}(\mR^3)$ satisfies 
\begin{align*}
    \sum_{j\in\mZ} \varphi \big(\f{\varsigma}{2^j}\big)=1, \qquad \forall \, \varsigma\in \mR^3\backslash \{0\}\,.
\end{align*}
We denote also $P_{\leq k}=\sum_{j\leq k} P_j\,.$

\section{Spectral preliminaries}
To give the definition of the spectral projector appearing in the statement of the linear stability result in Theorem \ref{thm-resol-L}, we need to first construct the resonant space associated to the continuous eigenvalues near $0.$ 
To do so, it is convenient to introduce  weighted energy spaces in the one-dimensional variable $x$ only but involving
in a quantitative way a parameter $\eta \in \mathbb{R}$ which is related to the transverse frequency:
$$L_a^2(\mR)=\colon L^2(\mR; e^{2ax} \d x), $$
\begin{align}\label{Hhalfeta}
  \dot{H}_{a,\eta}^{1}(\mR)= \bigg\{ f  \,\big|\, \|f\|_{ \dot{H}_{a,\eta}^{1}(\mR)}:= \bigg(\int \big|(\p_x, \eta)  f(x)\big|^2  e^{2ax} \d x\bigg)^{\f{1}{2}} <+\infty\bigg\}
\end{align}
and set  
\begin{equation}
\label{defspaceY}
Y_{a,\eta}=L_a^2(\mR)\times \dot{H}_{a,\eta}^{1}(\mR).
\end{equation}
We then consider the following unbounded operator $L_{c}({\eta})$ on $Y_{a,\eta}$
 defined for any smooth function $U(x)$ by  
$ L_{c}({\eta})(U(x))=e^{-i(y_{1}\zeta_1+y_2\zeta_2)}L_c\big(e^{i(y_{1}\zeta_1+y_2\zeta_2)}U(x)\big) \mathbb{I}_{\{\zeta_1^2+\zeta_2^2=\eta^2\}} ,$ so 
that $L_c$ takes the form
\beq\label{def-Leta}
L_c({\eta})= \left( \begin{array}{cc}
   \p_x(d_c\cdot)  &  -\p_x (\rho_c \p_x)+\rho_c\,\eta^2 \\[5pt]
 -\big[ \, h'(\rho_c) +(e^{\phi_c}+\eta^2-\p_x^2)^{-1}\big]   & d_c\p_x
\end{array}\right),
\eeq
where we have denoted $(e^{\phi_c}+\eta^2-\p_x^2)^{-1}: L_a^2(\mR)\rightarrow {H}_{a,\eta}^{2}(\mR):= \langle (\p_x,\eta)\rangle^{-2}L_a^2(\mR) $ as the solution map of the elliptic problem
\beqs 
(e^{\phi_c}+\eta^2-\p_x^2)\phi =n\, .
\eeqs
An important motivation for working in this weighted framework is  that the essential spectrum of the operator $L({\eta})$ which is
the imaginary axis in unweighted spaces is pushed to the left so that the first step in order to establish
 linear stability is to investigate the presence of eigenvalues close to the  imaginary axis. 
 We can  observe that when $\eta=0,$ $L(0)$ is the linearized operator for 1D Euler-Poisson system about  the line soliton $(n_c,\, \psi_c)$ studied in \cite{Bae-Kwon-linearEP}.  By the translational and Galilean invariances, there is  an eigenvalue $\lambda = 0$ with algebraic multiplicity two. Consequently, for nonzero and small $\eta,$  we expect that the operator $L(\eta)$  
 will have  two small  eigenvalues $\lambda(\eta)$ and $\overline{\lambda(\eta)}$ in the space $Y_{a,\eta}$ bifurcating from $\lambda(0)=0.$ 

    \begin{thm}\label{thm-resolmodes}
Let $0<a=\hat{a}\ep < \f{\sqrt{3}}{4}\ep.$ 
There exist $\epsilon_0>0,
\hat{\eta}_0=\hat{\eta}_0(\hat{a})>0$
such that for any $\ep\in (0,\ep_0], \eta \in[-\ep^2\hat{\eta}_0, \ep^2 \hat{\eta}_0],$ 
the operators $L_c(\eta)$ 
has two pairs of eigenmodes 
$(\lambda_c(\pm \eta), \, U_c(\cdot, \pm \eta))$ in the space $Y_{a, \eta},$ where
\beqs
\lambda_c(\eta)\in C^{\infty}\big([-\ep^2\hat{\eta}_0 , \ep^2\hat{\eta}_0 ]\big), \quad U_c(\cdot, \eta)\in  C^{\infty}\big(
[-\ep^2\hat{\eta}_0 , \ep^2\hat{\eta}_0 ], Y_{a,\eta}\big)\, . 
\eeqs
Moreover, for any $\eta \in [-\ep^2\hat{\eta}_0 , \ep^2\hat{\eta}_0 ],$ it holds that
\begin{align}
  & \lambda_c(\eta)=i\lambda_{1,c}\eta-\lambda_{2,c}\,\eta^2+\cO(\eta^2)\,, \label{exp-spectralcurve} \\
  &   U_c(\cdot,\eta)=
U_c^1+i\lambda_{1,c}\eta\, U_c^2+\cO(\eta^2) \, \,\,\text{ in } Y_{a,\eta}\, , \notag\\
&\overline{\lambda_c(\eta)}=\lambda_c(-\eta), \quad \overline{U_c(\cdot, \eta)}=U_c(\cdot, -\eta), \notag
\end{align}
where $\lambda_{1,c}, \lambda_{2,c} $ are two positive constants 
and
\beq\label{def-firsttwo}
U_c^1=\left( \begin{array}{c}
  n_c'    \\
   \psi_c '
\end{array}
\right), \qquad     U_c^2=-\left( \begin{array}{c}
  \p_c n_c    \\
  -\int_{x}^{+\infty} \p_c \psi_c'  
\end{array}
\right). 
\eeq

In the same way, the adjoint operator $L_c^{*}(\eta)$ has  eigenmodes under  the form $(\lambda_c(\mp \eta), U_c^{*}(\cdot, \pm \eta))$ where 
$\overline{U_c^{*}(\cdot, \eta)}=U_c^{*}(\cdot, -\eta) $ and 
\beqs 
U_c^{*}=
U_c^{*,1}-i\lambda_{1,c}\eta\, U_c^{*,2}+\cO(\eta^2) \, \,\,\text{ in } Y_{-a,\eta}\, 
\eeqs
with
\beq\label{def-firsttwo-adj}
U_c^{*,1}=\left( \begin{array}{c}
  \psi_c'    \\
   - n_c'  
\end{array}
\right), \qquad     U_c^{*,2}=\left( \begin{array}{c}
   \int_{-\infty}^x \p_c \psi_c'    \\
    -\p_c n_c
\end{array}
\right). 
\eeq
\end{thm}
The proof of this theorem will be presented in Appendix \ref{sec-resolnantmode}.

As shown in  Appendix A, one can then choose two smooth in $\eta$ functions
$(g_1(\cdot, \eta), g_2 (\cdot, \eta))$  such that  for $\eta \neq 0$  they form  a basis of 
the two dimensional space generated  in the space $Y_{a, \eta}$ by $U(\pm \eta)$ and for $\eta=0$ a basis of the generalized kernel of $L_c(0).$  Let $g_1^{*}(\cdot, \eta), g_2^{*} (\cdot, \eta)\in Y_{a, \eta}^{*}$ be the corresponding dual basis  defined in \eqref{defbasis-dual}. We can then 
define the spectral projector 
\begin{align}
  &  \mathbb{P}_c(\eta_0)f=\sum_{k=1, \, 2}\int_{ |\zeta|\leq \eta_0} \big \langle \cF_y(f) (\cdot, \zeta), g_k^{*}\big(\cdot, |\zeta|\big)\big\rangle_{_{Y_{a,|\zeta|}\times Y_{a,|\zeta|}^{*} }}  \,g_k\big(\cdot, |\zeta|\big) \, e^{i y\cdot\zeta}\, \d \zeta_1\d \zeta_2 \label{def-projectionP}\\
   & \mathbb{Q}_c(\eta_0)=\Id-  \mathbb{P}_c(\eta_0)\label{def-bQ}
\end{align}
which project the elements in $X_a$ 
onto the spectral subspaces associated with the continuous family of eigenvalues
 $\{\lambda(\pm \eta)\}_{|\eta|\leq \eta_0}$ and onto its orthogonal complement, respectively.
We refer to \eqref{def-bracketYa} for the definition of $ \langle \cdot \rangle_{Y_{{a, \eta}}\times Y_{a,\eta}^{*}}.$


\section{Decomposition of the solution and derivation of the equations for the modulation parameters}

This section marks the beginning of the proof of transverse nonlinear stability as stated in Theorem \ref{thm-nonlinear}. Its primary goal is to establish a precise decomposition of the solution and to derive both the modulation equations and their decay estimates.

Let us set  $Q_c=(n_c, \psi_c)^t,$ we write the solution to the system \eqref{EPION3d}, 
 $\Psi(t, \mathrm{x})=(n^i, \psi^i)^t (t, \mathrm{x})$ in the following form
 \beq \label{decomp}
\Psi(t, \mathrm{x})= Q_{c(t,y)}(z-\gamma(t,y))+ U(t, z, y), \qquad ( z=x-c_0t \,)\, . 
 \eeq
To simplify the notation, for a function $f: \mR\rightarrow \mR,$ we denote $f_{\gamma}(\cdot)=f(\cdot-\gamma).$ 

Let us  also set  $\tilde{c}=c-c_0.$
 Substracting the system $\eqref{EPION3d}$ by the equations \eqref{profile eq}
  satisfied by $Q_{c},$ we obtain the following equations for $U(t,z,y):=(n, \psi)^t(t,z,y)$
\beq\label{eq-U}
\begin{aligned}
  \pt U-L_{c,\gamma} U  = -\tilde{c}\,\p_z U+ R  + N( U, \tilde{c}, c_y, \gamma_y), 
\end{aligned}
\eeq
where the linear operator $L_{c, \gamma}$ is defined as
\beq\label{def-Lcgamma}
L_{c,\gamma}=L_c(Q_{c,\gamma}, \nabla):= \left( \begin{array}{cc}
  \p_z (d_{c,\gamma}\cdot)   &\qquad (1+n_{c,\gamma})\Delta+\p_z n_{c,\gamma}  \p_z \\
  h'(1+n_{c,\gamma})+I_{\phi_{c,\gamma}} 
  & d_{c,\gamma}\,\p_z
\end{array}
\right),
\eeq
the source term can be written as $R=R_1+R_2,$ with $R_1, R_2$ defined by 
\beq \label{def-R}
\begin{aligned}
 & R_1=-\pt c \, \p_c Q_{c,\gamma}+( \pt \gamma -\tilde{c})  Q'_{c,\gamma}, \qquad \\
&  R_2=- \left( \begin{array}{c} 
\div_y \big((1+n_{c,\gamma})\na_y \psi_{c,\gamma}\big) \\[3pt]
|\na_y \psi_{c,\gamma} |^2/2+I_{\phi_{c,\gamma},0}
 \Delta_y\phi_{c,\gamma} +
 I_{\phi_{c,\gamma}} I_{\phi_{c,\gamma},0}\big(\phi_{c,\gamma}' \Delta_y^2\gamma+\p_c \phi_{c,\gamma} \Delta_y^2 c\big)\,
\end{array}
\right)\, ,
\end{aligned}
\eeq
where  we have set 
$I_{\phi_{c,\gamma}}=(e^{\phi_{c,\gamma}}-\Delta)^{-1},\,I_{\phi_{c,\gamma},0}=(e^{\phi_{c,\gamma}}-\p_{z}^2)^{-1}.$

Finally,  the nonlinear term takes  the form
\beq\label{def-NLterm}
N=
-\left( \begin{array}{c}
\na_y n_{c,\gamma}\cdot\na_y\psi+\div_y \big(n \na_y\psi_{c,\gamma} \big)+\div(n\na\psi)\\[4pt]
 \na_y \psi_{c,\gamma}\cdot\na_y \psi+\f12|\na\psi|^2+H(n_{c,\gamma}, n)+G(\phi, \phi_{c,\gamma})
\end{array}
\right)
\eeq
with 
\begin{align*}
 &H(n_{c,\gamma}, n)=  h(1+n_{c,\gamma}+n)-h(1+n_{c,\gamma})-h'(1+n_{c,\gamma})n, \\
&G(\phi, \phi_{c,\gamma})=
 \phi-\phi_{c,\gamma}-
I_{\phi_{c,\gamma}} n- I_{\phi_{c,\gamma},0} \Delta_y \phi_{c,\gamma}- I_{\phi_{c,\gamma}} I_{\phi_{c,\gamma},0}\big(\phi_{c,\gamma}' \Delta_y^2\gamma+\p_c \phi_{c,\gamma} \Delta_y^2 c\big). 
\end{align*}

In order to fix the decompostion \eqref{decomp}, we need to determine the modulation parameters $c(t,y), \gamma(t,y),$ this will  be achieved by imposing appropriate orthogonality conditions on the perturbation term $U.$ Motivated by the linear asymptotic stability in $\mathbb{Q}_c(L_a^2\times \dot{H}_a^1)$ stated in Theorem \ref{thm-resol-L}, one would like  to require  $U$ to be 
 orthogonal to the space generated by the continuous resonant modes (which is  the image of projection $\mathbb{P}_c^a$). It is thus natural  to impose the following orthogonality condition:  
\beqs 
\int_{\mR^3} U(t,z_1+\gamma, y)\, g_{k}^{*}\big(z_1, |\zeta|, c(t,y)\big) \,e^{-iy\cdot\zeta}\, \d z_1 \d y =0, \quad  \forall\, \zeta, \,  \, |\zeta|\leq \eta_0,\, \text{ and } \,k=1,2,\,
\eeqs
where $\eta_0=\ep^2\hat{\eta}_0$ and ($g_{1}^{*}, g_2^{*})$ is a smooth basis  of  the eigenspace of $L_c^{*}(|\zeta|)$ associated to the  eigenvalues $\lambda(\pm |\zeta|)$ in  the weighted space $X_{-a},$ we refer to \eqref{defbasis-dual} for their precise definitions. Nevertheless, it is better in order  to get good decay properties for the parameters $(\gamma, \tilde{c})$ later, to  extend $g_{k}^{*}(z_1,\eta, c)$ from $\{|\eta|\leq \eta_0\}$ to $\{|\eta|\leq 2\eta_0\}$ by defining 
\beqs 
g_{k}^{*}(z_1,\eta, c)=2 g_{k}^{*}(z_1,\eta_0, c)-g_{k}^{*}(z_1,2\eta_0-\eta, c), \qquad \forall \, \eta \in[\eta_0, 2\eta_0].
\eeqs
We still denote the extended function as $g_{k}^{*}(z_1,\eta, c)$. Note that with values in $X_{-a},$ it is $C^{\infty}$ with respect to $\eta$ on $[0, \eta_0)\cup (\eta_0, 2\eta_0]$ and $C^1$ on $[\f{\eta_0}{2}, \f{3\eta_0}{2}].$ 

We then impose the  orthogonality conditions in the following form 
\beq\label{secularcondtion} 
\chi_{\eta_0}(\zeta)\, \int_{\mR^3} U(t,z_1+\gamma, y)\, g_{k}^{*}\big(z_1, |\zeta|, c(t,y)\big) \,e^{-iy\cdot\zeta}\, \d z_1 \d y =0, \quad   \,k=1,2,
\eeq
where $\chi_{\eta_0}(\cdot): \mR^2\rightarrow \mR$ is a cut-off function that equals to $1$ on $B(0,\eta_0)$ and vanishes on $B^c(0, 2\eta_0).$  Here $U=\Psi(t, \mathrm{x})- Q_{c(t,y)}(z-\gamma(t,y))$ solves the equation \eqref{eq-U}. 

\subsection{Local existence of modulation parameters}
Let us define the space 
\beq \label{def-spaceY}
Y=\big\{f\in L^2(\mR^2)\,\big|\, \Supp \cF_{y\rightarrow \zeta}(f)\subset B(0,\eta_0) \big\}.
\eeq
In this section, we prove that there exists a positive time  $T_0$ and  
 parameters $(\gamma,\tilde{c})(t,y)\in C^1([0,T_0], Y\times Y)$ such that  the
 decomposition \eqref{decomp} and the orthogonality condition 
\eqref{secularcondtion} hold true. We begin with the local well-posedness of  \eqref{EP3d-potential} around the
solitary waves $Q_{c_0}(x-c_0 t).$ 
\begin{thm}[Local well-posedness around solitary waves]\label{thm-local}

 Let $\tilde{\Psi}_0
 ={\Psi}_0-Q_{c_0}\in X_a^{N_1}\cap X_0^N $ with $N\geq 3, N_1\geq 1.$ 
There exists $T_0>0,$ such that the system \eqref{EP3d-potential} has a  unique solution such that
 $$\tilde{\Psi}(t, z, y)=\colon \Psi (t, z+c_0 t, y)- Q_{c_0}(z)\in C([0,T_0], X_a^{N_1}\cap X_0^N)\cap  C^1([0,T_0], X_a^{N_1-1}\cap X_0^{N-1}) $$ and it holds that 
 \beq \label{a-prior}
\sup_{t\in [0,T_0] }\|\tilde{\Psi}(t)\|_{X_a^{N_1}\cap X_0^N}\leq 2\, \|\tilde{\Psi}_0\|_{X_a^{N_1}\cap X_0^N}.
 \eeq
\end{thm}
\begin{rmk}
    Recall that $X_0^N:=\{(n,\psi)|(n,\na \psi)\in H^N\}.$ In particular, we do not need the initial  perturbation $\psi$
to be in $L^2.$ 
\end{rmk}

\begin{proof}
    It is direct to verify that $\tilde{\Psi}$ solves the equation
    \beqs 
(\pt  -L_{c_0} )\tilde{\Psi} =\left( \begin{array}{c}
-\div(n\na\psi)\\[4pt]
 \f12|\na\vp|^2+H(n_{c_0}, n)+\phi-\phi_{c_0}-(e^{\phi_{c_0}}-\Delta)^{-1} n  
\end{array}
\right).
    \eeqs
The local well-posedness of the above system in the unweighted sobolev space $X_0^N (N\geq 3)$ follows from standard  regularization arguments and a-priori estimates. We can then perform the energy estimates in the weighted space $X_a^{N_1}$ (one first truncates the weight and then pass to the limit after the energy estimate)
to show that this solution belongs also to the weighted space $X_a^{N_1}.$
 Moreover, it holds that 
\beqs 
\pt \|\tilde{\Psi}(t)\|_{X_a^{N_1}\cap X_0^N}^2\lesssim  \big(\|(n_{c_0},\p_z\Psi_{c_0})\|_{W_z^{1,\infty}}+\|\tilde{\Psi}(t)\|_{X_0^3} \big)\|\tilde{\Psi}(t)\|_{X_a^{N_1}\cap X_0^N}^2\, ,
\eeqs
which, combined with the Gr\"onwall inequality, leads to \eqref{a-prior}, upon choosing $T_0$ small enough.

\end{proof}

    We are now in  position to prove  the existence of the modulation parameters $(\gamma, \tilde{c})$ that satisfy \eqref{decomp} and \eqref{secularcondtion}. As in \cite{Mizumachi-KP-nonlinear}, it follows from  the Implicit Function Theorem. Let us introduce for $k=1, \, 2$  the functionals 
\beq \label{def-Fk}
F_k [U, \tilde{c}, \gamma](\zeta)= 
\chi_{\eta_0}(\zeta)
\int_{\mR^3} \bigg(U(z, y)+ Q_{c_0}(z)-Q_{c_0+\tilde{c}(y)}\big(z-\gamma(y)\big)\bigg) \cdot g_{k}^{*}\big(z-\gamma(y), |\zeta|, c(y)\big)\, e^{-iy\cdot\zeta} \,\d z \d y .
\eeq
We have  that 
$F_k$ maps $X_a\times Y\times Y$ into the space $\tilde{Y}=\{f\in L^2(\mR^2)\, \big|\,\Supp f\subset B(0,2\eta_0)\}$  
and that 
\beqs 
\int_{|\zeta|\leq 2\eta_0} |F_k(\zeta)|^2 \, \d \zeta \lesssim \big\|U\big\|_{X_a}^2+\big\|(\tilde{c}, \gamma)\big\|_{Y}^2 .
\eeqs
Moreover, it holds that 
\beqs 
D_{\tilde{c},\gamma}(F_1, F_2)[0,0,0]\left(\begin{array}{c}
   \tilde{c}  \\
    \gamma
\end{array}\right)=\sqrt{2\pi} \left(\begin{array}{cc}
 f_{11}    & f_{12}  \\
  f_{21}   & f_{22}
\end{array} \right)
\left(\begin{array}{c}
     \cF_y \tilde{c}  \\
       \cF_y \gamma
\end{array}\right),
\eeqs
with $f_{k 1}=- \int_{\mR} \p_c Q_{c_0}(z)\cdot g_k^{*}(z,\eta,c_0)\,\d z, \quad f_{k2}= \int_{\mR} \p_z Q_{c_0}(z)\cdot g_k^{*}(z,\eta,c_0)\,\d z.$ 
By \eqref{exp-g12dual}, \eqref{perop-betajk}
\begin{align*}
   & f_{11}=\langle U^2_{c_0}, g_1^{*}\rangle=\f{\beta_{22}}{\beta_{12}}(c_0)+\cO(\eta_0^2),\,\quad f_{12}=\langle U^1_{c_0}, g_1^{*}\rangle=1+\cO(\eta_0^2), \\
  &  f_{21}=\langle U^2_{c_0}, g_2^{*}\rangle=C_{*}\beta_{12}(c_0)+\cO(\eta_0^2),\, \quad
    f_{22}=\langle U^1_{c_0}, g_2^{*}\rangle=\cO(\eta_0^2),\,
\end{align*}
where $\beta_{12}(c_0)$ is strictly negative with a uniform (in $\ep$) upper bound.
Therefore, for $\eta_0$ small enough, $D_{\tilde{c},\gamma}(F_1, F_2)[0,0,0]$ is invertible on $B(Y\times Y, \tilde{Y}\times \tilde{Y} )$
  and the Implicit Function Theorem gives
\begin{lem}\label{lem-IFT}
    Let $\eta_0$ be small enough. There exists $\delta_1>0$ 
    such that for any $\|U\|_{X_a}\leq \delta_1,$
there exists unique $(\tilde{c},\gamma)\in Y\times Y$ such that
\beqs 
F_1[U, \tilde{c},\gamma ]=F_2[U, \tilde{c},\gamma ]=0 \qquad  \text{ in }\, \tilde{Y}.
\eeqs
   Moreover, the mapping $\{U\in X_a\big| \|U\|_{X_a}\leq \delta_1\} \ni U \rightarrow (\tilde{c},\gamma)\in Y\times Y$ is $C^1.$
\end{lem}

The existence of $(\tilde{c},\gamma)$ satisfying \eqref{decomp}, \eqref{secularcondtion} is then the consequence of Theorem \ref{thm-local} and lemma \ref{lem-IFT}. Indeed, assume that 
$\|\tilde{\Psi}_0\|_{X_a^{N_1}\cap X_0^N}\leq \delta_1/2,$ then it follows from  Theorem \ref{thm-local}  that
$\tilde{\Psi}(t,z, y)=\Psi(t, z+c_0t, y)-Q_{c_0}(z)\in C^1([0,T_0], X_a)$ with $\sup_{t\in [0,T_0] }\|\tilde{\Psi}(t)\|_{X_a}\leq \delta_1.$ Applying Lemma \ref{lem-IFT} with $U=\tilde{\Psi}(t,z, y),$
we find 
$(\tilde{c}(t,y), \gamma(t,y))\in C^1([0,T_0], Y\times Y)$ such that \eqref{secularcondtion} holds true. 
\begin{rmk}\label{rmk-intialdata-cgamma}
To get better decay properties for the modulation parameters $(\tilde{c}(t,y), \gamma(t,y))$ later,  it necessitates slightly more localization properties on their initial datum $(\tilde{c}(0,y), \gamma(0,y))$ which are in turn determined by the initial conditions  $\tilde{\Psi}_0.$ Denote $F_k^0$ the functional similar to \eqref{def-Fk} with $U$ changed to $\tilde{\Psi}_0.$
Suppose that $y \tilde{\Psi}_0\in X_a.$ Since $F_k^0$ defined in \eqref{def-Fk} maps $\{\tilde{\Psi}_0 | \langle y \rangle \tilde{\Psi}_0\in{X_a}\}\times \{(\tilde{c},\gamma)| \langle y \rangle (\tilde{c},\gamma) \in Y\times Y  \} $
 to $
 \{f\in H^1(\mR^2)\, \big|\,\Supp f\subset B(0,2\eta_0)\},$ we can use the Implicit Function Theorem again to find that upon choosing $\delta_1$ smaller,  
 if $\|\langle y \rangle \tilde{\Psi}_0\|_{X_a}\leq \delta_1, $ then there exists $\tilde{c}(0,y),\gamma(0,y)$
 such that
 $\langle y \rangle (\tilde{c},\gamma)(0,y) \in Y\times Y.$ Similarly, if it holds that 
   $ \|e^{ax}(n_0,\na\psi_0)\|_{L_y^1L_z^2}\leq \delta_1$ (recall that $\tilde{\Psi}_0(\cdot)=(n_0,\psi_0)^t))$
 for $\delta_1$ small enough, the found $(\tilde{c},\gamma)(0,y)$ satisfy that $\cF_y((\tilde{c},\gamma)(0,y))\in L^{\infty}(B(0,\eta_0)).$   
\end{rmk}

\subsection{Derivation of a PDE system for the evolution of the  modulation parameters}
In this subsection, we shall  derive a system that describes the evolution of the parameters $(\tilde{c},\gamma).$
Let us set  $\tilde{U}(t,z_1,y)=U(t,z_1+\gamma, y).$ It then follows from the equation \eqref{eq-U} that
\beq
\pt \tilde{U} -L_c(Q_c, \na) \tilde{U} =\big(\pt\gamma-\tilde{c}\big)\,\p_{z_1} \tilde{U}+ L_{c,1}(Q_c,\na)\tilde{U}+\tilde{R}+\tilde{N}
\eeq
where $\tilde{R}(t, z_1, y)=R(t, z_1+\gamma, y),\, \tilde{N}(t, z_1, y)=N(t, z_1+\gamma, y) $ and 
\beqs 
L_{c,1}(Q_c,\na)=\colon L_c\big(Q_c, \p_{z_1}, \na_y-\na_y \gamma\,\, \p_z\big)-L_c(Q_c, \na) \, .
\eeqs
Assume that  $(\gamma,\tilde{c})\in C^1([0,T], Y \times Y)$ is such that \eqref{secularcondtion} holds on the time interval $[0,T].$
Taking the time derivative of \eqref{secularcondtion}, we find the following identity that holds in $C\big([0,T_0], B(0, 2\eta_0)\big):$
\beq\label{Id-tY}
\chi_{\eta_0}(\zeta)\int_{\mR^3} \tilde{R}(t,z_1,y) \, g_k^{*}\big(z_1, |\zeta|, c(t,y)\big)\, e^{-iy\cdot\zeta}\, \d z_1 \d y :=\sum_{j=0}^4 I_j^k(t,\zeta)
\eeq
where
\beq
\begin{aligned}\label{def-I0-4}
  & I_0^k=-\chi_{\eta_0}(\zeta)\int_{\mR^3} \tilde{U}(t,z_1,y) \cdot \p_c g_{k}^{*} \big(z_1, |\zeta|, c(t,y)\big)\, (\pt c)(t,y)\, e^{-iy\cdot\zeta} \,  \d z_1 \d y , \\
  & I_1^k=-\chi_{\eta_0}(\zeta)\int_{\mR^3} \tilde{U}(t,z_1,y) \cdot  g_{k}^{*} \big(z_1, |\zeta|, c(t,y)\big)\, (\pt \gamma-\tilde{c})(t,y)\, e^{-iy\cdot\zeta} \,  \d z_1 \d y,\, \\
  & I_2^k=-\chi_{\eta_0}(\zeta)\int_{\mR^3} \tilde{U}(t,z_1,y)\cdot  L_c^{*}(Q_c,\nabla)\big( g_k^{*}(z_1,\eta, c)\, e^{-iy\cdot\zeta}  \big)\,  \d z_1 \d y,\, \\
  & I_3^k=-\chi_{\eta_0}(\zeta)\int_{\mR^3} L_{c,1}(Q_c, \nabla) \tilde{U} (t,z_1,y)
  \cdot g_{k}^{*} \big(z_1, |\zeta|, c(t,y)\big)\, e^{-iy\cdot\zeta} \,  \d z_1 \d y , \\
  & I_4^k=-\chi_{\eta_0}(\zeta) \int_{\mR^3}  \tilde{N} (t,z_1,y)
  \cdot g_{k}^{*} \big(z_1, |\zeta|, c(t,y)\big)\, e^{-iy\cdot\zeta} \,  \d z_1 \d y \, .
\end{aligned}
\eeq
It follows from the definition \eqref{def-R} and  \eqref{def-firsttwo} that
\beq\label{expression-R1}
\tilde{R}_1
=(\pt\gamma-\tilde{c}) Q_c'-\pt \tilde{c} \,\p_c Q_c 
=(\pt\gamma-\tilde{c})\,U_{c}^1
+\pt \tilde{c} \, U_c^2.
\eeq
Moreover, by using the equation $\p_{z_1}^2\phi_c=e^{\phi_c}-1-n_c,$ we find
through straightforward calculations that, the components  of the vector  $\tilde{R}_2$ take the following form
\begin{align*}
\tilde{R}_2^1&= - (1+n_c)\big(\Delta_y c \,\p_c -\Delta_y\gamma\, \p_{z_1}\big)\psi_c\\
&\quad - \sum_{j=1}^2 \bigg((1+n_c) (c_{y_j}\p_c-\gamma_{y_j}\p_{z_1})^2\psi_c+(c_{y_j}\p_c-\gamma_{y_j}\p_{z_1}) n_c\cdot\, (c_{y_j}\p_c-\gamma_{y_j}\p_{z_1})\psi_c\bigg), 
\end{align*} 
\begin{align*}
 \tilde{R}_2^2&=-I_{\phi_c,0}^2\big(\Delta_y c \,\p_c -\Delta_y\gamma\, \p_{z_1}\big)n_c+ I_{\phi_{c}} I_{\phi_{c},0}\big(\p_{z_1}\phi_{c} \Delta_y^2\gamma+\p_c \phi_{c} \Delta_y^2 c)\\
 &\quad -\sum_{j=1}^2\bigg( I_{\phi_c,0}^2\bigg((c_{y_j}\p_c-\gamma_{y_j}\p_z)^2\, n_c-e^{\phi_c}\big( (c_{y_j}\p_c-\gamma_{y_j}\p_z)\,\phi_c\big)^2\bigg)+\f12  
 \big( (c_{y_j}\p_c-\gamma_{y_j}\p_z)\,\psi_c\big)^2\bigg),
\end{align*}
where  we have set 
\beq\label{def-Iphic}
I_{\phi_c}=(e^{\phi_c}-\Delta)^{-1}, \,\qquad I_{\phi_c,0}=(e^{\phi_c}-\p_{z_1}^2)^{-1}.
\eeq
Consequently, we have that 
\beq\label{expression-R2}
\begin{aligned}
\tilde{R}_2&= \Delta_y c\, L_{c}^1\, U^2_c+\Delta_y \gamma \, L_{c}^1 \, U^1_c+\big(0, 
\Delta_y^2(\gamma, c)\cdot I_{\phi_{c}} I_{\phi_{c},0}(\p_{z_1} \phi_c,\p_c\phi_c) 
\big)^t 
\\
&\,\,+ |\na_y c|^2 F_3
+ \na_y c\cdot \na_y \gamma F_4
+|\na_y \gamma|^2 F_5
\end{aligned}
\eeq
where \beqs 
L_{c}^1=\left(\begin{array}{cc}
    0 &  1+n_c \\
    I_{\phi_c,0}^2  & 0
\end{array}\right)
\eeqs
 and $F_3- F_5$ are some quantities involving smoothly on modulations $\gamma,\,c ,$  the solitary waves $Q_c$ and their derivatives in $x$ and $c.$
It is not necessary to write down their exact expressions as they are not needed in the later computations.
However, it is useful to note, by utilizing the evenness of $n_c$ and $\phi_c$, and the oddness of $\psi_c$, that $F_3$ can be expressed as $(F_3^1, F_3^2)$, where $F_3^1$ is odd and $F_3^2$ is even. We thus get that 
$$\int_{\mR} F_3
\cdot \overline{U_c^{*,1}}\, \d z_1 =0 .$$
Let us introduce
\beq\label{split-tR}
\begin{aligned}
   & \tilde{R}_{L,1}:=(\pt\gamma-\tilde{c})\,U_{c}^1
+\pt \tilde{c} \, U_c^2+\Delta_y c\, L_{c}^1\, U^2_c+\Delta_y \gamma \, L_{c}^1 \, U^1_c  \,, \\
 &\tilde{R}_{L,2}:=\big(0, 
 I_{\phi_{c}} \big(\Delta_y^2(\gamma, c) \cdot I_{\phi_{c},0}(\p_{z_1}\phi_c,\p_c\phi_c) \big)
\big)^t , \\
 &\tilde{R}_{NL}:= |\na_y c|^2 F_3
+ \na_y c\cdot \na_y \gamma F_4
+|\na_y \gamma|^2 F_5, 
\end{aligned}
\eeq
and
\beqs 
g_{k1}^{*}(z_1,\eta, c_0)=\big(g_{k}^{*}(z_1,\eta, c_0)-g_{k}^{*}(z_1, 0, c_0)  \big)/\eta^{2}, \quad g_{k2}^{*}(z_1,\eta, c)=\big(g_{k}^{*}(z_1,\eta, c)-g_{k}^{*}(z_1, \eta, c_0)  \big)/\tilde{c} 
\eeqs
so that we can decompose 
\beqs
\tilde{R}=\tilde{R}_{L,1}+\tilde{R}_{L,2}+\tilde{R}_{NL},
\eeqs
\beq\label{usefulId}
g_k^{*}(z_1,\eta, c)=g_k^{*}(z_1, 0, c_0)+\eta^2 g_{k1}^{*}(z_1,\eta, c_0)+\tilde{c} \, g_{k2}^{*}(z_1,\eta, c)\,.
\eeq
We then compute 
\beq\label{inter-LHS}
\begin{aligned}
  & \chi_{\eta_0}(\zeta) \int_{\mR^3} \tilde{R}(t,z_1,y) \, g_k^{*}\big(z_1, |\zeta|, c(t,y)\big)\, e^{-iy\cdot\zeta}\, \d z_1 \d y\\
    &= \chi_{\eta_0}(\zeta)\, \cF_{y}\bigg(\int_{\mR} \tilde{R}|_{c=c_0}(t,z_1,y) \cdot g_k^{*}(z_1, 0, c_0)+\Delta_{y} (\tilde{R}_{L,1}+\tilde{R}_{L,2})|_{c=c_0})(t,z_1,y) \cdot g_{k1}^{*}(z_1,|\zeta|, c_0) \, \d z_1\bigg) (\zeta)\,\\
    &\quad 
    - I_5^k(t,\zeta)- I_6^k(t,\zeta), 
\end{aligned}
\eeq
where 
\beq\label{def-I56}
\begin{aligned}
       & I_5^k(t,\zeta)= \chi_{\eta_0}(\zeta)\int_{\mR^3} (\Delta_y\tilde{R}_{NL})(t,z_1,y) \, g_{k1}^{*}(z_1, |\zeta|, c_0)
   \, e^{-iy\cdot\zeta} \d z_1 \d y\, ,\\
 &  I_6^k(t,\zeta) = \chi_{\eta_0}(\zeta) \int_{\mR^3} \tilde{c}(t,y) \bigg( \f{\tilde{R}-\tilde{R}|_{c=0}}{\tilde{c}}(t,z_1,y)\, g_k^{*}(z_1, |\zeta|,c) + \tilde{R}|_{c=c_0}\, g_{k,2}^{*} (z_1, |\zeta|,c)\bigg)
   e^{-iy\cdot\zeta} \d z_1 \d y\, .
\end{aligned}
\eeq
By the definition \eqref{split-tR} and  the expansion \eqref{exp-g12dual}, it holds that 
\beqs 
\left(
\begin{array}{c}
   \int_{\mR} \tilde{R}_{L,1}|_{c=c_0}(t,z_1,y) \,g_1^{*}(z_1, 0, c_0) \, \d z_1  \\[3pt]
 \int_{\mR} \tilde{R}_{L,1}|_{c=c_0}(t,z_1,y) \,g_2^{*}(z_1, 0, c_0) \, \d z_1
\end{array}
\right)= B(c_0) \left(
\begin{array}{c}
\pt \gamma -\tilde{c} \\[3pt]
\pt \tilde{c}
\end{array}
\right)+D(c_0) \left(
\begin{array}{c}
\Delta_y \gamma \\[3pt]
\Delta_y\tilde{c}
\end{array}\right)
\eeqs
where the above  (constant)  matrices are defined by  
\beqs 
B(c_0)= \bigg( \int_{\mR} U_{c_0}^k(z_1)\cdot g_{j}^{*}(z_1, 0, c_0) \, \d z_1 \bigg)_{jk}=\left( \begin{array}{cc}
   1  & \beta_{22}(c_0)/\beta_{12}(c_0) \\[3pt]
   0  & C_{*} \beta_{12}(c_0)
\end{array}\right)
\eeqs
\beqs 
D(c_0)= \bigg( \int L_{c_0}^1 U_{c_0}^k(z_1)\cdot g_{j}^{*}(z_1, 0, c_0) \, \d z_1 \bigg)_{jk}:=\left( \begin{array}{cc}
   d_{11}  & d_{12}\\[3pt]
   d_{21}  & d_{22}
\end{array}\right).
\eeqs
In the above, $\beta_{12}\neq 0,$ so that $B(c_0)$ is invertible with determinant bounded away from zero uniformly in $\eta.$
Next, the nonlinear term 
 associated to $ \tilde{R}_{NL}$ in \eqref{split-tR} 
has the form
\beq \label{defquadterm}
\begin{aligned}
&\qquad\left(
\begin{array}{c}
   \int_{\mR} \tilde{R}_{NL}|_{c=c_0}(t,z_1,y) \,g_1^{*}(z_1, 0, c_0) \, \d z_1  \\[3pt]
 \int_{\mR} \tilde{R}_{NL}|_{c=c_0}(t,z_1,y) \,g_2^{*}(z_1, 0, c_0) \, \d z_1  
 \end{array}\right)\\
 &= |\na_y c|^2 \,d_3(c_0)+ \na_y c \cdot \na_y \gamma\, d_4(c_0)+ |\na_y \gamma|^2\, d_5(c_0):= \cQ(\na_y c, \na_y \gamma)
\end{aligned}
\eeq
where the constant vectors $d_3(c_0)-d_5(c_0)$ are defined by 
\begin{align*}
    d_j(c_0)=\left(
\begin{array}{c}
   \int_{\mR} F_j|_{c=c_0}(t,z_1,y) \,g_1^{*}(z_1, 0, c_0) \, \d z_1  \\[3pt]
 \int_{\mR}F_j|_{c=c_0}(t,z_1,y) \,g_2^{*}(z_1, 0, c_0) \, \d z_1  
 \end{array}\right):= \left(
\begin{array}{c}
d_{j1}\\
d_{j2}
\end{array}\right), \qquad (j=3,4,5).
\end{align*}
Note that it follows from \eqref{usefulId} that $d_{52}=0.$

Now, plugging \eqref{inter-LHS} into \eqref{Id-tY}, 
and taking the inverse Fourier transform, we obtain since  $\gamma, \tilde{c} \in C^1([0,T], Y)$ that for any $t\in [0,T],$
\beq\label{modeq-0}
\begin{aligned}
&  B_1(c_0, D_y) \left(  \begin{array}{c}
\pt \gamma -\tilde{c} \\[3pt]
\pt \tilde{c}
\end{array}
\right)+\Delta_y D(c_0) \left(
\begin{array}{c}
 \gamma \\[3pt]
\tilde{c}
\end{array}\right)+\Delta_y^2 \tilde{S}
\left(\begin{array}{c}
 \gamma \\[3pt]
\tilde{c}
\end{array}\right)\\
&=
-\chi_{\eta_0}(D) \cQ(\na_y \tilde{c}, \na_y \gamma)+
\cF_{\zeta}^{-1}\bigg(
\sum_{j=0}^6 I_j(t,\zeta)
\bigg),
\end{aligned}
\eeq
where $B_1(c_0, D_y)=B(c_0)+\Delta_y S, \, \tilde{S}= \tilde{S}_1+ \tilde{S}_2$ 
and $S, \tilde{S}_1, \tilde{S}_2$ are matrix valued 
operators defined by
\begin{align*}
   & S f= \cF_{\zeta}^{-1}\bigg( \bigg(  \int_{\mR} U_{c_0}^k(z_1) \cdot g_{j1}^{*}(z_1, |\zeta|, c_0) \, \d z_1 \bigg)_{jk} 
    \chi_{\eta_0}(\zeta)(\cF_y f)({\zeta} )\bigg),\\
  &  \tilde{S}_1 f= \cF_{\zeta}^{-1}\bigg( \bigg(  \int_{\mR} (L_{c_0}^1 U_{c_0}^k)(z_1) \cdot g_{j1}^{*}(z_1, |\zeta|, c_0) \, \d z_1 \bigg)_{jk} 
    \chi_{\eta_0}(\zeta)(\cF_y f)({\zeta} )\bigg),\\
& \tilde{S}_2 f= 
\cF_{\zeta}^{-1} \bigg(\bigg(\begin{array}{cc}
 \int_{\mR} (0, I_{\phi_{c_0},\zeta} I_{\phi_{c_0},0}\, \p_{z_1}\phi_{c_0})^t\cdot {g}_{1}^{*}(z_1, |\zeta|, c_0)\,\d z_1    & 
 \int_{\mR} (0, I_{\phi_{c_0},\zeta} I_{\phi_{c_0},0} \p_c \phi_{c_0})^t\cdot {g}_{1}^{*}(z_1, |\zeta|, c_0)\,\d z_1  \\[3pt]
 \int_{\mR} (0, I_{\phi_{c_0},\zeta} I_{\phi_{c_0},0}\, \p_{z_1}\phi_{c_0})^t\cdot {g}_{2}^{*}(z_1, |\zeta|, c_0)\,\d z_1    &
 \int_{\mR} (0, I_{\phi_{c_0},\zeta} I_{\phi_{c_0},0} \p_c \phi_{c_0})^t\cdot {g}_{2}^{*}(z_1, |\zeta|, c_0)\,\d z_1\end{array}\bigg) \\
 & \qquad\qquad \qquad\qquad 
 \chi_{\eta_0}(\zeta)(\cF_y f)({\zeta} )\bigg),
\end{align*}
with $I_{\phi_c,0}$ defined in \eqref{def-Iphic} 
and $I_{{\phi_{c_0}},\zeta}:=(e^{\phi_{c_0}}-\p_{z_1}^2+|\zeta|^2)^{-1}.$
Let $\eta_0$ be small enough so that $B_1(c_0, D_y)$ can be seen as a small perturbation of the constant matrix $B(c_0)$ an thus is invertible on $Y\times Y.$ 
Therefore, the identity \eqref{modeq-0}
is equivalent to 
\beq\label{modueq-1}
\bigg(\pt-A(c_0, D_y)\bigg)\left(\begin{array}{c}
 \gamma \\[3pt]
\tilde{c}
\end{array}\right)=-B_1^{-1}\, \chi_{\eta_0}(D) Q(\na_y \tilde{c}, \na_y \gamma)+ \sum_{j=0}^6\cR_j
\eeq
with $\cR_j=B_1^{-1}\cF_{\zeta}^{-1}\bigg(
I_j(t,\zeta)\bigg)$
and the symbol of the linear operator reads
\begin{align*}
A\big(c_0, \zeta\big)&=\underbrace{\bigg(\left(\begin{array}{cc}
   0  & 1  \\
    0 & 0
\end{array}\right)+|\zeta|^2 \,B^{-1}(c_0)D(c_0)\bigg)}_{:= A_0\big(c_0, \zeta\big)} +\underbrace{|\zeta|^4 \bigg( |\zeta|^{-2} \big(B_1^{-1}(c_0,\zeta)-B^{-1}(c_0)\big)  D(c_0)-\tilde{S}(\zeta)\bigg)}_{:=A_1\big(c_0, \zeta\big)}.
\end{align*} 
We have thus found  the system satisfied by $(\gamma,\tilde{c}),$ which is a dissipative wave system. Indeed,
direct computations show that 
\begin{align*}
   A_0\big(c_0, \zeta\big)=\left(\begin{array}{cc}
   0  & 1  \\
    0 & 0
\end{array}\right)+\f{|\zeta|^2}{C_{*}\beta_{12}}
   \left( \begin{array}{cc} 
   C_{*}\beta_{12}d_{11}-\f{\beta_{22}d_{21}}{\beta_{12}}   &  C_{*}\beta_{12} d_{21}-\f{\beta_{22}d_{22}}{\beta_{12}} \\[5pt] d_{21} &d_{22}
   \end{array} \right).
\end{align*}
Moreover, we derive from \eqref{Lbda1} \eqref{Lbda2} that, for $\zeta \neq 0$
\begin{align*}
 &\text{tr}  A_0\big(c_0, \zeta\big)=-2 \lambda_{2,c_0}|\zeta|^2<0, \\
& \det A_0 \big(c_0, \zeta\big)= -\f{d_{21}}{C_{*}\beta_{12}}
|\zeta|^2+\cO(|\zeta|^4)=\lambda_{1c_0}^2|\zeta|^2+\cO(|\zeta|^4)>0, 
\end{align*}
provided  $|\zeta|\leq \eta_0$ is small enough.   
The two eigenvalues of $A_0(c_0,\zeta)$ thus have the form 
\beqs 
\lambda_{\pm} (A_0)=-\lambda_{2,c}|\zeta|^2\pm i \sqrt{\lambda_{1c}^2|\zeta|^2+\cO(|\zeta|^4)}\,,
\eeqs
where $\lambda_{1, c_{0}}, \lambda_{2, c_{0}}>0$ 
appears in the expansion of the spectral curve \eqref{exp-spectralcurve}.
Consequently, we may expect that the solution to the linearized system of \eqref{modueq-1} enjoy some decay properties similar  to the ones of the  heat equation or the dissipative  wave system. 


\subsection{A-priori estimates for the modulation parameters}
Assume that $(\gamma, \tilde{c})$ exists \footnote{By Lemma \ref{lem-IFT} and Theorem \ref{thm-local}, if the initial perturbation $U_0$ is sufficiently small in $X_0^N\cap X_a^{N_1}$ with $N\geq 3, N_1\geq 1,$ then $T>0.$} on $C^1([0,T], Y\times Y),$ 
we derive some decay estimates for them in the interval $[0,T].$
It turns out to be more convenient to consider the system satisfied by $(|\na_y|\gamma, \tilde{c}).$ Therefore, after applying  $\diag \big(|\na_y|, \Id\big)$ on the system \eqref{modueq-1}, we find
\beq\label{modueq-new}
\begin{aligned}
\bigg(\pt-\tilde{A}(c_0, D_y)\bigg)\left(\begin{array}{c}
 |\na_y|\gamma \\[3pt]
\tilde{c}
\end{array}\right)=\diag\big(|\na_y|, \Id\big) \bigg(-B_1^{-1}\chi_{\eta_0}(D) Q(\na_y \tilde{c}, \na_y \gamma)+ \sum_{j=0}^6\cR_j\bigg)
\end{aligned}
\eeq
where $\tilde{A}(c_0, D_y)=\diag \big(|\na_y|, \Id\big)\,{A}(c_0, D_y)\,\diag \big(|\na_y|^{-1}, \Id\big).$
Denote  by $a_{ij}$ the entries of the matrix $A_0(c_0,\zeta),$ then
$\tilde{A}(c_0, \zeta)$ can be decomposed as:
\beqs 
\tilde{A}(c_0, \zeta)=\tilde{A}_0(c_0, \zeta)+\tilde{A}_1(c_0, \zeta)
\eeqs
with 
\begin{align}\label{deftA01}
 \tilde{A}_0(c_0, \zeta) =  \left( \begin{array}{cc}
  a_{11} & |\zeta| \\
   {a_{21}}/{|\zeta|}  & a_{22} 
    \end{array} \right), \qquad \tilde{A}_1(c_0, \zeta) =  \left( \begin{array}{cc}
  \cO(|\zeta|^4) & \cO(|\zeta|^3) \\
  \cO(|\zeta|^3)  &  \cO(|\zeta|^4)
    \end{array} \right).
\end{align}
To continue, we need 
to prove some decay properties of the linear system $\pt f-\tilde{A}(c_0, D_y)f=0.$
We define the space
\beq \label{def-spaceY1}
Y_1=\big\{f\in L^2(\mR^2)\,\big|\, \Supp \cF_y f\subset B(0, 2\eta_0) \big\},
\eeq
endowed with the norm $\|\cdot\|_{Y_1}=\|\cF(\cdot)\|_{L^{2}(B(0,2\eta_0))}.$ Additionally, define the space for the initial data
\beq \label{def-spaceY2}
 Y_2=\big\{f\in Y_1\,\big|\,  \cF_{y} f\in L^{\infty}\big( B(0, 2\eta_0)\big) \big\} .  
\eeq
endowed with the norm
  $\|\cdot\|_{Y_2}=\|\cF(\cdot)\|_{L^{\infty}(B(0,2\eta_0))}.$ 
  Moreover, for a vector function $f=(f_1,f_2)^t$ belonging to $X\times X$, we denote for short $\|f\|_{X}=\|f_1\|_{X}+\|f_2\|_{X}.$ 
\begin{lem}
 Assume that $f_0\in Y_2$ and that  $\eta_0$ is  small enough.
 Let $\cL(t): Y_2\ni  f_0 \rightarrow f(t)\in Y_1 $ be the solution map of the linear system 
    \beq\label{modu-lineareq}
\pt f-\tilde{A}(c_0, D_y)f=0, \qquad f|_{t=0}=
f_0\,.
    \eeq
Then it holds that for any $k\in \mathbb{N},$
    \begin{align}
       & \|\p_y^k \cL(t) f_0\|_{L^2{(\mR^2)}}\lesssim (1+t)^{-\f{k}{2}}\min \big\{(1+t)^{-\f{1}{2}} \|f_0\|_{Y_2},  \,\|f_0\|_{Y_1}\big\}\label{lineares-L2},\\ 
        &\|\p_y^k\big( y \, \cL(t) f_0\big)\|_{L^2{(\mR^2)}}\lesssim 
       (1+t)^{-\f{k}{2}} \big(\|f_0\|_{Y_2}+\min \big\{(1+t)^{-\f{1}{2}} \|y f_0\|_{Y_2},  \,\|y f_0\|_{Y_1}\big\} \big),\label{lineares-weight}.
 \end{align} 

\end{lem}

\begin{proof}
    We first claim that it suffices to show \eqref{lineares-L2} and \eqref{lineares-weight} when 
    $\cL(t)$ is replaced with $\cL_0(t)$ where 
    $\cL_0(t)$ is the semigroup generated by the linear operator $\tilde{A}_0(c_0, D_y).$ Indeed, we
    write $f(t)=\cL_0(t) f_0+\int_0^t \cL_0(t-s) (\tilde{A}_1 f) (s)\, \d s $ and
    use the estimate for $\cL_0(t)$ and the definition of $\tilde{A}_1$ in \eqref{deftA01} to obtain 
    \begin{align*}
     \bigg\|\int_0^t \p_y^k \cL_0(t-s) (\tilde{A}_1 f) (s) \d s \bigg\|_{L^2(\mR^2)}
        &\leq \int_0^t \big\| |\na_y|^{\f{5}{2}} \cL_0(t-s) \big(|\na_y|^{-\f{5}{2}}\tilde{A}_1 \p_y^k f\big) (s)\big\|_{L^2(\mR^2)}\, \d s\\
        &\lesssim |\eta_0|^{\f12}\int_0^t (1+t-s)^{-\f{5}{4}}\, \d s \sup_{0\leq s\leq t}\|\p_y^k f(s)\|_{L^2(\mR^2)}.
    \end{align*}
It thus holds that for $\eta_0$ small enough,
\beqs 
\|\p_y^kf(t)\|_{L^2(\mR^2)}\lesssim \|\p_y^k\cL_0(t) f_0\|_{L^2(\mR^2)}.
\eeqs
Similarly, by using the Hardy-Littleword-Sobolev inequality and  the Young  inequality, we find that for $\eta_0$ sufficiently small,
 \begin{align*}
 \bigg\|\int_0^t  y\, \cL_0(t-s) (\tilde{A}_1 f) (s) \d s \bigg\|_{L^2(\mR^2)}&\lesssim 
 \sup_{0\leq s\leq t}\big\|(1+y)|\na|^{-\f52}\tilde{A}_1 f (s)\big\|_{Y_1}\\
 &\lesssim  \sup_{0\leq s\leq t} \big(|\eta_0|^{\f12}\|y f(s)\|_{Y_1}+ \|f(s)\|_{Y_1}^{\f12} \|y f(s)\|_{Y_1}^{\f12}\big)\\
 &\leq \f12 \sup_{0\leq s\leq t} \|y f(s)\|_{Y_1}+C\|f\|_{Y_1},
\end{align*}  
which yields that 
\beqs 
\|y f(t)\|_{L^2(\mR^2)}\lesssim \|y\cL_0(t)f_0\|_{L^2(\mR^2)}+\|f_0\|_{Y_1}\,.
\eeqs
In a similary way,  for $k>0$, we can get  
\beqs 
\|\p_y^k(y f)(t)\|_{L^2(\mR^2)}\lesssim \|\p_y^k(y\cL_0(t)f_0)\|_{L^2(\mR^2)}+\|\p_y^k f_0\|_{Y_1}.
\eeqs
  Now, it suffices to prove  \eqref{lineares-L2} and \eqref{lineares-weight} with $\cL(t)$ replaced by $\cL_0(t).$  This can be obtained by studying the matrix   $e^{-t \tilde{A}_0(c_0, \zeta)}.$
  Some algebraic computations show that 
  \begin{align}\label{semi-matrix}
      e^{t \tilde{A}_0(c_0, \zeta)}=e^{-\lambda_{2,c_0}|\zeta|^2 t}\,\f{1}{\lambda_{1,c_0}\omega}
     \left( \begin{array}{cc}
      \lambda_1\omega \cos (\beta t)-\nu|\zeta|\sin (\beta t)     & -i\sin (\beta t) \\[4pt]
       i \lambda_{1,c_0}\sin (\beta t)   & \lambda_1\omega \cos (\beta t)+\nu|\zeta|\sin (\beta t)  
      \end{array}\right)
  \end{align}
  where $\lambda_{2,c_0}=2\p_{\eta}^2\lambda_{c_0}(\eta)>0$
  appears in the expansion of the spectral curve \eqref{exp-spectralcurve},
  $\nu:= \f{a_{11}-a_{22}}{{2|\zeta|^2}}, \, \omega:=\sqrt{1-(\f{\nu |\zeta|}{\lambda_{1,c_0} }\big)^2},\, \beta:=\lambda_{1,c_0} |\zeta|\omega$ 
   Note that for $|\zeta|\leq \eta_0$ small enough, $\omega \in \mR.$ Thanks to the explicit expression \eqref{semi-matrix} and  
 the Plancherel identity, we obtain easily 
  \begin{align*}
       & \|\p_y^k \cL_0(t) f_0\|_{L^2{(\mR^2)}}\lesssim (1+t)^{-\f{k}{2}}\min \big\{(1+t)^{-\f{1}{2}} \|f_0\|_{Y_2},  \,\|f_0\|_{Y_1}\big\}\big),\\
        &\|\p_y^k\big( y \, \cL_0(t) f_0\big)\|_{L^2{(\mR^2)}}\lesssim 
       (1+t)^{-\f{k}{2}} \big(\|f_0\|_{Y_2}+\min \big\{(1+t)^{-\f{1}{2}} \|y f_0\|_{Y_2},  \,\|y f_0\|_{Y_1}\big\}\big).
 \end{align*} 
\end{proof}
\begin{rmk}
Recall that $A_0(c_0,\zeta)=\diag \big(|\zeta|^{-1},\Id\big) \tilde{A}_0(c_0,\zeta)\, \diag \big(|\zeta|,\Id\big).$ Thanks to the expression \eqref{semi-matrix}, we have also that 
\begin{align}
\big\| \big(e^{t A_0(c_0,D_y)} f\big)_1\big\|_{L^{2}}
\lesssim 
   \big( \|f_1\|_{L^2}+\min\big\{ \|\cF f_2\|_{L_{\zeta}^{\infty}}, \||\zeta|^{-1}\cF f_2\|_{L_{\zeta}^2} \big\} \big), \label{lineares-L2-o}\\
   \big\| \big(e^{t A_0(c_0,D_y)} f\big)_1\big\|_{L^{\infty}}\lesssim
   t^{-\f12}\big( \|f_1\|_{L^2}+\min\big\{ \|\cF f_2\|_{L_{\zeta}^{\infty}}, \||\zeta|^{-1}\cF f_2\|_{L_{\zeta}^2} \big\} \big) .  \label{lineares-Linfty-o}
\end{align}
    
\end{rmk}

We are now in position to establish the decay properties for the modulation parameters $(\tilde{c}, \gamma)$. Let us introduce the following norm to assess their decay:
\beq\label{def-decays-modu}
\begin{aligned}
\cN_{\tilde{c},\gamma}(T)&:= \sup_{0\leq t\leq T} \sum_{k=0}^2 (1+t)^{\f{k+1}{2}} \|\p_y^k (\tilde{c}, \na_y \gamma)(t) \|_{L^2}+ \sum_{\ell=0}^3(1+t)^{\f{\ell}{2}} \|y\,\p_y^{\ell} (\tilde{c}, \na_y \gamma) (t)\|_{L^2}
\\
&\qquad +\sup_{0\leq t\leq T} \big( \|\gamma(t)\|_{L^2}+(1+t)^{\f12} \|\gamma(t)\|_{L^{\infty}}\big)\, .
\end{aligned} 
\eeq
Note that by interpolation, we have that for any
$k=0,1; p\in (1, +\infty] $ but $p\neq +\infty$ when $k=1,$
\begin{align}\label{cor-interpolation}
   \|\,\p_y^k (\tilde{c}, \na_y \gamma)(t) \|_{L^p}\lesssim (1+t)^{-\big(\f{k}{2}+1-\f1p\big)}\cN_{\tilde{c},\gamma}(T), \quad \forall\,\, 0\leq t\leq T. \, 
\end{align}
Before presenting the statement and proof of the next proposition, it is useful to define a quantity involving the weighted norm of the perturbation
\beq\label{def-weighted-U}
\begin{aligned}
  \cN_{U}(T)=\colon   \sup_{0\leq t\leq T}&(1+t)\|U(t)\|_{X_a^{M}}+(1+t)^{\f12}\|y\,U(t)\|_{X_a^{M}}\\
  &+(1+t)^{\f32}\big(\|\nabla_y U(t)\|_{X_a^{M-1}}+\| y\nabla_y^2 U(t)\|_{X_a^{M-2}}\big), \qquad (M\geq 3)
\end{aligned}
\eeq
A priori estimates for this quantity will be established in the next section. 
We also recall the definition of  $\cN_T$ from \eqref{def-cNT}, 
which clearly satisfies:
\begin{align*}
    \cN_T\geq \cN_{\tilde{c},\gamma}(T)+\cN_{U}(T).
\end{align*}

\begin{prop}\label{prop-modulation}
Let $(\tilde{c}, \gamma)(0,\cdot)\in Y\times Y$  determined by  Lemma \ref{lem-IFT}
and  
\beq\label{def-cMcgamma0}
\cM_{\tilde{c},\gamma}(0):
=\|\langle \cdot \rangle ((\tilde{c},\gamma)(0,\cdot))\|_{Y_1}+\|(\tilde{c},\gamma)(0,\cdot)\|_{Y_2}<+\infty.
\eeq
Assume that $(\tilde{c}, \gamma)$ exist in $C^1([0,T], Y\times Y),$ ${U}$ exists in $C^1([0,T], X_a^1).$ 
There exists 
    $\delta_2$ small enough such that  if
 $\cN(T)\leq \delta<{\delta_2}$ 
    then it holds that 
    \begin{align}\label{Ncgamma}
        \cN_{\tilde{c},\gamma}(T)\lesssim \cM_{\tilde{c},\gamma}(0)+ \cN(T)^2+ \cN(T)^4.
    \end{align} 
\end{prop}
\begin{rmk}\label{rmk-cgamma0}
   By Remark \ref{rmk-intialdata-cgamma}, it holds that 
   \begin{align*}
       \cM_{\tilde{c},\gamma}(0)\lesssim C\big(\|\langle \cdot \rangle \tilde{\Psi}_0(\cdot)\|_{L^2}+\|e^{ax}(n_0,\na\psi_0)\|_{L_y^1L_x^2}\big), \, \qquad (\tilde{\Psi}_0(\cdot)=(n_0,\psi_0)^t).
   \end{align*}
\end{rmk}
\begin{proof}
Let us recall that $(|\na_y|\gamma, \tilde{c})$ and $(\gamma, \tilde{c})$ solves the equations \eqref{modueq-new} and \eqref{modeq-0} with 
initial datum $(\tilde{c}, \gamma)(0,\cdot).$  Denote for simplicity
    $\cR_7=-B_1^{-1}\chi_{\eta_0}(D) Q(\na_y \tilde{c}, \na_y \gamma)$
    and $\tilde{\cR}_j=\diag\big(|\na_y|, \Id\big) \cR_j (j=0,\cdots 7)$ where $Q(\na_y \tilde{c}, \na_y \gamma)$ is the quadratic term defined in \eqref{defquadterm}. 
  Applying the semigroup estimates \eqref{lineares-L2} \eqref{lineares-weight} and 
    \eqref{lineares-L2-o}, \eqref{lineares-Linfty-o},  we find that
    \begin{align}\label{aprior-firstround}
         \cN_{\tilde{c},\gamma}(T)\lesssim \cM_{\tilde{c},\gamma}(0)+\sum_{j=0}^7 \sup_{0\leq t\leq T} (1+t)^{\f32} \big(\|\tilde{\cR}_j\|_{Y_2}+ \|({\cR}_j)_1\|_{Y_1} \big)+  (1+t) \|y \tilde{\cR}_j\|_{Y_2}.
    \end{align}
 Therefore, we just need  to  prove that the right hand side can be controlled by $\cN(T)^2+ \cN(T)^4.$
 Let us first control the quadratic terms $\tilde{\cR}_7$ and $({\cR}_7)_1,$ 
 In view of the expressions \eqref{defquadterm}, we have by the Plancherel identity, the  H\"older inequality as well as \eqref{cor-interpolation} that, for any $t\in [0,T],$
\begin{align*}
   \| \tilde{\cR}_7(t)\|_{Y_1}\lesssim \big\|\big(|\na_y \tilde{c}|^2, \na_y  \tilde{c}\cdot\na_y \gamma, |\na_y \gamma|^2\big)(t)\big\|_{L^2}\lesssim \|\na_y ( \tilde{c}, \gamma)(t)\|_{L^{4}}^2 \lesssim (1+t)^{-\f32} \cN_{\tilde{c},\gamma}(T)^2\, .
\end{align*}
 Moreover, since $d_{52}=0$ and $\|B_1^{-1}\chi_{\eta_0}\|_{L_{\zeta}^{\infty}}<+\infty,$ we deduce from the  Hausdorff-Young's inequality that 
 \begin{align*}
    \|\tilde{\cR}_7(t)\|_{Y_2}&\lesssim \||\na_y|\big(|\na_y \tilde{c}|^2, \na_y  \tilde{c}\cdot\na_y \gamma, |\na_y \gamma|^2 \big)(t) \|_{L^1}+\|\big(|\na_y \tilde{c}|^2 , \na_y  \tilde{c}\cdot\na_y \gamma\big)(t)\|_{L^1} \\
    &\lesssim (1+t)^{-\f32}\cN_{\tilde{c},\gamma}(T)^2.
 \end{align*}
Similarly, as for $\eta_0$ small enough, $\|\p_{\zeta} (B_1^{-1}\chi_{\eta_0})\|_{L_{\zeta}^{\infty}}<+\infty,$ it holds that 
\begin{align*}
    \|y\,\tilde{\cR}_7(t)\|_{Y_2}&\lesssim \|(1+y)\,\diag(|\na_y|, \Id) Q(\na_y \tilde{c}, \na_y \gamma) (t)\|_{L^1}
    \lesssim (1+t)^{-1}\cN_{\tilde{c},\gamma}(T)^2.
 \end{align*}

Now, 
it remains to control $\tilde{R}_0-\tilde{R}_6,$ or more precisely, in view of \eqref{aprior-firstround} and the definition $\cR_j=B_1^{-1}\cF_{\zeta}^{-1}\bigg(
I_j(t,\zeta)\bigg),$ the following estimate
\begin{align*}
 \sup_{0\leq t\leq T} (1+t)^{\f32} 
\| I_j(t) \|_{L_{\zeta}^{\infty}}+
  (1+t) \big\|\p_{\zeta}I_j(t)\big\|_{L_{\zeta}^{\infty}}, \quad \forall \, 0\leq j\leq 6.
\end{align*}
Let us control $I_0-I_6$ term by term, which are defined in \eqref{def-I0-4} and \eqref{def-I56}. 
First, utilizing the fact 
\beq \label{property-gstar}
\sup_{j\in \mathbb{N}, |\zeta|\leq 2\eta_0, \, |c-c_0|\leq \delta}\big\|\p_c^{j} g_k^{*}(\cdot, |\zeta|, c)\big\|_{X_{-a}}<+\infty,
\eeq 
we can control $I_0, I_1$  directly by 
\beq\label{es-I01}
\begin{aligned}
    \|(I_0, I_1)(t)\|_{L_{\zeta}^{\infty}}&\lesssim\|e^{a z_1}\tilde{U}(t, z_1, y)\|_{L_{z_1,y}^2} \big\|(\pt \gamma-\tilde{c}, \pt \tilde{c})(t)\big\|_{L_y^2} \\
&\lesssim e^{a\|\gamma\|_{L^{\infty}}} \|U\|_{X_a} \big\|(\pt \gamma-\tilde{c}, \pt \tilde{c})\big\|_{L_y^2} .
\end{aligned}
\eeq
For the term $I_2,$ we write $I_2=\big(I_2-I_2|_{c=c_0}\big)+I_2|_{c=c_0}$ and control them separately. Thanks to \eqref{property-gstar}, we control the first one as 
\beqs 
\big\|\big(I_2-I_2|_{c=c_0}\big)(t)\big\|_{L_{\zeta}^{\infty}}\lesssim   e^{a\|\gamma\|_{L^{\infty}}} \|U(t)\|_{X_a}\|\tilde{c}(t)\|_{L_y^2}.
\eeqs
For the second one, we use the identity $e^{iy\cdot\zeta}L_{c_0}^{*}(Q_{c_0},\nabla)e^{-iy\cdot\zeta}
=L_{c_0}^{*}(|\zeta|)$ 
as well as the definition of $g_{k}^{*}(\cdot, |\zeta|, c_0)$ to find that, there exists  smooth functions $\mathfrak{K}_{k1}(|\zeta|), \mathfrak{K}_{k2}(|\zeta|)$ 
such that
\begin{align*}
    L_{c_0}^{*}(|\zeta|)g_{k}^{*}(\cdot, |\zeta|, c_0)=
    \mathfrak{K}_{k1}(|\zeta|)\,g_{1}^{*}(\cdot, |\zeta|, c_0)+\mathfrak{K}_{k2}(|\zeta|)\,g_{2}^{*}(\cdot, |\zeta|, c_0).
\end{align*}
Therefore, thanks to the identity \eqref{secularcondtion}, it holds also 
\beqs 
 \|I_2|_{c=c_0}(t)\|_{L_{\zeta}^{\infty}}\lesssim e^{a\|\gamma\|_{L^{\infty}}} \|U(t)\|_{X_a}\|\tilde{c}(t)\|_{L_y^2}.
\eeqs
We thus find that, by using the definition of $\cN_U(T)$ in \eqref{def-weighted-U}, 
\begin{align}
    \|I_2(t)\|_{L_{\zeta}^{\infty}} \lesssim (1+t)^{-\f32} \cN_{U}(T)\cN_{\tilde{c},\gamma}(T) 
\end{align}
whenever $\|\gamma\|_{L^{\infty}}\leq C\delta \leq 1,$ which will be always assumed through the proof.
Next, for the term $I_3^k(t),$ it follows from the definition of $L_{c,1}$ that 
\begin{align}
 \|I_3(t)\|_{L_{\zeta}^{\infty}} \lesssim \|U(t)\|_{X_a}\|\na_y \gamma (t)\|_{L_y^2} \lesssim (1+t)^{-\f32} \cN_{U}(T)\cN_{\tilde{c},\gamma}(T).
\end{align}
Let's first address $I_5$ and $I_6$ before estimating $I_4,$ as it is the most intricate.
Thanks again to \eqref{property-gstar}, the term $I_5$ can be controlled as 
\beq
\begin{aligned}
    & \|I_5(t)\|_{L_{\zeta}^{\infty}} \lesssim \|\Delta_y Q(\na_y  \tilde{c}, \na_y\gamma )(t)\|_{L_y^1}\\
     &\lesssim \|\Delta_y (\na_y  \tilde{c}, \na_y\gamma )\cdot  (\na_y  \tilde{c}, \na_y\gamma ) \|_{L_y^1}+\|\na_y (\na_y  \tilde{c}, \na_y\gamma )\|_{L_y^2}^2\lesssim (1+t)^{-2} \cN_{\tilde{c},\gamma}(T)^2.
\end{aligned}
\eeq
Moreover, in light of the expressions \eqref{expression-R1}, we can estimate $I_6$ as
\begin{align}
    \|I_6(t)\|_{L_{\zeta}^{\infty}}&\lesssim \|\tilde{c}(t)\|_{L^2} \big(\|(\pt\gamma-\tilde{c} ,\pt c)(t)\|_{L^2}+ \||\na_y \tilde{c}|^2, |\na_y\gamma|^2, \na_y \tilde{c}\cdot\na_y \gamma)(t)\|_{L^2}+\|\Delta_y^2(\gamma,\tilde{c})(t)\|_{L^2} \big)\notag\\
    &\lesssim \|\tilde{c}\|_{L^2} \|(\pt\gamma-\tilde{c} ,\pt c)(t)\|_{L^2}+(1+t)^{-2}\big(\cN_{\tilde{c},\gamma}(T)^2+\cN_{\tilde{c},\gamma}(T)^3\big).\label{es-I6}
\end{align}
It remains to control the term $I_4=(I_4^1, I_4^2)^t$ where $$I_4^k=\chi_{\eta_0}(\zeta) \int_{\mR^3}  {N} (t,z_1+\gamma,y)
  \cdot g_{k}^{*} \big(z_1, |\zeta|, c(t,y)\big)\, e^{-iy\cdot\zeta} \, \d z_1 \d y \, .$$  
Looking at  the explicit expression of $N$ in \eqref{def-NLterm}, we observe that there are  three types of terms: 
the multiplications of $ \na_y U, U$ and modulation parameters $\na_y(\tilde{c},\gamma), |\na_y(\tilde{ c}, \gamma)|^2, \Delta_y(\tilde{ c}, \gamma);$ the quadratic type  terms of the form 
$\div(n\na\psi), |\na \psi|^2, H(n, n_{c,\gamma})$ and the term $G(\phi,\phi_{c,\gamma}).$ We denote them respectively as $N_1, N_2, N_3$, and denote their contributions in the definition of $I_4$ as $I_{41}-I_{43}$. Using \eqref{property-gstar}, we control $I_{41}$ as
\begin{align}\label{es-I41}
  \|I_{41}(t)\|_{L_{\zeta}^{\infty}} \lesssim   \|e^{a(z-\gamma)}N_1\|_{L_y^1L_{z}^2}\lesssim \|(U,\na_y U)\|_{X_a} \|(\na_y, \Delta)(\tilde{ c}, \gamma)\|_{L_y^2}\lesssim (1+t)^{-\f32} \cN_{U}(T)\cN_{\tilde{c},\gamma}(T) .
\end{align}
For $I_{42},$ we use the expansion \eqref{usefulId} to split it into three terms $I_{42}=I_{420}+I_{421}+I_{422}$ where  
\begin{align*}
  &  I_{420}^k=-\chi_{\eta_0}(\zeta) \int_{\mR^3}  {N}_2(t,z_1+\gamma,y)
  \cdot g_{k}^{*} \big(z_1, 0, c_0\big)\, e^{-iy\cdot\zeta} \, \d z_1 \d y \, , \\
  & I_{421}^k=-\chi_{\eta_0}(\zeta) \int_{\mR^3}  {N}_2(t,z_1+\gamma,y)
  \cdot |\zeta|^2 g_{k1}^{*} \big(z_1, |\zeta|, c_0\big)\, e^{-iy\cdot\zeta} \, \d z_1 \d y \, ,\\
  & I_{422}^k=-\chi_{\eta_0}(\zeta) \int_{\mR^3}  {N}_2 (t,z_1+\gamma,y)
  \cdot \tilde{c} \, g_{k2}^{*} \big(z_1, |\zeta|, c(t,y)\big)\, e^{-iy\cdot\zeta} \, \d z_1 \d y \, .
\end{align*}
The last sub-term $I_{422}$ can be bounded easily by
\begin{align*}
    \| I_{422}(t)\|_{L_{\zeta}^{\infty}} \lesssim \|\na(n ,\na\psi)(t)\|_{L^{\infty}_{z,y}}\|U(t)\|_{X_a}\|\tilde{c}(t)\|_{L^2}\lesssim (1+t)^{-\f52} \cN^2_{U}(T)\cN_{\tilde{c},\gamma}(T).
\end{align*}
Moreover, choosing $a$ small enough such that  
$g_{k}^{*} \big(z_1, 0, c_0\big)\in X_{-2a},$ one can  control $I_{420}$ as
\begin{align*}
    \| I_{420}(t)\|_{L_{\zeta}^{\infty}} \lesssim \|e^{2a(z-\gamma)}N_2\|_{L_y^1L_{z}^2}\lesssim \|(n,\na_x\psi)(t)\|_{H_{a}^1}^2\lesssim (1+t)^{-2} \cN^2_{U}(T).
\end{align*}
Finally, we use the identity 
  $|\zeta|^2 e^{-iy\cdot\zeta}=
  -\Delta_y (e^{-iy\cdot\zeta})$
to integrate by parts in $y$ twice to obtain that 
\begin{align*}
 \| I_{421}(t)\|_{L_{\zeta}^{\infty}}& \lesssim \|e^{a(z-\gamma)}\p_ y N_1\|_{L_y^1L_{z}^2}\\
 &\lesssim \|\p_y (n, \na \psi)\|_{X_a^1} \|\p_y(n, \p_x\psi)\|_{H^2}+\| e^{az}\p_y^2(\Id,\na)(n,\na \psi)\|_{L_{z}^2L_y^4}\|(\Id,\na)(n,\p_x\psi)\|_{L_z^{\infty}L_y^4}\\
 &\lesssim (1+t)^{-\f32} \cN^2_{U}(T).
\end{align*}
where in the last line, we have used the inequalities
$$\| e^{az}f\|_{L_{z}^2L_y^4}\lesssim \|\langle y\rangle f\|_{L_a^2},\, \|f\|_{L_z^{\infty}L_y^4}\lesssim \|f\|_{W^{1,4}}$$ as well as the fact
that  
$$\|\na\p_x\psi\|_{H^2}\lesssim \cN(T),\quad \|\p_y U\|_{X_a^1}+ \|y\p_y^2 U\|_{X_a^2}\lesssim (1+t)^{-\f32}\cN_{U}(T). $$ 

Combining the above  estimates, we find that 
\begin{align}\label{es-I42}
     \|I_{42}(t)\|_{L_{\zeta}^{\infty}} \lesssim (1+t)^{-\f32} \cN^2_{U}(T).
\end{align}
For the term $I_{43},$ we use the equations $\Delta \phi=e^{\phi}-1-n-n_c$ and $\p_{z_1}^2\phi_{c,\gamma}=e^{\phi_{c,\gamma}}-1-n_{c,\gamma}$ to find 
\begin{align}\label{errorphi-cgamma}
  \tilde{\phi}:= \phi- \phi_{c,\gamma}=I_{\phi_{c,\gamma}}\bigg(n+\Delta_y\phi_{c,\gamma}+e^{\phi_{c,\gamma}}\big(e^{\tilde{\phi}}-1-\tilde{\phi}\big)\bigg)
\end{align}
where we recall that  $I_{\phi_{c,\gamma}}:= (e^{\phi_{c,\gamma}}-\Delta)^{-1}.$
Consequently, we can write  
\beq \label{expansion-G}
G(\phi,\phi_{c,\gamma})=G_1(\phi,\phi_{c,\gamma})+G_2(\phi,\phi_{c,\gamma})
\eeq
where 
\begin{align*}
   &  G_1(\phi,\phi_{c,\gamma})= I_{\phi_{c,\gamma}} e^{\phi_{c,\gamma}}\big(e^{\tilde{\phi}}-1-\tilde{\phi}\big),\\
  & G_2(\phi,\phi_{c,\gamma}) =(I_{\phi_{c,\gamma}} -I_{\phi_{c,\gamma},0}) \Delta_y \phi_{c,\gamma}-
     I_{\phi_{c,\gamma}} I_{\phi_{c,\gamma},0}\big(\phi_{c,\gamma}' \Delta_y^2\gamma+\p_c \phi_{c,\gamma} \Delta_y^2 c\big). 
\end{align*}
We denote their contributions in the definition of $I_{43}$ as $I_{431}, I_{432}$. Since $G_1$ is at least quadratic in terms of $n$ and $\Delta_y \phi_c,$ we can follow the same arguments as in the estimate of $I_{41}, I_{42}$ to find 
\begin{align*}
    \|I_{431}(t)\|_{L_{\zeta}^{\infty}} \lesssim (1+t)^{-\f32} \big( \cN(T)^2+\cN(T)^3\big).
\end{align*}
Moreover, straightforward computations show that 
$$G_2(\phi,\phi_{c,\gamma})= I_{\phi_{c,\gamma}} [\Delta_y,  I_{\phi_{c,\gamma},0}]\Delta_y \phi_{c,\gamma}+ I_{\phi_{c,\gamma}} I_{\phi_{c,\gamma},0}\big(\Delta_y^2\phi_{c,\gamma}-\phi_{c,\gamma}' \Delta_y^2\gamma-\p_c \phi_{c,\gamma} \Delta_y^2 c\big)$$ 
which is quadratic in terms of $\na_y (\gamma, \tilde{c})$ and their derivatives. We can thus control $I_{432}$ as 
\begin{align*}
   \|I_{432}(t)\|_{L_{\zeta}^{\infty}} \lesssim   \|e^{a(z-\gamma)}G_2\|_{L_z^2L_y^1}\lesssim \||\na_y(\gamma, \tilde{c})\|_{L_y^2}^2\lesssim (1+t)^{-\f32}\cN^2_{\tilde{c},\gamma}(T).
\end{align*}
Therefore, it holds that
 $\|I_{43}(t)\|_{L_{\zeta}^{\infty}} \lesssim (1+t)^{-\f32} \big( \cN(T)^2+\cN(T)^3\big),$ which, together with \eqref{es-I41}, \eqref{es-I42}, enables us to conclude
\beq \label{es-I4}
\|I_4(t)\|_{L_{\zeta}^{\infty}} \lesssim (1+t)^{-\f32}\big(\cN(T)^2+\cN(T)^3\big).
\eeq

Now, combining the estimates \eqref{es-I01}-\eqref{es-I6} and \eqref{es-I4}, we 
find that 
\beq\label{es-I0-6}
\begin{aligned}
  & \sum_{j=0}^6 \|I_j(t)\|_{L_{\zeta}^{2}}\lesssim
\sum_{j} \|I_j(t)\|_{L_{\zeta}^{\infty}}\\
&\lesssim \big(\|\tilde{c}(t)\|_{L^2}+\|{U}(t)\|_{X_a}\big)\big\|(\pt\gamma-\tilde{c},\pt\tilde{c})(t)\big\|_{L^2}+(1+t)^{-\f32}\big(\cN(T)^2+\cN(T)^3\big). 
\end{aligned}
\eeq
By using  the equation \eqref{modeq-0}, it holds that
\begin{align*}
   & \big\|(\pt\gamma-\tilde{c},\pt\tilde{c})(t)\big\|_{L^2}\lesssim \|\nabla_y^2(\gamma, \tilde{c})(t)\|_{L^2}+\|Q(\na_y \tilde{c}, \na_y \gamma)\|_{L^2}+\sum_{j=0}^6 \|I_j(t,\zeta)\|_{L_{\zeta}^2}\\
    & \lesssim (1+t)^{-1} \big(\cN_{\tilde{c},\gamma}(T)+ \cN^2
    (T)+\cN(T)^3\big)+\big(\|\tilde{c}(t)\|_{L^2}+\|{U}(t)\|_{X_a}^2\big)\big\|(\pt\gamma-\tilde{c},\pt\tilde{c})(t)\big\|_{L^2}.
\end{align*}
Consequently,  there exists ${\delta_2}$ small enough, such that as long as
$\|\tilde{c}(t)\|_{L^2}+\|{U}(t)\|_{X_a}^2\leq \cN(T)\leq {\delta_2}$  it holds that 
\beq\label{gammtct}
  \big\|(\pt\gamma-\tilde{c},\pt\tilde{c})(t)\big\|_{L^2}\lesssim   (1+t)^{-1} \big(\cN_{\tilde{c},\gamma}(T)+\cN(T)^2+\cN(T)^3\big).
\eeq
Plugging this last estimate  into \eqref{es-I0-6}, we find that 
\begin{align*}
    \sum_{j=0}^6 \|I_j(t)\|_{L_{\zeta}^{\infty}}\lesssim (1+t)^{-\f32}\big(\cN(T)^2+\cN(T)^4\big).
\end{align*}
Following closely the above arguments we can also prove  for $I_0-I_6,$ that 
\beqs 
\big\|\p_{\zeta}I_j(t)\big\|_{L_{\zeta}^{\infty}}\lesssim (1+t)^{-1}\big(\cN(T)^2+\cN(T)^4\big).
\eeqs
We omit the details and thus finish the proof.
\end{proof}
\begin{rmk}
By \eqref{gammtct} and \eqref{Ncgamma}, it holds that for any $t\in[0,T]$
\beq \label{es-ptc}
 \big\|(\pt\gamma-\tilde{c},\pt\tilde{c})(t)\big\|_{L^2}\lesssim   (1+t)^{-1} \big(\cM_{\tilde{c},\gamma}(0)+
       \cN(T)^2+\cN(T)^4\big).
\eeq
Moreover, following  similar arguments as  in the proof of the above estimate, we find also that, for any $0\leq t\leq T,$
    \begin{align}
          \big\|y\,(\pt\gamma-\tilde{c},\pt\tilde{c})(t)\big\|_{L^2}\lesssim   (1+t)^{-\f12} \big(\cM_{\tilde{c},\gamma}(0)+
       \cN(T)^2+\cN(T)^4\big)\,,   \label{es-yptc} \\
       \|y(\p_y,\pt)(\pt\gamma-\tilde{c},\pt\tilde{c})(t)\big\|_{L^2}\lesssim   (1+t)^{-1} \big(\cM_{\tilde{c},\gamma}(0)+
       \cN(T)^2+\cN(T)^4\big)\,, \label{es-ypyptc} \\
       \big\|(\p_y,\pt, y\p_y\pt, y\p_y^2)(\pt\gamma-\tilde{c},\pt\tilde{c})(t)\big\|_{L^2}\lesssim   (1+t)^{-\f32} \big(
      \cM_{\tilde{c},\gamma}(0)+ \cN(T)^2+\cN(T)^4\big)\,. \label{es-nablayptc}
    \end{align}
\end{rmk}
\section{Decay of $U$ in the weighted space } 
In this section, we focus on the decay estimate of  $U$ in the weighted space $X_a^{M}$ and derive the a  priori estimates for the quantity $\cN_{U}(T)$ defined in \eqref{def-weighted-U}.
The main result of this section is the following:
\begin{prop}\label{prop-weightednorm}
   Let $\cN (T) $ be defined in \eqref{def-cNT}.
   Assume that $(\tilde{c}, \gamma)$ exists in $C([0,T], Y\times Y),$ ${U}$ exists in $C([0,T], X_a^{M}).$ 
      There exists  a constant ${\delta}_3>0,$ 
     such that if $\cN(T)\leq \delta<{\delta}_3,$
    then it holds that 
    \begin{align*}
    \cN_{U}(T)\lesssim 
    \cN_U(0)+\cM_{\tilde{c},\gamma}(0)+\cN(T)^2, 
    \end{align*} 
where $\cM_{\tilde{c},\gamma}(0)$ is defined in \eqref{def-cMcgamma0}.
\end{prop}
\begin{proof}
We will provide a detailed analysis of the estimate for the first quantity $(1+t)\|U(t)\|_{X_a^{M}}$ and then outline the control of the remaining quantities. The proof is split into the following two steps: 

 $\bullet$ Step 1, Low frequency estimate--control of zeroth order weighted norm $X_a^{0}$ by using semigroup estimates,

 $\bullet$ Step 2, High frequency high order estimates using energy estimates. \\[3pt]
Indeed, for low frequencies, we can readily use  the semigroup estimate \eqref{semigroup} for the linearized operator $L_{c_0}$. However, to avoid derivative losses at high frequencies, we need to perform  energy estimates relying on  appropriate energy functionals in order  to capture the stronger damping effect of the constant coefficient part of the system.

 $\bullet$ Step 1, Low frequency estimate.
 By \eqref{eq-U}, $U$ solves the equation
 \begin{align}\label{eqU-rewrite}
 (\pt -L_{c_0})U=\big(L_{c,\gamma}-L_{c_0}-\tilde{c}\p_z\big) U+ R+ N 
 \end{align}
where $L_{c,\gamma}, R, U$ are defined in \eqref{def-Lcgamma}-\eqref{def-NLterm}. 
To estimate the right hand side, we have at first that  
\begin{align*}
  \|(L_{c,\gamma}-L_{c_0}-\tilde{c}\p_z\big) U\|_{X_a^0}  \lesssim \|(\tilde{c},\gamma)\|_{L^{\infty}}\|U\|_{X_a^1}\lesssim (1+t)^{-\f32} \cN_{\tilde{c},\gamma}(T) \cN_U(T).
\end{align*}
Next, 
we deal with the term $R,$ for which we refer to
\eqref{expression-R1}, \eqref{expression-R2} for the definition (recall that $\tilde{R}(t,z,y)=R(t,z+\gamma,y)$). 
Since $(\gamma, \tilde{c})$ has  low frequency support, we have, by using the definition of $\cN_{\tilde{c},\gamma}(T)$ in \eqref{def-decays-modu} and the estimate \eqref{es-ptc} that
\beq\label{es-R-XaM}
  \|R\|_{X_a^{M}}\lesssim \|(\pt \gamma-\tilde{c}, \tilde{c})\|_{L_y^2}+\||\na_y (\tilde{c},\gamma)|^2\|_{L_y^2}\lesssim (1+t)^{-1} 
 \big(\cM_{\tilde{c},\gamma}(0)+\cN(T)^2\big).
 \eeq
 Next, it follows from the definition \eqref{def-NLterm} and the expansion \eqref{expansion-G} that
 \begin{align*}
     \|N\|_{X_a^0}\lesssim  \big(\|\na_y(\gamma, \tilde{c})\|_{L_y^{\infty}}+\|\na (n,\na\psi)\|_{L_{z_1,y}^{\infty}}\big) \|U\|_{X_a^1}+\|G(\phi,\phi_{c,\gamma})\|_{X_a^0}\lesssim (1+t)^{-2}\cN(T)^2.\,
 \end{align*}
Applying the semigroup estimate \eqref{semigroup} for $L_{c_0}$ while setting  $\alpha_{0}=\beta\,\ep^3,$ we obtain that
\begin{align*}
   \|\mathbb{Q}_{c_0}^a  U(t)\|_{X_a^0}&\lesssim e^{-\alpha_0 t} \|\mathbb{Q}_{c_0}^a U(0)\|_{X_a^0}+
\int_0^t e^{-\alpha_0 (t-s)} (1+s)^{-1} \big(\cM_{\tilde{c},\gamma}(0)+\cN^2(s)\big)\, \d s\\
&\lesssim e^{-\alpha_0 t} \|U(0)\|_{X_a^0}+\alpha_0^{-1}(1+t)^{-1}\big(\cM_{\tilde{c},\gamma}(0)+\cN(T)^2\big). 
\end{align*} 
Finally, we derive from the orthogonality condition \eqref{secularcondtion} that 
\beqs 
 \|\mathbb{P}_{c_0}^a  U(t)\|_{X_a^0}\lesssim \|(\tilde{c}, \gamma)(t)\|_{L^2}
\|U(t)\|_{X_a^0}
 \lesssim  \cN_{\tilde{c},\gamma}
 (T)\|U(t)\|_{X_a^0}\leq  \f12  \|U(t)\|_{X_a^0}. 
\eeqs
 as long as $ \cN_{\tilde{c},\gamma}(T)\leq \delta$ is small enough. Consequently, we derive that 
\beq\label{UXa0}
 \|U(t)\|_{X_a^0}\lesssim 2 \|\mathbb{Q}_{c_0}^a  U(t)\|_{X_a^0}\lesssim e^{-\alpha_0 t} \|U(0)\|_{X_a^0}+\alpha_0^{-1}(1+t)^{-1}\big(\cM_{\tilde{c},\gamma}(0)+\cN(T)^2\big).
\eeq

 $\bullet$ Step 2. High order weighted norms.  
 Let us set  $U=(n,\psi)^t$ and  $u=\na\psi,$ then by \eqref{eq-U}, $(n_a, u_a)^t=\colon e^{az}(n, u)^t$ solves
 \begin{align*}
 \left\{
 \begin{array}{l}
    \pt n_a-c_0\, (\p_z-a) n_a+ \div_a\big( \rho \,u_a+ n_a\na\psi_{c,\gamma} \big)=R_a^1,\\[3pt]
       \pt u_a-c_0 \, (\p_z-a) u_a+\tilde{u}\cdot\na_a u_a+u_a\cdot \na \na\psi_{c,\gamma}+e^{az}\na \big(h(\rho)-h(1+n_{c,\gamma})+G(\phi,\phi_{c,\gamma})\big)=\na_a R_a^2, 
 \end{array}\right.
 \end{align*}
     where 
     \beq\label{somenotations}
     R_a=e^{az}R,\quad  \rho= 1+n_{c,\gamma}+n, \quad  \tilde{u}=u+\na\psi_{c,\gamma}, \quad \na_a=(\p_z-a,\na_y)^t,\quad \div_a F=\na_a \cdot F.
     \eeq

 For $K$ large enough, we introduce a cut-off function $\tilde{\chi}_{_K}: \mR^3\ni (\xi,\zeta)\rightarrow \tilde{\chi}_{_K}(\xi,\zeta)\in  \mR,$ which vanishes on $B_K(0)$ and equals to $1$ on $B_{2K}^c(0).$ Define $\tilde{\chi}_{_{K,M}}=:\tilde{\chi}_{_K} (\xi^2+\zeta^2)^{{M}/{2}}.$ Applying  $\tilde{\chi}_{_{K,M}}$ on  the above system  for $(n_a, u_a)^t,$ we find that $(\dot{n}_a, \dot{u}_a)^t=\colon \tilde{\chi}_{_{K,M}}(n_a, u_a)^t$ solves the following system
 \beq\label{eqdotnua}
 \begin{aligned}
\pt \left( \begin{array}{c}    \dot{n}_a \\ \dot{u}_a  \end{array}\right)
 +
   \left( \begin{array}{cc}
   - c_0\, (\p_z-a)+
   \tilde{u}\cdot\na_a & \div_a\big( \rho \,\cdot )\,   \\ 
 \na_a \sigma_a^2(D) +\big( h'(\rho)-h'(1)\big)\na_a
  & -c_0\, (\p_z-a)+ \tilde{u}\cdot\na_a
   \end{array}
   \right) 
    \left(
    \begin{array}{c}
        \dot{n}_a \\
    \dot{u}_a
    \end{array}\right)=\cC\\[3pt]
 \end{aligned}
 \eeq
 where $\sigma_a(D):=\sqrt{h'(1)+(\Id-\Delta_a)^{-1}}$ 
 and
 $\cC=(\cC^1, \cC')^t=:\sum_{j=0}^3 \cC_j$ with
 \beq\label{def-cc}
 \begin{aligned}
  &   \cC_0=\tilde{\chi}_{_{K,M}} \left(
    \begin{array}{c}
        R_a^1 \\[3pt]
    \na_a R_a^2
    \end{array}\right), \quad \cC_1= \tilde{\chi}_{_{K,M}} \left(
    \begin{array}{c}
    n_a\,\Delta \psi_{c,\gamma} \\[3pt]
     u_a \cdot \na \na \psi_{c,\gamma}
    \end{array}\right), \\
  &  \cC_2=\tilde{\chi}_{_{K,M}} \left(
    \begin{array}{c}
        0 \\
     e^{az} \bigg(\na G(\phi,\phi_{c,\gamma}) -(h'(\rho)-h'(1+n_{c,\gamma}))\na n_{c,\gamma} \bigg)
    \end{array}\right), \\
   & \cC_3= \left(
    \begin{array}{c}
  \div_a \big([\tilde{\chi}_{_{K,M}}, n_{c,\gamma}] u_a\big)+[\tilde{\chi}_{_{K,M}}, 
   \tilde{u}]\na_a n_a+[\tilde{\chi}_{_{K,M}}, n]\div_a u_a+\na n\cdot \dot{u}_a\\[3pt]
 \, [\tilde{\chi}_{_{K,M}},  \tilde{u}]\na_a u_a  
    \end{array}\right) .
 \end{aligned}
 \eeq
Let us define the  energy functional $$\cE(\dot{n}_a,  \dot{u}_a):=\f12 \int |\sigma_a(D)\dot{n}_a|^2+ \big(h'(\rho)-h'(1)\big)|\dot{n}_a|^2+ \rho |\dot{u}_a|^2\, \d z \d y.$$
 We find by using the equation \eqref{eqdotnua} the following energy identity
 \begin{align}\label{def-cJ15}
    \pt \,\cE(\dot{n}_a,  \dot{u}_a)=\cJ_1+\cdots \cJ_5\, ,
\end{align}
where
\begin{align*}
    \cJ_1&=-2a c_0 \,\cE(\dot{n}_a,  \dot{u}_a)- \,\Re \int 2a\,\big(\sigma_a^2(D)\dot{n}_a+\big(h'(\rho)-h'(1)\big)\dot{n}_a\big) \cdot \rho \,\dot{u}_a \\
    & \qquad \qquad \qquad \qquad \qquad \qquad\, +\overline{\sigma_a(D)\dot{n}_a} \big(\sigma_a(D)-\sigma_{-a}(D)\big)\, \div_a (\rho \,\dot{u}_a) \,\d z \d y\, ,
   \\
    \cJ_2&= \,  \Re \int \overline{\sigma_a(D)\dot{n}_a} \,\,[\sigma_a, \tilde{u}]\na_a \dot{n}_a  \, \d z \d y\,,\\
\cJ_3&=\f{1}{2} \int (\div \tilde{u} +2a\, \tilde{u} )\,|\sigma_a(D) \dot{n}_a|^2+\bigg(\div\big((h'(\rho)-h'(1))\tilde{u}\big)+2a (h'(\rho)-h'(1))\tilde{u}-c_0\p_z h'(\rho)\bigg)|\dot{n}_a|^2\\
&\qquad \qquad \qquad \qquad\qquad \qquad+\big(\div\big(\rho\,\tilde{u}\big)+2a \rho\,\tilde{u}-c_0\p_z\rho\big)\,|\dot{u}_a|^2  \,\d z \d y\,, \\
\cJ_4&=\, \Re\int 
\big(\sigma_a^2(D)\dot{n}_a+\big(h'(\rho)-h'(1)\big)\dot{n}_a\big) \div (\rho \,\dot{u}_a )+\rho \,\dot{u}_a \cdot \big(\na \sigma_a^2(D)\dot{n}_a+\big(h'(\rho)-h'(1)\big)\na\dot{n}_a\big)\,\d z \d y\, ,\\
\cJ_5&= \Re \int
\bigg[\overline{\sigma_a(D) \dot{n}_a}\, \sigma_a(D) \cdot+ (h'(\rho)-h'(1)){\dot{n}_a}\cdot\bigg] \cC^1+ \rho \,\dot{u}_a \cdot \cC' \,\,\d z \d y\,.
\end{align*}

We now control $\cJ_1-\cJ_5$ term by term. As in the study of resolvent estimate for the linearized equation in  Section \ref{subsection-Ur}, we expect to detect some damping effects from the term $\cJ_1$ while the remaining terms can be considered as  remainders.  At first, by the Cauchy-Schwarz inequality, we have 
\beqs
\cJ_1\leq -2 a \cE(\dot{n}_a,  \dot{u}_a)\bigg(c_0-\sup_{\mathrm{x}\in \mR^3,\, (\xi,\zeta)\in B_{K}^c(0)}
\sqrt{\rho(\mathrm{x})}\,\big|\sigma_a(\xi,\zeta)\big|+\sqrt{\rho(x)(h'(\rho)-h'(1))}+a^{-1}\big|\Im \sigma_a (\xi,\zeta)\big|\big|(\xi+ia,\zeta)\big|\bigg). 
\eeqs
Since $\|n_{c,\gamma}\|_{L_{\mathrm{x}}^{\infty}}=\cO(\ep^2), \|n\|_{L_{\mathrm{x}}^{\infty}}\leq \cN_U(T)\leq \delta,$ it holds that
\beqs 
\sqrt{\rho(\mathrm{x})}\leq 1+\|n_{c,\gamma}\|_{L_{\mathrm{x}}^{\infty}}+\|n\|_{L_{\mathrm{x}}^{\infty}}\leq 1+C\ep^2+\delta, \quad 
\sqrt{h'(\rho)-h'(1)}=\cO(\delta+\ep^2). 
\eeqs
Moreover, let $J=(1+\xi^2-a^2+|\zeta|^2)^2+4a^2\xi^2.$ Choosing $K$ sufficiently large such that for any $|(\xi,\zeta)|\geq K,$
\beqs 
    \big|\sigma_a(\xi,\zeta)\big|\leq \sqrt{h'(1)+J}\leq \sqrt{h'(1)+{1}/{20}\,},\quad a^{-1}\big|\Im \sigma_a(\xi,\zeta)\big|\big|(\xi+ia,\zeta)\big|\leq\f{|\xi|\big|(\xi+ia,\zeta)\big|}{J\Re\sigma_a}\leq \f{1}{40 \sqrt{h'(1)}},
\eeqs
it then holds that 
\begin{align*}
    \cJ_1\leq  -2 a \cE(\dot{n}_a,  \dot{u}_a)\bigg(c_0-\sqrt{h'(1)+1/10}+\cO(\delta+\ep^2) \bigg).
\end{align*}
From Theorem \ref{thm-existence}, we have  $c_0>\sqrt{h'(1)+1}.$ Therefore, by choosing $\delta$ and $\ep$ small enough,  we can find a uniform  constant $\kappa_0>0$ 
that is independent of $\delta, \ep,$ such that 
\begin{align}\label{es-cJ1}
    \cJ_1 \leq -2\,a \,\kappa_0\, \cE(\dot{n}_a,  \dot{u}_a)\,.
\end{align}
It remains to control the other terms $\cJ_2-\cJ_5.$ 
To continue, it is useful to have an estimate for $\|\tilde{u}(t)\|_{L_t^{\infty}W_{\mathrm{x}}^{1,\infty}},$ which, however, cannot be obtained from the Sobolev embedding since we do not expect the boundedness of $u$ in the usual Sobolev space.
 However, by the definition \eqref{somenotations}
 and relation \eqref{identiy-import-intro}, we can write further
$$\tilde{u}_1=u_1+\psi_{c}'(z_1)=\tilde{v}_1+w_1+\psi_{c}'(z_1), \qquad \tilde{u}_j=\tilde{v}_j+w_j-\psi_c'(z_1)\p_{y_{j-1}}\gamma, \,\, (j=2,3),$$
where $\tilde{v}=\na\tilde{\vp}$ and $\tilde{\vp}, w$ are defined in \eqref{def-tvp-intro},  \eqref{def-w-intro}. 
From these relations we obtain that 
\begin{align}\label{Lip-tildeu}
\|\tilde{u}(t)\|_{W_{\mathrm{x}}^{1,\infty}}\lesssim \|\na (\tilde{v}, w)(t)\|_{H^2}+\|\psi_{c}'\|_{L_{\mathrm{x}}^{\infty}}(1+\|\na_y\gamma(t)\|_{L_y^{\infty}})\lesssim \cN(T)+\ep^2\lesssim (\delta+\ep^2).
\end{align}
We are now in position to estimate $\cJ_2-\cJ_5.$
Applying the commutator estimate \eqref{commutator-crude} and  \eqref{Lip-tildeu},  we get that
\beq \label{es-cJ2}
|\cJ_2|\lesssim \|\nabla \tilde{u}\|_{L_{\mathrm{x}}^{\infty}} \|\dot{n}_a\|_{L^2}^2 \leq  (\ep^2+\delta)\|\dot{n}_a\|_{L^2}^2.
\eeq
Moreover, in view of the expression of $\cJ_3,$
it holds that 
\beq\label{es-cJ3}
|\cJ_3|\lesssim \big(\| \tilde{u}\|_{W_{\mathrm{x}}^{1,\infty}}\big(1+\|(n_{c,\gamma}+n)\|_{W_{\mathrm{x}}^{1,\infty}}\big) +\|(n_{c,\gamma}+n)\|_{W_{\mathrm{x}}^{1,\infty}}\big)  \cE(\dot{n}_a,  \dot{u}_a)\lesssim (\ep^2+\delta)\cE(\dot{n}_a,  \dot{u}_a).
\eeq
Next, integration by parts and the Cauchy-Schwarz inequality yield that 
\beqs 
|\cJ_4|\lesssim \|(n_{c,\gamma}+n)\|_{W_{\mathrm{x}}^{1,\infty}}(\|\dot{n}_a\|_{L^2}\|\dot{u}_a\|_{L^2} )\lesssim (\ep^2+\delta)\cE(\dot{n}_a,  \dot{u}_a).
\eeqs
Finally, the last term $\cJ_5$  can be bounded by
$
\sqrt{\cE(\dot{n}_a,  \dot{u}_a)} \,\|\cC\|_{L^2},$ we are thus left to estimate $\|\cC_j\|_{L^2}$ for $j=0,1,2,3$ which are defined in \eqref{def-cc}. 
First, thanks to \eqref{es-R-XaM}, it holds that for any $t\in [0,T],$
\beq\label{es-cC0}
\|\,\cC_0(t)\|_{L^2}\lesssim \|R\|_{X_a^M}\lesssim (1+t)^{-1}\big(\cM_{\tilde{c},\gamma}(0)+\cN(T)^2\big).
\eeq
The next term $\cC_1$ can be controlled roughly by 
\begin{align}\label{es-cC1}
    \|\cC_1(t)\|_{L^2}\lesssim \|(n_a, u_a)\|_{H^M} \|\nabla\nabla \psi_{c,\gamma}\|_{W^{M,\infty}}.
\end{align}
Moreover, by the expansion \eqref{expansion-G} and the fact that  $I_{\phi_{c,\gamma}}$ maps $L_a^2$ to $H_a^2,$ we have 
\beq\label{es-cC2}
\begin{aligned}
    \|\cC_2(t)\|_{L^2}&\lesssim \|G(\phi, \phi_{c,\gamma} )\|_{H_a^{M+1}}+\|n\|_{H_a^M}\|\na n_{c,\gamma}\|_{W^{M,\infty}}\\
    & \lesssim \|n_a\|_{H^{M}} \big(\|n\|_{W^{[\f{M}{2}],+\infty}} + \|\na n_{c,\gamma}\|_{W^{M,\infty}} \big)+ (1+t)^{-\f32} \cN(T)^2.
\end{aligned}
\eeq
Finally, thanks to the commutator estimate \eqref{commutator-clasical} and the Sobolev embedding, we obtain
\begin{align*}
    \|[\tilde{\chi}_{_{K,M}},  \tilde{u}]\na_a f\|_{L^2}\lesssim \big(\|\na\psi'_{c,\gamma}\|_{W^{M-1,\infty}}+\|\na(u_1, \tilde{u}_2,\tilde{u}_3 )\|_{H^{M-1}}\big)\|f\|_{H^M}.
\end{align*}
We thus get for any $M\geq 3,$
\begin{align}\label{es-cC3}
     \|\cC_3(t)\|_{L^2}&\lesssim \big(\|\na (n_{c,\gamma}, \psi'_{c,\gamma})\|_{W^{M,\infty}}+\|\na (n, u_1, \tilde{u}_2,\tilde{u}_3)\|_{H^{M-1}}
     \big)\|(u_a,n_a)\|_{H^M}.
\end{align}
Combining \eqref{es-cC0}-\eqref{es-cC3}, we find that for $M\geq 3,$
\begin{align*}
\|\cC(t)\|_{L^2}&\lesssim \big(\|\na (n_{c,\gamma}, \psi'_{c,\gamma})\|_{W^{M+1,\infty}}+\|\na (n, u_1, \tilde{u}_2,\tilde{u}_3)\|_{H^{M-1}}\big)\|(n_a, u_a)\|_{H^M}+ (1+t)^{-1}\cN(T)^2\\
&\lesssim (\ep^2+\delta)\big(\|(n_a, u_a)\|_{L^2}+\sqrt{\cE(\dot{n}_a,  \dot{u}_a)} \big)+(1+t)^{-1}\cN(T)^2\, .
\end{align*}
This, together with \eqref{es-cJ1}-\eqref{es-cJ2}, yields that
\beqs 
\sum_{k=1}^5 \cJ_k \leq -\big(2a \kappa_0- C(\ep^2+\delta)\big)\cE(\dot{n}_a,  \dot{u}_a)+C\big(\|U\|_{X_a^0}+(1+t)^{-1}\cN(T)^2\big)\sqrt{\cE(\dot{n}_a,  \dot{u}_a)}\,.
\eeqs
Recall that the constant $a,$  which appears in the definition of the weighted space, is chosen to be of order  $\cO(\ep)$ by Theorem \ref{thm-resol-L}. Consequently, we find, by choosing $\ep$ and $\delta$ small enough that
\beqs 
\pt \sqrt{\cE(\dot{n}_a,  \dot{u}_a)}+ \f{a \kappa_0}{2}\,\sqrt{\cE(\dot{n}_a,  \dot{u}_a)}\lesssim \|U\|_{X_a^0}
+(1+t)^{-1}\cN(T)^2.
\eeqs
Using the Gr\"onwall inequality, we achieve that 
\begin{align*}
\|\tilde{\chi}_{_{K}}(D) U(t)\|_{X_a^M}\approx \sqrt{\cE(\dot{n}_a(t),  \dot{u}_a(t))}&\lesssim e^{-\f{a\kappa_0}{2} t}  \|U(0)\|_{X_a^M}+a^{-1}\big(\sup_{0\leq s\leq t}\|U(s)\|_{X_a^0}+(1+t)^{-1}\cN(T)^2\big). 
\end{align*}
This, together with \eqref{UXa0}, leads to 
\beqs 
     \|U(t)\|_{X_a^M}\lesssim e^{-\min\{\f{a\kappa_0}{2}, \alpha_{0}\}t} \|U(0)\|_{X_a^M}+ (a\alpha_0)^{-1}(1+t)^{-1}\big(\cM_{\tilde{c},\gamma}(0)+\cN(T)^2\big).
\eeqs
This ends the proof of  the estimate of $ \|U(t)\|_{X_a^M}.$ 

In the following, we sketch the estimates of $ \|y U(t)\|_{X_a^M}, \, \|\nabla_y U(t)\|_{X_a^{M-1}}, \, \|y\nabla_y^2 U(t)\|_{X_a^{M-1}},$ whose decay essentially depend on those of the source terms $y R,$ $\partial_y R, y\p_y^2 R$ in the weighted space.
As in the estimate of  $ \|U(t)\|_{X_a^M},$ we need to derive the equations for $yU, \nabla_y U$ and $ y\nabla^2 U, $ then use  the semigroup estimate \eqref{semigroup} for low order  norms, while performing  energy estimates for high-frequencies and the high order norm. Since the high frequency estimates would be very similar to that of $U(t),$ we focus 
on the estimate in the weighted  energy space.

First, by using the equation \eqref{eqU-rewrite}, we obtain
\begin{align*}\label{}
& (\pt -L_{c_0})(y_jU)=\left(\begin{array}{c}
      \big(1+n_{c_0}\big)\, \p_{y_j}\psi  \\
       2 I_{\phi_{c_0}}\p_{y_j} n
 \end{array} \right)+y_j\big(L_{c,\gamma}-L_{c_0}-\tilde{c}\p_z\big) U+y_j R+ y_j N , \\
 &(\pt -L_{c_0})(\nabla_{y}U)=\na_y \big(L_{c,\gamma}-L_{c_0}-\tilde{c}\p_z\big) U+\na_y R+ \na_y N .
 \end{align*}
Again, since $I_{\phi_{c_0}}$ maps $L_a^2$ to $H_a^2,$ 
it holds  that 
\begin{align*}
\bigg\|\left(
    \begin{array}{c}
      \big(1+n_{c_0}\big)\, \p_{y_j}\psi  \\
       2 I_{\phi_{c_0}}\p_{y_j} n
 \end{array} \right)\bigg\|_{X_a^0}\lesssim \|U\|_{X_a^0}\lesssim (1+t)^{-1} \big(\|U(0)\|_{X_a^0}+\cN(T)^2\big).
\end{align*}
Moreover, by using the definition of the norm $\cN(T)$, we have the bounds
\begin{align*}
  & \| y \big(L_{c,\gamma}-L_{c_0}-\tilde{c}\p_z\big) U\|_{X_a^0}\lesssim \|\big(\gamma, \tilde{c}\big)\|_{L_y^{\infty}}\big(\|y U\|_{X_a^0}+\|U\|_{X_a^0}\big)\lesssim (1+t)^{-1} \cN(T)^2, \\
  &  \| \na_y \big(L_{c,\gamma}-L_{c_0}-\tilde{c}\p_z\big) U\|_{X_a^0}\lesssim \|\big(\gamma, \tilde{c}\big)\|_{L_y^{\infty}}\|\na_y U\|_{X_a^1}+\|\na_y\big(\gamma, \tilde{c}\big)\|_{L_y^{\infty}}\|U\|_{X_a^1}\lesssim (1+t)^{-2} \cN(T)^2\, .
\end{align*}
Next, thanks to \eqref{es-yptc}, \eqref{es-nablayptc} we have  
\begin{align*}
    \|y R \|_{X_a^0}\lesssim (1+t)^{-\f12} \cN(T)^2, \quad 
     \|\na_y R \|_{X_a^0}\lesssim (1+t)^{-\f32} \cN(T)^2.
\end{align*}
Finally, from the same arguments as in the estimate of  $N,$ it holds that 
\begin{align*}
    \|y N\|_{X_a^0}\lesssim  (1+t)^{-\f32} \cN(T)^2,\quad \|\na_y N\|_{X_a^0}\lesssim  (1+t)^{-2} \cN(T)^2.
\end{align*}
The above estimates, in conjugation with the semigroup estimate \eqref{semigroup}, enable us to find that
\beq\label{es-Qc0U}
\begin{aligned}
    & \|\mathbb{Q}_{c_0}^a  (y U)(t)\|_{X_a^0}\lesssim e^{-\alpha_0 t} \|\mathbb{Q}_{c_0}^a (y U)(0)\|_{X_a^0}
+(1+t)^{-\f12}\cN(T)^2, \\
&\|\mathbb{Q}_{c_0}^a  \na_y U(t)\|_{X_a^0}\lesssim e^{-\alpha_0 t} \|\mathbb{Q}_{c_0}^a \na_y U(0)\|_{X_a^0}+(1+t)^{-\f32}\cN(T)^2.
\end{aligned}
\eeq
It now remains to  derive estimates for  $\mathbb{P}_{c_0}^a  (y U), \mathbb{P}_{c_0}^a(\na_y U).$ 
Let us set  $$C_{k}[f](\zeta)=\int_{\mR^3} f(z,y) e^{-i y\cdot \zeta} g_{k}^{*}(z, |\zeta|, c_0) \, \d z \d y .$$
Applying  $D_{\zeta}=\f{1}{i}{\p_{\zeta}}$ on the orthogonality condition \eqref{secularcondtion}, we find that for any $|\zeta|\leq \eta_0,$
\begin{align*}
   & C_k[y U](\zeta)= \int_{\mR^3} U (t, z, y) (D_{\zeta} g_k^{*})\big(z-\gamma, |\zeta|, c(t,y)\big) e^{-iy\cdot\zeta}\, \d z \d y \\
    &-\int_{\mR^3} yU(t, z, y)\bigg( \gamma(y)\int_0^1 \p_zg_{k}^{*}(z-\theta\gamma,|\zeta|,c)\d \theta +\tilde{c}(y)\f{g_{k}^{*}(z-\gamma,|\zeta|,c)-g_{k}^{*}(z-\gamma,|\zeta|,c_0)}{\tilde{c}}\bigg)e^{-iy\cdot\zeta} \d z \d y.
\end{align*}
Using the evenness of $g_{k}^{*}$ in $\eta,$ and the Parseval identity,  we find 
\begin{align*}
 \|\mathbb{P}_{c_0}^a  (y U) \|_{X_a}\lesssim   \| C_k[y U](\zeta)\|_{L_{\zeta}^2(B_0(\eta_0))}&\lesssim |\eta_0| \|U\|_{X_a}(1+\|(\tilde{c},\gamma)\|_{L_y^2})+\|yU\|_{X_a} \|(\tilde{c}, \gamma) \|_{L^2}\\
 &\leq \f{1}{2}\big(\|U\|_{X_a}+\|y U\|_{X_a}\big)
\end{align*}
as long as $\eta_0$  and $\cN_{\tilde{c},\gamma}(T)\leq \delta$ are small enough. Similarly, by multiplying by  $i\zeta$  the orthogonality conditions \eqref{secularcondtion}, we also obtain  by assuming that  $\cN_{\tilde{c},\gamma}(T)$ is  small enough that 
\begin{align*}
     \|\mathbb{P}_{c_0}^a  (\na_y U) \|_{X_a}\lesssim   \| C_k[\na_y U](\zeta)\|_{L_{\zeta}^2(B_0(\eta_0))}&\lesssim  \|U\|_{X_a}\|\na_y (\tilde{c},\gamma)\|_{L_y^2}+\|\na_y U\|_{X_a} \|(\tilde{c}, \gamma) \|_{L^2}\\
 &\leq \f{1}{2}\big(\|U\|_{X_a}+\|\na_y U\|_{X_a}\big).
\end{align*}
We thus deduce that $ \|(y U) \|_{X_a}\lesssim \|\mathbb{Q}_{c_0}^a  (y U)\|_{X_a}+\|U\|_{X_a},\, \|\na_y U \|_{X_a}\lesssim \|\mathbb{Q}_{c_0}^a  \na_y U\|_{X_a}+\|U\|_{X_a}. $  Combined with \eqref{es-Qc0U}, this yields  the desired estimates for $yU$ and $\na_y U.$
\end{proof}

\section{Energy estimates in the usual Sobolev space--Further decomposition}\label{energyes}
As mentioned in the introduction, the quantity 
$U$ cannot be bounded in the usual Sobolev space 
$H^M$ because of the nontrivial dependence of  $\psi_c(\cdot)$ on
$c$ and the fact that $\psi_{c}$ does not tend to zero at $-\infty$.  Indeed, the equation for  $\na_y\psi$ involves a source term $c_t \na_y c\, \p_c^2 \psi_c,$ which unfortunately does not belong to the usual Sobolev space $H^M$ since $\p_c^2\psi_c$ does not vanish at $-\infty.$  To solve the problem, we need to perform 
a  decomposition different from \eqref{decomp}. As proposed in the introduction,
we write 
\begin{align*}
 \Psi(t,x, y)=\left(\begin{array}{c}
          n^i \\
          \psi^i
 \end{array}\right)= \left(\begin{array}{c}
          n_{c(t,y)} \big(z-\gamma(t,y)\big)  \\[3pt]
          \underline{\psi_c}(t, z, y)
 \end{array}\right)+\left(\begin{array}{c}
          \varrho \\
          \vp
 \end{array}\right)(t,x,y), \quad (z=x-c_0t),
\end{align*}
where $ \underline{\psi_c}(t, z, y)=\colon \psi_{c_0}\big(z-\gamma(t,y)\big)+\widetilde{\psi}(t,z,y),$ $\widetilde{\psi}$ being defined in the following way.
Consider a divergence-free vector function 
\beq \label{def-w}
w=-\curl (-\Delta)^{-1}\curl \bigg( \psi_{c}'(\cdot-\gamma)\left(\begin{array}{c}
          1 \\
    -\na_y\gamma
 \end{array}\right)\bigg).
\eeq 
From this choice, we immediately get   that $ (\psi_{c}'-\psi_{c_0}')(\cdot-\gamma)(1,-\na_y\gamma)^t+w$ is curl-free and  belongs to $C([0,T],L^2(\mR^3)).$ Therefore, one can find $\widetilde{\psi}\in C([0,T],\dot{H}^1(\mR^3))$ such that  
\begin{align*}
\na \widetilde{\psi}(t, z,y) = (\psi_c'-\psi_{c_0}')(z-\gamma) \left(\begin{array}{c}
          1 \\
    -\na_y\gamma
 \end{array}\right)+w(t,z,y).
\end{align*}
It follows from this definition that
\begin{align*}
 \underline{v_c}(t,z,y) :=  \na \underline{\psi_c}(t, z, y)=\underline{u_c}(t, z-\gamma, y)+w(t,z,y), \quad \text{ with }\, \,\,\underline{u_c}(t,z_1,y)=:\psi_c'(z_1) \left(\begin{array}{c}
          1 \\
    -\na_y\gamma
 \end{array}\right).
\end{align*}
By construction, it also  holds  that 
\begin{align*}
    w(t,z,y)+\na\vp(t, z+c_0t, y )=\na\psi(t, z, y)+\p_c \psi_c(z-\gamma) \left(\begin{array}{c}
          0 \\
    \na_y\, c
 \end{array}\right),
\end{align*}
where $\psi(t, z, y)=\psi^i(t, x, y)-\psi_c(z-\gamma).$
We thus have by applying Propositions \ref{prop-modulation} and \ref{prop-weightednorm} that 
\begin{align}\label{w+v-weighted}
   \| \big(w+\na\vp(\cdot+c_0 t)\big)(t)\|_{H_a^N}\lesssim (1+t)^{-1} \big(\cM(0)+\cN(T)^2+ \cN(T)^4\big), \qquad \forall \,t\in [0,T]. 
\end{align}

Let us set  $z_1=x-c_0t -\gamma.$ By using the profile equations
 $$\big(-c n_c'+\p_x \big((1+n_c)\psi_c'\big)\big)(z_1)=0, \quad 
\big(-c\psi_c'+\f{|\psi_c'|^2}{2}+h(1+n_c)-\phi_c\big)(z_1)=0,$$ we find that $\varrho$ and $ v=\colon \na \vp$ solve
\beq\label{eq-vr-v}
\left\{ 
\begin{array}{l}
  \pt \varrho+\div \big(\rho \, v+\vr\, \underline{v_c}(z) 
  \big) = r_1^0+r_2^0 ,    \\[3pt]
 \pt v+  \big(\underline{v_c}(z)+v\big)\cdot\na v+v\cdot \na \underline{v_c}(z)+ \na \big(h(\rho)-h(1+n_c(z_1))\big)+\na (\phi-\phi_c(z_1))=r_1'+r_2',\\[3pt]
 \Delta \phi=e^{\phi}-\rho,
\end{array}
\right.
\eeq
where $\rho(t, x, y)=1+n_c(z_1)+\varrho(t,x,y)$ and 
\beq\label{def-source-sob}
\begin{aligned}
  &  r_1^0=-\pt c\, (\p_c n_c )(z_1)+(\pt \gamma-\tilde{c}) n_c'(z_1), \quad r_2^0=\div_y (\psi_c'(z_1)\na_y \gamma)-\div \big(n_c(z_1)w(z)\big), \\
  &  r_1'= -(\pt-c_0\p_z)w(z)-\left( \begin{array}{c}
      \pt c\, \p_c \psi_c'(z_1)-(\pt \gamma-\tilde{c})\psi_c''(z_1)    \\[3pt]
          (\pt c\,\p_c \psi_c'(z_1)-\psi_c'\pt)\na_y\gamma+\na_y(c\, \psi_{c}')
    \end{array}\right),\\
& r_2'=-\nabla \bigg((\underline{u_c}\cdot w)(z)+\f{1}{2}|w|^2 (z)+\f12|\psi_c'|^2(z_1)|\na_y\gamma|^2\bigg).
\end{aligned}
\eeq
 With slight abuse of notation, we sometimes denote a four-component vector as $r=(r^0, r')$ where $r'\in \mR^3.$

In this section and the following ones, our aim is to get a priori  estimates for $(\vr, v)$, as depicted by the following quantities:
\beq\label{def-norms-vr-v}
\begin{aligned}
    \cN_{(\vr,v)}&(T):= \sup_{0\leq t\leq T} \bigg((\log(2+t))^{-1}\|(\vr,v)(t)\|_{H^M}+\|\na(\vr,v)(t)\|_{H^{M-1}}\\
   & +(1+t)^{-\vartheta} \|(\vr, v)\|_{\dot{H}^{-\f12}}+(1+t)^{(\f{3}{4}-\f2p-\vartheta)}\|(\vr, v)\|_{\dot{W}^{-1,p}}
 +(1+t)^{(1+\iota)} \big\||\na|^{\f{3}{8_{\kappa,1}}} \nabla(\vr, v)\big\|_{W^{1,8_{\kappa}}}\bigg),
\end{aligned}
\eeq
where $M\geq 12$ and
$\kappa, \iota, \vartheta, 8_{\kappa}, 8_{\kappa,1}$ will be precised in the next section. 
We also define the following norms for the initial data 
\begin{align*}
   \cM(0)=\cN_U(0)+\cM_{\tilde{c},\gamma}(0)+ \cM_{(\vr,v)}(0), 
\end{align*}
where $\cM_{\tilde{c},\gamma}(0)$ is defined in \eqref{def-cMcgamma0} and 
\beqs 
\cM_{(\vr,v)}(0):= \|(\vr, v)(0)\|_{H^{M}\cap \dot{H}^{-1}}+\|(\vr, v)(0)\|_{W^{\f{29}{4},1}}.
\eeqs

In the current section, we focus on the propagation of the energy norm listed in the first line of \eqref{def-norms-vr-v}. 
Before doing so, it is useful to collect some properties for the auxiliary function $w.$ 
\begin{prop}\label{prop-w}
    Let $w$ be defined in \eqref{def-w}.  Assume that $(\tilde{c}, \gamma)$ exists in $C([0,T], Y\times Y).$
   Then for any $t\in [0,T],$ the following estimates hold true: for any $p\in[2,+\infty),$  
    \beq\label{es-w}
      \begin{aligned}
    &(1+t)^{(\f76-\f1p)}\|w\|_{W^{M+1,p}} 
    +(1+t)^{\f53} \|w_1\|_{H^{M+1}}+ (1+t)^{\f76}  \|(\pt,\na_y)\, w(t)\|_{H^M}\\
    &\quad+(1+t)\|\na w(t)\|_{H^M} +(1+t)^{\f32}\|(\pt, \na)(\na_y w, w_1)\|_{H^{M-1}}\lesssim \cN_{\tilde{c},\gamma}(T)+\cN(T)^2. 
 \end{aligned}
 \eeq
\end{prop}
\begin{proof}
    We compute 
    \beq \label{comp-curlF}
    \curl \bigg( \psi_{c}'(\cdot-\gamma)\big(1,-\na_y\gamma\big)^t \bigg)=\p_c \psi_c'(\cdot-\gamma)\bigg(\p_{y_2}c\,\p_{y_1}
    \gamma-\p_{y_1}c\,\p_{y_2}\gamma,\p_{y_2}c,-\p_{y_1}c\bigg)^t .
     \eeq 
Since $(\tilde{c},\gamma)(y)$ has compact Fourier support and $\psi_c'$ is smooth, it follows from the definition of $w$ in \eqref{def-w} and an application of the Hardy–Littlewood–Sobolev inequality that 
    \begin{align*}
   &\|w(t)\|_{W^{M+1,p}}  \lesssim \|(\na_y c, \, \p_{y_2}c\,\p_{y_1}
    \gamma)\|_{L_y^{{1}/({\f13+\f1p})}} .
    \end{align*}
Consequently, by applying the interpolation inequalities 
\begin{align*}
&\|f\|_{L_y^{{1}/({\f13+\f1p})}}\lesssim \|f\|_{L_y^2}^{\theta_p}\|yf\|_{L_y^2}^{1-\theta_p}, \qquad \big(2\leq p\leq 6,\, \theta_p=\f43-\f2p\big), \\
&\|f\|_{L_y^{{1}/({\f13+\f1p})}}\lesssim \|f\|_{L_y^2}^{\vartheta_p}\|\na_y f\|_{L_y^2}^{1-\vartheta_p}, \qquad \big(6\leq p< +\infty,\, \vartheta_p=\f23+\f2p\big) ,
\end{align*} 
and using  the definition of $\cN_{\tilde{c},\gamma}(T)$ in \eqref{def-decays-modu}, we obtain that 
\beqs
\|w\|_{W^{M+1,p}}\lesssim (1+t)^{-(\f76-\f1p)} \big(\cN_{\tilde{c},\gamma}(T)+\cN(T)^2\big).
\eeqs
Similarly, we derive from the estimates 
\beqs 
\|w_1(t)\|_{H^{M+1}}  \lesssim \|\p_{y_2}c\,\p_{y_1}  \gamma\|_{L_y^{6/5}}, \quad \|(\pt,\na_y) w(t)\|_{H^{M}}  \lesssim \|(\pt,\na_y) (\na_y c, \, \p_{y_2}c\,\p_{y_1} \gamma)\|_{L_y^{6/5}} \, ,
\eeqs
 the control of $w_1$ and $\pt w, \na_y w.$ 
 Moreover,
it follows from  standard elliptic regularity theory that 
  \beqs 
\|\p_z w\|_{H^{M+1}}\lesssim \|(\na_y c, \, \p_{y_2}c\,\p_{y_1}\gamma)\|_{L^2}\lesssim (1+t)^{-1} \big(\cN_{\tilde{c},\gamma}(T)+\cN(T)^2\big).
\eeqs
The estimates for the remaining terms on the right-hand side of \eqref{es-w}  
follow from  similar arguments, and we omit the details.
\end{proof}

By using the above  proposition and Proposition \ref{prop-modulation}, we get  the following properties for the source terms $r_1=(r_1^0, r_1')^t, r_2=(r_2^0, r_2')^t$ in \eqref{eq-vr-v}.
\begin{cor}\label{cor-sourceterm}
Under the same assumptions as in  Proposition \ref{prop-w}, it holds that
\beq\label{es-r1r2}
(1+t)\|(r_1,r_2)\|_{H^M}+(1+t)^{\f32}\|(\na_y, \pt)(r_1,r_2)\|_{H^{M-1}}\lesssim \cN_{\tilde{c},\gamma}(T)+\cN(T)^2\, .
\eeq
\end{cor}
\begin{proof}
It suffices to show that each term of $r_j^0, r_j' (j=1,2)$ defined in \eqref{def-source-sob} enjoys  the desired estimates. 
This  is a  consequence of the definition of $\cN_{\tilde{c},\gamma}(T)$ in \eqref{def-decays-modu},  the estimates \eqref{es-ptc}, \eqref{es-nablayptc} and Proposition \ref{prop-w}. 
For instance, we see that the term
$$\div(n_c(z_1)w)=n_c'(z_1)\, w_1+(\p_c n_c \na_y c-n_c'\na_y\gamma)\cdot (w_2, w_3)^t,$$ enjoys the desired property thanks to Proposition \ref{prop-w}. The other terms can be checked similarly, we omit the details.
\end{proof}

We are now in position to perform the energy estimates for  $(\varrho, v)^t.$
\begin{prop}\label{prop-energy}
Let $\cN(T)$ be defined in \eqref{def-cNT}.
Suppose that $(\vr, v)\in C([0,T], H^M)$ and 
$\cN(T)\leq \delta.$ There exists $\delta_4$ small enough, such that for any $\delta\in (0,  \delta_4],$
it holds that 
\begin{align}
    &\sup_{0\leq t\leq T}(\log(2+t))^{-1}\|(\vr, v)(t)\|_{H^M}\lesssim 
    \cM(0)+ \cN(T)^2, \label{sobes-1}\\
   & \sup_{0\leq t\leq T}\|\na (\vr, v)(t)\|_{H^{M-1}}\lesssim \cM(0) 
   + \cN(T)^2.\label{sobes-der}
\end{align}
\end{prop}
 In view of the two estimates \eqref{sobes-1} and \eqref{sobes-der}, only the 
$L^2$ norm of $(\vr, v)$ exhibits a logarithmic growth. However, for pedagogical  reasons, we will first present the proof of \eqref{sobes-1}, then we explain the main changes in order to obtain \eqref{sobes-der}. 
\begin{proof} 
 Let $\alpha\in \mathbb{N}^3$ with $|\alpha|\leq M.$ 
Define the energy functional 
\begin{align*}
    \dot{\cE}_{\alpha}(\vr, v)=\f12 \int h'(\rho)|\p^{\alpha}\vr|^2+|\sqrt{(\Id-\Delta)^{-1}}\p^{\alpha}
    \vr|^2+\rho\, |\p^{\alpha} v|^2 \, \d x \d y .
\end{align*}
Taking the time derivative of 
the above quantity and using the equation \eqref{eq-vr-v}, we find the energy identity
\begin{align*}
   \pt \dot{\cE}_{\alpha}(\vr, v)=\sum_{j=0}^5\cT_j^{\alpha}\, ,
\end{align*}
where 
\begin{align*}
   & \cT_0^{\alpha}=  \int \big(h'(\rho)+(\Id-\Delta)^{-1}\big)\, \p^{\alpha} \vr \,\p^{\alpha}(r_1^0+r_2^0)+\rho\, \p^{\alpha } v \cdot \p^{\alpha}(r_1'+r_2') \, \d x\d y,\\
   & \cT_1^{\alpha}= \f12 \int \div \big( \rho(\underline{v_c}(z)+v)\big)|\p^{\alpha} v|^2+\div \big( h'(\rho)(\underline{v_c}(z)+v)\big)|\p^{\alpha} \vr|^2\\
  & \qquad\qquad\qquad\qquad\qquad +2\div \big((\underline{v_c}(z)+v) (\Id-\Delta)^{-1}\p^{\alpha}\rho\big)\p^{\alpha}\rho\,\d x\d y,\\
  & \cT_2^{\alpha}=- \int  \rho \,\p^{\alpha} v \cdot \na \bigg( \p^{\alpha}\big(h(\rho)-h(1+n_c(z_1))\big)-h'(\rho)\p^{\alpha}\vr\bigg)\,\d x\d y, \\
  &\cT_3^{\alpha}=- \int  \rho \,\p^{\alpha} v \cdot  \na\p^{\alpha}\big(\phi-\phi_c(z_1)- (\Id-\Delta)^{-1}\vr \big)\,\d x\d y,\\
  &\cT_4^{\alpha}= -\int \big(h'(\rho)+(\Id-\Delta)^{-1}\big)\, \p^{\alpha} \vr \bigg([\p^{\alpha}, \underline{v_c}(z)+v] \na\vr+[\p^{\alpha},\rho]\div v+\p^{\alpha}\big(v\cdot \nabla n_c(z_1)+\vr \div  \underline{v_c}(z)\big)\bigg), \\
  & \cT_5^{\alpha}= - \int \rho \,\p^{\alpha} v \bigg([\p^{\alpha}, \underline{v_c}(z)+v] \na v+\p^{\alpha}(v\cdot \nabla v_c(z))\bigg)\, \d x \d y.
\end{align*}
Hereafter, unless otherwise  specified,
each quantity and norm is evaluated at time $t.$
We will control $\cT_0^{\alpha}-\cT_1^{\alpha}$ term by term. First, since 
\beqs 
\rho(t,x,y)=1+n_{c}(x-c_0t-z_1)+\vr(t,x,y)=1+\cO(\ep^2+\delta),
\eeqs
it holds that for $\ep,\delta$ small enough, 
$\f12 \leq \rho(t,x)\leq \f32$ and $\sup_{t, x,y} h'(\rho)=\sup_{s\in [\f12.\f32]} h'(s)\leq C.$
We thus have by the Cauchy-Schwarz inequality and the fact that  $(\Id-\Delta)^{-1}$ maps $L^2(\mR^3)$ to $H^2(\mR^3)$ that
\begin{align}\label{es-cT0}
    |\cT_0^{\alpha}|\lesssim \|(\vr, v)\|_{H^{|\alpha|}} \|(r_1, r_2)\|_{H^{|\alpha|}}.
\end{align}
 
Let us now address the second term $\cT_1^{\alpha}$, we shall focus particularly  on  the first quantity   in its definition which is the most difficult to handle.
Since we have that   $\underline{v_c}(z)=\underline{u_c}(z_1)+w(z),$ $\div\,  w=0,$ we can write
\beqs 
\div \big( \rho(\underline{v_c}(z)+v)\big)=\na (n_c(z_1)+\vr)\cdot (\underline{u_c}(z_1)+w+v) +\rho \div (\underline{u_c}(z_1)+v),
\eeqs
 where $\underline{u_c}(z_1)=\psi_c'(z_1) \left(\begin{array}{c}
          1 \\
    -\na_y\gamma
 \end{array}\right).$ Since $n_c, \psi_c'$ are exponentially localized, it holds that,
 for $a$ sufficiently small, 
 \beq \label{use-weightnorm}
 \|e^{-2a (x-c_0t)}(\Id,\na)\big(n_c(z_1), u_c(z_1)\big)\|_{L_{t, x,y}^{\infty}}\lesssim \|e^{2a \gamma}\|_{L_{t,y}^{\infty}}\leq e^{2 a \delta} 
 \eeq
as long as $\|\gamma\|_{L_{t,y}^{\infty}}\leq \delta.$ 
The term that requires to pay  more attention  to  is 
$\int \p_x\big(\rho \psi_c'(z_1)\big)|\p^{\alpha}v|^2$ 
since $\p_x\big(\rho \psi_c'(z_1)\big)$ does not decay in time and $v$ is not in the weighted space.
However, motivated by \eqref{w+v-weighted}, we can write 
\beqs 
\p^{\alpha}v=\p^{\alpha}(v+w(z))- \p^{\alpha} w(z).
\eeqs
Therefore, for any  $|\alpha|\geq 1,$ 
\beqs 
\int \p_x\big(\rho \psi_c'(z_1)\big)|\p^{\alpha}v|^2\lesssim \|v\|_{H^{|\alpha|}}\big(\|\na w\|_{H^{|\alpha|-1}}+\|v(\cdot+z_0 t)+w\|_{H_a^{|\alpha|}}\big).
\eeqs
It is worth noting that in the right hand side it appears $\|\na w\|_{H^{|\alpha|-1}}$ (rather than $ \|w\|_{H^{|\alpha|}}$) which has the critical time decay.
On the other hand, if $\alpha=0,$ we integrate by parts to find the similar estimate
\beqs 
\int \p_x\big(\rho \psi_c'(z_1)\big)|v|^2\lesssim \|v\|_{L^2}\big(\|\na w\|_{L^2}+\|v(\cdot+z_0 t)+w\|_{H_a^{1}}\big).
\eeqs
Combining the above ingredients, we find that 
\begin{align*}
   & \f12 \int \div \big( \rho(\underline{v_c}(z)+v)\big)|\p^{\alpha} v|^2\, \d x \d y\\
    &\lesssim 
    \big(\cA_{1,\infty}\|v\|_{H^{|\alpha|}}+\|\na w\|_{H^{|\alpha|}}+\|v(\cdot+z_0 t)+w\|_{H_a^{|\alpha|}}\big)\|v\|_{H^{|\alpha|}},
\end{align*}
where we have set  
\beq \label{def-ca1infty}
\cA_{1,\infty}=\cA_{1,\infty}(t)=\|\na\vr\|_{L^{\infty}}(\ep^2+\|(v,w(z))\|_{L^{\infty}})+\|\na (v,w)\|_{L^{\infty}}.
\eeq
Similarly and more easily, the next two quantities of $\cT_1^{\alpha}$ can be controlled  by 
\beqs 
\cA_{1,\infty}\|\vr\|_{H^{|\alpha|}}^2+\|\vr(\cdot+c_0t)\|_{H_a^{|\alpha|}}^2=\cA_{1,\infty}\|\vr\|_{H^{|\alpha|}}^2+\|n\|_{H_a^{|\alpha|}}^2.
\eeqs
Note that by definition, we have  $\vr(\cdot+c_0 t)=n(\cdot).$ 
Therefore, we obtain that 
\begin{align}\label{es-cT1}
 |\cT_1^{\alpha}|\lesssim \|(\vr,v)\|_{H^{|\alpha|}}\bigg( \cA_{1,\infty} \|(\vr,v)\|_{H^{|\alpha|}}+\|\na w\|_{H^{|\alpha|}}+\|(v(\cdot+z_0 t)+w, n)\|_{H_a^{|\alpha|}}\bigg).
\end{align}

We now switch to  the  term $\cT_2^{\alpha}.$ On the one hand, when $\alpha=0,$ we use \eqref{use-weightnorm} to obtain  that 
\begin{align*}
\|\na\big(h(\rho)-h(1+n_c(z_1))-h'(\rho)\vr\big)\|_{L^2}&\lesssim \|\na n_c(z_1) (h'(\rho)-h(1+n_c(z_1))\|_{L^2}+\|h''(\vr)\na\vr\, \vr\|_{L^2}\\
&\lesssim \|n\|_{L_a^2}+\|\na\vr\|_{L^{\infty}}\|\vr\|_{L^2}.
\end{align*}
On the other hand, when $1\leq |\alpha|\leq M,$ since 
\beqs 
\p^{\alpha}\big(h(\rho)-h(1+n_c(z_1))\big)-h'(\rho)\p^{\alpha}\vr=\sum_{\beta<\alpha, \,|\beta|\geq 1} C_{\beta,\alpha}\bigg(\p^{\beta} (h'(\rho))\p^{\alpha-\beta}\vr +\p^{\beta} \big(h(\rho)-h(1+n_c(z_1))\big)\p^{\alpha-\beta}n_c\bigg),
\eeqs
it holds by the Kato-Ponce inequality that 
\beqs 
\|\na\big(\p^{\alpha}(h(\rho)-h(1+n_c(z_1)))-h'(\rho)\p^{\alpha}\vr\big)\|_{L^2}\lesssim \|\na \vr\|_{L^{\infty}}\|\na \vr\|_{H^{|\alpha|-1}}+\|n\|_{H_a^{|\alpha|}}.
\eeqs
Consequently, the application of the Cauchy-Schwarz inequality as well as Proposition \ref{prop-weightednorm} gives
\beq
\begin{aligned}
|\cT_2^{\alpha}|&\lesssim \|v\|_{H^{|\alpha|}} \big(\|\na \vr\|_{L^{\infty}}\|\vr\|_{H^{|\alpha|}}+\|n\|_{H_a^{|\alpha|}}\big).
\end{aligned}
\eeq

For  the next term $\cT_3^{\alpha},$ we use an  expansion similar to \eqref{errorphi-cgamma}, we write 
\begin{align*}
    \tilde{\phi}=\colon \phi- \phi_{c}(z_1)=I_{\phi_{c}(z_1)}\bigg(\vr+\Delta_y \phi_c(z_1)+e^{\phi_{c}(z_1)}\big(e^{\tilde{\phi}}-1-\tilde{\phi}\big)\bigg),
\end{align*}
and use the identity $I_{\phi_{c}(z_1)}-(\Id-\Delta)^{-1}=(\Id-\Delta)^{-1}(1-e^{\phi_c(z_1)})I_{\phi_{c}(z_1)}$ and 
the exponential localization of $\phi_c(z_1)$
to find 
\begin{align}\label{diff-phi-vr}
    \|\na \big(\tilde{\phi}-(\Id-\Delta)^{-1}\vr\big)\|_{H^{|\alpha|}}\lesssim \| n \|_{H_a^{|\alpha|}}+\|\na_y^2(\gamma, c)\|_{L_y^2}+\|\na_y( \gamma, c)\|_{L^4}^2+\|\na \vr\|_{L^{\infty}}\|\vr\|_{H^{|\alpha|}}.
\end{align}
Note that we have used again that $I_{\phi_c(z_1)}, (\Id-\Delta)^{-1}$ map $L_a^2$ to $H_a^2$ and map $L^2$ to $H^2.$ It thus follows again from the 
 Cauchy-Schwarz inequality that 
 \beq
\begin{aligned}
|\cT_3^{\alpha}|&\lesssim \|v\|_{H^{|\alpha|}} \|\na \big(\tilde{\phi}-(\Id-\Delta)^{-1}\vr\big)\|_{H^{|\alpha|}}\\
&\lesssim \|v\|_{H^{|\alpha|}} \big( \|\na \vr\|_{L^{\infty}}\|\vr\|_{H^{|\alpha|}}+\|n\|_{H_a^{|\alpha|}}+(1+t)^{-1}\cN_{\tilde{c},\gamma}(T) \big).
\end{aligned}
\eeq
Next, for the term $\cT_4^{\alpha},$ we can apply the commutator estimates in Lemma \ref{lem-appendix-commutator} to find that 
\begin{align*}
    \|[\p^{\alpha}, \underline{v_c}(z)+v] \na\vr\|_{L^2}+\|[\p^{\alpha},\rho]\div v\|_{L^2}\lesssim \|\na(\vr, v, w(z))\|_{L^{\infty}}\|(\vr, v,w(z))\|_{H^{|\alpha|}} +\| \na n\|_{H_a^{|\alpha|-1}}. 
\end{align*}
Moreover, we can control $\p^{\alpha}\big(v\cdot \nabla n_c(z_1)\big)$ as
\beqs 
\|\p^{\alpha}\big(v\cdot \nabla n_c(z_1)\big)\|_{L^2}\lesssim \|v(\cdot+c_0t)+w\|_{H_a^{|\alpha|}}+\|w_1\|_{H^{|\alpha|}}+\|(w_2,w_3)\|_{H^{|\alpha|}}\|\na_y(c, \gamma)\|_{L_y^{\infty}}.
\eeqs
Collecting these estimates, we derive that
 \beq\label{es-cT4}
\begin{aligned}
|\cT_4^{\alpha}
|&\lesssim \|\vr\|_{H^{|\alpha|}} \big(\|\na(\vr, v, w(z))\|_{L^{\infty}}\|(\vr, v,w(z))\|_{H^{|\alpha|}}
+\|(n, v(\cdot+c_0t)+w)\|_{H_a^{|\alpha|}}\\
&\qquad \qquad \qquad \qquad\qquad \qquad+\|w_1\|_{H^{|\alpha|}}+\|(w_2,w_3)\|_{H^{|\alpha|}}\|\na_y(c, \gamma)\|_{L_y^{\infty}}\big).
\end{aligned}
\eeq
Following  similar arguments, we can bound the last term  as
\beq\label{es-cT5}
\begin{aligned}
    |\cT_5^{\alpha}|&\lesssim  \|v\|_{H^{|\alpha|}} \big(\|\na( v, w(z))\|_{L^{\infty}}\|( v,w(z))\|_{H^{|\alpha|}}+\|v(\cdot+c_0t)+w\|_{H_a^{|\alpha|}}\\
&\qquad \qquad \qquad \qquad\qquad \qquad+\|\na w\|_{H^{{|\alpha|-1}}}+\|w_1\|_{H^{|\alpha|}}+\|(w_2,w_3)\|_{H^{|\alpha|}}\|\na_y(c, \gamma)\|_{L_y^{\infty}}\big).
\end{aligned}
\eeq
Let us set
$$\cE_M =\sum_{|\alpha|\leq M} \dot{\cE}_{\alpha}(\vr, v)
\approx \|(\vr, v)\|_{H^{M}}^2.$$
Combining the estimates \eqref{es-cT0}, \eqref{es-cT1}-\eqref{es-cT5} and summing  up for $|\alpha|\leq M,$ we obtain the following energy inequality
\begin{align*}
    \pt \cE_M (t)&\lesssim \sqrt{\cE_M (t)} \bigg( \cA_{1,\infty}(t)\big(\sqrt{\cE_M (t)}+\|w\|_{H^{M}}\big)+\|(n, v(\cdot+c_0t)+w)\|_{H_a^{|\alpha|}}\\
    &\qquad\qquad\qquad+\|(r_1,r_2)\|_{H^M}+\|(w_1,\na w)\|_{H^{M}}+(1+\|w\|_{H^{M}})
    (1+t)^{-1}\cN_{\tilde{c},\gamma}(T)
    \bigg).
\end{align*}
In view of the definition of $\cN(T)$ in \eqref{def-cNT}, $\cN_{(\vr,v)}(T)$ in \eqref{def-norms-vr-v}
and $\cA_{1,\infty}(t)$ in \eqref{def-ca1infty},
it holds that 
\begin{align*}
    \cA_{1,\infty}(t)\lesssim (1+t)^{-(1+\iota)}\cN(T), \quad \forall\, t\in [0,T].
\end{align*}
 Applying the Gr\"onwall inequality and using  Proposition \ref{prop-w} for the estimates of $w,$  Corollary \ref{cor-sourceterm} for the estimates of $(r_1,r_2)$, Proposition \ref{prop-weightednorm} and Estimate \eqref{w+v-weighted} for the weighted estimates of $n$ and $v(\cdot+c_0 t)+w,$  we get 
\begin{align*}
   \sqrt{\cE_M (t)}\lesssim  \sqrt{\cE_M (0)}+ \int_0^t (1+s)^{-1} \big( \cN_{c,\gamma}(T)+\cM(0)+\cN(T)^2 \big)\, \d s.
\end{align*}
We thus obtain \eqref{sobes-1} by applying Proposition \ref{prop-modulation}. 

Let us now prove  the second estimate \eqref{sobes-der}
for  $\na(\vr, v).$ The idea is to first obtain the desired estimates for $\na_y(\vr, v)$ and $\pt (\vr, v),$ and then use the equation to recover the estimates for $\p_x(\vr, v).$ Let $Z$ denote for $\p_{y_1}, \p_{y_2}, \pt+c_0\p_x,$ acting $Z$ on the equation \eqref{eq-vr-v} for $(\vr, v),$ we find the following system for $(\dot{\vr},\dot{v})=(Z\vr, Z v):$
\beq\label{eq-dot-vr-v}
\left\{ 
\begin{array}{l}
  \pt \dot{\varrho}+\rho \div \dot{v}+(v+\underline{v_c}(z))\cdot\na \dot{\vr}=Z(r_1^0+r_2^0)-F_1^0-F_2^0 ,    \\[3pt]
 \pt \dot{v}+ (v+\underline{v_c}(z))\cdot\na \dot{v} +\na (h'(\rho)\dot{\vr}+(\Id-\Delta)^{-1}\dot{\vr})=Z(r_1'+r_2')-F_1'-F_2'-F_3',
\end{array}
\right.
\eeq
where 
\begin{align*}
   & F_1^0=\dot{v}\cdot \na\vr+\dot{\vr}\div v+ \div(\vr\,\dot{w}(z)),\quad F_1'=(\dot{w}(z)+\dot{v})\cdot \na v+\dot{v}\cdot\na w(z)+v\cdot\na \dot{w}(z),\, \\
  &  F_2^0=\dot{v}\cdot\na n_c(z_1)+\dot{\vr} \div \big(\underline{u_c}(z_1)\big)+\div\big(\vr\, Z\big(\underline{u_c}(z_1)\big)+v Z (n_c(z_1))\big)\\
   & F_2'=Z\big(\underline{u_c}(z_1)\big)\cdot\na v+v\cdot\na Z \big(\underline{u_c}(z_1)\big)+\dot{v}\cdot\na \big(\underline{u_c}(z_1)\big)+\na \big(\big(h'(\rho)-h'(1+n_c(z_1))\big)Z \big(n_c(z_1)\big)\big),\\
    & F_3'=Z\na \bigg((\phi-\phi_c(z_1))-(\Id-\Delta)^{-1}{\vr}\bigg).
\end{align*}
Note that we have used the notation $\dot{w}(z)= (Zw)(z).$
Define the energy functional $$  {\cE}_{M-1}(\dot{\vr}, \dot{v})=\sum_{|\alpha|\leq M-1}\f12 \int h'(\rho)|\p^{\alpha}\dot{\vr}|^2+|\sqrt{(\Id-\Delta)^{-1}}\p^{\alpha}
   \dot{\vr}|^2+\rho\, |\p^{\alpha} \dot{v}|^2 \, \d x \d y .$$
   We can perform   energy estimates  to find that 
   \begin{align*}
      & \pt \cE_{M-1}(\dot{\vr}, \dot{v})=\sum_{|\alpha|\leq M-1} \bigg(\int \na\rho\cdot \p^{\alpha} \dot{v} \big(( h'(\rho)+(\Id-\Delta)^{-1})\p^{\alpha}\dot{\vr}\big)\\
       &\qquad -\int \big(h'(\rho)+(\Id-\Delta)^{-1}\big)\, \p^{\alpha} \dot{\vr} \,\big([\p^{\alpha}, \underline{v_c}(z)+v] \na\dot{\vr}+[\p^{\alpha},\rho]\div\dot{v}\big) +\rho \,\p^{\alpha}\dot{v}\cdot [\p^{\alpha}, (\underline{v_c}(z)+v)\cdot\na] \dot{v}\\
       & + \f12 \int \div \big( \rho(\underline{v_c}(z)+v)\big)|\p^{\alpha} \dot{v}|^2+\div \big( h'(\rho)(\underline{v_c}(z)+v)\big)|\p^{\alpha} \dot{\vr}|^2+2\div \big((\underline{v_c}(z)+v)\p^{\alpha} (\Id-\Delta)^{-1}\dot{\rho}\big)\p^{\alpha}\dot{\rho} \\
       & + \int \big(h'(\rho)+(\Id-\Delta)^{-1}\big)\,  \p^{\alpha} \dot{\vr} \cdot \p^{\alpha}\big(Z(r_1^0+r_2^0)-F_1^0-F_2^0\big)+\rho\, \p^{\alpha} \dot{v}\cdot \p^{\alpha} \big(Z(r_1'+r_2')-F_1'-F_2'-F_3'\big)\bigg).
   \end{align*}
As we did for  the estimate of $\|(\vr, v)\|_{H^M},$ the 
first three terms in the right hand side can be bounded by 
\begin{align*}
  & \|(\dot{\vr}, \dot{v})\|_{H^{M-1}}\bigg(\cA_{1,\infty,t} \big(\|(\dot{\vr}, \dot{v})\|_{H^{M-1}}+\|w\|_{H^{M-1}}\big)+ \|Z(n, {v}(\cdot+c_0 t)+{w})\|_{H_{a}^{M-1}}+\|Zw\|_{H^{M-1}} \bigg)\\
  & \lesssim \|(\dot{\vr}, \dot{v})\|_{H^{M-1}} \bigg((1+t)^{-(1+\iota)}\cN(T)^2+(1+t)^{-\f32} (\cM(0)+\cN(T)^2)+(1+t)^{-\f76}\cN_{c,\gamma}\bigg) 
\end{align*} 
where we have applied Propositions \ref{prop-weightednorm} and  \ref{prop-w} in the second line. 
Moreover, we  claim that the last term can be controlled by  
\beq\label{es-6.16}
\begin{aligned}
&\|(\dot{\vr}, \dot{v})\|_{H^{M-1}}\bigg(\|Z(r_1, r_2)\|_{H^{M-1}}+\|(F_1, F_2, F_3')\|_{H^{M-1}} \bigg)\\
&\lesssim \|(\dot{\vr}, \dot{v})\|_{H^{M-1}} (1+t)^{-(1+\iota)}\big(\cM(0)+\cN(T)^2+\cN_{\tilde{c},\gamma}\big).
\end{aligned} 
\eeq
The above two estimates enable us to find the desired estimate
\begin{align*}
    \|(\dot{\vr}, \dot{v})(t)\|_{H^{M-1}}\approx \sqrt{ \cE_{M-1}(\dot{\vr}(t), \dot{v}(t))}\lesssim 
    \cM(0)+\cN_{c,\gamma}+\cN(T)^2.
\end{align*}
Let us now show \eqref{es-6.16}. First, by Corollary \ref{cor-sourceterm}, it holds that $\|Z(r_1, r_2)\|_{H^{M-1}}\lesssim (1+t)^{-\f32}(\cN_{\tilde{c},\gamma}(T)+\cN(T)).$
We now control $F_1,F_2, F_3'.$ By the Kato-Ponce inequality, 
\begin{align*}
    \|F_1\|_{H^{M-1}}\lesssim \cA_{1,\infty} \|(\vr , v , w)\|_{H^M}\lesssim (1+t)^{-(1+\iota)} \cN(T)^2.
\end{align*}
Moreover, since $Z \underline{u_c}(z_1)=-\underline{u_c}'(z_1) Z\gamma+\p_c \underline{u_c}(z_1) Z c, $ we can control $F_2$ in the following way
\begin{align*}
   \| F_2\|_{H^{M-1}} &\lesssim \|Z(n, {v}(\cdot+c_0 t)+{w})\|_{H_{a}^{M-1}}+\|Zw\|_{H^{M-1}}\\
   &\qquad +\big(\|(n, {v}(\cdot+c_0 t)+{w}) \|_{H_{a}^{M
   }}+\|w\|_{H^{M-1}}\big) \|Z(\gamma, c)\|_{L^{\infty}_y}\\
   &\lesssim (1+t)^{-\f32} (\cM(0)+\cN(T)^2)+(1+t)^{-\f76}\cN_{c,\gamma}.
\end{align*}
Similar to the estimate of $\na \big((\phi-\phi_c(z_1))-(\Id-\Delta)^{-1}{\vr}\big)$ in \eqref{diff-phi-vr},
we can bound $F_3'$ as 
\begin{align*}
\|F_3'\|_{H^{M-1}}&\lesssim  \| n \|_{H_a^{M}}\|\na_y(\gamma,c)\|_{L_y^2}+\|\na \vr\|_{L^{\infty}}\|\vr\|_{H^{|\alpha|}}\\
&\quad +\|Z\na_y^2(\gamma, c)\|_{L_y^2}+\|\na_y( \gamma, c)\|_{L_y^4}\|Z\na_y( \gamma, c))\|_{L_y^4}+\|\na_y( \gamma, c)\|_{L_y^4}^2\|Z( \gamma, c)\|_{L_y^{\infty}}\\
&\lesssim (1+t)^{-\f32}\big(\cN_{\tilde{c},\gamma}(T)+\cN(T)^2\big).
\end{align*}
The assertion made in \eqref{es-6.16} is now proved. Consequently, thanks again to Proposition \ref{prop-modulation},
we have established the estimate \eqref{sobes-der} by replacing $\nabla$ with $\nabla_{y_1}$ or $\pt+c_0\p_x.$ Moreover, since $v$ is curl-free, we have also the same estimate for $\p_x (v_2, v_3).$
To get the desired  estimate for $\partial_x (\vr, v_1)$, we need to use the following lemma. 
\begin{lem}\label{lem-recoverpx}
Assume that $(n, \tilde{v}_1)$ solve the following system 
\beq\label{odesystem}
-c_0 \p_z n+\p_z \tilde{v}_1=H_1, \quad -c_0 \p_z \tilde{v}_1+\p_z \big(h'(1)+(\Id-\p_z^2)^{-1}\big)n =H_2.
\eeq
It holds that for any $m\geq 0,$ any $p\in [1, +\infty],$ 
   \beq \label{crucial-recover}
\|\p_z (n, \tilde{v}_1)\|_{W^{m,p}}\lesssim \|(H_1, H_2)\|_{W^{m,p}}.
   \eeq
\end{lem}
Let us postpone the proof of this lemma and finish the 
estimate for $\partial_x (\vr, v_1).$ Let 
\beq\label{def-n-tv1}
n(z)=\vr(z+c_0t), \quad \tilde{v}_1(z)=v_1(z+c_0t)
\eeq
so that it  suffices to control $\|\partial_z (n, \tilde{v}_1)\|_{H^{M-1}}.$ Using  the system \eqref{eq-vr-v}, we find that $(n, \tilde{v}_1)$ solve the system \eqref{odesystem}
 with 
\beq\label{def-H1H2}
\begin{aligned}
    H_1&=-\pt n -\div_y \tilde{v}_y- \p_z(n\, \tilde{v}_1)- \div_y(n\, \tilde{v}_y)-\div (n w)-\div \big(n_c(z_1)(v+w)+n\underline{u_c}\big) 
    \\
    &\qquad \qquad \qquad 
    +r_1^0+\div_y (\psi_c'(z_1)\na_y \gamma)\,,\\
    H_2&=-\pt \tilde{v}_1+(w+v)\cdot\na(w+v)_1+\big(\underline{u_c}(z)\cdot \na(w+\tilde{v})_1+(w+\tilde{v})\cdot\na \psi_c')\big)
    \\
    &- \big((h'(\rho)-h'(1))\p_z n
    +(h'(\rho)-h'(1+n_c(z_1)))n_c'(z_1) \big)-\p_z(\phi-\phi_c(z_1)-(\Id-\p_x^2)^{-1} n)\\
    & \qquad +(r_1')_1+\f12\p_x(|\psi_c'|^2(z_1)|\na_y\gamma|^2)\,,
\end{aligned}
\eeq
where we denote $(\tilde{v}_1, \tilde{v}_y)^t(z)=\tilde{v}(z)=v(z+c_0t).$ We have  that 
\begin{align*}
    \|(H_1, H_2)\|_{H^{M-1}}&\lesssim 
 \|(\na_y, \pt) (n, \tilde{v})\|_{H^{M-1}}+   \|(n,\tilde{v},w)\|_{W^{1,\infty}}
    \|(n,\tilde{v},w)\|_{H^{M}}\\
    &\quad +\|(n, v(\cdot+c_0t)+w)\|_{H_a^{M}}
    +\|r_1\|_{H^M}+ \||\na_y(\gamma, c)|^2, \na_y^2\gamma\|_{L^2}.
\end{align*}
Applying the estimate \eqref{crucial-recover} and using  the estimates already obtained for $(\na_y, \pt)(n, \tilde{v}),$ we find the desired estimate
\beqs 
\|\p_z (n, \tilde{v}_1)\|_{H^{M-1}}\lesssim \cM(0)+\cN_{\tilde{c},\gamma}(T)+\cN(T)^2.
\eeqs
\end{proof}
\begin{proof} [Proof of Lemma \ref{lem-recoverpx}]
The simple combination of the two equations of \eqref{odesystem} yields the following equation for $\p_z n:$
\beqs 
(b -d\,\p_z^2)\,\p_z n= -(1-\p_z^2)(c_0 H_1+H_2)
\eeqs
where the constants $b=c_0^2-h'(1)-1>0, d=c_0^2-h'(1)>1$ thanks to the assumption $c_0> h'(1)+1.$ The solution of the above ODE is given by 
\beqs 
\p_z n (z)=- \f{1}{\sqrt{b\,d}}\int_{\mR}  e^{-\pi \sqrt{b/d}\,|z-z'|} (1-\p_{z'}^2)(c_0 H_1+H_2)(z')\, \d z'.
\eeqs
It then follows from integration by parts and Young's inequality that 
\beqs 
\|\p_z n \|_{W^{m,p}}\lesssim_{b,d} 
\|(H_1, H_2)\|_{W^{m,p}}, \quad \forall\, p\in [1,+\infty].
\eeqs
The desired estimate for $\p_z \tilde{v}_1$ finally  follows from  the first equation of \eqref{odesystem}. 
\end{proof}

\section{Time decay  estimates }\label{sec-decayes}
In this section, we focus on the time decay estimate for the radiation term $(\vr, v),$ which satisfies equation \eqref{eq-vr-v}. 
Specifically, our goal is  to  estimate the quantities appearing in the second line of $\cN_{(\vr,v)}(T)$ defined in \eqref{def-norms-vr-v},  which is crucial to close the energy estimates performed in Section \ref{energyes}. The following theorem presents the main result of this section.

\begin{prop}\label{prop-decayes}
Let $(\vr, v)$ be a solution to equation \eqref{eq-vr-v}.
 Assume that $(\tilde{c}, \gamma)\in C^1([0,T], Y\times Y),$   $(\vr, v)\in C([0,T], H^M).$
  Then there exists $\kappa_0\in(0,\f{1}{100}],$ such that for any $\kappa\in(0, \kappa_0],$ any $\vartheta\in(0,\kappa^2],$ any $p\in[2,6),$
 the following estimate holds : 
 \beq\label{decayes-vrv}
 \begin{aligned}
      &(1+t)^{-\vartheta} \|(\vr, v)\|_{\dot{H}^{-\f12}}+(1+t)^{(\f{3}{4}-\f2p-\vartheta)}
      \|(\vr, v)\|_{\dot{W}^{-1,p}}
 +(1+t)^{(1+\iota)} \big\||\na|^{\f{3}{8_{\kappa,1}}} \nabla(\vr, v)\big\|_{W^{1,8_{\kappa}}} \\
&\qquad\qquad\qquad\qquad\qquad\qquad\qquad\qquad\qquad\lesssim \cM(0)+\cN_{\tilde{c},\gamma}(T)+ \cN(T)^2
 \end{aligned}
 \eeq
 where $8_{\kappa}=\f{8}{1-\kappa}, 8_{\kappa,1}=\f{8}{1-\f{10}{9}\kappa}, \iota=\f{\kappa}{24}.$ 
\end{prop}
\begin{rmk}
   Establishing the propagation of regularity for $(\vr, v)$ in the negative Sobolev spaces $\dot{H}^{-\f12}$
   and $\dot{W}^{-1,p}$ with controlled growth and decay respectively 
   as indicated by the first two quantities, is crucial for deriving favorable time decay estimates, as implied by the third quantity.
   This behavior in the negative Sobolev spaces differs from the case where the background wave is zero, as in \cite{G-P-global}, where the $\dot{H}^{-1}$ norm remains uniformly bounded. In the present setting, however, due to source terms involving  only the solitary wave and  the modulation parameters in the equation  for $(\vr,v),$ one can at best expect mild growth for  the $\dot{H}^{-\f12}$
   norm. This gives rise to additional challenges in proving integrable time decay of $(\vr,v)$ in  $L^p$ for $p>8.$ 
\end{rmk}
\begin{proof}

We first claim that to prove the decay of $|\na|^{\f{3}{8_{\kappa,1}}} \nabla(\vr, v)$ in $W^{1,8_{\kappa}},$
it suffices to show the corresponding estimate for $(\na_y, \pt+c_0\p_x)(\vr, v):$
\begin{align} \label{decay-timeder}
    (1+t)^{(1+\f{\kappa}{24})} \big\||\na|^{\f{3}{8_{\kappa,1}}} (\na_y, \pt+c_0\p_x)(\vr, v)\big\|_{W^{1,8_{\kappa}}} \lesssim \cM(0)+\cN_{\tilde{c},\gamma}(T)+ \cN(T)^2. 
\end{align}
Indeed, once this is proven, we can apply Lemma \ref{lem-recoverpx} to obtain that 
\beqs 
\||\na|^{\f{3}{8_{\kappa,1}}} \p_x(\vr, v_1)\|_{L^{8_{\kappa}}}\lesssim \||\na|^{\f{3}{8_{\kappa,1}}} \p_z(n, \tilde{v}_1)\|_{W^{1,8_{\kappa}}} \lesssim \||\na|^{\f{3}{8_{\kappa,1}}}(H_1, H_2)\|_{W^{1, 8_{\kappa}}}, 
\eeqs
where $(n, \tilde{v}_1)$ and $H_1, H_2$ are defined in  \eqref{def-n-tv1}, \eqref{def-H1H2} respectively. By applying Lemma \ref{lem-H1H2} for $H_1, H_2$ and \eqref{decay-timeder}, we find that 
\beqs 
\||\na|^{\f{3}{8_{\kappa,1}}}(H_1, H_2)\|_{W^{1,8_{\kappa}}}\lesssim (1+t)^{-(1+\f{\kappa}{24})} \big( \cM(0)+\cN_{\tilde{c},\gamma}(T)+ \cN(T)^2 \big)
\eeqs
and thus the claim is proved. 

  To continue, let us rewrite the system \eqref{eq-vr-v} into the following form
\begin{align}\label{new-eq-vr-v}
\pt \left(\begin{array}{c}
     \vr  \\
     v 
\end{array}\right) +\left( \begin{array}{cc}
     0  &  \div \\
     \na \big(h'(1)+(1-\Delta)^{-1}\big) 
     & 0
  \end{array}  \right)\left(\begin{array}{c}
     \vr  \\
     v 
\end{array}\right) = 
\mathfrak{F}_1+\mathfrak{F}_2+\mathfrak{F}_3+\mathfrak{F}_4,
\end{align}
where 
\beq\label{def-fF}
\begin{aligned}
  & \mathfrak{F}_1= \left(\begin{array}{c}
     r_1^0+r_2^0+\div\big(n_c(z_1) w(z)\big) \\[3pt]
    r_1'+r_2'+\na (\underline{u_c}\cdot w)(z)+\na \big(I_{\phi_c(z_1)}\Delta_y (\phi_c(z_1))\big)
\end{array}\right), \\
&\mathfrak{F}_2= -\left(\begin{array}{c}
    \div\big( n_c(z_1)(v+w(z))+ \vr( \underline{u_c}+w)(z)\big)  \\[3pt]
    \na \big(\underline{u_c}\cdot(v+w(z))+w(z)\cdot v\big)+\na \big((h'(1+n_c)-h'(1))n\big)+\na \big((I_{\phi_c(z_1)}-I_0)n\big)
\end{array}\right), \\
 &\mathfrak{F}_3= -\left(\begin{array}{c}
    \div(\vr\, v)  \\[3pt]
     \f{1}{2}\na\big(|v|^2+h''(1)\vr^2+I_{0}(I_{0}\vr)^2 \big)
\end{array}\right),  \, \,\,\mathfrak{F}_4=-\left(\begin{array}{c}
 0  \\[3pt]
    \na \big(H(n_c(z_1), \vr) +\tilde{G}(\phi_c(z_1), \vr) \big)  
\end{array}\right),
\end{aligned}
\eeq
with $I_{\phi_c(z_1)}=(e^{\phi_c(z_1)}-\Delta)^{-1},\, I_0=(\Id-\Delta)^{-1}$, and 
\beqs 
H(n_c(z_1), \vr)=h(1+n_c(z_1)+\vr)-h(1+n_c(z_1))-h'(1+n_c)\vr-h''(1)\vr^2/2\,,
\eeqs
\beqs 
\tilde{G}(\phi_c(z_1), \vr) =\big(\phi-\phi_c(z_1)-I_{\phi_c(z_1)} \big(\vr+\Delta_y (\phi_c(z_1))\big)-\f12 I_{0}(I_{0}\vr)^2\big).
\eeqs
Let $P(D)=\sqrt{-\Delta\big(h'(1)+(1-\Delta)^{-1}\big) }.$ 
To symmetrize the system \eqref{new-eq-vr-v}, we
define the new unknown $\alpha=\vr+i \f{\div}{P(D)}v,$ which solves the equation
\beq\label{eq-alpha}
\pt \alpha- i P(D)\alpha=\mathfrak{N}_1+\mathfrak{N}_2+\mathfrak{N}_3+\mathfrak{N}_4,
\eeq
with $\mathfrak{N}_j=\mathfrak{F}_j^0+ i \f{\div}{P(D) }\mathfrak{F}_j'.$ 
As the velocity $v$ is curl-free, we can recover $(\vr, v)$ from $\alpha, 
\bar{\alpha}, $ we have:
\beq\label{rel-v-alpha}
\vr=\f12(\alpha+\bar{\alpha}),\qquad \, v=\sqrt{h'(1)+(1-\Delta)^{-1}} \f{\na}{|\na|}\f{\bar{\alpha}-\alpha}{2\, i}.
\eeq
By Hörmander-Mikhlin Theorem (see for instance Theorem 5.2.7 in \cite{Book-grafakos}), it holds that for any $p\in (1,+\infty),$
\beqs 
\|(\vr,v)\|_{L^p}\lesssim \|\alpha\|_{L^p}.
\eeqs
Consequently, 
to prove \eqref{decay-timeder}, it suffices to show that there exists $\kappa, \vartheta>0$ small enough, such that
\begin{align}
    & \|\alpha\|_{\dot{H}^{-\f12}}\lesssim  (1+t)^{\vartheta} \big(\cM(0)+\cN_{\tilde{c},\gamma}(T)+ \cN(T)^2\big), \label{negative-sobolev2} \\
    & \|P_{\leq -5}\,\alpha\|_{\dot{W}^{-1,p}}\lesssim 
    (1+t)^{-(\f{3}{4}-\f2p-\vartheta)}\big(\cM(0)+\cN_{\tilde{c},\gamma}(T)+ \cN(T)^2\big), \quad (p\in [2,6)),\label{negative-sobolev} \\
    & \big\||\na|^{\f{3}{8_{\kappa,1}}} (\na_y, \pt+c_0\p_x)
     \alpha \big\|_{W^{1,8_{\kappa}}} \lesssim  (1+t)^{-(1+\f{\kappa}{24})}  \big(\cM(0)+\cN_{\tilde{c},\gamma}(T)+ \cN(T)^2\big). 
     \label{decayes-Zder}
\end{align}
We shall establish the needed estimates for $\mathfrak{N}_1-\mathfrak{N}_4$ (or equivalently $\mathfrak{F}_1-\mathfrak{F}_4$) in Subsection \ref{subsection-fF124}. Based on these estimates, we can  first give the proof of \eqref{negative-sobolev2}-\eqref{negative-sobolev}.
By Duhamel formula, we have
\beqs 
\alpha=e^{it P(D)}\alpha_0+\int_0^t e^{i(t-s) P(D)} \big(\mathfrak{N}_1+\mathfrak{N}_2+\mathfrak{N}_3+\mathfrak{N}_4\big)\, \d s. 
\eeqs 
By the estimates \eqref{cF1-half}, \eqref{fF2-half}, \eqref{fF3-half}, \eqref{esfF4-neghalf}, it holds that 
\beqs 
\|\mathfrak{N}_j(s)\|_{\dot{H}^{-\f12}}\lesssim \|\mathfrak{F}_j(s)\|_{\dot{H}^{-\f12}}\lesssim
(1+s)^{-(1-\vartheta)} \big(\cM(0)+\cN_{\tilde{c},\gamma}(T)+\cN(T)^2\big), \quad (j=1, 2, 3, 4).
\eeqs
By using also that 
 $e^{it P(D)}$ is an isometry on $\dot{H}^{-1/2},$ we thus get
\eqref{negative-sobolev}. 
In order to prove  \eqref{negative-sobolev2}, it suffices to establish the following two refined estimates 
\begin{align}
  & \|P_{\leq -5}\alpha_s\|_{\dot{W}^{-1,p}}\lesssim 
    (1+t)^{-(\f{3}{4}-\f2p-\vartheta)} \big(\cN_{\tilde{c},\gamma}(T)+\cN(T)^2\big), \quad \forall\, 2\leq p<6 , \label{es-alphas} \\
  &  \|P_{\leq -5} \alpha_r\|_{\dot{W}^{-1,p}}\lesssim (1+t)^{-(1-\f{7}{3p})}\big(\cM(0)+\cN_{\tilde{c},\gamma}(T)+\cN(T)^2\big),  \quad \forall\, 2\leq p\leq  4, \label{es-alphar}
\end{align}
where 
 $$\alpha_s:=\int_0^t e^{i(t-s)P(D)} \mathfrak{N}_1(s)\, \d s$$
 is the contribution of $\alpha$ resulting from the source term $\mathfrak{N}_1$ and  $\alpha_r=\alpha- \alpha_s.$ Note that thanks to the Sobolev embedding, we can derive from \eqref{es-alphar} that for $p\in (4,6),$ $$ \|P_{\leq -5} \alpha_r\|_{\dot{W}^{-1,p}}\lesssim \|P_{\leq -5} \alpha_r\|_{\dot{W}^{-(\f{1}{4}+\f3p),4}}\lesssim (1+t)^{-\f{5}{12}
 }\big(\cM(0)+\cN_{\tilde{c},\gamma}(T)+\cN(T)^2\big).$$
Note  that we thus have that  $\alpha_r$ has better decay properties  than $\alpha_s$ for any $p\in [2,6).$
 
To prove \eqref{es-alphas}, it is convenient to split
$\mathfrak{N}_{1}=\mathfrak{N}_{1}^1+\mathfrak{N}_{1}^2$ where $\mathfrak{N}_{1}^j:=(\mathfrak{F}_{1}^{j})^{0}+ i \f{\div}{P(D) }(\mathfrak{F}_1^{j})^{'},$ associated with suitable splitting of $\mathfrak{F}_{1}:=\mathfrak{F}_{1}^1+\mathfrak{F}_{1}^2$ that will be introduced in [(2), Lemma \ref{lem-fF1}].
We write $$\alpha_s=\int_0^t e^{i(t-s)P(D)} \big(\mathfrak{N}_{1}^1+\mathfrak{N}_{1}^2(s)\big)\, \d s:=
\alpha_s^1+\alpha_s^2.$$
Applying the dispersive estimate \eqref{disper-low} we find, on the one hand by the estimate \eqref{cF1-minus1} that 
for any $p\in[2,6),$
\begin{align*}
    \big\|P_{\leq -5} \f{\alpha_s^1}{|\na|}\big\|_{L^p}&\lesssim \int_0^t (1+t-s)^{-\f{3}{2}(1-\f2p)} \big\||\na|^{-(\f12+\f1p)}\mathfrak{F}_1^1(s)\big\|_{L^{p'}}\, \d s \\
    &\lesssim \int_0^t (1+t-s)^{-\f{3}{2}(1-\f2p)} (1+s)^{-(\f14+\f1p-\vartheta)}\, \d s \,
    \big(\cN_{\tilde{c},\gamma}(T)+\cN(T)^2\big)\\
    &\lesssim (1+t)^{-(\f{3}{4}-\f2p-\vartheta)} \big(\cM_{\tilde{c},\gamma}(0)+\cN(T)^2\big),
\end{align*}
and on the other  hand, by \eqref{cF1-minus1-1} that  for any $p\in[2,4),$
\begin{align*}
     \big\|P_{\leq -5} \f{\alpha_s^2}{|\na|}\big\|_{L^p}\lesssim (1+t)^{-\f{3}{2}(1-\f2p+\vartheta)} \big(\cM_{\tilde{c},\gamma}(0)+\cN(T)^2\big).
\end{align*}
The estimate \eqref{es-alphas} then stem from the above two inequalities and the Sobolev embedding.

Similarly, thanks to the  estimates \eqref{es-fF2-Lp}, \eqref{fF3-negative1-low}, \eqref{fF4-negativeLp},  there holds for any $p\in[2,4],$
\begin{align*}
    \big\|P_{\leq -5} \f{\alpha_r}{|\na|}\big\|_{L^p}&\lesssim 
    (1+t)^{-\f32(1-\f2p)}\||\na|^{-(\f12+\f1p)}\alpha(0)\|_{L^{p'}}+
    \int_0^t (1+t-s)^{-\f{3}{2}(1-\f2p)} \big\||\na|^{-(\f12+\f1p)}\sum_{j=1}^3\mathfrak{N}_j(s)\big\|_{L^{p'}}\, \d s \\
    &\lesssim  (1+t)^{-\f32(1-\f2p)} \cM(0)+
    \int_0^t (1+t-s)^{-\f{3}{2}(1-\f2p)} (1+s)^{-(\f12+\f{2}{3p})}\, \d s \, \big(\cM(0)
    +\cN(T)^2\big)\\
    &\lesssim (1+t)^{-(1-\f{7}{3p})} \big(\cM(0)
    +\cN(T)^2\big).
\end{align*}

We are now left to prove \eqref{decayes-Zder} which is lengthy and technical, we postpone its proof to Subsection \ref{subsection-decayes}.



\end{proof} 

\subsection{Estimates for $\mathfrak{F}_1-\mathfrak{F}_4$} \label{subsection-fF124}
In this subsection, we show various estimates for $\mathfrak{F}_1-\mathfrak{F}_4$ that are used in the proof of \eqref{prop-decayes}, 
including the ${\dot{H}^{-\f12}}$ estimate of $\mathfrak{F}_1-\mathfrak{F}_4$ and $W^{m_{\kappa},8_{\kappa}'}(m_{\kappa}=
1+\f{3}{8_{\kappa,1}}+\f{3(3+\kappa)}{4})$
 estimates for $Z(\mathfrak{F}_1, \mathfrak{F}_2, \mathfrak{F}_4)$ where $Z\in \{\na_y, \pt+c_0\p_x \}.$  The latter will also be useful in the proof of estimate \eqref{decayes-vrv}. 

\subsubsection{Estimate of $\mathfrak{F}_1$}
The term $\mathfrak{F}_1$ contains only the terms depending on $(\tilde{c},\gamma)$ and $w,$ and  requires particular attention in the estimation.
\begin{lem}\label{lem-fF1}
    The following properties hold.\\[4pt]
(1) We have the following estimate for $\mathfrak{F}_1$  in $\dot{H}^{-\f12}:$
\beq \label{cF1-half}
\|\mathfrak{F}_1\|_{\dot{H}^{-\f12}}\lesssim  (1+t)^{-(1-\vartheta)} \big(\cN_{\tilde{c},\gamma}(T)+\cN(T)^2\big), \quad \forall\, \vartheta>0\, .
  \eeq
(2) We can split $\mathfrak{F}_1=\mathfrak{F}_1^1+\mathfrak{F}_1^2,$ where $\mathfrak{F}_1^1, \mathfrak{F}_1^2$ satisfy the following (different) estimates 
  \begin{align}
    &  \|\mathfrak{F}_1^1\|_{\dot{W}^{-(\f12+\f1p),p'}}\lesssim  (1+t)^{-(\f14+\f1p-\vartheta)} \big(\cN_{\tilde{c},\gamma}(T)+\cN(T)^2\big), \quad \forall \,p\in [2,6], \,\forall \vartheta>0, \label{cF1-minus1}\\
 & \|\mathfrak{F}_1^2\|_{\dot{W}^{-(\f12+\f1p),p'}}\lesssim (1+t)^{-1} \big(\cN_{\tilde{c},\gamma}(T)+\cN(T)^2\big), \quad \forall \,p\in [2,4). \label{cF1-minus1-1}
  \end{align} 
  Moreover, it holds that for any $p\in [2,4),$
  \begin{align}\label{es-ZfF}
     \|Z\mathfrak{F}_1\|_{\dot{W}^{-(\f12+\f1p),p'}}\lesssim  (1+t)^{-(\f34+\f1p-\vartheta)} \big(\cN_{\tilde{c},\gamma}(T)+\cN(T)^2\big).
  \end{align}
  
(3) Finally, it holds that
    \begin{align}\label{fF1-Lp}
       \|\mathfrak{F}_1(t)\|_{W^{k,r}}\lesssim (1+t)^{-(\f32-\f1r)} \big(\cN_{\tilde{c},\gamma}(T)+\cN(T)^2\big),\quad \forall \, k\geq 0,\, r\in(1, +\infty),
    \end{align}
  \begin{align}  \label{es-fF1} \|Z\mathfrak{F}_1\|_{W^{m_{\kappa},8_{\kappa}'}}\lesssim (1+t)^{-\f{9-\kappa}{8}}  \big(\cN_{\tilde{c},\gamma}(T)+\cN(T)^2\big).
  \end{align}
\end{lem}

\begin{proof}
We will provide comprehensive details regarding the proof of \eqref{cF1-half} while briefly outlining the proof of the remaining estimates.
In view of the definitions of $r_1, r_2$ in \eqref{def-source-sob}, we observe that $\mathfrak{F}_1$ is made of  the following typical terms:
\beq\label{cF1typicalterms}
\p_c Q_{c}(z_1) \pt c, \quad (\pt-c_0\p_z) w(z),  \quad \p_z\big(Q_{c}(z_1) (\pt\gamma-\tilde{c})\big), \quad \nabla |w|^2(z)\, ,
\eeq
where $Q_c$ denotes $n_c$ or $\psi_c'$ or $\phi_c$ and thus is exponentially localized. 

Thanks to \eqref{es-w},
the last term can be bounded easily as 
\beqs 
\big\|\na |w|^2\big\|_{\dot{H}^{-1/2}}\lesssim \|w\cdot\na w\|_{L^{{3}/{2}}}\lesssim \|w\|_{L^6}\|\na w\|_{L^2}\lesssim (1+t)^{-2}\cN(T)^2.
\eeqs
Moreover, thanks to the  derivative in front,  the third term can be bounded  in $\dot{H}^{-1/2}$ by $\|\pt \gamma -\tilde{c}\|_{L_y^2},$ which, in view of \eqref{gammtct}, can be further controlled by $(1+t)^{-1}\big(\cN_{\tilde{c},\gamma}(T)+\cN(T)^2\big).$
The first two terms require a more careful analysis
since the  direct use of the Sobolev embedding only leads to the  $(1+t)^{-5/6}$ decay  which is too slow. To resolve the problem, we need to 
to use the localization of the wave $Q_c(z_1).$ For the first term in \eqref{cF1typicalterms}, let us write 
\begin{align*}
    \p_c Q_{c}(z_1) \pt c= \p_c Q_{c_0}(z) \pt c+(\p_c Q_c-\p_c Q_{c_0})(z)\pt c-\p_z \int_0^1 \p_c Q_c (z-\gamma\theta)\, \d\theta\, \gamma \,\pt c
\end{align*}
where we have set  $\p_c Q_{c_0}=\p_c Q_{c}|_{c=c_0}$
The last two terms above  can be controlled in $\dot{H}^{-1/2}$ directly by 
$$\|(\tilde{c},\gamma)\,\pt c\|_{L_y^{3/2}}\lesssim \|(\tilde{c},\gamma)\|_{L_y^{6}}\|\pt c\|_{L_y^{2}}\lesssim (1+t)^{-\f43}\cN(T)^2. $$ To estimate the first one, we  use the Parseval's identity to obtain 
\begin{align*}
& \big\|P_{\geq 0}|\na|^{-\f12}\big( \p_c Q_{c_0}  \pt c\big)\big\|_{L^2}\lesssim \big\| \p_c Q_{c_0}  \pt c\big\|_{L^2}\lesssim \|\pt c\|_{L_y^2},
\\
  & \big\|P_{\leq 0}|\na|^{-\f12}\big( \p_c Q_{c_0}  \pt c\big)\big\|_{L^2}^2\leq \int_{\mR^2} |(\cF_y \pt c)(\zeta)|^2\int_{|\xi|\leq 4} \f{|(\cF_z\p_c Q_{c_0})(\xi)|^2}{(|\zeta|^2+|\xi|^2)^{1/2}}\, \d\xi\d \zeta\\
   &\lesssim  \|\p_c Q_{c_0}\|_{L_z^1}^2  \int_{\mR^2} |(\cF_y \pt c)(\zeta)|^2 \log (1+|\zeta|^{-1})\, \d \zeta\lesssim \big\||\na_y|^{-2\vartheta} \pt c\big\|_{L_y^2}^2
\end{align*}
for any $\vartheta>0.$
It follows the Hardy-Sobolev-Young inequality and the estimates \eqref{gammtct}, \eqref{es-yptc} that  
\beqs 
\big\||\na_y|^{-2\vartheta} \pt c\big\|_{L_y^2}\lesssim  \|\pt c\|_{L_y^{{2}/{(1+2\vartheta)}}}\lesssim \|\pt c\|_{L_y^2}^{1-2\vartheta}\|y \pt c\|_{L_y^2}^{2\vartheta}\lesssim (1+t)^{-\vartheta} \big(\cN_{\tilde{c},\gamma}(T)+\cN(T)^2\big).
\eeqs
Let us now sketch the estimate of the term $(\pt-c_0\p_z )w$ in \eqref{cF1typicalterms}. By the definition \eqref{def-w} and the expression \eqref{comp-curlF}, it holds that 
$$\|\p_z w\|_{\dot{H}^{-\f12}}\lesssim \|\psi_c'\, (\na_y c, \na_y c\cdot\na_y \gamma)\|_{\dot{H}^{-\f12}}.$$
Moreover, we compute 
\begin{align*}
    \curl \pt \underline{u_c}=\curl \left(\begin{array}{c}
       0    \\
       \pt\big(\psi_c'(z_1)\na_y\gamma\big)   
\end{array}\right)+\left(\begin{array}{c}
    0\\
   \na_y^{\perp} \big(\p_c\psi_c'(z_1)\pt c \big)
\end{array}\right)+\left(\begin{array}{c}
    0\\
-\p_z\na_y^{\perp} \big(\psi_c'(z_1)\pt\gamma\big)
\end{array}\right),
\end{align*}
where $\na_y^{\perp}=(\p_{y_2},-\p_{y_1})^t.$
Since $\curl(-\Delta)^{-1}\na \in B(\dot{H}^{-\f12}),$ it holds that 
\beqs
\|\p_t w\|_{\dot{H}^{-\f12}}
\lesssim \|(\psi_c'(z_1)\pt\na_y\gamma,\p_c\psi_c'(z_1)\pt c)\|_{\dot{H}^{-\f12}} + \|(\psi_c''(z_1), \p_c\psi_c')\pt(\gamma,c)\na_y(\gamma,c)\|_{\dot{H}^{-\f12}}.
\eeqs
The terms $\psi_c'(z_1)(\na_y c,\pt c,\pt\na_y\gamma)$
 can be estimated in the same manner as 
  $\p_c Q_{c}(z_1) \pt c.$ The remaining quantities are quadratic in terms of $\gamma$ and $\tilde{c}$ and thus are easier to handle. To summarize, we have obtained \eqref{cF1-half}. 
  
Proof of \eqref{cF1-minus1}-\eqref{cF1-minus1-1}. 
Similar to the proof of \eqref{cF1-half}, when controlling $\mathfrak{F}_1$ in $\dot{W}^{-(\f12+\f1p),p'}$, which exhibits lower integrability and more singularities at low frequencies, the terms that require careful treatment are the following:
\begin{align}\label{fF11}
  \p_c Q_{c_0}(z) \pt c, \quad \p_z\big(Q_{c}(z_1) (\pt\gamma-\tilde{c})\big), \quad -\curl (-\Delta)^{-1} \bigg(\curl  \left(\begin{array}{c}
       0    \\
      \psi_c'(z_1) \pt\na_y\gamma  
\end{array}\right)+\na_y^{\perp} \left(\begin{array}{c}
    0\\
   \p_c\psi_c'(z_1)\pt c 
\end{array}\right)\bigg).
\end{align}
Using  the same arguments as in the 
$\dot{H}^{-\f12}$  estimate only enables us to prove \eqref{cF1-minus1} for $p\in[2,4).$ 
Nevertheless, thanks to the equation \eqref{modueq-1}, we observe that
$$\pt c=a_{21}\p_y^2\gamma+\text{ good terms, }$$ where 
`good terms' have favorable decay $(1+t)^{-(1+\f1p)}$ in $L^p.$
We thus define $\mathfrak{F}_1^1$ as the  superposition of terms listed in \eqref{fF11} with $\pt c$ replaced by $\p_y^2\gamma.$ 
The typical term $\p_c Q_{c_0}(z) \p_y^2\gamma$ can thus be bounded as 
\begin{align*}
    \|\p_c Q_{c_0}(z) \p_y^2\gamma\|_{\dot{W}^{-(\f12+\f1p),p'}}\lesssim  \big\|\f{\p_y^2 \gamma}{|\na_y|^{1/2+2\vartheta}}\big\|_{L_y^{p'}}\lesssim (1+t)^{-(\f14+\f1p-\vartheta)} \cN_{\tilde{c},\gamma}(T).
\end{align*}
The other terms in \eqref{fF11} can be treated in a similar or simpler manner. For instance the term $\p_z\big(Q_{c}(z_1) (\pt\gamma-\tilde{c})\big)$ an be easily controlled thanks to the appearance of perfect derivative $\p_z.$  
The remaining quantity
$\mathfrak{F}_1^2=\mathfrak{F}_1-\mathfrak{F}_1^1$  is only made of quantities 
which have  favorable time decay and thus have  the property
\begin{align*}
    \|\mathfrak{F}_1^2\|_{{\dot{W}^{-(\f12+\f1p),p'}}}\lesssim   \|\mathfrak{F}_1^2\|_{L^{\f{1}{7/6-{2}/{3p}}}}\lesssim (1+t)^{-1}  \big(\cN_{\tilde{c},\gamma}(T)+\cN(T)^2\big).
\end{align*}

Proof of \eqref{es-ZfF}. It follows from  similar arguments as above, by observing that the  $Z$ derivative gives an  extra $(1+t)^{-\f12}$ decay.

Proof of \eqref{fF1-Lp}. Let us control for instance the first typical term in \eqref{cF1typicalterms}. Thanks to \eqref{es-ptc}, \eqref{es-nablayptc}
\begin{align*}
   \| \p_c Q_{c}(z_1) \pt c\|_{W^{k,r}}\lesssim \|\pt c\|_{L_y^r}\lesssim \|\pt c\|_{L^2}^{\f2r}\|\na_y\pt  c\|_{L^2}^{1-\f2r}\lesssim (1+t)^{-(\f{3}{2}-\f1r)} \cN(T)^2.
\end{align*}

 Proof of \eqref{es-fF1}. Let us control for instance the first typical term as
\begin{align*}
    \|Z (\p_c Q_{c}(z_1)\pt c)\|_{W^{m_{\kappa},8_{\kappa}'}}&\lesssim \|\pt Z c\|_{L_y^{8_{\kappa}'}}+\|(Z c, Z\gamma)\pt c\|_{L_y^{8_{\kappa}'}}\\
    &\lesssim \|\pt Zc\|_{L_y^2}^{\f{2}{8_{\kappa}}}\| y\pt  Zc\|_{L_y^2}^{1-\f{2}{8_{\kappa}}}+\|\pt c\|_{L^2}\|(Zc,Z\gamma)\|_{L^{\f{8}{3+\kappa}}}.
\end{align*}
In view of the estimates \eqref{es-ypyptc}, \eqref{es-nablayptc} and \eqref{cor-interpolation}, 
we can estimate the above right hand side by 
$(1+t)^{-\f{9-\kappa}{8}} \cN(T)^2.$ 
  \end{proof}
\begin{rmk}
    It follows from the dispersive estimate \eqref{disper-low} and the estimate \eqref{es-ZfF} that for any $p\in[2,4),$
    \begin{align}\label{es-Zalphas}
         \big\|P_{\leq -5} {Z\alpha_s}\big\|_{\dot{W}^{-1,p}}\lesssim (1+t)^{-\f{3}{2}(1-\f2p)}\big(\cN_{\tilde{c},\gamma}(T)+\cN(T)^2\big).
    \end{align}
\end{rmk}
\subsubsection{Estimate of $\mathfrak{F}_2$}
\begin{lem}
The following estimates hold.
 \begin{align}\label{fF2-half}
    \|\mathfrak{F}_2\|_{\dot{H}^{-\f12}} \lesssim  (1+t)^{-1}\big(\cM(0)+\cN_{\tilde{c},\gamma}(T)+ \cN(T)^2\big).
\end{align}  
For any  any $k\in [0, M-1]$, 
\begin{align}
    \big\|\f{\mathfrak{F}_2}{|\na|}\big\|_{\dot{W}^{k,p}}\lesssim (1+t)^{-(\f32-\f1p)}\big(\cM(0)+\cN_{\tilde{c},\gamma}(T)+ \cN(T)^2\big), \quad \forall\, p\in [4/3,3],\, k\in [0, M-1], 
    \label{es-fF2-Lp} \\
    \big\|\f{Z\mathfrak{F}_2}{|\na|}\big\|_{\dot{W}^{k,p}}\lesssim (1+t)^{-
    \f32}\big(\cM(0)+\cN_{\tilde{c},\gamma}(T)+ \cN(T)^2\big), \quad \forall\, p\in [2,3],\, k\in [0, M-2],   \label{es-ZfF2-Lp}
\end{align}
Moreover, it holds that
\beq \label{es-fF2}
\|Z\mathfrak{F}_2\|_{W^{m_{\kappa},8_{\kappa}'}}\lesssim (1+t)^{-\f{9-\kappa}{8}}\big(\cM(0)+\cN_{\tilde{c},\gamma}(T)+\cN(T)^2\big).
\eeq
\end{lem}
\begin{proof}
Proof of \eqref{fF2-half}.
Let $n(z)=\vr(z+c_0t), \, \tilde{v}(z)=v(z+c_0t).$
In view of the very definition of $\mathfrak{F}_2$ in \eqref{def-fF}, we can control it crudely
\begin{align*}
   \|\mathfrak{F}_2\|_{\dot{H}^{-\f12}} &\lesssim \|Q_c (z_1) (n, I_0 n, \tilde{v}+w)(z)\|_{\dot{H}^{\f12}}+\|w(z) \,(n, \tilde{v})\|_{\dot{H}^{\f12}}\\
   &\lesssim \|(n, \tilde{v}+w)\|_{H_a^1}+\|w\|_{W^{1,\infty}} \|(n, \tilde{v})\|_{H^1}. 
\end{align*}
Thanks to Proposition \ref{prop-weightednorm}, estimate \eqref{w+v-weighted} for the weighted norm and Proposition \ref{prop-w} for $w,$ we find 
\eqref{fF2-half}.

Proof of \eqref{es-fF2-Lp}-\eqref{es-ZfF2-Lp}.
Since $Q_c(\cdot)$ is sufficiently localized,
we have by using Sobolev embedding, interpolation that
\begin{align*}
    \|Q_c(z_1) f\|_{L^p}\lesssim \|e^{az}(\Id,\p_z)f\|_{L_y^pL_z^2}\lesssim \|f\|_{H_a^1}^{\f2p}\|\na_y f\|_{H_a^1}^{1-\f2p}, \quad (\forall\,  2\leq p<+\infty), \\
     \|Q_c(z_1) f\|_{L^p}\lesssim \|Q_c(z_1) \langle z\rangle f\|_{L_y^pL_z^2}\lesssim \|f\|_{H_a^1}^{2-\f2p}\|y f\|_{H_a^1}^{\f2p-1}, \quad (\forall\, 1<p\leq 2).
\end{align*}
Therefore, it holds that for $2\leq p\leq 3$
\begin{align*}
     \big\|\f{\mathfrak{F}_2}{|\na|}\big\|_{\dot{W}^{k,p}}\lesssim
     \|(n, \tilde{v}+w)\|_{H_a^{k+1}}^{\f2p}\|\na_y (n, \tilde{v}+w)\|_{H_a^{k+1}}^{1-\f2p}+ \|w(z) \,(n, \tilde{v})\|_{W^{k,p}}, 
\end{align*}
and for $\f43\leq p\leq 2,$
\begin{align*}
     \big\|\f{\mathfrak{F}_2}{|\na|}\big\|_{\dot{W}^{k,p}}\lesssim
     \|(n, \tilde{v}+w)\|_{H_a^{k+1}}^{2-\f2p}\|y (n, \tilde{v}+w)\|_{H_a^{k+1}}^{\f2p-1}+ \|w(z) \,(n, \tilde{v})\|_{W^{k,p}}.
\end{align*}
The estimate \eqref{es-fF2-Lp} then follows from Propositions \ref{prop-weightednorm}, \ref{prop-w}.
The estimate \eqref{es-ZfF2-Lp} follows from the same arguments by using the fact that  $(\na_y,\pt )(n, \tilde{v}+w)$ has better decay.

Proof of \eqref{es-fF2}.
It holds that  
\begin{align*}
\|Z\mathfrak{F}_2\|_{W^{m_{\kappa},8_{\kappa}'}}&\lesssim
\|Q_c (z-\gamma) \tilde{Z}(n, I_0 n, \tilde{v}+w)(z)\|_{W^{m_{\kappa}+1,8_{\kappa}'}}+\| \tilde{Z}(c,\gamma)\tilde{Q}_c(n,\tilde{v}+w),\|_{W^{m_{\kappa}+1,8_{\kappa}'}}\\
&\qquad\qquad \qquad\qquad\qquad\qquad\qquad\qquad\qquad +\|  \tilde{Z}\na(w(z) \,(n, \tilde{v}))\|_{W^{m_{\kappa},8_{\kappa}'}},
\end{align*}
where $ \tilde{Z}\in \{\na_y, \pt \}$ and $\tilde{Q}_c$ here stands for  $\p_c \tilde{Q}_c$ or $\tilde{Q}_c'.$
The first term can be controlled by using interpolation and Proposition \ref{prop-weightednorm}, 
$$\|\tilde{Z}(n, \tilde{v}+w)\|_{H_a^{m_{\kappa}+1}}^{\f{2}{8_{\kappa}}} \|y \tilde{Z}(n, \tilde{v}+w)\|_{H_a^{m_{\kappa}+1}}^{1-\f{2}{8_{\kappa}}}\lesssim (1+t)^{-\f{9-\kappa}{8}}  \big(\cM(0)+\cN_{\tilde{c},\gamma}(T)+\cN(T)^2\big).$$
The second term is bounded by
$\|\tilde{Z}(c,\gamma)\|_{L_y^{\f{8}{3+\kappa}}}\|\tilde{Z}(n, \tilde{v}+w)\|_{H_a^{m_{\kappa}+1}}\lesssim (1+t)^{-\f{13-\kappa}{8}}\cN(T)^2.$
Moreover, using  \eqref{es-w} for the estimate of $w,$ the third term can be controlled by  
\beqs 
\|(\na,\pt) w\|_{W^{m_{\kappa}+1,\f{8}{3+\kappa}}}\|(n, \tilde{v}))\|_{H^{m_{\kappa}+1}}+\|w\|_{W^{m_{\kappa},24}}\|\tilde{Z}\na (n, \tilde{v}))\|_{W^{m_{\kappa},6/5}}\lesssim (1+t)^{-\f{9}{8}} \cN(T)^2.
\eeqs
Note that we have used the Sobolev embedding $L^{6/5} \hookrightarrow H^{-1}.$ To summarize, we obtain \eqref{es-fF2}.
\end{proof}
\subsubsection{Estimate of $\mathfrak{F}_3$}
To estimate $\mathfrak{F}_3,$ it is useful to first 
state some decay estimates for $\alpha$ in $L^p (2\leq p\leq +\infty).$ 
\begin{lem} 
(1) (Low frequency estimate).
  Let $p\in[2,+\infty]$ and $\theta$ as in  \eqref{relation-theta-p}, it holds that
\begin{align}\label{alphaLP-low}
  \|P_{\leq 0}\,\alpha\|_{\dot{W}^{\theta,p}}\lesssim (1+t)^{-\f{1}{3}\big(2+\theta-\f3p\big)(1+\f{\kappa}{18})+\cO(\kappa^2)}\, \cN(T)\, .
\end{align}
More preciserly,  we have that 
\beq\label{L4-LF}
\begin{aligned}
 \|P_{\leq 0}\,\alpha\|_{L^4}\lesssim (1+t)^{-\f{5}{12}\big(1+\f{\kappa}{18}+\cO(\kappa^2) \big)} \cN(T),\quad 
 \|P_{\leq 0}\,\alpha\|_{{\dot{W}^{\f12,4}}}\lesssim (1+t)^{-\f{7}{12}\big(1+\f{\kappa}{18}+\cO(\kappa^2) \big)} \cN(T).
\end{aligned}
\eeq
(2). (High regularity estimate).
Let $\ell_{\kappa}=1+\f{3}{8_{\kappa,1}},  \beta_{\kappa}=\f{4}{3+\kappa}\big(1-\f2p\big).$
 It holds that, for any $2\leq p\leq 8_{\kappa},$ 
 any $k\geq 0$ such that $M\geq \ell_{\kappa}+1+k/(1-\beta_{\kappa}),$
\begin{align}\label{highreg-interp}
    \|P_{\geq -10}\,\alpha\|_{W^{\ell_{\kappa}+1+k,p}}\lesssim (1+t)^{-\f43(1-\f2p)+\cO(\kappa)} \cN(T). \quad 
\end{align}

\end{lem}
\begin{proof}
The estimate \eqref{alphaLP-low} is a slightly rougher version of \eqref{interpolation-import} proved in Lemma \ref{lem-interpo-low}.  
Specifically, we find by applying the interpolation \eqref{interpolation-import} that
\begin{align*}
    \|P_{\leq 0}\alpha\|_{L^4}\lesssim \big\|P_{\leq 0} 
  \alpha\big\|_{\dot{W}^{\ell_{\kappa}, 8_{\kappa}}}^{\theta_{\kappa}}\|P_{\leq 0}\alpha\|_{\dot{H}^{-1/2}}^{1-\theta_{\kappa}}, \quad
     \|P_{\leq 0}\alpha\|_{\dot{W}^{\f12,4}}\lesssim \big\|P_{\leq 0} 
\alpha\big\|_{\dot{W}^{\ell_{\kappa},8_{\kappa}}}^{\vartheta_{\kappa}}\|P_{\leq 0}\alpha\|_{\dot{H}^{-1/2}}^{1-\vartheta_{\kappa}},
\end{align*}
where 
\begin{align*}
    \theta_{\kappa}=\f{5}{12}\big(1+\f{\kappa}{72} \big)+\cO(\kappa^2),\, \quad  \vartheta_{\kappa}=\f{7}{12} \big(1+\f{\kappa}{72} \big)+\cO(\kappa^2).
\end{align*}
   The estimate \eqref{highreg-interp} is the consequence of the Gagliardo-Nirenberg inequality
\begin{align*}
    \|P_{\geq -10} \, \alpha\|_{\dot{W}^{\ell_{\kappa}+1+k,p}}\lesssim \|P_{\geq -10}\,\alpha\|_{W^{\ell_{\kappa}+1,8_{\kappa}}}^{\beta_{\kappa}} \|P_{\geq -10}\,\alpha\|_{H^M}^{1-\beta_{\kappa}}
\end{align*}
where $k$ is such that $M\geq \ell_{\kappa}+1+k/(1-\beta_{\kappa}).$

\end{proof}
Based on the above lemma, we are able to show the following estimates for  $\mathfrak{F}_3:$
\begin{lem}
 Let $\kappa>0$ be small enough,  it holds that, 
\beq\label{fF3-half}
 \|\mathfrak{F}_3\|_{\dot{H}^{-\f12}}\lesssim (1+t)^{-(1+\f{\kappa}{36})} \cN(T)^2,
 \eeq
 and for any $2\leq p\leq 4,$ 
 \beq\label{fF3-negative1-low}
\|P_{\leq -5}\mathfrak{F}_3
\|_{\dot{W}^{-(\f12+\f1p), p'}}\lesssim (1+t)^{-(\f12+\f{2}{3p})} \cN(T)^2.
 \eeq
\end{lem}
\begin{proof}
    Applying the Kato-Ponce inequality, we find that 
\begin{align*}
   \|\mathfrak{F}_3\|_{\dot{H}^{-\f12}}\lesssim 
   \|\alpha^2\|_{\dot{H}^{\f12}}\lesssim \|\alpha\|_{\dot{W}^{\f12, 4}
   }\|\alpha\|_{L^{4}},
\end{align*}
The estimate \eqref{fF3-half} then follows from  \eqref{L4-LF} and \eqref{highreg-interp}.
To prove \eqref{fF3-negative1-low}, we apply \eqref{alphaLP-low} \eqref{highreg-interp} to find 
\begin{align*}
   &\|P_{\leq -5}\mathfrak{F}_3
\|_{\dot{W}^{-(\f12+\f1p), p'}}\lesssim \|P_{\leq -5} |\na|^{\f12-\f1p}(\cR\alpha)^2\|_{L^{p'}}\\
&\lesssim \|\alpha\|_{L^{5/2}} \|\alpha\|_{\dot{W}^{\f12-\f1p,\f{1}{1/p'-2/5}}}\lesssim (1+t)^{-(\f12+\f{2}{3p})(1+\f{\kappa}{18})+\cO(\kappa^2)} \cN(T)^2.
\end{align*}

\end{proof}

\subsubsection{Estimate of $\mathfrak{F}_4$}
\begin{lem}
The following three estimates hold
    \begin{align}\label{esfF4-neghalf}
   (1+t)^{\f53+\cO(\kappa)} \|\mathfrak{F}_4\|_{\dot{H}^{-\f12}}+ (1+t)^{\f43+\cO(\kappa)}\big\|\f{\mathfrak{F}_4}{|\na|}\big\|_{H^{M}} 
   \lesssim \cN(T)^2+\cN(T)^3, 
\end{align}
\begin{align}\label{fF4-negativeLp}
\|P_{\leq -5}\,\mathfrak{F}_4
\|_{\dot{W}^{-(\f12+\f1p), p'}}\lesssim (1+t)^{-(\f76+\f{2}{3p})+\cO(\kappa)} \big(\cN(T)^2+\cN(T)^3\big), \quad \forall \, 2\leq p\leq 4,
\end{align}
   \begin{align}\label{es-fF4}
\|Z\mathfrak{F}_4\|_{W^{m_{\kappa},8_{\kappa}'}}\lesssim
(1+t)^{-\f{7}{6}+\cO(\kappa)} \big(\cN(T)^2+\cN(T)^3\big).
\end{align}
where $\mathcal{O}(\kappa)$ means a quantity whose absolute can be bounded by $C\kappa$ for some $C>0$ independent of $\kappa.$ 
\end{lem}

\begin{proof}
By Taylor expansion, we have
\beq\label{exp-H}
H(n_c(z_1), \vr)=\f{\vr^3}{2}\int_0^1 h'''(1+n_c(z_1)+\theta\vr) (1-\theta)^2 \, \d\theta,
\eeq
we can thus control $\na H(n_c(z_1), \vr)$ by using  \eqref{alphaLP-low}-\eqref{highreg-interp}:
\begin{align*}
    \|\na H(n_c(z_1), \vr)\|_{\dot{H}^{-\f12}}\lesssim \|\vr^3\|_{\dot{H}^{\f12}} \|h(\cdot)\|_{C^4\big([\f12,\f32]\big)}
&\lesssim \|\vr\|_{\dot{W}^{\f12,4}}\|\vr\|_{L^4}\|\vr\|_{L^{\infty}}
\\
&\lesssim (1+t)^{-(\f53+\cO(\kappa))} \cN(T)^3.
\end{align*}
Concerning the term $\na \tilde{G}(\phi_c(z_1), \vr),$ 
we apply the expansion \eqref{errorphi-cgamma} to write 
\beq\label{exp-tG}
\begin{aligned}
\tilde{G}(\phi_c(z_1), \vr)&=\f12 \big(I_{\phi_c(z_1)}\big(I_{\phi_c(z_1)}\vr\big)^2-I_0(I_0 \vr)^2\big)\\
&+\f12 \big(I_{\phi_c(z_1)}(I_{\phi_c(z_1)}\Delta_y \phi_{c}(z_1))^2+2I_{\phi_c(z_1)} \vr\, I_{\phi_c(z_1)}\Delta_y \phi_{c}(z_1)\big)+\text{cubic terms in } \vr, \Delta_y \phi_{c}
\end{aligned} 
\eeq
Consequently, it holds that 
\begin{align*}
    \|\na \tilde{G}(\phi_c(z_1), \vr)\|_{\dot{H}^{-\f12}}\lesssim \|\tilde{G}(\phi_c(z_1), \vr)\|_{\dot{H}^{\f12}}&\lesssim \|\vr(\cdot+c_0 t)\|_{H_a^1}\|\vr\|_{W^{1,\infty}}+\|\Delta_y \phi_{c}\|_{H^1} \|(\Delta_y \phi_c, \vr)\|_{W^{1,\infty}},
\end{align*} 
which, combined with the estimate
$\|\Delta_y \phi_{c}\|_{H^3}\lesssim \|\big(\Delta_y(\gamma, c), \na_y\gamma\cdot\na_y c\big)\|_{L^2}\lesssim (1+t)^{-1}\cN(T),$ can be bounded further by $ (1+t)^{-(\f53+\cO(\kappa))}\cN(T)^2.$ To summarize, we have finished the estimate of $\mathfrak{F}_4$ in $\dot{H}^{-\f12}.$ The estimate of $\big\|\f{\mathfrak{F}_4}{|\na|}\big\|_{H^{M}}$ can be controlled in the similar way. For instance we have
\beqs 
\|H(n_c(z_1), \vr) \|_{H^{M}}\lesssim \|\vr\|_{L^{\infty}}^2\|\vr\|_{H^M}\lesssim (1+t)^{-(\f43+\cO(\kappa))}\cN(T)^2\, .
\eeqs
The estimate \eqref{fF4-negativeLp} can be derived using the expansions \eqref{exp-H} and \eqref{exp-tG}, following a proof similar to that of \eqref{fF3-negative1-low}. The details are omitted.

Let us now prove \eqref{es-fF4}. 
We can control the cubic term $Z\na \vr^3$ roughly as
\beqs 
\|\na Z \vr^3\|_{W^{m_{\kappa},8_{\kappa}'}}\lesssim 
\|Z \vr^3\|_{W^{m_{\kappa}+1,8_{\kappa}'}}\lesssim 
 \|\vr\|_{W^{m_{\kappa}+1,\f{8}{3+\kappa}}}\|\vr\|_{L^{\f{10}{3}}}\|Z\vr\|_{L^5}+\|P_{\geq 1} Z \vr\|_{W^{m_{\kappa}+1,\f{8}{3+\kappa}}}
\|\vr\|_{L^4}^2. 
\eeqs
On the one hand, in view of \eqref{alphaLP-low},  it holds that
\begin{align*}
  &\|P_{\leq 0}\, \vr\,\|_{L^{\f{8}{3+\kappa}}}\lesssim  (1+t)^{-(\f{7}{24}+\cO(\kappa))} \cN(T),\\
& \|P_{\leq 0}\, \vr\,\|_{L^{\f{10}{3}}}\lesssim  (1+t)^{-(\f{11}{30}+\cO(\kappa))} \cN(T), \quad  \|P_{\leq 0} \na \vr\,\|_{L^{5}}\lesssim  (1+t)^{-(\f{4}{5}+\cO(\kappa))} \cN(T).
\end{align*}
On the other hand, we apply
\eqref{highreg-interp} to find that
\begin{align*}
    &\|P_{\geq 0}\, \vr\,\|_{L^{\f{10}{3}}}\lesssim  (1+t)^{-(\f{8}{15}+\cO(\kappa))} \cN(T), \quad  \|P_{\geq 0} \na \vr\,\|_{L^{5}}\lesssim  (1+t)^{-(\f{4}{5}+\cO(\kappa))} \cN(T),\\
   & \|P_{\geq 0}\vr\|_{\dot{W}^{m_{\kappa}+2,\f{8}{3+\kappa}}}
\lesssim (1+t)^{-(\f13+\cO(\kappa))}\cN(T),
\end{align*}
as long as $M\geq \ell_{\kappa}+\f{19}{4}+\cO(\kappa).$
The previous estimate, together with \eqref{L4-LF}, leads to the desired estimate for $\na Z\vr^3.$
By more computations, we can also  get  that
the term $Z\tilde{G}$ can be bounded by $(1+t)^{-2}(\cN(T)^2+\cN(T)^3),$
we omit the details.
    
\end{proof}
\subsubsection{Supplementary results} 
We state in this short subsection a lemma which will be  often  used  later in the  proof. 
\begin{lem}
    Let $Z\in \{\na_y, \pt+c_0\p_x\}.$ Under the same assumption as in Proposition \ref{prop-decayes},
    it holds that for any $p\in[2,+\infty), k\leq M-2,$
    \begin{align}\label{zalpha/na}
        \big\|\f{Z\alpha}{|\na|}(t)\big\|_{W^{k,p}}\lesssim \|\alpha(t)\|_{W^{k,p}}+(1+t)^{-(\f23+\cO(\kappa))} \big(\cM(0)+\cN(T)^2\big).
    \end{align}
\end{lem}
\begin{proof}
   As  it is clear when $Z=\na_y,$ we focus on the case 
   $Z=\pt+c_0\p_x.$ We use the equation \eqref{eq-alpha} to obtain that 
   $(\pt+c_0\p_x)\alpha=(iP(D)+c_0\p_x)\alpha+\sum_j \mathfrak{N}_j.$ Since $\f{P(D)}{|\na|}, \f{\p_x}{|\na|}\in B(L^p),$ it suffices to show that 
   \begin{align*}
       \big\|\f{\mathfrak{N}_j}{|\na|}\big\|_{W^{k,p}}\lesssim (1+t)^{-\f23+\cO(\kappa)} \big(\cM(0)+\cN(T)^2\big).
   \end{align*}
 We can control roughly the term $\cN_3\approx |\na|\cR(\cR \alpha)^2$  by using the product estimate \eqref{productineq}  
\begin{align*}
   \big\|\f{\mathfrak{N}_3}{|\na|}\big\|_{W^{k,p}}\lesssim 
   \|\cR \alpha\|_{L^{\f{1}{\kappa}}}\|\cR \alpha\|_{W^{k,\f{p}{1-p\kappa}}}\lesssim  \| \alpha\|_{L^{\f{1}{\kappa}}}\| \alpha\|_{H^M}
   \lesssim 
   (1+t)^{-\f23+\cO(\kappa)} 
   \cN(T)^2.
\end{align*}
Note that we have by the interpolation \eqref{interpolation-import} that 
\begin{align*}
     \| \alpha\|_{L^{\f{1}{\kappa}}}\lesssim  (1+t)^{-\f23+\cO(\kappa)}\cN(T).
\end{align*}
By using  similar computations as  in the the previous subsections, one can show that 
\begin{align*}
     \big\|\f{\mathfrak{N}_j}{|\na|}\big\|_{W^{k,p}}\lesssim  
   (1+t)^{-1+\cO(\kappa)} \big(\cM(0)+
   \cN(T)^2\big), \quad j=1,2,4.
\end{align*}
This ends the proof.

\end{proof}
\subsection{Time decay estimates in $L^{8_{\kappa}}$--proof of  \eqref{decayes-Zder}.}\label{subsection-decayes}
In this subsection, we prove  the decay estimate \eqref{decayes-Zder}.
Let $Z\in \{\na_y, \pt+c_0\p_x\}.$
Applying the dispersive estimate \eqref{dispersive-L8} to 
\beqs \label{Formula-alpha}
Z\alpha=e^{it P(D)}Z\alpha_0+\int_0^t e^{i(t-s) P(D)} Z\big(\mathfrak{N}_1+\mathfrak{N}_2+\mathfrak{N}_3+\mathfrak{N}_4\big)\, \d s,
\eeqs 
we obtain that 
\begin{align*}
 &\big\|Z\alpha-\int_0^t e^{i(t-s) P(D)} Z\mathfrak{N}_3\,\d s\big\|_{W^{\ell_{\kappa}, 8_{\kappa}}}\\
 &\lesssim (1+t)^{-(1+\f{\kappa}{3})}\|Z\alpha_0\|_{W^{m_{\kappa},8_{\kappa}'}} +\int_0^t (1+t-s)^{-(1+\f{\kappa}{3})}\|Z(\mathfrak{N}_1, \mathfrak{N}_2, \mathfrak{N}_4)(s)\|_{W^{m_{\kappa},8_{\kappa}'}} \, \d s, \quad 
\end{align*}
where $\ell_{\kappa}=1+\f{3}{8_{\kappa,1}}, m_{\kappa}=\ell_{\kappa}+\f{3(3+\kappa)}{4}.$  
Thanks to the estimates \eqref{es-fF1}, \eqref{es-fF2}, \eqref{es-fF4} that have been shown in Subsection \ref{subsection-fF124}, it holds that for $\kappa$ small enough,
\begin{align*} 
\|Z(\mathfrak{N}_1, \mathfrak{N}_2, \mathfrak{N}_4)(s)\|_{W^{m_{\kappa},8_{\kappa}'}}\lesssim \|Z(\mathfrak{F}_1, \mathfrak{F}_2, \mathfrak{F}_4)(s)\|_{W^{m_{\kappa},8_{\kappa}'}}\lesssim  (1+s)^{-\f{9-\kappa}{8}} \big(\cM(0)+\cN_{c,\gamma}(T)+ \cN(T)^2\big).
\end{align*}
The previous two estimates then lead to
\beq\label{decayL8-N124}
\big\|Z\alpha-\int_0^t e^{i(t-s) P(D)} Z\mathfrak{N}_3\,\d s\big\|_{W^{1+\f{3}{8_{\kappa,1}}, 8_{\kappa}}}\lesssim (1+t)^{-(1+\f{\kappa}{3}})\big(\cM(0)+\cN_{\tilde{c},\gamma}(T)+ \cN(T)^2\big).
\eeq 
It thus remains to prove the decay estimate for the term 
$\int_0^t e^{i(t-s) P(D)} \mathfrak{N}_3\, \d s,$ which is the most involved part of the analysis. Since $\mathfrak{N}_3$ is quadratic in terms of $(\vr, v),$ the direct application of the dispersive estimate is not enough to close the estimate. As in the study of the stability of the constant equilibrium \cite{G-P-global}, we need to use  the normal form transformation to `change' the quadratic terms into cubic ones.

 Plugging the relation \eqref{rel-v-alpha} into
the definition $$\mathfrak{N}_3=\div(\vr \, v)+ \f{i}{2} \f{\div}{P(D) }\nabla\big( |v|^2+h''(1)\vr^2+I_{0}(I_{0}\vr)^2\big),$$  we find that $\mathfrak{N}_3$ takes the form 
\begin{align*}
 \mathfrak{N}_3=\sum_{\mu,\nu\in\{\pm\}}\cF^{-1} \bigg( \int_{\mR^3} m_{\mu\nu}(\varsigma, \varsigma') \,\alpha^{\mu}(\varsigma-\varsigma') \,\alpha^{\nu} (\varsigma')\, \d \varsigma'\bigg)
\end{align*}
where for a function $f$ we denote $f^{+}=f, f^{-}=\overline{f}.$ Moreover, $m_{\mu,\nu}$ 
is a sum  of terms under  the form  $|\varsigma|n_0(\varsigma)n_1(\varsigma-\varsigma')n_2(\varsigma'),$ where $n_0, n_1, n_2$ are multipliers that satisfy the  homogeneous  $0$ condition
$|\na^{\gamma} n_j (\cdot)|\lesssim |\cdot|^{-|\gamma|},\, \forall \gamma\in \mR^3,$
and thus are bounded in $L^p(\mathbb{R}^3)$ for any $p\in(1,+\infty).$
Let $\beta^{\mu}=e^{-\mu i t P(D)} \alpha^{\mu},$ then we can write 
\begin{align*}
\int_0^t e^{i(t-s) P(D)} \mathfrak{N}_3\, \d s=e^{itP(D)}\sum_{\mu,\nu\in\{\pm\}}\cF^{-1} \bigg(\int_0^t \int_{\mR^3} e^{-is \phi^{\mu\nu}} m_{\mu\nu}(\varsigma, \varsigma') \,\beta^{\mu}(s,\varsigma-\varsigma') \,\beta^{\nu} (s, \varsigma')\, \d \varsigma'\d s\bigg)
\end{align*}
where $\phi^{\mu\nu}(\varsigma, \varsigma')=P(\varsigma)-\mu P(\varsigma-\varsigma')-\nu P(\varsigma').$ 
Acting $Z$ on the above identity and 
 integrating by parts in time, we find, by using the equation $\p_s \beta^{\mu}=e^{-\mu i t P(D)}(\sum_{j=1}^4\mathfrak{N}_j)^{\mu}$ that 
 \beq\label{expandFn3}
\begin{aligned}
\int_0^t e^{i(t-s) P(D)} Z \mathfrak{N}_3\, \d s&=
\sum_{\mu,\nu\in\{\pm\}} \bigg(e^{itP(D)}B^{\mu\nu}\big(Z\alpha^{\mu}(0),\alpha^{\nu}(0)\big)- B^{\mu\nu}\big(Z\alpha^{\mu}(t),\alpha^{\nu}(t)\big)+\text{symmetric terms}\\
&\qquad + \sum_{j=1}^4 \int_0^t e^{i(t-s)P(D)} \big(B^{\mu\nu}\big(Z\mathfrak{N}_j^{\mu}(s), \alpha^{\nu}(s)\big)+ B^{\mu\nu}\big(Z\alpha^{\mu}(s), \mathfrak{N}_j^{\nu}(s)\big)\big)\,\d s\bigg)
\end{aligned}
\eeq
 where 
 \begin{align*}
    \text{symmetric terms}=\sum_{\mu,\nu\in\{\pm\}} \bigg(e^{itP(D)}B^{\mu\nu}\big(\alpha^{\mu}(0), Z\alpha^{\nu}(0)\big)- B^{\mu\nu}\big(\alpha^{\mu}(t), Z\alpha^{\nu}(t)\big)\bigg)\\
    =\sum_{\mu,\nu\in\{\pm\}} \bigg(e^{itP(D)}\tilde{B}^{\mu\nu}\big(Z\alpha^{\mu}(0),\alpha^{\nu}(0)\big)- \tilde{B}^{\mu\nu}\big(Z\alpha^{\mu}(t),\alpha^{\nu}(t)\big)\bigg)
\end{align*}
and 
 \begin{align*}
     B^{\mu\nu}(f,g)=\cF^{-1}\bigg(\int_{\mR^3}\f{m_{\mu\nu}} {i\,\phi^{\mu\nu}} (\varsigma-\varsigma', \varsigma') \,f(\varsigma-\varsigma') \,g (\varsigma')\, \d \varsigma' \bigg), \\
      \tilde{B}^{\mu\nu}(f,g)=\cF^{-1}\bigg(\int_{\mR^3}\f{m_{\nu\mu}(\varsigma, \varsigma-\varsigma')} {i\,\phi^{\mu\nu}(\varsigma, \varsigma') } \,f(\varsigma-\varsigma') \,g (\varsigma')\, \d \varsigma' \bigg).
 \end{align*}
 Note that $B^{\mu\nu}$ and $\tilde{B}^{\mu\nu}$ differ from each other by multipliers of type $\cR(\varsigma-\varsigma')\cR(\varsigma')$ where $\cR$ is bounded in any $L^p(\mR^3), (1<p<+\infty),$ 
and thus have the same bilinear estimates. Let us also observe that
$$\max\{|\varsigma-\varsigma'|, |\varsigma'| \}^{-1}\lesssim |\phi^{\mu\nu}|^{-1}\lesssim \min\{|\varsigma-\varsigma'|, |\varsigma'| \}^{-1}$$ and thus $|\phi^{\mu\nu}|^{-1}$  
has some singularities in low frequencies. This leads to  unfavorable bilinear estimates for $B^{\mu\nu}(f,g)$ when both $f$ and $g$ have  low frequencies. For instance, by \eqref{bilinear-worst}, it holds that for any $r\in [2,3],$
\begin{align*}
    \|B^{\mu\nu}(P_{\leq -5}f, P_{\leq -5}g)\|_{L^r}\lesssim \big\|P_{\leq -5}\f{f}{|\na|}\big\|_{L^p} \|P_{\leq -5}\f{g}{|\nabla|}\|_{L^q}, \quad 
     \big(\f{1}{p}+\f{1}{q}=
     \f{1}{3}+\f{1}{2r}+\ep, \, \, 0< \ep<<1 \big),
\end{align*}
 Note that not only we have some singularities in low frequency, but also we lose some integrability. Moreover, 
 we do not have the corresponding estimate for $r>3$ (or allowing the case $r>3$ produces more singularities on $f$ and $g$ in low frequency). 
 These leads  to  technical difficulties, especially when controlling the boundary term 
 $B^{\mu\nu}\big(Z\alpha^{\mu}(t), \alpha^{\nu}(t)\big).$ Indeed, by Bernstein inequality, for $k<<j\leq 0,$ we have
 \begin{align*}
 2^{\f{3}{8_{\kappa,1}}j}  \|  B^{\mu\nu}\big(P_j Z\alpha^{\nu}(t), P_k\alpha^{\mu}(t))\|_{L^{8_{\kappa}}}&\lesssim 2^{-\f{\kappa}{24}j}\|  B^{\mu\nu}\big(P_j Z\alpha^{\mu}(t), P_k\alpha^{\nu}(t))\|_{L^{3}}\\
 &\lesssim 
 2^{-\f{\kappa}{24}j}\|P_j Z\alpha^{\mu}(t)\|_{L^{8^{-}}}2^{-k}\|P_k \alpha^{\nu}(t)\|_{L^{8/3}}.
 \end{align*}
The a priori estimates \eqref{decayes-vrv} are not enough to get a suitable decay estimate.
Indeed, by \eqref{decayes-vrv} and  the interpolation estimate  \eqref{interpolation-import}, it holds that
$\|\alpha^{\nu}\|_{\dot{W}^{-1, 8/3}}\lesssim \cN(T)$ 
and $\|Z\alpha^{\mu}(t)\|_{L^8}$ decays at best like $(1+t)^{-7/8}.$ To resolve the problem, our strategy is to split 
$$ Z\alpha=\tilde{\alpha}+\sum_{\mu,\nu\in\{\pm\}}B_{LL}^{\mu\nu}(Z\alpha^{\mu},\alpha^{\nu})+\tilde{B}_{LL}^{\mu\nu}(Z\alpha^{\mu}, \alpha^{\nu}),$$
 where 
 \beqs
B_{LL}^{\mu\nu}  (f, g)\colon = 
\sum_{ k+5\leq j\leq -5} B^{\mu\nu}(P_j f, P_{k} \,g),\, \quad  \tilde{B}_{LL}^{\mu\nu}  \colon =\sum_{k+5\leq j\leq -5}\tilde{B}^{\mu\nu}(P_j f, P_{k} \,g),\,
 \eeqs
We will prove on the one hand that 
\beq \label{tildealpha}
\|\tilde{\alpha}(t)\|_{W^{\ell_{\kappa}
,8_{\kappa}}}\lesssim (1+t)^{-(1+\f{\kappa}{3})}\big(\cM(0)+\cN(T)^2 \big), \quad (\ell_{\kappa}=1+{3}/{8_{\kappa,1}}),
\eeq
and on the other hand, by using \eqref{tildealpha} that
\beq\label{bilinear-LL} 
\||\na|^{\f{3}{8_{\kappa,1}}} (B_{LL}^{\mu\nu}, \tilde{B}_{LL}^{\mu\nu})(Z\alpha^{\mu},\alpha^{\nu})\|_{L^{8_{\kappa}}}\lesssim (1+t)^{-(1+\f{\kappa}{24})}\big(\cM(0)+\cN(T)^2 \big),\, \quad \forall \, \mu,\nu\in\{\pm\}.
\eeq
Let us first derive \eqref{bilinear-LL} by assuming that 
\eqref{tildealpha} is true. Since we will use the  estimates that hold for all $\mu,\nu \in\{\pm\},$ we
 skip the superscript $\mu,\nu$ for the sake of clarity. 
On the one hand, by Bernstein inequality and the bilinear estimate \eqref{bilinear-worst} (taking $\ep=\f{\kappa}{72}$), we have  
\begin{align*}
 \||\na|^{\f{3}{8_{\kappa,1}}} B_{LL}(\tilde{\alpha},\alpha)\|_{L^{8_{\kappa}}}
& \lesssim \big\| B_{LL}\big(|\na|\tilde{\alpha}, {\alpha}
\big)\big\|_{L^{\f{3}{1+\kappa/24}}}\lesssim \|P_{\leq -5}\,\tilde{\alpha}\|_{L^{8_{\kappa}}} \big\|P_{\leq -10}\f{\alpha}{|\na|}\big\|_{L^{\f{8}{3+7\kappa/6}}}\, .
\end{align*}
Applying \eqref{negative-sobolev} with $p=\f{8}{3+7\kappa/6}$
and using the estimate \eqref{tildealpha}, we obtain that 
\beqs
\||\na|^{\f{3}{8_{\kappa,1}}} B_{LL}(\tilde{\alpha},\alpha)\|_{L^{8_{\kappa}}}\lesssim (1+t)^{-(1+\f{\kappa}{24})}
\big(\cM(0)+\cN(T)^2 \big)^2\, .
\eeqs
On the other hand, we deal with the remaining term by applying the bilinear estimate \eqref{bilinear-worst} twice
\begin{align*}
 &   \||\na|^{\f{3}{8_{\kappa,1}}} B_{LL}\big(B_{LL}(Z\alpha, \alpha),\alpha\big)\|_{L^{8_{\kappa}}}
 \lesssim \big\|  B_{LL}\big(|\na|B_{LL}(Z\alpha, \alpha),\alpha\big)\big\|_{L^{\f{3}{1+\kappa/24}}}\\
&\lesssim  \|B_{LL}(Z\alpha, \alpha)\|_{L^{3}} \big\|P_{\leq 5}\f{\alpha_s}{|\na|}\big\|_{L^{\f{6}{1+\kappa/12}}}\lesssim \big\|P_{\leq -10}\, \f{Z\alpha}{|\na|}\big\|_{L^3}\,\big\|P_{\leq 5}\f{\alpha}{|\na|}\big\|_{L^{\f{6}{1+\kappa/12}}}^2.
\end{align*}
We can derive from \eqref{zalpha/na} and \eqref{alphaLP-low} that
\beqs
\big\|P_{\leq -10}\, \f{Z\alpha}{|\na|}\big\|_{L^3}\lesssim \big\|P_{\leq -10}\, \alpha\big\|_{L^3}+\sum_{j=1}^4\big\|\f{\mathfrak{N}_j}{|\na|}\big\|_{L^3}\lesssim (1+t)^{-\f13+\cO(\kappa)}\cN(T).
\eeqs
This, combined with the estimate \eqref{negative-sobolev} with $p=\f{6}{1+\kappa/12},$ leads to 
\begin{align*}
      \||\na|^{\f{3}{8_{\kappa,1}}} B_{LL}\big(B_{LL}(Z\alpha, \alpha),\alpha\big)\|_{L^{8_{\kappa}}}
 \lesssim (1+t)^{-\f76+\cO(\kappa)} \big(\cM(0)+\cN(T)^3\big).
\end{align*}
The same estimate holds also when $B_{LL}$ is replaced by $\tilde{B}_{LL}.$ We thus finish the proof of \eqref{bilinear-LL}.

It now remains to prove \eqref{tildealpha}. Thanks to the estimate \eqref{decayL8-N124} and the identity \eqref{expandFn3}, we are left to bound the  terms (we skip again the superscript $\mu,\nu$)
$$e^{itP(D)}B\big(Z\alpha(0),\alpha(0)\big),\qquad \mathbb{\cQ}_0(t):= (B-B_{LL})\big(Z\alpha(t),\alpha(t)\big)$$
and 
\beqs 
\mathbb{\cQ}_l(t)=\colon \int_0^t e^{i(t-s)P(D)} \bigg(B\big(Z\alpha(s), \mathfrak{N}_l(s)\big)+B\big(Z\mathfrak{N}_l(s), \alpha(s)\big)\bigg)\,\d s , \quad (l=1,2,3,4).
\eeqs
Applying the dispersive estimate \eqref{dispersive-L8} and bilinear estimates \eqref{BL-High-less2} and \eqref{BL-low-moreder}, we obtain 
\begin{align*}
&\|e^{itP(D)}B\big(Z\alpha(0),\alpha(0)\big)\|_{W^{\ell_{\kappa},8_{\kappa}}}\lesssim (1+t)^{-(1+\f{\kappa}{3})}\| B\big(Z\alpha(0),\alpha(0)\big)\|_{W^{m_{\kappa},8_{\kappa}'}} \\
&\lesssim (1+t)^{-(1+\f{\kappa}{3})} \bigg(\big\|\f{\alpha(0)}{|\na|}\big\|_{W^{\tilde{m}_{\kappa},\f{8}{3+\kappa}}}\big\|\f{Z\alpha(0)}{|\na|}\big\|_{H^{\tilde{m}_{\kappa}}}+ \big\|\f{\alpha(0)}{|\na|}\big\|_{L^{\f{24}{11+\kappa}}}\big\|\f{Z\alpha(0)}{|\na|}\big\|_{L^2}\bigg)\\
&\lesssim (1+t)^{-(1+\f{\kappa}{3})} \|\alpha(0)\|_{\dot{H}^{-1}\cap H^{\tilde{m}_{\kappa}}}^2\lesssim (1+t)^{-(1+\f{\kappa}{3})}  \cM^2(0),
\end{align*}
where $\tilde{m}_{\kappa}=m_{\kappa}+\f{5}{2}+\kappa.$
Note that we have used $(\pt\alpha)(0)=\mathfrak{N}_2(0)+\mathfrak{N}_3(0)+\mathfrak{N}_4(0).$
The estimates of $\cQ_0-\cQ_4$ will be presented in the following five subsections.
\subsubsection{Estimate of $\cQ_0(t)$}
We write 
\begin{align*}
B-B_{LL}\big(
f, g\big)=&
\bigg(\sum_{j-5\leq k\leq -10}+\sum_{k\geq -10}
+\sum_{j\geq -5 \geq k+5 }\bigg)B(P_j f, P_k g)\\
&=\colon
B_{LL}^1(f,g)+B_{H,2}(f,g)+B_{HL}(f,g).
\end{align*}
Let us begin with the control of $B_{LL}^1(Z\alpha(t), \alpha(t)).$ 
Thanks to the bilinear estimate \eqref{bilinear-worst}, 
we have
\begin{align*}
  \| B_{LL}^1(Z\alpha_s(t), \alpha(t))\|_{W^{\ell_{\kappa},8_{\kappa}}} &\lesssim 
    \| B_{LL}^1\big(Z\alpha_s(t), |\na|^{\f{5+3\kappa}{8}}\alpha(t)\big)\|_{L^{3}}\lesssim \big\|{Z\alpha_s(t)}\|_{\dot{W}^{-1,\f{4}{1+\kappa}}}\|\alpha(t)\|_{\dot{W}^{-\f{3(1-\kappa)}{8},4}} ,\\
     \| B_{LL}^1(Z\alpha_r(t), \alpha(t))\|_{W^{\ell_{\kappa},8_{\kappa}}},
    &\lesssim 
    \| B_{LL}^1\big(\f{Z\alpha_r(t)}{|\na|}, |\na|^{\f{13+3\kappa}{8}}\alpha(t)\big)\|_{L^{3}}\lesssim \big\|\f{Z\alpha_r(t)}{|\na|}\|_{\dot{W}^{-1,\f{4}{1+\kappa}}}\|\alpha(t)\|_{\dot{W}^{\f{5+3\kappa}{8},{4}}}.
\end{align*}
By 
\eqref{alphaLP-low}, it holds that
\begin{align}\label{someinterp}
    \|\alpha(t)\|_{\dot{W}^{-\f{3(1-\kappa)}{8},
    4}}\lesssim (1+t)^{-\f{7}{24}+\cO(\kappa)}\cN(T), \quad  \|\alpha(t)\|_{\dot{W}^{\f{5+3\kappa}{8},
    4}}\lesssim (1+t)^{-\f{5}{8}+\cO(\kappa)}\cN(T). 
\end{align}
Moreover, since $\pt \alpha_r=iP(D)\alpha_r+\mathfrak{N}_2+\mathfrak{N}_3+\mathfrak{N}_4,$ we have for any $Z\in \{\na_y, \pt+c_0\p_x\},$
\beqs 
\big\|\f{Z\alpha_r(t)}{|\na|}\|_{\dot{W}^{-1,\f{4}{1+\kappa}}}\lesssim \|\alpha_r\|_{\dot{W}^{-1,\f{4}{1+\kappa}}}+\sum_{j=2}^4 \big\|\f{\mathfrak{N}_j}{|\na|}\big\|_{L^{\f{12}{7+3\kappa}}}.
\eeqs
Since $\mathfrak{N}_2-\mathfrak{N}_4$ can be written as the   gradient of some function, it is easy to obtain that 
\begin{align*}
    \big\|\f{\mathfrak{N}_j}{|\na|}\big\|_{L^{\f{12}{7+3\kappa}}}\lesssim (1+t)^{-\f23+\cO(\kappa)} \cN(T),
\end{align*}
which, combined with the estimates \eqref{es-alphar}, \eqref{es-Zalphas}  and \eqref{someinterp}, leads to 
\begin{align*}
 \| B_{LL}^1(Z\alpha(t), \alpha(t))\|_{W^{\ell_{\kappa},8_{\kappa}}}  \lesssim (1+t)^{-\f{25}{24}+\cO(\kappa)} \cN(T)^2.
\end{align*}

Next, for the term $B_{H,2}(Z\alpha(t), \alpha(t)),$ 
we use the bilinear estimate  \eqref{BL-High2} to obtain that 
\begin{align*}
\| B_{H,2}(Z\alpha(t), \alpha(t))\|_{W^{\ell_{\kappa},8_{\kappa}}} &\lesssim \|B_{H,2}(Z\alpha(t), \alpha(t))\|_{H^{\ell_{\kappa}+\f{3(3+\kappa)}{8}}}\\
&\lesssim \big\|\f{Z\alpha(t)}{|\na|}\big\|_{W^{2,\f{4}{1+\kappa}}}\|P_{\geq -10}{\alpha(t)}\|_{W^{A_{\kappa},4}}+\big\|P_{\geq -10}\f{Z\alpha(t)}{|\na|}\big\|_{W^{A_{\kappa}+1,\f83}}\|P_{\geq -10}\f{{\alpha(t)}}{|\na|}\|_{W^{2,\f{8}{1+\kappa}}} 
\end{align*}
where $A_{\kappa}=\ell_{\kappa}+\f{3(3+2\kappa)}{8}+\f52.$ Let us deal with the first quantity, the second one can be handled similarly (actually  it behaves better).
By \eqref{highreg-interp}, we have 
\begin{align*}
    \big\|P_{\geq -10}|\na|^{A_{\kappa}}{\alpha(t)}\big\|_{L^4}  
    \lesssim (1+t)^{-\f23+\cO(\kappa)}\cN(T),
\end{align*}
as long as $M\geq \ell_{\kappa}+\f{71}{8}+\cO(\kappa)=\f{41}{4}+\cO(\kappa).$ 
Moreover, it follows from \eqref{zalpha/na} and  \eqref{alphaLP-low} that
\beqs 
\big\|\f{Z\alpha(t)}{|\na|}\big\|_{W^{2,\f{4}{1+\kappa}}}\lesssim \|\alpha(t)\|_{W^{2,\f{4}{1+\kappa}}}+(1+t)^{-\f{2}{3}+\cO(\kappa)}\cN(T) 
\lesssim (1+t)^{-\f{5}{12}+\cO(\kappa)}\cN(T).
\eeqs
We thus obtain that 
\begin{align*}
    \| B_{H,2}(Z\alpha(t), \alpha(t))\|_{W^{\ell_{\kappa},8_{\kappa}}} \lesssim (1+t)^{-\f{13}{12}+\cO(\kappa)}\cN(T)^2.
\end{align*}

Finally, let us estimate $B_{HL}(Z\alpha(t), \alpha(t)).$
Applying the bilinear estimate \eqref{BL-HL} and the Sobolev embedding, we find
\begin{align*}
    \| B_{HL}(Z\alpha(t), \alpha(t))\|_{W^{\ell_{\kappa},8_{\kappa}}} \lesssim
    \|P_{\geq -5} Z\alpha(t)\|_{W^{\ell_{\kappa}+1^{+}, 8_{\kappa}}} \big\|P_{\leq -10}{\alpha(t)}\big\|_{L^{\f{3}{1+\kappa}}}
    \lesssim  
    \|P_{\geq -5} Z\alpha(t)\|_{W^{C_{\kappa},\f92}} \big\|P_{\leq -10}
    {\alpha(t)}\big\|_{L^{\f{3}{1+\kappa}}},
\end{align*}
where $C_{\kappa}=\ell_{\kappa}+\f{31}{24}+\kappa.$
Note that thanks to \eqref{highreg-interp}, we have
\begin{align*}
    \big\|P_{\geq -10}|\na|^{C_{\kappa}+1}\alpha(t)\big\|_{L^{9/2}}  \lesssim 
    (1+t)^{-\f{20}{27}+\cO(\kappa)}\cN(T), 
\end{align*}
as long as 
$M\geq \ell_{\kappa}+\f{279}{56} +\cO(\kappa)
=\f{103}{14}+\cO(\kappa).$
This, together with the estimate
 $$\|\alpha(t)\|_{L^{\f{3}{1+\kappa}}}\lesssim (1+t)^{-\f13+\cO(\kappa)}\cN(T),$$ yields 
\beqs 
 \|B_{HL}(Z\alpha(t), \alpha(t))\|_{W^{\ell_{\kappa},8_{\kappa}}} \lesssim (1+t)^{-\f{29}{27}+\cO(\kappa)}\cN(T)^2.
\eeqs
\subsubsection{Estimate of $\cQ_1(t)$} 
Let us recall that
$$\mathbb{\cQ}_1(t):= \int_0^t e^{i(t-s)P(D)}\big(B\big(Z\alpha(s), \mathfrak{N}_1(s)\big)+B\big(Z\mathfrak{N}_1(s), \alpha(s)\big)\big)\,\d s,$$
where 
$\mathfrak{N}_1=\mathfrak{F}_1^0+ i \f{\div}{P(D) }\mathfrak{F}_1',$ $\mathfrak{F}_1$ being defined in \eqref{def-fF}.

To get good  decay estimates for $\mathbb{\cQ}_1(t)$, it is essential to distinguish between high and low output frequencies. When focusing on high frequencies $\{|\varsigma|\geq \frac{1}{32}\}$, the bilinear estimate \eqref{BL-High-less2} maintains integrability, allowing us to directly close the estimate using \eqref{BL-High-less2}. However, when concentrating on low frequencies $\{|\varsigma|\leq \frac{1}{32}\}$, directly applying the bilinear estimate \eqref{BL-low-moreder} with $\theta=0$ inevitably leads to a loss of integrability, preventing us to get  the desired decay estimates. To address this issue, we crucially use  the dispersive estimate \eqref{disper-low} to gain additional derivatives, thereby improving the bilinear estimates.

Let us detail the estimate of 
$$\mathbb{\cQ}_{11}(t):= \int_0^t e^{i(t-s)P(D)} B\big(Z\alpha(s), \mathfrak{N}_1(s)\big)\, \d s,$$ the other term is similar (just changing $Z\alpha$ to $\alpha$ and $\mathfrak{N}_1$ to $Z\mathfrak{N}_1$).
On the one hand, we apply the bilinear estimate \eqref{BL-High-less2} to obtain  that
\beq\label{cQ2-highfreq}
\begin{aligned}
   & \big\|P_{\geq -5}  B\big(Z\alpha(s), \mathfrak{N}_1(s)\big) \big\|_{W^{m_{\kappa},8_{\kappa}'}}\\
  &  \lesssim \big\|P_{\geq -10}\f{{Z\alpha}}{|\na|}(s)\big\|_{W^{
  B_\kappa,\f{8}{3+\kappa}}}\big\|\f{\mathfrak{N}_1(s)}{|\na|}\big\|_{H^2}+ \big\|P_{\geq -10}\f{\mathfrak{N}_1(s)}{|\na|}\big\|_{W^{B_\kappa,\f{8}{3+\kappa}}}\big\|\f{Z\alpha}{|\na|}\big\|_{H^2}, \, 
\end{aligned}
\eeq
where $m_{\kappa}=\ell_{\kappa}+\f{3(3+\kappa)}{4},$ $B_{\kappa}=m_{\kappa}+\f52+\kappa.$
Thanks to \eqref{highreg-interp}, we find as long as $M\geq 8+\cO(\kappa)$
\begin{align*}
    \|P_{\geq-10}\alpha(s)\|_{W^{m_{\kappa}+\f52+\kappa,\f{8}{3+\kappa}}}\lesssim (1+t)^{-\f13+\cO(\kappa)} \cN(T)^2.
\end{align*}
Combined with \eqref{zalpha/na}, \eqref{cF1-minus1}, \eqref{fF1-Lp}, this yields
\begin{align*}
    \big\|P_{\geq -5}  B\big(Z\alpha(s), \mathfrak{N}_1(s)\big) \big\|_{W^{m_{\kappa},8_{\kappa}'}}\lesssim (1+s)^{-\f{13}{12}+\cO(\kappa)} \cN(T)^2\, .
\end{align*}
With the same estimates as in \eqref{cQ2-highfreq},
it can be proven  that 
\begin{align*}
    \big\|  B\big(P_{\geq -5}Z\alpha(s), \mathfrak{N}_1(s)\big),  B\big(Z\alpha(s), P_{\geq -5}\mathfrak{N}_1(s)\big)\big\|_{W^{m_{\kappa},8_{\kappa}'}}\lesssim (1+s)^{-\f{13}{12}+\cO(\kappa)} \cN(T)^2.
\end{align*}
We thus deduce from the dispersive estimate \eqref{dispersive-L8}
that
\begin{align*}
  \big\|Q_{11}-\int_0^t e^{i(t-s)P(D)}P_{\leq -5}B\big(P_{\leq -5}Z\alpha(s), P_{\leq -5}\mathfrak{N}_1(s)\big)\,\d s \big\|_{W^{\ell_{\kappa},8_{\kappa}}}
  \lesssim (1+t)^{-(1+\f{\kappa}{3})} \cN(T)^2.
\end{align*}
On the other hand, when localizing to  low frequencies, we  use  the Sobolev embedding $\dot{W}^{\f{1-\kappa}{8},\f{6}{1-
\kappa}}\hookrightarrow L^{8_{\kappa}}$
and  the dispersive estimate \eqref{disper-low} to find that 
\beq\label{cQ1-ll}
\begin{aligned}
    & \big\| \int_0^t e^{i(t-s)P(D)}P_{\leq -5}B\big(P_{\leq -5}Z\alpha(s), P_{\leq -5}\mathfrak{N}_1(s)\big)\,\d s \big\|_{W^{\ell_{\kappa},8_{\kappa}}}\\
    &\lesssim \int_0^t (1+t-s)^{-(1+\f{\kappa}{2})}  \big\| |\na|^{\f{11+\kappa}{24}} B\big(P_{\leq -5}Z\alpha(s), P_{\leq -5}\mathfrak{N}_1(s)\big) \big\|_{L^{\f{6}{5+\kappa}}} \d s .
\end{aligned}
\eeq
It follows from 
 the bilinear estimate \eqref{BL-low-moreder}  with $\theta=\f{11+\kappa}{24}, \ep=\f{\kappa}{144}$ 
 that
\begin{align}\label{appli-BL-LL}
    \big\| |\na|^{\f{11+\kappa}{24}} B\big(P_{\leq -5}Z\alpha(s), P_{\leq -5}\mathfrak{N}_1(s)\big) \big\|_{L^{\f{6}{5+\kappa}}}\lesssim \big\|P_{\leq -5}\f{Z\alpha(s)}{|\na|}\big\|_{L^{\f{144}{49+24\kappa}}} 
    \big\|P_{\leq -5}\f{\mathfrak{N}_1(s)}{|\na|} \big\|_{L^2},
\end{align}
which, by \eqref{zalpha/na}, \eqref{alphaLP-low}, and \eqref{cF1-minus1}, can be controlled by $(1+s)^{-\f{161}{144}+\cO(\kappa)}\big(\cM(0)+\cN(T)^2\big).$ 
Plugging this estimate into \eqref{cQ1-ll}, we find the desired estimate 
\beqs 
\big\| \int_0^t e^{i(t-s)P(D)}P_{\leq -5}B\big(P_{\leq -5}Z\alpha(s), P_{\leq -5}\mathfrak{N}_1(s)\big)\,\d s \big\|_{W^{\ell_{\kappa},8_{\kappa}}}\lesssim (1+t)^{-(1+\f{\kappa}{2})}\big(\cM(0)+\cN(T)^2\big).
\eeqs

\subsubsection{Estimate of $\cQ_2(t),\, \cQ_4(t)$}
For  $j=2$ or $4,$ we apply the bilinear estimates 
\eqref{BL-High-less2}, \eqref{BL-low-moreder} to have crudely
\begin{align*}
&\big\|B\big(Z\alpha(s), \mathfrak{N}_j(s)\big)+B\big(Z\mathfrak{N}_j(s), \alpha(s)\big)\,\big\|_{W^{m_{\kappa},8'_{\kappa}}}\\
&\lesssim \big\|\f{Z\alpha(s)}{|\na|} \big\|_{H^{B_{\kappa}}}\big\|\f{\mathfrak{N}_j(s)}{|\na|} \big\|_{W^{B_{\kappa},\f{24}{11+\kappa}}}+\big\|\f{
\alpha(s)}{|\na|} \big\|_{H^{B_{\kappa}}}\big\|\f{Z\mathfrak{N}_j(s)}{|\na|} \big\|_{W^{B_{\kappa},\f{24}{11+\kappa}}},
\end{align*}
where $B_{\kappa}=m_{\kappa}+\f52+\kappa.$ Applying 
\eqref{zalpha/na}, \eqref{es-fF2-Lp}, \eqref{es-ZfF2-Lp} for $j=2$ and \eqref{esfF4-neghalf} for $j=4,$ we find that 
\begin{align*}
    \big\|B\big(Z\alpha(s), \mathfrak{N}_j(s)\big)+B\big(Z\mathfrak{N}_j(s), \alpha(s)\big)\,\big\|_{W^{m_{\kappa},8'_{\kappa}}}\lesssim (1+s)^{-\f{25}{24}+\cO(\kappa)} (\cM(0)+\cN(T)^2).
\end{align*}
It then follows from the dispersive estimate \eqref{dispersive-L8} that $$\|\cQ_2(t),\, \cQ_4(t)\|_{W^{\ell_{\kappa},8_{\kappa}}}\lesssim (1+t)^{-(1+\f{\kappa}{3})}\big(\cM(0)+\cN(T)^2\big).$$

\subsubsection{Estimate of $\cQ_3(t)$}
Let us set
$$\cQ_{31}:= \int_0^t e^{i(t-s)P(D)} B\big(Z\alpha(s), \mathfrak{N}_3(s)\big)\, \d s,\, \cQ_{32}=\mathbb{\cQ}_{3}(t)-{\cQ}_{31}(t).$$
We can use the same arguments as in the  estimate of $\cQ_1(t).$ 
We first apply \eqref{alphaLP-low} and \eqref{highreg-interp} to get that
\begin{align}
& \big\|\f{\mathfrak{N}_3(s)}{|\na|}\big\|_{H^{2}}\lesssim \|\alpha^2\|_{H^2}\lesssim \|\alpha\|_{L^4}\|\alpha\|_{W^{2,4}}\lesssim (1+s)^{-\f56+\cO(\kappa)} \cN(T)^2, \label{fN3/na}\\
 & \big\|P_{\geq -10} \f{ \mathfrak{N}_3(s)}{|\na|}
 \big\|_{W^{B_{\kappa},\f{8}{3+\kappa}}}\lesssim \|\alpha\|_{L^{\f{24}{1+3\kappa}}} \|P_{\geq -20}\,\alpha\|_{W^{B_{\kappa},3}}\lesssim (1+s)^{-(\f58+\f49+\cO(\kappa))} \cN(T)^2 \notag
\end{align}
as long as $M\geq (\ell_{\kappa}+1)+\f{27}{4}+\cO(\kappa)=\f{73}{8}+\cO(\kappa).$ It thus follows from the same computation as in \eqref{cQ2-highfreq} that
\beq\label{es-cQ31}
\begin{aligned}
  &\big\|Q_{31}-\int_0^t e^{i(t-s)P(D)}P_{\leq -5}B\big(P_{\leq -5}Z\alpha(s), P_{\leq -5}\mathfrak{N}_3(s)\big)\,\d s \big\|_{W^{\ell_{\kappa},8_{\kappa}}}\\
 & \lesssim \int_0^t (1+t-s)^{-(1+\f{\kappa}{3})}(1+s)^{-\f{77}{72}+\cO(\kappa)}\cN(T)^2\, \d s \, \lesssim (1+t)^{-(1+\f{\kappa}{3})} \cN(T)^2.
\end{aligned}
\eeq
The other term $\cQ_{32}$ can be dealt with similarly, but since $\f{\alpha}{|\na|}$ has growth in $L^2$ at rate $(1+t)^{\f14},$ it is better to put it  in  $L^{\f{8}{3+\kappa}}$  when using  bilinear estimates. For instance, we can have by \eqref{BL-High-less2} that
\begin{align*}
   & \big\|P_{\geq -5}  B\big(\alpha(s), Z\mathfrak{N}_3(s)\big) \big\|_{W^{m_{\kappa},8_{\kappa}'}}\\
  & \lesssim \big\|P_{\geq -10}\,{{\alpha}}(s)\big\|_{W^{
  B_\kappa-1,\f{8}{3+\kappa}}}\big\|\f{Z\mathfrak{N}_3(s)}{|\na|}\big\|_{H^2}+ \big\|P_{\geq -10}\f{Z\mathfrak{N}_3(s)}{|\na|}\big\|_{H^{B_\kappa}}\big\|\f{\alpha}{|\na|}\big\|_{W^{2,\f{8}{3+\kappa}}}.
\end{align*}
The first term can be controlled by $(1+s)^{-(\f{1}{3}+\f56)+\cO(\kappa)}\cN(T)^2,$ while for the second one, 
we need to use 
the estimate
\begin{align*}
    \big\|P_{\geq -10}\f{Z\mathfrak{N}_3(s)}{|\na|}\big\|_{H^{B_\kappa}}\lesssim \|\alpha\|_{L^{\f{14}{3}}}\|(\na \alpha, \pt\alpha)\|_{W^{B_{\kappa},\f72}}\lesssim (1+s)^{-(\f{19}{42}+\f47)+\cO(\kappa)}\cN(T)^2
\end{align*}
as long as $M\geq \ell_{\kappa}+1+\f{35}{4}+\cO(\kappa)=\f{89}{8}+\cO(\kappa).$ We thus get a similar estimate as in \eqref{es-cQ31}.

We now sketch the low-frequency estimates, which can be obtained using arguments similar to those employed for $\cQ_1.$ 
Thanks to \eqref{zalpha/na}, \eqref{alphaLP-low} and \eqref{fN3/na}, we
obtain, in the same spirit as the estimate \eqref{appli-BL-LL}, that
\begin{align*}
    \big\| |\na|^{\f{11+\kappa}{24}} B\big(P_{\leq -5}Z\alpha(s), P_{\leq -5}\mathfrak{N}_3(s)\big) \big\|_{L^{\f{6}{5+\kappa}}}\lesssim (1+s)^{-(\f{47}{144}+\f56)+\cO(\kappa)} \cN(T)^2\,.
\end{align*}
Next, it follows from \eqref{alphaLP-low} and \eqref{highreg-interp} that
\begin{align*}
     \big\|P_{\leq -5}\f{Z\mathfrak{N}_3(s)}{|\na|} \big\|_{L^2}\lesssim \|\alpha\|_{L^{\f{10}{3}}} \|(\pt\alpha, \na\alpha)\|_{L^5}\lesssim (1+s)^{-(\f{11}{30}+\f45)+\cO(\kappa)}\cN(T)^2.
\end{align*}
This last estimate, together with the fact $\|\alpha\|_{\dot{W}^{-1,{\f{144}{49+24\kappa}}}}\lesssim (1+t)^{\f{1}{72}+\cO(\kappa)}\cN(T),$ enables  us to derive that 
\beqs 
\big\| |\na|^{\f{11+\kappa}{24}} B\big(P_{\leq -5}\alpha(s), P_{\leq -5}Z\mathfrak{N}_3(s)\big) \big\|_{L^{\f{6}{5+\kappa}}}\lesssim (1+s)^{-\f{17}{15}+\cO(\kappa)} \cN(T)^2\,.
\eeqs
It then follows from the dispersive estimate \eqref{disper-low} that
\begin{align*}
    &\big\| \int_0^t e^{i(t-s)P(D)}P_{\leq -5}\big(B\big(P_{\leq -5}Z\alpha(s), P_{\leq -5}\mathfrak{N}_3(s)\big)+B\big(P_{\leq -5}\alpha(s), P_{\leq -5}Z\mathfrak{N}_3(s)\big)\big)\,\d s\,\big\|_{L^{8_{\kappa}}}\\
    &\lesssim (1+t)^{-(1+\f{\kappa}{2})}\big(\cM(0)+\cN(T)^2\big).
\end{align*}
The estimate of $\cQ_3$ is now complete.

\section{Proof of Theorem \ref{thm-nonlinear}.}
Based on the a priori estimates obtained in Section 4-7, we are now ready to prove Theorem \ref{thm-nonlinear}.

By Theorem \ref{thm-local} and Lemma \ref{lem-IFT}, there exists $T_0>0$ such that $(\tilde{c},\gamma)$ exists in $C^1([0,T_0], Y\times Y).$ Given this, it follows from standard arguments—particularly the energy estimates used in the proof of Proposition \ref{prop-energy}—that the system \eqref{eq-vr-v} admits a unique solution $(\vr, v)$ on $C([0,T], H^M)$ with $0<T\leq T_0.$
Moreover, by continuity, the local existence of 
$U$ established in Theorem \ref{thm-local}, 
 and of $(\tilde{c},\gamma)$ from Lemma \ref{lem-IFT}  
can be extended as long as $\sup_{t\in[0,T]}\|\na (n,\nabla\psi)(t)\|_{L_{\mathrm{x}}^{\infty}}<+\infty.$ 
In addition, by Propositions \ref{prop-modulation}, \ref{prop-weightednorm}, \ref{prop-energy}, and \ref{prop-decayes}, along with Remark \ref{rmk-cgamma0}, there is a  non-decreasing continuous function $f: \mR\rightarrow \mR,$ and a constant $C>0$ such that
\begin{align*}
    \cN(T)\leq f\big(\cM(0)
    \big) +C\cN(T)^2,
\end{align*}
which implies $\cN(T)\leq 2 f \big(\cM(0)
\big)$ provided $\cM(0)
\leq \delta$ is sufficiently small. It then follows from Lemma \ref{lem-IFT}
and a standard continuity argument that the solution exists globally, ie. $T=+\infty.$ 
Finally, the time decay estimates \eqref{decayes-thm} follow from the identity \eqref{identiy-import-intro}, Proposition \ref{prop-w}, and the definition of $\cN(T)$ in \eqref{def-cNT}. 

\section{Transverse Linear asymptotic stability--Proof of Theorem \ref{thm-resol-L} }\label{sec-prooflinear}
In this section, we  prove the uniform resolvent estimate \eqref{uni-resol}, which yields the exponential decay of the semigroup \eqref{semigroup} stated in Theorem \ref{thm-resol-L}, and thus the transverse linear asymptotic stability.  In the proof, it is more convenient to
consider the transformed operator 
 \beq\label{def-La}
 {L}_{c,a} =\colon e^{ax} L_c \,e^{-a x}=\left( \begin{array}{cc}
   (\p_x-a)(d_c\cdot)  & -(\p_x-a) \big(\rho_c (\p_x-a)\big)-\rho_c\Delta_y^2\\ 
-\big[ \, h'(\rho_c) +\big(e^{\phi_c}-(\p_x-a)^2-\p_y^2\big)^{-1}\big]& d_c(\p_x-a)
\end{array}\right) \,\, 
\eeq
in the unweighted space  $X:= L^2(\mR^3)\times H_{*}^{1}(\mR^3)$ where 
\begin{align}
  H_{*}^{{1}}(\mR^3)= \bigg\{ f  \,\big|\, \|f\|_{ H_{\star}^{1}(\mR^3)}:= \bigg(\int \bigg(\big|{\nabla_{x,y}}  f\big|^2+a^2 |f|^2 \bigg)(x,y)  \, \d x\d y\bigg)^{\f{1}{2}} <+\infty\bigg\}\, .
\end{align}
Note that $\dot H_{*}^1(\mR^3)$ still depends on $a$ but we neglect this dependence in the notation  in order to distinguish it with the weighted space $\dot{H}_{a}^1(\mR^3).$ 
Since the space $X$ can be identified with  $e^{ax}X_a,$ the study of $L_c$ on the weighted space $X_a$ is equivalent to the study of $L_{c,a}$ on the unweighted space $X.$ Consequently, Theorem \ref{thm-resol-L} is equivalent to the following theorem: 
\begin{thm}\label{thm-resol-La}
Let $X:= L^2(\mR^3)\times H_{\star}^{1}(\mR^3), \, \mathbb{Q}_{c}^a(\eta_0)=e^{ax}\,  \mathbb{Q}_c(\eta_0) \, e^{-ax}.$ There exist $\ep_0>0, \beta_0=\beta_0(\hat{a})>0, \hat{\eta}_0=\hat{\eta}_0(\hat{a})>0, C_0>0$ such that for any $0<\ep\leq \ep_0, 0<\beta\leq \beta_0,\, $ and any $\lambda \in \Omega_{\beta,\,\ep},$ 
 the operator
 $(\lambda-L_{c,a})$ is invertible on the space $\mathbb{Q}_c^a(\ep^2\hat{\eta}_0)X$ and
 \beq\label{uni-resol-La}
\|(\lambda-L_a)^{-1}\|_{B(\mathbb{Q}_c^a(\ep^2\hat{\eta}_0)X)}\leq C_0\, .
 \eeq
\end{thm}
As explained in the introduction, in order to obtain the above resolvent estimates, we have to use different arguments in 
 various frequency regions.
  For $d=1,2,$ let
 $\mathbb{I}_K(\cdot):\mR^d\rightarrow \mR$ be the characteristic function of the interval or disk $B_K(0)$ and $\tilde{\mathbb{I}}_K=1-\mathbb{I}_K.$ 
For a large number $A$ such that $A\ep^2<1,$ we  define the spaces 
$$X^{UH}=\tilde{\mathbb{I}}_{2}(D_y) X ,
\quad 
X^I_{A}=\mathbb{I}_2\tilde{\mathbb{I}}
_{A\ep^2}(D_y)X, \quad  X_A^L={\mathbb{I}}
_{A\ep^2}(D_y)X$$
    which are subspaces of $X$  whose elements are  localized respectively in  the uniform high, intermediate  and low transverse frequency regions $R^{UH},$ $R^I$ and $R^L$ 
    defined in \eqref{defregions}.
Theorem \ref{thm-resol-La} will be the consequence of Propositions \ref{prop-UTH}, \ref{prop-ITF}, \ref{prop-LTH} for the spectral stability 
and uniform resolvent bounds in the spaces $X^{UH},\, X^I_A, \,\mathbb{Q}_a(\ep^2\hat{\eta}_0)X_A^L$ respectively.

\subsection{Spectral stability for the uniform high transverse frequency}
In this subsection, we show the following result concerning the spectral stability in $X^{UH}.$
\begin{prop}\label{prop-UTH}
Suppose that $0<\beta\leq \f12$
and $\ep$ is sufficiently small. 
 For any $\lambda\in \Omega_{\beta, \ep},$ the operator
    $\big(\lambda- 
 L_a\big)$ is invertible on the space $X^{UH}$ and its inverse has the bound: 
 \beq 
 \|\big(\lambda-
 L_{c,a}\big)^{-1}\|_{B(X^{UH})} \leq  C\ep^{-1}\,,
 \eeq
 where the constant $C>0$ is independent of $\ep.$
\end{prop}
\begin{proof}
    We set  ${\mu}_a(D)=\sqrt{-(\p_x-a)^2-\Delta_y^2}$ and we observe that  for any function $f\in X$
  with transverse frequency $|\zeta|\geq 2,$ we have that $\|f\|_{X}\approx \|(f_1, {\mu}_a(D) f_2)\|_{L^2(\mR^3)}.$ Consequently, to prove Proposition \ref{prop-UTH}, it suffices to show that  
$\Omega_{\beta, \ep}\subset \rho(\cL_a; \cX^{UH}),$
with the bound
\beq 
 \|\big(\lambda-
 \cL_a\big)^{-1}\|_{B(\cX^{UH})} \leq  C\ep^{-1}, \quad \forall \, \lambda \in 
 \Omega_{\beta, \ep},
 \eeq
where $\cX^{UH}=\tilde{\mathbb{I}}_2(D_y) (L^2(\mR^3)\times L^2(\mR^3))$ and $\cL_a= \diag(1, \mu_a(D) )\, L_{c,a} \,\diag\big(1, \mu_a^{-1}(D) \big).$

Note that we skip the dependence of the operator on $c$  in the notation for convenience.

For any $F\in \cX^{UH},$ any $\lambda\in\Omega_{\beta, \ep},$ consider the resolvent problem: 
 \beqs 
\big( \lambda-\cL_a\big) U=F,
    \eeqs
we shall prove the existence of $U\in \cX^{UH}$ and the resolvent estimate simultaneously. 

Define
\beq\label{defcLzeta} 
\begin{aligned}
&\cL_a(\eta)=\colon e^{-iy\cdot \zeta}\cL_a e^{iy \cdot \zeta}\mathrm{1}_{\{\eta= |\zeta|\}}\\
&=\left( \begin{array}{cc}
   (\p_x-a)(d_c\cdot)  &  \rho_c\, \mu_{a,\eta}(D_x) -\p_x\rho_c(\p_x -a)\mu_{a,\eta}(D_x)^{-1}
   \\[5pt]
 -\mu_{a,\eta}(D_x) 
 \big[ \, h'(\rho_c) + \big(e^{\phi_c}
 +\mu_{a,\eta}(D_x)^2\big)^{-1}\big]&\, d_c(\p_x-a)+\big[\mu_{a,\eta} (D_x), d_c\big](\p_x-a) \mu_{a,\eta} (D_x)^{-1} 
\end{array}\right)\, ,
\end{aligned}
\eeq  
where ${\mu}_{a,\eta}(D_{x})=\sqrt{-(\p_x-a)^2+\eta^2},$  and 
$\big(e^{\phi_c}
 +\mu_{a,\eta}^2(D_x)\big)^{-1}: L^2(\mR)\rightarrow H^2(\mR)$ denotes the solution map associated to the elliptic problem
 \beqs 
\big(e^{\phi_c}+\eta^2-(\p_x-a)^2\big) f=n,
 \eeqs
 for $ n\in  L^2(\mR)$. Note that the solvability of the above equation is ensured by the Lax-Milgram theorem provided that $\ep$ is small enough. 
  Moreover, $\big(e^{\phi_c}
 +\mu_{a,\eta}^2(D_x)\big)^{-1}$ can be identified with a pseudo-differential operator with principal  symbol $$
 \f{1}{e^{\phi_c}(x)+\eta^2+(\xi+ia)^2}\,, \qquad  (\,|\eta|\geq 2\,).$$
 
For the convenience, we extend $\cL_a(\eta)$ for $\eta\in \mathbb{R}$ with $ \mathcal{L}_{a}(- \eta)=  \mathcal{L}_{a}( \eta)$ 
By taking the Fourier transform in the transverse variable,  we have to solve  the problem 
\beq \label{resol-HTF}
\big( \lambda-{\cL}_a(\eta)\big) u=f, \quad \, \forall \,|\eta|\geq 2
\eeq
in the space $\cY= L^2(\mR)\times L^2(\mR),$ and  to prove  the resolvent bounds  $\|u\|_{\cY}\lesssim \ep^{-1}\|f\|_{\cY}$ uniformly for $\eta$ such that $|\eta|\geq 2.$ 

Consider  the function 
\beq\label{sigmaaeta}
\sigma_{a,\eta}(x, \xi)=\sqrt{ h'(\rho_c)+\big(e^{\phi_c}(x)+\eta^2+(\xi+ia)^2\big)^{-1}}\,, \qquad  (\,|\eta|\geq 2\,).
\eeq
which is smooth and non-vanishing for any $(x,\xi)\in \mR^2$ 
as long as $\ep$ is sufficiently small. Moreover, when $|\eta|\geq 2,$ both $\sigma_{a,\eta}(x, \xi)$ and $\sigma^{-1}_{a,\eta}(x, \xi)$
belong to the symbol class $S^0$ 
so that the associated pseudo-differential operator $\op (\sigma_{a,\eta}), \op(\sigma_{a,\eta}^{-1}) $
are bounded operators on $L^2(\mR).$ Moreover, by using  standard results on the composition of pseudo-differential operators and the fact that $\|\p_x(\rho_c, d_c)\|_{W_x^{k,\infty}(\mR)}\lesssim \ep^3$ for $k$   
large enough, it holds that: 
\beqs 
\op(\sigma_{a,\eta})\circ  \op(1/\sigma_{a,\eta})=\Id_{L^2(\mR)} + R_{1\eta}\, , \quad  \op(1/\sigma_{a,\eta})\circ \op(\sigma_{a,\eta}) =\Id_{L^2(\mR)}+ R_{2\eta} \,,
\eeqs
where  $R_{1\eta}, R_{2\eta}$ are two bounded operators in $L^2(\mR)$ with  operator norms bounded by  $C\ep^3$ (C being independent of $\eta$).

Let the operator valued matrices $P_1(\eta)$ and $P_2(\eta)$ be defined by:
  \begin{align*}
  P_1(\eta)=   \left(\begin{array}{cc}
   \op({1}/{\sigma_{a,\eta}})   & \op({1}/{\sigma_{a,\eta}})  \\  
  -i/\sqrt{\rho_c} & i/\sqrt{\rho_c} 
\end{array}  \right), \,\,\quad  P_2= \f12 \left(\begin{array}{cc} \op(\sigma_{a,\eta})  &  i \sqrt{\rho_c}    \\ \op(\sigma_{a,\eta})  & - i\sqrt{\rho_c} \end{array}  \right),
\end{align*}
we can  then diagonalize ${\cL}_a(\eta)$ as:
\beq\label{diagnolzation}
{\cL}_a(\eta)= P_1(\eta) \,\,\op \left(\begin{array}{cc}
   \lambda_{-}^1(x,\xi,\eta)  &  \\
     &   \lambda_{+}^1(x,\xi,\eta)
\end{array}  \right) P_2(\eta)+\cR_{\eta}\, ,
\eeq
where  \beq \label{deflpm}
\lambda_{\pm}^1=i \bigg(d_c(x)(\xi+ia)\pm \sqrt{\rho_c}(x)\,\mu_{a,\eta}\,\sigma_{a,\eta}(x,\xi)\bigg). 
\eeq 
Hereafter, we denote  by 
$\cR_{\eta}$ 
a generic bounded linear operator in
$\mathcal{Y}=L^2(\mR)\times L^2(\mR)$ with  norm is of order  $\ep^3$ that is to say  
\beqs
\|\cR_{\eta}\|_{B(\cY)}\leq C \ep^3,\quad  \, \forall \, |\eta|\geq 2\, ,
\eeqs
 which may change from line to line.
 
Applying $P_2(\eta)$ to 
the resolvent equation $(\lambda-{\cL}_a(\eta))u={f}$ and noticing that $P_2(\eta) P_1(\eta)=\Id_{\cY}+\cR_{\eta},$ we find:  \beqs\label{resol-HTF-1} 
 \op \left(\begin{array}{cc}
  \lambda- \lambda_{-}^1(x,\xi,\eta)   &  \\
     &  \lambda- \lambda_{+}^1(x,\xi,\eta) 
\end{array}  \right) 
(P_2(\eta)u)=P_2(\eta)\tilde{ f}+\cR_{\eta} u\,.
\eeqs
Applying \eqref{muasax} for the region $R_{\eta}^{UH},$ and using the fact that  $d_c=c-\p_x\psi_c=c+\cO(\ep^2),$ we have, 
for any $a=\hat{a}\ep\in (0,1/2),$ 
any $|\eta|\geq 2,$ any $\ep$ small enough, that  there exists a 
constant 
$\kappa>0$
which is independent of $\eta, \ep,$ such that
$$\Re \lambda_{\pm}^1 =-a d_c(x) \pm \sqrt{\rho_c}(x)\, \Im (\mu_{a,\eta}\,\sigma_{a,\eta})(x,\xi)  \leq -\kappa a\,, \qquad \forall \,(x,\xi)\in\mR^2.$$ Therefore, for any $\lambda\in \Omega_{\beta,\ep},$ it holds, upon choosing $\ep$ small enough, that
\beqs 
\Re (\lambda-\lambda_{\pm})(x,\xi) \geq \kappa a/2>0\, , \qquad \forall \,(x,\xi)\in\mR^2.
\eeqs
We thus get by using standard results from pseudodifferential calculus that the operators $Op(\lambda - \lambda_{\pm})$ are invertible.
More precisely, by  using for example  Lemma A.1 of  \cite{RS-WW} (see also Theorem 4.29 in \cite{Zworski-book}) which gives  that
for a complex-valued
$S^0$ symbol $a(x,\xi)$  such that
$\gamma:=\inf_{(x,\xi)\in\mR^2}|a(x,\xi)|>0$ and $\sup_{(x,\xi)\in\mR^2}|\p_x^k a (x,\xi)|$ is  small enough for a sufficiently large number $k\geq 2,$
of derivatives  then the
  pseudo-differential operator $\op(a)$ is invertible
on $L^2(\mR)$ with the norm of the inverse operator controlled by  $2\gamma^{-1},$ we have that $\op(\lambda-\lambda_{\pm})$ 
are invertible and 
their inverses have the bounds
\begin{align}
    \|\big(\op \big(\lambda-
    \lambda_{\pm}^1(\cdot, \cdot, \eta) 
    \big) \big)^{-1}\|_{B(L^2(\mR))}\leq 4 (\kappa a)^{-1} \lesssim   
    \ep^{-1}. \label{ev-HT-p} 
    \end{align}
The existence of $u$ solving the problem \eqref{resol-HTF}
 follows from \eqref{ev-HT-p} and the invertibility of 
$P_2(\eta).$ Moreover, upon choosing $\ep$ sufficiently small and using  the fact that  $P_1(\eta)P_2(\eta)=\Id+\cR_{\eta},$ we get  that $\|u\|_{\cY}\lesssim \ep^{-1}\|{f}\|_{\cY}\,.$ The  proof is thus finished.
\end{proof}

\subsection{Spectral stability for the intermediate transverse frequency}
In this subsection, we  establish the  resolvent estimates in the space $X_A^I.$ Since $\mu_a(\xi,\zeta)$ may vanish on the  nontrivial curves $\{\xi=0,\, |\zeta|= a),$ the diagonalization procedure like \eqref{diagnolzation} does not work anymore since $(\p_x-a)\mu_{a,\eta}^{-1}$ is not a bounded operator on $L^2(\mR).$ The case where $\mu_{a, \eta}^{-1}$ has singularities thus requires careful treatment. To establish resolvent bounds within this intermediate transverse frequency regime, we shall 
use  an energy-based approach, achieved through energy estimates  based on carefully designed  energy functionals adapted to different
frequency subregions.

\begin{prop}\label{prop-ITF}
Suppose that $0<\beta\leq \f12.$
There exists a large number $A>0$ 
such that for any $\ep$ sufficiently small so that 
$A^2\ep<1,$
 the following hold true: for any $\lambda\in \Omega_{\beta,\ep},$ the operator
$\big(\lambda- 
 L_{c,a}\big)$ is invertible on the space $X_{A}^I$ and its inverse has the bound: 
\begin{align}
 \|\big(\lambda- 
 L_{c,a}\big)^{-1}\|_{B(X_{A}^I)}  \leq C A^{-1}
 \ep^{-3}, \label{ITF}
\end{align}
where $C$ is a constant independent of $A, \ep.$
\end{prop}
\begin{proof}
For any $\lambda \in \Omega_{\beta, \ep},$ $F\in X_A^I,$ consider the resolvent equation 
 \beq\label{reolveq-inter}
(\lambda I -L_{c,a})\, U=F\, .  
 \eeq
Let us first assume the existence of $U \in X_A^I$ and prove the resolvent estimates 
\beq\label{resolvent-inter}
\|U\|_{X}\lesssim A^{-1}\ep^{-3}\,\|F\|_{X}. 
\eeq
Define the set on which $\mu_a(\xi,\zeta)$ 
has degeneracies (i.e. $\mu_a(\xi,\zeta)=0$ for some $(\xi,\zeta)$):
\beq \label{def-singset}
\cS_{K, \vartheta}=\bigg\{(\xi,\zeta_1,\zeta_2)\in \mR^3\big|\,|\xi
|\leq K\ep, \, \vartheta \sqrt{\xi^2+a^2}\leq 
|\zeta|\leq 2\bigg\}, \quad ( \,0<\vartheta<1 \text{ small, } K(K+1) \in [A/2, A]  ),
\eeq 
and $\pi_{s}(\cdot), \, \pi_r(\cdot): \mR^3\rightarrow \mR$ be  respectively the characteristic functions  of the sets 
\beq \label{def-reguset}
  \cS_{K, \vartheta}\, , \qquad \cS_{K, \vartheta}^c=\big\{(\xi, \zeta_1,\zeta_2)\in \mR^3\big|\,  
  A\ep^2\leq 
{|\zeta|}\leq 2\big\} \backslash \cS_{K, \vartheta} \eeq
We split the solution $U$ to the problem \eqref{reolveq-inter} into two parts $U=U_s+U_r,$ where
\beqs 
 U_s= \pi_{s}(D) U, \qquad U_r= \pi_{r}(D)U 
\eeqs
denote  respectively a  low frequency in $x$  ‘singular' part and a  ‘regular’ part. In order to prove \eqref{resolvent-inter}, it is enough to show the following estimates on $ U_s $ and $ U_r:$
\begin{align}
     \|U_s\|_{X}&\lesssim_{\vartheta}  \ep^{-1}\|F\|_{X}+\ep(K\ep+1)  \| U_r\|_{X}, \label{es-us}\\
       \|U_r\|_{X}&\lesssim K^{-2} \ep^{-3}\big(\|F\|_{X}+\ep^2 (K\ep+1)\| U_s \|_{X}\big). \label{es-ur}
\end{align}
Indeed, we derive from the above estimates  that
\beqs 
\|U\|_{X}\leq C K^{-2}\ep^{-3}  \|F\|_{X} + C(K^{-1}+\ep)^2\|U\|_{X}
\eeqs
where $C>0$ is independent of $K$ and $\ep.$  The estimate \eqref{resolvent-inter} then follows by choosing first $K$ large enough, and then $\ep$ sufficiently small.
We will prove \eqref{es-us} and \eqref{es-ur} below  in the following subsubsections. 

Once the above estimates are obtained, to finish the proof of the Proposition \ref{prop-ITF}, we just need to prove  the existence for the resolvent equation \eqref{reolveq-inter}, that is  any element of $
\Omega_{\beta,\ep}$  lies in the resolvent set of $L_a$ in the space $X_A^I.$  As in  \cite{RS-WW,Pego-Sun}, we can invoke
 an abstract result from Spectral Theory (Theorem III.6.7, \cite{Book-Kato}) in order to reduce the problem to the
 proof of  the existence of one element of the resolvent set in $\Omega_{\beta, \ep}$.
We shall thus prove that $\lambda=1$  lies in the resolvent set of $L_{c,a}.$
We split the operator $L_{c, a}$ into two parts: 
$L_{c,a}=L_{a}^0+L_{a}^1,$ where
\beq\label{def-La0}
 {L}_{a}^0
 =\left( \begin{array}{cc}
  c( \p_x-a)  &  -\Delta_a \\[5pt]
 -h'(1)-(\Id -\Delta_a)^{-1}  & c(\p_x-a)
\end{array}\right), 
\eeq
\beq\label{def-La-1}
 L_a^1
 =-\left( \begin{array}{cc}
   (\p_x-a)(\p_x\psi_c\cdot)  &   \div_a (n_c \nabla_a)  \\[5pt]
 h'(\rho_c)-h'(1)+ (e^{\phi_c}-\Delta_a)^{-1}- (\Id -\Delta_a)^{-1}  & \p_x\psi_c(\p_x-a)
 \end{array}\right)\, ,
\eeq
where $\nabla_a=(\p_x-a, \na_y)^t,\, \Delta_a=(\p_x-a)^2+\Delta_y^2=\div_a \na_a.$
It is then sufficient to prove on the  one hand that $\Id- L_{a}^0$ is invertible in $X_A^I$ with the norm of the inverse uniformly bounded  in $\ep$ and on the other hand that $(\Id- L_{a}^0)^{-1}L_a^1\in B(X_A^I)$ with the norm 
\beq\label{smallperLa1}
\|(\Id- L_{a}^0)^{-1}L_a^1\|_{ B(X_A^I)}\lesssim \ep^2.
\eeq
Since $L_a^0$ is an operator with constant coefficient, the invertibility of the operator $\Id-L_{a}^0$ is equivalent to the invertibility of the complex valued matrix 
$\Id_{2} - L_{a}^0(\xi,\zeta),$ $L_{a}^0(\xi,\zeta)$ being the symbol of $L_{a}^0(D).$

Let us set  
$$\mu_a(\xi,\zeta)=\sqrt{|\zeta|^2+(\xi+ia)^2}, \quad   \sigma_a(\xi,\zeta)=\sqrt{h'(1)+(1+|\zeta|^2+(\xi+ia)^2)^{-1}},$$ 
the eigenvalues of $L_{a}^0(\xi,\zeta)$ are given by  
$$\lambda_{\pm}^0(\xi,\zeta)=
ic(\xi+ia)\pm \sigma_a \sqrt{- 
\mu_a^2\,}(\xi,\zeta).$$
 By \eqref{Relambdapm} in Appendix A, 
it holds that $\Re \lambda_{\pm}^0\leq 0,$ for any $(\xi,\zeta)\in \mR^3,$ provided  $a=\hat{a}\ep$ and $\f{\hat{a}}{1-\hat{a}^2}\leq \sqrt{h'(1)+1}.$ Consequently, we find that
$\Re( 1- \lambda_{\pm}^0)$ never vanishes, which ensures the invertibility of 
$\Id_{2} - L_{a}^0(\xi,\eta).$ Moreover, straightforward computation show that
\beq\label{inverse1-La0}
\begin{aligned}
\big(\Id_{2} - L_{a}^0(\xi,\zeta)\big)^{-1}&= \f{1}{(1-\lambda^0_+)(1-\lambda^0_-)}\left( \begin{array}{cc}
  \f{(1-\lambda^0_+)+(1-\lambda^0_-)}{2}  & \f{(1-\lambda^0_-)-(1-\lambda^0_+)}{2i\,\sigma_a}  \mu_a   \\[5pt]
 \sigma_a^2&\f{(1-\lambda^0_+)+(1-\lambda^0_-)}{2} 
\end{array}\right) .
\end{aligned}
\eeq
In light of this explicit expression, we have, by remembering  $X_A^I=\mathbb{I}_2\tilde{\mathbb{I}}
_{A\ep^2}(D_y)(L^2(\mR^3)\times H_{\star}^{1}(\mR^3)),$  that 
\begin{align}\label{id-la0-inverse}
    \|\big(\Id_{2} - L_{a}^0(\xi,\zeta)\big)^{-1}\|_{B(X_A^I)}\lesssim \max \sup_{\xi\in \mR, |\eta|\leq 2} \bigg\{ \f{1}{|1-\lambda^0_+|}\,,  \f{1}{|1-\lambda^0_-|}\,, \f{|\sigma_a|^2}{|(1-\lambda^0_+)(1-\lambda^0_-)|}\bigg\}<+\infty.
\end{align}
Note that we have used that  $\|\mu_a f\|_{L^2(\mR^3)}\lesssim \|f\|_{H_{*}^{1}(\mR^3)},\,$  and that $|\sigma_a|\leq c$, 
$|\Im (1-\lambda^0_{\pm}) |\thickapprox (1+\xi^2)^{\f12}$ for $|\xi|\gg 1$ and $|\zeta|\leq 2.$

To show \eqref{smallperLa1}, it suffices to check  the following estimates
\begin{align*}
    \big\|\f{1}{1-\lambda_{\pm}}(\p_x-a)(\p_x\psi_c\cdot)+\f{\mu_a}{(1-\lambda_{\pm})\sigma_a}(L_a^1)_{21}\big\|_{B(L^2,L^2)}\lesssim \ep^2,  \\
    \big\|\f{\div_a}{1-\lambda_{\pm}}(n_c\na_a)+\f{1}{1-\lambda_{\pm}}\big(\p_x\psi_c(\p_x-a)\big)\big\|_{B(H_{*}^{1},L^2)}\lesssim \ep^2, \\
  \big  \|\big(\f{\sigma_a^2}{(1-\lambda_{+})(1-\lambda_{-})}(\p_x-a)+ \f{1}{1-\lambda_{\pm}}(L_a^1)_{21}\big) (\p_x \psi_c\cdot)\big\|_{B(L^2,H_{*}^{1})}\lesssim \ep^2, \\
  \big \|\big(\f{\sigma_a^2}{(1-\lambda_{+})(1-\lambda_{-})}\div_a(n_c\na_a)+ \f{1}{1-\lambda_{\pm}}(\p_x\psi_c(\p_x-a))\big\|_{B(H_{*}^{1},H_{*}^{1})} \lesssim \ep^2.
\end{align*} 
where 
$$-(L_a^1)_{21}= h'(\rho_c)-h'(1)+ (e^{\phi_c}-\Delta_a)^{-1}- (\Id -\Delta_a)^{-1}\in \cO_{B(L^2(\mR^3))}(\ep^2).$$ However, they are direct consequences of the facts $|(1-\lambda^0_{\pm}) |\thickapprox (1+\xi^2)^{\f12}$ for $|\xi|\gg 1, |\zeta|\leq 2$ and  
$\|(\p_x\psi_c, n_c)\|_{W_x^{1,\infty}}\lesssim \ep^2.$

 To summarize, we have shown that $1\in \rho(L_a; X_A^I)$ and thus the proof of Proposition \ref{prop-ITF} is complete
 upon proving \eqref{es-us} and \eqref{es-ur}.
\end{proof}
\subsubsection{Proof of \eqref{es-us} }
Applying  the Fourier multiplier $\pi_s(D)$ on both sides of \eqref{reolveq-inter}, we find that $U_s$ solves the equation
\beqs 
  \big(\lambda-{L}_a^0 \big)U_{s}=\pi_s(D)(F+{L}_a^1 U)\, ,
\eeqs
where $L_a^0, L_a^1$ are defined in \eqref{def-La0} and 
\eqref{def-La-1}. Let us set  $ F_{\lambda}=e^{\lambda t} \pi_s(D)(F+{L}_a^1 U)\in  C([0,\infty); X).$
Since $ V= e^{\lambda t} U_s$ solves the evolution problem
 \beq\label{evoeq}
  \big( \pt  -  {L}_a^0 \big)V= F_{\lambda}, 
    \eeq
we first prove some a priori estimates for the equation \eqref{evoeq}.
Note that we could perform the energy estimates directly on the eigenvalue problem, 
but the structure is somewhat more clear if we go through the time evolution problem. Let us set  $V=(V_1, V_2)^t$ and define the energy functional $$\cE_s (V(t)):= \f12 \int \big(|\sigma_a(D)V_1(t)|^2+|\na_a V_2(t)|^2\big)\, 
\, 
\d \mathrm{x}\, , $$ 
where 
\beq\label{defsigaD}
\sigma_a(D)=\sqrt{h'(1)+(\Id-\Delta_a)^{-1}}.
\eeq
We get, after performing the straightforward energy estimate for the equation
\eqref{evoeq},  
\beq\label{EI}
\begin{aligned}
\pt \,\cE_s(V
)& = -2a c\, \cE_s(V)+ 2 a \, \Re\int  \sigma_a(D)^2 V_1 \, \overline{(\p_x-a) V_2} \,
  \,\d \mathrm{x}   \\
  &\quad +2a \, \Re \int  \big(\sigma_a(D)-\sigma_{-a}(D)\big) V_1\, \overline{\div_a \na_a V_2}  \,\d \mathrm{x}\\
  &\quad + \Re \int \overline{\sigma_a(D) V_1}\, \sigma_a (D)F_{\lambda1}+\overline{\na_a V_2}\cdot \na_a F_{\lambda2} \,\d \mathrm{x}\, := I_1+I_2+I_3.
\end{aligned} 
\eeq
Since $\pi_s(\xi,\eta)$ is supported in   $\{|\zeta|\geq {\vartheta}|\xi+ia|\},$ we have
$$\|(\p_x-a) V_2\|_{L^2}\leq \f{1}{\sqrt{1+{\vartheta^2}}}\|\na_a V_2\|_{L^2}.$$ 
It thus follows from the Cauchy-Schwarz inequality and \eqref{Resia} that
\beqs 
I_1\geq -2a \cE_s(V) \bigg(c- \f{1}{\sqrt{1+{\vartheta^2}}}\sup_{(\xi,\zeta)\in \mR^3} |\sigma_a(\xi,\zeta)|\bigg)\geq -\vartheta^2 a \, \cE_s(V)\,.
\eeqs
for $\vartheta$ small enough and $a=\hat{a}\ep$ with $\f{\hat{a}}{1-\hat{a}^2}\leq \sqrt{h'(1)+1}.$

Next, it holds that for any  $|\xi|\leq K\ep,$  any $\zeta\in\mR^2,$ 
\begin{align*}
   | \sigma_a(\xi,\zeta)-\sigma_{-a}(\xi,\zeta)|= 2\,|\Im \sigma_a(\xi,\zeta)|\leq \f{a |\xi|}{h'(1)(1-a^2)}\lesssim K\ep^2,
\end{align*}
which leads to 
\beqs 
|I_2|\lesssim K\ep^2 (K\ep+1) \, \cE_s(V).
\eeqs
Assuming $K^4\ep\leq 1,$ we have, upon choosing $\ep$ small enough, 
\beq\label{es-I12}
I_1+I_2\geq  -\f12 {\vartheta^2} a\, \cE_s(V)\,.
\eeq
Finally, by using again the Cauchy-Schwarz inequality 
\beq \label{es-I3}
I_3\lesssim  \|(\sigma_a(D) F_{\lambda1},\na_a F_{\lambda2})\|_{L^2} \|(\sigma_a (D)V_1, \na_a V_2)\|_{L^2}\lesssim\|F_{\lambda}\|_{X}  \sqrt{\cE_s}(V). 
\eeq
Plugging the estimates \eqref{es-I12}, \eqref{es-I3} into \eqref{EI} and 
 applying the Gr\"onwall inequality, we find
\beqs
\sqrt{\cE_s}(V(t)) \leq e^{-\vartheta^2 a t/2} \sqrt{\cE_a}(V(0)) 
+C \int_0^t  e^{-\vartheta^2 a (t-s)/2}  \|F_{\lambda}(s)\|_{X} \,\d s.
\eeqs
As $e^{\lambda t} U_s$ solves \eqref{evoeq} with $ F_{\lambda}=e^{\lambda t} \pi_s(D)(F+{L}_a^1 U),$
we then derive that 
\beqs
\sqrt{\cE_s}(U_s) 
\leq e^{-\alpha_{\lambda}t}\sqrt{\cE_s}(U_s) +C \big(\|F\|_{X}+ \|\pi_s(D)L_a^1 U\|_{X}\big)  \int_0^t e^{-\alpha_{\lambda}(t-s)}\, \d s\, 
\eeqs
where $\alpha_{\lambda}=\Re \lambda+\vartheta^2 a/2. $ As $a=\hat{a}\ep$ and $\Re \lambda>-\beta \ep^3,$ one has that $\alpha_{\lambda}\geq \vartheta^2 a/8>0 $ as long as $\ep$ is  small enough.
Consequently, by taking the limit $t \rightarrow + \infty$, we find
\beqs 
\|U_s\|_{X} \approx \sqrt{\cE_s}(U_s) \lesssim \ep^{-1}\big(\|F\|_{X}+\|\pi_s(D) L_a^1 U\|_{X}\big).
\eeqs
By combining the above estimate  with the estimate \eqref{La1pis},
we get  \eqref{es-us}.

\subsubsection{Proof of \eqref{es-ur} }\label{subsection-Ur}
We first state a  commutator estimate that will be used in the  proof.
It is a simple generalization of the Proposition 5.1 of \cite{Pego-Sun} in one dimension:
\begin{lem}\label{lem-commutator}
    Let $\mathcal{A}, \mathcal{B}, \mathcal{W} $ be three Fourier multipliers on the space $L^2(\mR^2)$ with symbols $A, B, W$ and let $f\in C^{\infty}(\mR)$ and $\cF(f)$ be sufficiently localized, for any $s\in \mR$ such that $C_{s}$
     and $C_{f,s}$ below are finite,  it holds that 
    \beq \label{es-commutator}
\|\mathcal{A}[ \mathcal{B}, f] \mathcal{W} \|_{B(L^2(\mR^2))}\lesssim C_s \, C_{f,s}
    \eeq
   where    
$$C_s=\sup_{(\xi,\xi',\zeta)\in \mR^4}\f{A(\xi,\zeta) \big(B(\xi,\zeta)-B(\xi',\zeta)\big)W(\xi',\zeta)}{|\xi-\xi'|^s}, \qquad  C_{f,s}= \int_{\mR} |\xi|^s |\cF(f)|(\xi) \,\d \xi.$$
\end{lem}

To prove \eqref{es-ur}, we apply 
 the Fourier multiplier $\pi_r(D)$ on both sides of \eqref{reolveq-inter}, we find that $U_r$ solves the equation
\beqs 
  \big(\lambda-\pi_r(D) L_a \pi_r(D) \big)U_{r}=\pi_r(D)\big(F+
  {L}_a^1 U_s\big),
\eeqs
where $L_a^1$ is defined in \eqref{def-La-1}. Let us set  $W_{\lambda}=\pi_r(D)  e^{\lambda t}\big(F+{L}_a^1 U_s\big).$ As in the estimate for $U_s ,$ we consider the evolution problem: 
\beq\label{evo-ur}
  \big(\p_t- \pi_r (D) L_{c,a} \pi_r (D)\big)V= W_{\lambda}\,.
\eeq
with $V= e^{\lambda t} U_{r}$.
Since the damping mechanism is weaker in this region (proportional to $\ep^3$), 
we need to use  an energy functional  which involves the background waves $n_c, \phi_c$ and is equivalent to 
the norm $\|\cdot\|_{X}.$ An appropriate functional is 
\begin{align}
\cE_r (V(t))&= \f12 \int \bigg(\big|\sigma_a(D)V_1(t)\big|^2+\big(h'(\rho_c)-h'(1)\big)\big|V_1(t)\big|^2+\phi_c \big|I_a (D)V_1\big|^2+\rho_c \big|\mu_a(D)V_2(t)\big|^2\bigg)\, 
\, \d \mathrm{x}\notag\\
&:= \f12 \int Q(V(t))\,(\mathrm{x}) \, \d \mathrm{x} \label{cEt}
\end{align} 
where $\sigma_a(D)$ is defined in \eqref{defsigaD} and $ I_a(D):= (\Id -\Delta_a)^{-1} \in B(L^2(\mR^3), H^2(\mR^3)).$ Thanks to the facts $$\sqrt{h'(1)}\leq |\sigma_a(\xi,\zeta)|\leq \sqrt{h'(1)+\f{1}{1-a^2}}\,,  \quad \rho_c=1+n_c=1+\cO(\ep^2), \quad \phi_c=\cO(\ep^2),\quad \,\,\forall\, (x,\xi,\zeta)\in \mR^4 $$
and $$|\mu_a(\xi,\zeta)|\approx \sqrt{\xi^2+a^2+|\zeta|^2}, \qquad \forall\, (\xi,\zeta)\in S_{K,\vartheta}^c\, ,$$
we see that, upon choosing $\ep$ small enough, 
$\cE_r (V(t))\approx \|V(t)\|_{X}.$
By computing the time derivative of the above functional and  by using  the equation \eqref{evo-ur}, we find the energy identity
\begin{align}\label{def-ck18}
    \pt \,\cE_r (V)=\cK_1+\cdots \cK_8\, ,
\end{align}
where 
\begin{align*}
    \cK_1&=-a \int( c-\p_x\psi_c(x))\,Q(V(t))\,(\mathrm{x}) \, 
+\overline{\sigma_a(D) V_1}\,\big( \mu_a\sigma_a -\mu_{-a}\sigma_{-a}\big) (D)(\rho_c\,\mu_a(D)\, V_2) \, \,
    \d \mathrm{x},\\
    \cK_2&= \,  \Re \int \rho_c\, \overline{\mu_a(D)V_2}\,\, \pi_r(D)\mu_a(D)\big((e^{\phi_c}-\Delta_{a})^{-1}-I_a(D)-I_a(D)\phi_c I_a(D)\big) V_2\, \,  \d \mathrm{x},\\
     \cK_3&=\, \Re  \int  \phi_c\, \overline{I_a(D) V_1} \,
 (I_a\mu_a-I_{-a}\mu_{-a})(D)\pi_r(D) \big( \rho_c\, {\mu_a(D)V_2}\big)\, \d \mathrm{x},\\
\cK_4&=\f{1}{2} \int \p_x d_c \bigg( |\sigma_a(D)V_1|^2+(h'(\rho_c)-h'(1))|V_1|^2+\phi_c|I_a(D)V_1|^2\bigg)\\
&\qquad \qquad -d_c\bigg(\p_x h'(\rho_c)\,|V_1|^2+\p_x\phi_c \,|I_a(D)V_1|^2\bigg) - \p_x (\rho_c\, d_c) |\mu_a(D)V_2|^2 \, \d \mathrm{x}, \\
\cK_5&=\, \Re\int \rho_c \,\overline{\mu_a(D)V_2}\,\big[\mu_a(D), d_c\big](\p_x-a) V_2 \\
    & \qquad\qquad \qquad\qquad +\overline{\sigma_a(D) V_1} (\p_x-a)\big[\sigma_a(D), d_c\big] V_1 + \phi_c \, \overline{I_a(D) V_1} (\p_x-a)\big[I_a(D), d_c\big]V_1\, \, \d \mathrm{x} \, ,
    \end{align*}
    \begin{align*}
\cK_6&= \Re \int  \bigg[\overline{\sigma_a(D) V_1}\, \sigma_a(D) \cdot+ (h'(\rho_c)-h'(1))\overline{ V_1}\cdot+ \phi_c\, \overline{I_a(D) V_1} I_a(D)\cdot\bigg] \pi_r(D)\big[\rho_c,\mu_a(D)\big]\mu_a(D)V_2 \,\d \mathrm{x} , \\ 
\cK_7 &=-\Re \int  \mu_{-a}(D)\big[\pi_r(D),\rho_c\big]\pi_r(D)\, \overline{\mu_a(D)V_2} \, \cdot \p_x\psi_c (\p_x-a)V_2  \\
& \qquad\qquad \qquad\qquad  +\bigg((h'(\rho_c)-h'(1))\overline{ V_1}+ \phi_c\, \overline{I_a(D) V_1} I_a(D)\cdot \bigg) (\p_x-a)\big[\pi_r(D), \p_x\psi_c\big]\pi_r(D)\,V_1\,\,\d \mathrm{x} , \\ 
\cK_8&= \Re \int   \rho_c\,\overline{ 
\mu_a(D) V_2}\, W_{\lambda2} \,\\
&
+\bigg[\overline{\sigma_a(D) V_1}\, \sigma_a(D) \cdot+ (h'(\rho_c)-h'(1))\overline{ V_1}\cdot+ \phi_c\, \overline{I_a(D) V_1} I_a(D)\cdot\bigg] \pi_r(D)\big(W_{\lambda1}+(\p_x n_c)(\p_x-a)V_2\big)\, \d \mathrm{x}\,.
\end{align*}

We now control $K_1-K_8$ term by term. The point of the following is to show that one can get some damping effect from the term $\cK_1$ while the remaining terms can be considered as the remainders.

At first,  by  the Cauchy-Schwarz inequality along with Lemma \ref{lemmuasa} and  the fact that  $\|(\p_x\psi_c, n_c)\|_{L^{\infty}_x}\lesssim \ep^2,$ we can find a constant $C>0$ such that, for $A$ large enough and $\ep$ sufficiently small,
\beq\label{es-ck1}
\cK_1\leq -2\bigg(a \inf_{x\in\mR} (c-\p_x\psi_c(x))-\sup_{x\in \mR,\, (\xi,\zeta)\in S_{K,\delta}^c}
\sqrt{\rho_c(x)}\big|\Im (\mu_a\,\sigma_a)(\xi,\zeta)\big|\bigg)\cE_r(V(t))\leq -C A \ep^3 \cE_r(V(t)).
\eeq
For the second term, we use the formula $A^{-1}-B^{-1}=A^{-1}(B-A)B^{-1},$ the fact that $\phi_c(x)=\cO(\ep^2)$ and the Taylor expansion for $e^x\, (x\ll 1)$ to obtain that for $\ep$ small enough
\beqs
\|(e^{\phi_c}-\Delta_{a})^{-1}-I_a(D)-I_a(D)\phi_c I_a(D)\|_{B(L^2(\mR^3), H^{4}(\mR^3))}\lesssim \ep^4,
\eeqs
from which we derive that
\beq
|\cK_2|\lesssim \ep^4\, \cE_r(V(t))\, .
\eeq
Note that hereafter $'\lesssim'$ denotes $\leq C$ for some $C$ independent of $A$ and $\ep.$
Next, on the one hand, by the Cauchy-Schwarz inequality and Parseval's identity, it holds that 
\beqs 
|\cK_3|\lesssim  \sup_{(\xi,\zeta)\in \mR^3} |\Im (I_a \mu_a)(\xi,\zeta)| \, \|\phi_c\|_{L_x^{\infty}} \, \cE_r(V(t)) . 
\eeqs
On the other hand, it follows from straightforward algebraic calculations that 
\beqs
 \sup_{(\xi,\zeta)\in \mR^3} |\Im (I_a \mu_a)(\xi,\zeta)| \lesssim \f{a}{1-a^2+\xi^2+\eta^2}\lesssim \ep.
\eeqs
We thus find that 
\beq 
|\cK_3|\lesssim \ep^3\, \cE_r(V(t))\, .
\eeq
Moreover, by using the fact 
$\|\p_x(n_c,\phi_c, \p_x\psi_c)\|_{L_x^{\infty}}\lesssim \ep^3,$ we easily see that
\beq 
|\cK_4|\lesssim \ep^3\, \cE_r(V(t))\, .
\eeq
For the next two terms $\cK_5$ and $\cK_6,$ by applying the commutator estimate \eqref{es-commutator}  for $s=1$ and remembering that $\p_x\psi_c,  n_c$ are exponentially localized, we  have  that 
\beqs
\bigg\|\bigg(\big[\mu_a(D), d_c\big], (\p_x-a)\big[\sigma_a(D), d_c\big], (\p_x-a)\big[I_a(D), d_c\big]\bigg)\bigg\|_{B(L^2(\mR))}\lesssim \|\cF_{x\rightarrow\xi}(\p_x^2 \psi_c)\|_{L_{\xi}^1}\lesssim \ep^3,
\eeqs
\beqs 
\big\|\big[\rho_c, \mu_a(D)\big]\big\|_{B(L^2(\mR))}\lesssim \|\cF_{x\rightarrow\xi}(\p_x n_c)\|_{L_{\xi}^1}\lesssim \ep^3.
\eeqs
By using also  the Cauchy-Schwarz inequality we then find  the desired estimate 
\beq 
|\cK_5+\cK_6|\lesssim \ep^3\, \cE_r(V(t))\, .
\eeq
Similarly, by  using  the commutator estimate \eqref{es-commutator} for $s=0,$ 
we can control the commutators appearing in $\cK_7$ as:
\begin{align*}
     \big\|\mu_{-a}(D)\big[\pi_r(D),\rho_c\big]\pi_r(D)\big\|_{B(L^2(\mR))}\lesssim \sup_{|\xi|\leq K\ep, |\zeta|\leq 2} |\mu_a(\xi,\zeta)|\, \|\cF_{x\rightarrow\xi} (n_c)\|_{L_{\xi}^1}\lesssim K\ep^3, \\
    \big\|  (\p_x-a)\big[\pi_r(D), \p_x\psi_c\big]\pi_r(D) \big\|_{B(L^2(\mR))}\lesssim \sup_{|\xi|\leq K\ep, |\zeta|\leq 2} |\xi+ia|\,\|\cF_{x\rightarrow\xi} (\p_x\psi_c)\|_{L_{\xi}^1}\lesssim K\ep^3.
\end{align*}
Consequently, by using again that $\|(n_c,\phi_c, \p_x\psi_c)\|_{L_x^{\infty}}\lesssim \ep^2$ and $K\ep\leq 1,$ 
we derive that 
\beq 
|\cK_7|\lesssim \ep^4\, \cE_r(V(t))\, .
\eeq
Finally, another application of the  Cauchy-Schwarz inequality yields that 
\beq\label{es-ck8}
|\cK_8| \lesssim \sqrt{\cE_r(V(t))} \big(\|W_{\lambda}\|_{X}+\ep^3\,\|\na_a V_2(t)\|_{L^2}\big)\leq \sqrt{\cE_r(V(t))} \big(\|W_{\lambda}\|_{X}+\ep^3  \sqrt{\cE_r(V(t))} 
\eeq
Combining the estimates \eqref{es-ck1}-\eqref{es-ck8}, we find that, upon choosing $A$ large enough, 
\beq\label{EI-cEr}
\pt \,\cE_r(V(t))\leq -\f{C}{2} A\ep^3 \cE_r(V(t))+C_1 \sqrt{\cE_r(V(t))}\|W_{\lambda}\|_{X}\,.
\eeq
It then follows from the Gr\"onwall inequality that: 
\beqs
\sqrt{\cE_r}(V(t)) \leq e^{-CA\ep^3 t/2} \sqrt{\cE_a}(V(0)) 
+C_1 \int_0^t  e^{-CA\ep^3  (t-s)/2}  \|W_{\lambda}(s)\|_{X} \,\d s\,.
\eeqs
Plugging $V(t)=e^{\lambda t} U_r$ into the above estimate and since we had set  $W_{\lambda}=\pi_r(D)  e^{\lambda t}\big(F+{L}_a^1 U_s\big),$ we find the inequality
\beqs
\sqrt{\cE_r}(U_r) 
\leq e^{-\beta_{\lambda}t}\sqrt{\cE_r}(U_r) +C_1\big(\|F\|_{X}+  \|\pi_r(D){L}_a^1 (U_h+U_s)\|_{X}\big)  \int_0^t e^{-\beta_{\lambda}(t-s)}\, \d s\, ,
\eeqs
where $\beta_{\lambda}=\Re \lambda+CA\ep^3/2> CA\ep^3/4$ 
for any $ \lambda \in \Omega_{\beta,\ep}$ provided that $A$ is large enough. 
As a result, we deduce 
\beqs 
\sqrt{\cE_r}(U_r) \lesssim (A\ep^3)^{-1} \big(\|F\|_{X}+  \|\pi_r(D){L}_a^1 U_s\|_{X}\big)\,.
\eeqs
The estimate \eqref{es-ur} then follows by applying \eqref{La1pis}. 

\begin{rmk}
    We could have alternatively considered a functional of the form  \beqs 
\tilde{\cE}_r (V(t)):= \f12 \int \bigg(\big|\op(\tilde{\sigma}_a)V_1(t)\big|^2
+\rho_c \big|\mu_a(D)V_2(t)\big|^2\bigg)\, 
\, \d \mathrm{x}
\eeqs
with the symbol
$\tilde{\sigma}_a(x, \xi,\zeta)=\sqrt{h'(\rho_c)+(e^{\phi_c}+|\zeta|^2+(\xi+ia)^2)^{-1}}.$ 
However, with this definition, in order to perform energy estimates, we  would need to use full  pseudodifferential calculus
 to estimate the commutators
whereas in the above version we only needed to commute Fourier multipliers and products which is technically simpler.
 We thus worked instead on the functional ${\cE}_r(V(t))$. Note that 
the operator $$\sigma_0^2(D)+h'(\rho)-h'(1)+I_a\,\phi_c\,I_a=h'(\rho_c)+I_a+I_a\,\phi_c\,I_a$$ is a good approximation of $\op(\tilde{\sigma}_a^2)$ in the sense that their difference has an operator norm in $B(L^2(\mR^3))$
  of order $\cO(\ep^4).$ It thus holds that 
\beqs 
|\tilde{\cE}_r (V(t))-{\cE}_r (V(t))|\lesssim \ep^4 \,\|V\|_{X},
\eeqs
which is admissible since the the damping effects in the region $S_{K,\delta}^c$ is propotional to $\ep^3.$
 
\end{rmk}

\subsection{Spectral stability for  low transverse frequencies.} 
In this subsection, we  consider the resolvent estimates on the space  $X_A^L=\colon {\mathbb{I}}
_{A\ep^2}(D_y)X.$ 
\begin{prop}\label{prop-LTH}
 Assume that $0<\beta\leq \f12$ and $\ep$ and $\eta_0$ are small enough. Then $\Omega_{\beta,\ep}\subset \rho( L_a; \mathbb{Q}_c^a(\ep^2\hat{\eta}_0)X_A^L).$  Moreover,
for any $\lambda \in \Omega_{\beta, \ep},\,$ any $ U, F \in \mathbb{Q}_c^a(\ep^2\eta_0)\bI_{A\ep^2}(D_y)X,$ such that $(\lambda I-L_{c,a}) U=F,$ it holds that 
    \beq\label{resol-lowfreq}
\|U\|_{X}\lesssim \ep^{-3}\|F\|_{X}. 
    \eeq
\end{prop}
To prove the above proposition, we need 
the following resolvent estimate in the high longitudinal frequency region $\{|\xi|\geq K\ep\}:$
\begin{lem}\label{lem-HHF}
Let $U \in \tilde{\bI}_{K\ep}(D_x)X_A^L,\, F\in X$ which satisfy
$\big(\lambda I-
L_{c,a}\big) U=F,$
 then it holds that, for $K>1$ sufficiently large and $\ep$ small such that $K^4\ep\leq 1,$
\begin{align} \label{es-HHF}
\|U\|_{X}\lesssim K^{-2}\ep^{-3}\,\|F\|_{X}.
\end{align}  
\end{lem}
\begin{proof}
    The proof is almost identical to that of the estimate \eqref{es-ur} for $U_r$ which is
    detailed in subsection \ref{subsection-Ur}. It suffices to take the Fourier projection $ \pi_{hl}(D)=\tilde{\bI}_{K\ep}(D_x){\bI}_{A\ep^2}(D_y)$ on the equation $\big(\lambda I-
L_a\big) U=F$ to obtain that
\beqs 
 \big(\lambda I- \pi_{hl}(D)
L_a  \pi_{hl}(D)\big) U=\pi_{hl}(D) F\,.
\eeqs
We can then repeat the computations in subsection \ref{subsection-Ur} to find, 
\beqs 
\pt \big(\cE_r(e^{\lambda t} U)\big)=\tilde{\cK}_1+\cdots+\tilde{\cK}_8
\eeqs
where $\tilde{\cK}_1-\tilde{\cK}_8$ are defined by replacing $\pi_r(D)$ via $\pi_{hl}(D),$ $V$ via $e^{\lambda t} U,$ $W_{\lambda}$ by $e^{\lambda t }\pi_{hl}(D) F$ in the definition of ${\cK}_1-{\cK}_8$ (see \eqref{def-ck18}). 
The terms $\tilde{\cK}_2-\tilde{\cK}_8$ can be controlled  in 
the same way as we did for ${\cK}_2-{\cK}_8.$ The only difference is the term $\tilde{\cK}_1.$ Instead of using the properties of $|\Im \mu_a\sigma_a|$ both in the regions $R_{\xi}^H, R_{\zeta,2}^I$ stated in \eqref{muasa}, we use  only its properties  in the region $\{ (\xi, \zeta)\in \mR^3\,\big||\xi|\geq K\ep, \, |\zeta|\leq A\ep^2 \}\subset R_{\xi}^H.$ 

To conclude, we are able to get the similar energy inequality as \eqref{EI-cEr}, which leads to the desired resolvent bound.

\end{proof}
We are now in position to prove Proposition \ref{prop-LTH}. As in the proof of \ref{prop-ITF},  it suffices to prove 
the uniform resolvent bound \eqref{es-HHF}.
\begin{proof}[Proof of \eqref{es-HHF}.]
We shall perform different estimates  for  low longitudinal frequencies (${|\xi|\leq K\epsilon}$) and high   longitudinal frequencies (${|\xi|\geq K\epsilon}$). In the high-frequency regime, we apply Lemma \ref{lem-HHF}, while in the low-frequency regime, we use  the KP-II approximation. Due to the excessively large resolvent bound in the low longitudinal  frequency regime, which is of  order $\epsilon^{-3}$ (without  the factor $K^{-2}$ in front), we need to use a smooth cutoff function to localize in frequency.

Let $\chi, \chi_1, \chi_2:\mR \rightarrow [0,1]$ 
be three cut-off functions with the properties: 
\beq\label{def-smoothcutoff}
\begin{aligned}
 & \Supp \chi\subset [-2K\ep, 2K\ep] \, , \quad   \chi \equiv 1 \text{ on } [-K\ep, K\ep] \,, \\
&\Supp \chi_1 \subset [-4K\ep, 4K\ep] \, , \quad   \chi_1 \equiv 1 \text{ on } [-3K\ep, 3K\ep]\, ,  \\
 & \Supp \chi_2\subset [-8K\ep, 8K\ep]\, ,  \quad   \chi_2 \equiv 1 \text{ on } [-4K\ep, 4K\ep]\, ,
\end{aligned}
\eeq
where $K\geq 1$ is large and $\ep$ is small, such that $K^4\ep\leq 1.$

We  define then
$$ U^H=(1-{\chi})(D_x)\, U,\, \qquad U^L=\chi_1(D_x) \,U,$$
which solve the equations
\begin{align}\label{eqULUH}
\left\{ \begin{array}{l}
\big (\lambda I- 
L_{c,a}\big)\, U^H=(1-\chi)(D_x) F-[{\chi}(D_x), \, L_a^1] \,U, \\[5pt]
\big(\lambda- L_{c,a}\big)\,  
{\chi}_{2}(D_x)\, U^L= [{\chi}_1(D_x) , \, L_a^1] \,U+ {\chi}(D_x) F\, .
\end{array} \right.
\end{align}
We refer to \eqref{def-La-1} for the definition of $L_a^1.$ 
Applying Lemma \ref{lem-HHF} and the estimate \eqref{la1smooth},  we first have that:
\beqs
\begin{aligned}
    \| U^H \|_{X}&\lesssim K^{-2}\ep^{-3}\big(
    \big\|[{\chi}(D_x), \, L_a^1] \, U\big\|_{X} +\|F\|_{X}\big)\\
    & \lesssim K^{-2}(K\ep+1) \, \|U\|_{X}+K^{-2}\ep^{-3}\|F\|_{X}\,.
\end{aligned}
\eeqs
It follows from the fact 
 $\chi_{1} \chi= \chi$ 
 that $\|U\|_{X}\lesssim \|U^L\|_{X}+\|U^H\|_{X}.$ 
Therefore, upon choosing $K$ large enough, 
\beq\label{es-UH}
 \| U^H \|_{X}  \lesssim K^{-2}(K\ep+1) \, \|U^L\|_{X}+K^{-2}\ep^{-3}\|F\|_{X}\,.
\eeq
It now remains to bound  $U^L.$ Since it is localized both in  low longitudinal and transverse frequencies,
we can diagonalize smoothly  the constant matrix 
\begin{align}\label{def-inverseP}
{L}_{a}^0
 =\left( \begin{array}{cc}
  c( \p_x-a)  &  -\Delta_a \\[5pt]
 -\sigma_a^2(D) & c(\p_x-a)
\end{array}\right) =  P  \left( \begin{array}{cc}
   \lambda_{+}^0  & 0   \\[5pt]
     0   &  \lambda_{-}^0
\end{array}\right) P^{-1}, \quad  P^{-1}= \f12 \left(\begin{array}{cc}
   \sigma_a  & -\sqrt{-\mu_a^2}  \\[5pt]
  \sigma_a & \sqrt{-\mu_a^2}
\end{array}\right),
\end{align}
where the symbol
$$\lambda_{\pm}^0(\xi,\zeta)=ic(\xi+ia)\pm \sqrt{- 
\mu_a^2\, \sigma_a^2 \,}(\xi,\zeta)\,.$$ 
Note that we have
\beq\label{equivalence-ll}
\|\bI_{A\ep^2}(D_y) f\|_{X}\approx \|
\bI_{A\ep^2}(D_y)P^{-1} f \|_{\cX}\,, \quad \big(\cX=L^2(\mR^2)\times L^2(\mR^2)\big) ,
\eeq
and it thus suffices  to get an estimate for $\tilde{U}^L=\colon P^{-1}U^L $ in the space $\cX.$

Let us set  
\beqs
R(x, D)=P^{-1} L_a^1\, P=\big( R_{ij}\big)_{1\leq i, j\leq 2} \, , \qquad \tilde{F}=P^{-1}\big([{\chi}(D_x) , \, L_a^1] \,U+ {\chi}(D_x) F \big).
\eeqs
In particular, $R_{11}$ takes the form
\beq\label{defR}
\begin{aligned}
    R_{11}(x, D)&=-\f12 \bigg(\sigma_a(D) (\p_x-a) (\p_x\psi_c\, \cdot)\sigma_a^{-1}(D)+ \sqrt{-\mu_a^2\,}(D)\,\p_x\psi_c\, (\p_x-a)\, \sqrt{-\mu_a^2}^{-1}(D) \\
    & -\,\sigma_a(D) \div_a(n_c\na_a)  \sqrt{-\mu_a^2}^{-1}(D)- \sqrt{-\mu_a^2}(D) \,\big( h'(\rho_c)-h'(1)+I_{a,\phi_c}-I_a\big) \, \sigma_a^{-1}(D) \bigg)\, ,
\end{aligned}
\eeq
where we denote for short $I_{a,\phi_c}=(e^{\phi_c}-\Delta_a)^{-1}, \, I_a=(\Id -\Delta_a)^{-1}.$
Moreover, thanks to   \eqref{equivalence-ll} and  \eqref{La1-ll}, we have that 
\beq\label{estimate-RxD}
\|\,R\,\,{\bI}_{4K\ep}(D_x){\bI}_{A\ep^2}(D_y)\|_{B(\cX)}\lesssim (K\ep^3+K^4\ep^5)\lesssim K\ep^3.
\eeq
We take $P^{-1}$ on both sides of  $\eqref{eqULUH}_2$ and translate it into the following system:
\beq \label{eq-tUL}
\begin{aligned}
\left\{ \begin{array}{l}
     \big(\lambda- \lambda_{+}^0(D)+R_{11}(x, D) \big)\, 
{\chi}_{2}(D_x)\,\tilde{ U}_1^L= -R_{12}\, \tilde{ U}_2^L + \tilde{F}_1\,,   \\[5pt]
      \big(\lambda- \lambda_{-}^0(D) \big)\, 
\,\tilde{ U}_2^L=-\big( R_{21}\, \tilde{ U}_1^L+ R_{22}\, \tilde{ U}_2^L \big)+ \tilde{F}_2\,.
\end{array}\right.
\end{aligned}
\eeq
Since $\Re \lambda_{-}^0\leq -a,$ it holds that for any $\lambda\in \Omega_{\ep},$ $\Re (\lambda-\lambda_{-}^0)\geq -a-\beta\ep^3\geq a/2$ as long as  $\ep$ is small enough. Consequently, 
we get by energy estimates that
\beq\label{U2tL}
\|\tilde{ U}_2^L\|_{L^2}\lesssim \ep^{-1}\big( \|\tilde{F}_2\|_{L^2}+ K\ep^3 \|\tilde{U}^L\|_{\cX} \big).
\eeq
To estimate $U_1^L,$ we need to use the KP-II approximation. 
As a preparation, we
introduce the scaling operator 
$S_{\ep}: L^2(\mR^3)\rightarrow L^2(\mR^3)$ as $S_{\ep}f(x,y)=\ep^{-5/2}f(\ep^{-1}x, \ep^{-2}y).$ 
Define then $$\cP=S_{\ep}^{-1}\,\mathbb{P}^{\hat{a}}_{\kp}(\hat{\eta}_0)\,S_{\ep}\,,\qquad  \cQ=S_{\ep}^{-1} \,\mathbb{Q}_{\kp}^{\hat{a}}(\hat{\eta}_0)\,S_{\ep}\, .$$
  where $\mathbb{P}^{\hat{a}}_{\kp}$ is the projection (see \eqref{projection-KP} for precise definition) to the resonant space associated to the operator $L_{\kp}$ and  $\mathbb{Q}^{\hat{a}}_{\kp}=\Id -\mathbb{P}^{\hat{a}}_{\kp}.$  We first claim that it suffices to prove the following two properties: 
\begin{align}\label{prop-puL}
    \|\cP \, \tilde{U}_1^L\|_{L^2}\lesssim (\hat{\eta}_0^2+ \ep^2+K^{-3})
    \,\|\tilde{U}_1^L\|_{L^2}+\|\tilde{U}^L_2\|_{L^2}+\| U^H\|_{X}\, ,
\end{align}
    \begin{align}\label{prop-QuL}
    \|\cQ \, \tilde{U}_1^L\|_{L^2}\lesssim 
    K^{-3}\,\|\tilde{U}_1^L\|_{L^2} + K\|\tilde{U}^L_2\|_{L^2}+
    \ep^{-3}\,\|\tilde{F}_1\|_{L^2}\, .
\end{align}
Indeed, once these two estimates are proven, we can  conclude that for $K$ large enough, $\ep, \hat{\eta}_0$ small enough, 
\beqs 
\|\tilde{U}_1^L\|_{L^2}\lesssim K  \|\tilde{U}^L_2\|_{L^2}+\|U^H\|_{X}+ \ep^{-3}\,\|\tilde{F}_1\|_{L^2}.
\eeqs
By combining the above estimate  with the estimate \eqref{U2tL}, yields, for $K^4\ep\leq 1 $ and $\ep$ small enough, 
\begin{align*}
    \|U^L\|_{X}\lesssim \|\tilde{U}^L\|_{\cX}&\lesssim \|U^H\|_{X}+ \ep^{-3}\,\|\tilde{F}\|_{\cX}, \\
    & \lesssim \|U^H\|_{X}+ \ep^{-3}\,\big(\|F\|_{X}+ \|[{\chi}_1(D_x) , \, L_a^1] \,U\|_{X}\big).
\end{align*}
As is shown in \eqref{la1smooth-1}, 
\beqs 
\|[{\chi}_1(D_x) , \, L_a^1] \,\bI_2(D_y)\,U\|_{X}\lesssim \ep^3\,\big( \|U^H\|_{X}+ K^{-1}
\|U^L\|_{L^2}\big).
\eeqs
Therefore, choosing $K$ large enough, we find that
\beqs 
 \|U^L\|_{X}\lesssim \|U^H\|_{X}+ \ep^{-3}\,\|F\|_{X}.
\eeqs
Together with \eqref{es-UH}, this  allows us to conclude \eqref{resol-lowfreq} upon choosing $K$ larger and $\ep$ smaller if necessary. The proof is   thus finished  once \eqref{prop-puL}, \eqref{prop-QuL} are established. This is  done in the following two subsubsections.
\end{proof}
\subsubsection{Proof of  the estimate \eqref{prop-puL}}
Recall that $Y_{a,\eta}=L_a^2(\mR)\times \dot{H}^{1}_{a,\eta}(\mR),$ where $\dot{H}^{1}_{a,\eta}(\mR)$ is defined in \eqref{Hhalfeta}. 
The projection operator associated with  $L_{c}$  in the weighted space $X_a$ is defined as
\beqs 
\mathbb{P}_c(\eta_0) f=\sum_{k=1}^2\int_{|\zeta|\leq \eta_0} 
\big\langle\mathcal{F}_yf(\cdot,\zeta),\, g_k^{*}(\cdot,\eta,c)\big\rangle_{{Y_{a,\eta}\times Y_{a,\eta}^{*}}}\,
g_k(\cdot,\eta, c)\,e^{iy\cdot\zeta} \,\mathbb{I}_{\{\eta=|\zeta|\}}\,\d \zeta\, ,
\eeqs
 where $g_k(x,\eta,c), g_k^{*}(x, \eta,c) $ are the basis and dual basis of $L_{c,a}(\eta)$ defined in \eqref{defbasis-dual}. 
Let $$Y_{\eta}= e^{a x} Y_{a, \eta}, \qquad 
g_{k,a}(x, \eta,c)=e^{a x} g_k(x, \eta,c),\, \qquad g^{*}_{k,a}(x, \eta,c)=e^{-a x} g_k^{*}(x, \eta,c), $$ 
the corresponding basis for  the transformed operator $L_{c,a} =e^{a x}\,  L_{c}\, e^{-a x} $
 in the unweighted space $X$  has the form
\beqs
\mathbb{P}_c^a(\eta_0)  f =e^{-a x} \mathbb{P}_a(\eta_0) e^{ax} f=\sum_{k=1}^2\int_{|\zeta|\leq \eta_0} 
\big\langle \mathcal{F}_y f(\cdot,\zeta),\, g_{k,a}^{*}(\cdot,\eta)\big\rangle_{Y_{\eta}\times Y_{\eta}^{*}}\,
g_{k,a}(\cdot,\eta)\,e^{iy\cdot\zeta}\,\mathbb{I}_{\{\eta=|\zeta|\}}\,\d \zeta\, .
\eeqs
Note that we skip the dependence of $g_{k,a}, g^{*}_{k,a}$ on $c,$ since it is not very important here.
At this stage,  it is useful to 
introduce the projection operator on the continuous resonant modes of  the  linearized operator $L_{\kp}$ defined in \eqref{linear-kpII} 
\begin{align}\label{projection-KP}
    \mathbb{P}_{\kp}^{\hat{a}}(\hat{\eta}_0)  f =\sum_{k=1}^2\int_{|\hat{\zeta}|\leq \hat{\eta}_0}
\big\langle \mathcal{F}_y f(\cdot, \hat{\zeta}),\, g_{0k,\hat{a}}^{*}(\cdot,\hat{\eta})\big\rangle_{L^2(\mR)\times L^2(\mR)}\,
g_{0k,\hat{a}}(\cdot,\hat{\eta})\,e^{iy\cdot\hat{\zeta}} \,\mathbb{I}_{\{\hat{\eta}=|\hat{\zeta}|\}}\,\d\hat{ \zeta}\, ,
\end{align}
where $g_{0k,\hat{a}}(x,\hat{\eta})=e^{\hat{a} x} g_{0k}(x,\hat{\eta}),\, g^{*}_{0k,\hat{a}}(x,\hat{\eta})=e^{-\hat{a} x} g^{*}_{0k}(x,\hat{\eta}) $ 
and $g_{0k,\hat{a}}(x,\hat{\eta}),  \,g^{*}_{0k,\hat{a}}(x,\hat{\eta})$ are generalized eigenmodes for $L_{\kp}(\hat{\eta})$ defined in \eqref{base-kp}.

We are now ready to prove \eqref{prop-puL}. Using the fact that $\mathbb{P}_a(\eta_0) U=0, $ we write
\beqs 
\cP \tilde{U}_1^L=\cP U_1^L-\big[ P^{-1}(D)\, \mathbb{P}_c^a(\eta_0) P(D)  \tilde{U}^L\big]_1-\big[ P^{-1}(D)\, \mathbb{P}_c^a(\eta_0) (1-\chi_1)(D)U \big]_1
\eeqs
where we use $\big[\,\,\big]_1$ to denote the first element of a vector. The last term in the above identity can be controlled directly, thanks to the equivalence of norms \eqref{equivalence-ll}, we have  $\|(1-\chi_1)(D)U \|_{X}\lesssim \|U^H\|_X.$ It thus remains to control the first two terms.
By the definitions and the conjugation,  we have on the  one hand
\begin{align*}
   & \qquad \big(S_{\ep} \, P^{-1}(D) \,\mathbb{P}_c^a(\eta_0) P(D) \tilde{U}^L\big)(x, y)\\
   & =\sum_{k=1}^2 \ep^{-{5}/{2}}\int_{|\zeta|\leq \eta_0} e^{i y\cdot\zeta/\ep^2}  \bigg\langle(\cF_y \tilde{U}^L)(\cdot, \zeta),\, P^{*}(D, \zeta)\, g_{k,a}^{*}(\cdot, |\zeta|,c)\bigg\rangle P^{-1}(\ep D_x, \zeta) \, g_{k,a}(\ep^{-1}x, |\zeta|)\, \d \eta\\
   &=  \sum_{k=1}^2 \int_{|\hat{\zeta}|\leq \hat{\eta}_0} e^{i {y}\cdot\hat{\zeta}}  \bigg\langle\cF_y( S_{\ep}\tilde{U}^L)(\cdot, \hat{\zeta}),\,P^{*}(\ep D_{{x}}, \ep^2\hat{\zeta})\, g_{k,a}^{*}(\ep^{-1}\cdot, \ep^2|\hat{\zeta}|,c) \bigg\rangle\, \ep^{-1} P^{-1}(\ep D_{{x}}, \ep^2\hat{\zeta}) \, g_{k,a}(\ep^{-1}{x}, \ep^2|\hat{\zeta}|)\, \d \hat{\zeta}\, ,
\end{align*}
and on the other hand, 
\begin{align*}
    &\qquad \big(\,\mathbb{P}_{\kp}^{\hat{a}}(\eta_0) \, S_{\ep}\tilde{U}_1^L\big)(x, y) \\
   &=   \sum_{k=1}^2 \int_{-\hat{\eta}_0}^{\hat{\eta}_0} e^{i {y}\cdot\hat{\zeta}}  \bigg\langle\cF_y( S_{\ep}\tilde{U}_1^L)(\cdot, \hat{\zeta}), \, g_{0k,\hat{a}}^{*}(\cdot, |\hat{\zeta}|) \bigg\rangle\,   \, g_{0k,\hat{a}}({x}, |\hat{\zeta}|)\, \d \hat{\zeta}\, .
\end{align*}
Note that the inner product in the above identities is taken in $L^2(\mR).$
Moreover, since $g_{0k, \hat{a}}$ and $g_{k,a}(\ep^{-1}\cdot)$ are smooth and exponentially decaying in $x,$ for any $\hat{\eta}\in [0, \hat{\eta}_0],$
it holds that 
\begin{align*}
& \| \big(1- \chi_2(\ep D_x)  \big) (g_{0k,\hat{a}}(\cdot, \hat{\eta}))\|_{L^2}\lesssim K^{-3}\|
 g_{0k,\hat{a}}(\cdot,  \hat{\eta})\|_{H^3}\lesssim K^{-3},  \\
&\| \big(1- \chi_2(\ep D_x)  \big) (g_{k,a}(\ep^{-1}\cdot, \ep^2 \hat{\eta}))\|_{L^2}\lesssim K^{-3}\|
g_{k,a}(\ep^{-1}\cdot, \ep^2 \hat{\eta}))\|_{H^3}\lesssim K^{-3}. 
\end{align*}
Consequently, since $S_{\ep}$ is a diffeomorphism on $L^2(\mR^3),$ it suffices to show that for any $|\hat{\zeta}|\leq  \hat{\eta}_0,$
\begin{align*}
    \bigg\|\chi_2(\ep D_x)\bigg(\ep^{-1}\big[
    P^{-1}(D_x, \ep^2 \hat{\zeta}) g_{k,a}\big]_1(\ep^{-1}\cdot, \ep^2 \hat{\zeta})-g_{0k, \hat{a}}(\cdot, |\hat{\zeta}|)\bigg)\bigg\|_{L^2(\mR)}\lesssim (\ep^2+\hat{\eta}_0^2)\,,  \\
     \bigg\|\chi_2(\ep D_x)\bigg(\big[P^{*}(D_x, \ep^2 \hat{\zeta}) g^{*}_{k,a}
     \big]_1(\ep^{-1}\cdot, \ep^2 \hat{\zeta})-g_{0k, \hat{a}}^{*}(\cdot, |\hat{\zeta}|)\bigg)\bigg\|_{L^2(\mR)}\lesssim (\ep^2+\hat{\eta}_0^2)\,. 
\end{align*}
In view of the definition of $P^{-1}$ in \eqref{def-inverseP} as well as the fact that on the support of $\chi_2(\xi)\bI_{A\ep^2}(\eta),$ we have
$$\sigma_a(\xi,\zeta)=\sqrt{h'(1)+1} +\cO(\ep^2),
\quad  -\sqrt{-\mu_a^2\,}(\xi,\zeta)=\ep(i\xi-a) +\cO(K^3\ep^3)\,,$$
it amounts to verify that for any $0\leq \hat{\eta}\leq  \hat{\eta}_0,$
\begin{align*}
    & \big\|(2\ep)^{-1}\big(\sqrt{h'(1)+1}\,g_{k,a}^1+
    (\p_x-a)\,g_{k,a}^2\big)(\ep^{-1}\cdot, \ep^2 \hat{\eta})-g_{0k, \hat{a}}(\cdot, \hat{\eta})\big\|_{L^2(\mR)}\lesssim (\ep^2+\hat{\eta}_0^2)\,,  \\
   &  \big\|\big(\sqrt{h'(1)+1}^{-1}\,g^{*,1}_{k,a} -
   (\p_x-a)^{-1} \,g^{*,2}_{k,a}\big)(\ep^{-1}\cdot, \ep^2 \hat{\eta})-g_{0k, \hat{a}}^{*}(\cdot, \hat{\eta})\big\|_{L^2(\mR)}\lesssim (\ep^2+\hat{\eta}_0^2)\,.
\end{align*}
Since they rely on direct but somewhat tedious algebraic calculations, we leave the details to Appendix \ref{approx-resonants-sec2}.
\subsubsection{Proof of the  estimate \eqref{prop-QuL}} 
Let us define  the operators
$$\lambda_{+,\,\ep}^0(D)= \lambda_{+}^0(\ep D_x, \ep^2 D_y), \quad R_{11}^{\ep}(x,D)=R_{11}\big(\ep^{-1}{x},\, \ep D_x,\, \ep^2 D_y\big), \quad \chi_{{\kp}}(D)=\chi_2(\ep D_x) \bI_{A}(D_y). $$
For the last one, note that we have used that  $ \bI_{A}(D_y) =  \bI_{A\ep^2}(\ep^2D_y)$  to simplify.
By applying  the scaling operator $S_{\ep}$ on both sides of 
$\eqref{eq-tUL}_1,$ 
we find that $u:= S_{\ep}\tilde{U}_1^L$ solves:
\beqs 
\big(\lambda- \lambda_{+,\ep}^0(D)- R_{11}^{\ep}(x,D)\big)\,\chi_{{\kp}}(D)\, u=f
\eeqs
where $ f=S_{\ep}(\tilde{F}_1-R_{12}\, \tilde{ U}_2^L).$

Localizing in  the low frequency regimes, one expects that the linear operator $\lambda- \lambda_{+,\ep}^0(D)- R_{11}^{\ep}(x,D)$ can be well approximated by the linearized operator of the KP-II equation, this is justified in the following lemma.
\begin{lem}\label{lem-KPapprox}
    Let $\lambda=\ep^3 \Lambda, $ and $L_{\kp}^a=e^{\hat{a}x}L_{\kp}e^{-\hat{a}x},$ 
    where $L_{\kp}$ is defined in \eqref{linear-kpII}.
    It holds that 
    \beq\label{id--KPapprox}
\big(\lambda- \lambda_{+,\,\ep}^0(D)- R_{11}^{\ep}(x,D)\big)\chi_{{\kp}}(D)={\ep^3} (\Lambda- L_{\kp}^{\hat{a}})\chi_{{\kp}}(D)+\cR
    \eeq
    where $\cR $ 
    enjoys the property
    \beq\label{es-cR}
    \|\cR\|_{ B(L^2(\mR^3))}\lesssim K^{-3}\,\ep^{3}.
    \eeq
\end{lem}
Once this  lemma is proven, 
the estimate \eqref{prop-QuL} is the consequence of 
the invertibility of $ (\Lambda- L_{\kp}^{\hat{a}})\chi_{{\kp}}(D)$ on the space $\mathbb{Q}_{\kp}^a(\hat{\eta}_0)L^2$ which follows from Proposition 3.2,
\cite{Mizumachi-KP-nonlinear}: 
\begin{prop}[Proposition 3.2, \cite{Mizumachi-KP-nonlinear}]
     There is $\hat{\beta}_0$ such that for any $\Lambda$ such that 
$\Re \Lambda > -\hat{\beta}_0,$
\beq\label{revert-Lkp}
\|(\Lambda- L_{\kp}^{\hat{a}})^{-1}\chi_{{\kp}}(D)\|_{B(\mathbb{Q}_{\kp}^{\hat{a}}(\hat{\eta}_0)L^2)}<+\infty\,.
\eeq
\end{prop}

Indeed, since $\mathbb{Q}_{\kp}^{\hat{a}} L_{\kp}^{\hat{a}}=L_{\kp}^{\hat{a}} \mathbb{Q}_{\kp}^{\hat{a}} $ and $\chi_{\kp}(D) u=u,$ we have 
that 
\beqs 
\ep^3 (\Lambda- L_{\kp}^{\hat{a}})\,\mathbb{Q}_{\kp}^{\hat{a}} u =\mathbb{Q}_{\kp}^{\hat{a}}\,(f-\cR).
\eeqs
It then follows from \eqref{revert-Lkp} and \eqref{es-cR} that 
\beqs 
\|\mathbb{Q}_{\kp}^a u \|_{L^2}\lesssim K^{-3}\, \|u\|_{L^2}+\ep^{-3}\|f\|_{L^2}.
\eeqs
Together with \eqref{estimate-RxD} and  the fact that $S_{\ep}$ is isometric on $L^2(\mR^3),$ this yields \eqref{prop-QuL}.

It now remains to prove  Lemma \ref{lem-KPapprox}. Let us set
\beq\label{defmuep}
\mu_{\ep}({\xi, \eta})=\ep\sqrt{(\xi+i\hat{a})^2+\ep^2\eta^2}\,, \qquad\sigma_{\ep}=\sqrt{h'(1)+(1+\mu_{\ep}^2)^{-1}}\, .
\eeq
For any $(\xi,\eta)\in\Supp\chi_{\kp}(\xi,\eta)=\big\{(\xi,\eta)\big|\, |\xi|\leq 8K, |\eta|\leq A\leq 2K^2\big\}, $
 it holds that
\beq\label{expand-musigmaep}
\begin{aligned}
    \sqrt{-\mu_{\ep}^2(\xi, \eta)}&=-i\ep(\xi+ i\hat{a}) \sqrt{1+\f{\ep^2\eta^2}{(\xi+ i\hat{a})^2}}=-i\ep(\xi+ i\hat{a}) \bigg( 1+\f{\ep^2 \eta^2}{2(\xi+i \hat{a})}+\cO(K^4\ep^4)\bigg), \\
    \sigma_{\ep}&=V-\f{\ep^2}{2V}(\xi+ i\hat{a})^2+\cO(K^4\ep^4)\, ,\qquad \big(V=\sqrt{h'(1)+1}\big)\, .
\end{aligned}
\eeq
Consequently, by remembering that $c=V+\ep^2$ we obtain that 
\begin{align*}
    \lambda_{+,\ep}^0(\xi, \eta)&= i\,c\,\ep (\xi+ i\hat{a})+
\sqrt{-\mu_{\ep}^2 \sigma_{\ep}^2\,}\\
&= 
\ep^3\bigg(i (\xi+i \hat{a})+\f{i(\xi+i\hat{a})^3}{2V}+\f{V\eta^2}{2i(\xi+ i\hat{a})}\bigg)+\cO(K^5\ep^5) \\
& ={\ep^3}\Lambda_{\kp,0}^{\hat{a}}(\xi, \eta)+\cO(K^5\ep^5)\,.
\end{align*}
We thus proved that
\beq\label{KP-nonlinear}
 \lambda_{-,\ep}^0(D)\chi_{\kp}(D)={\ep^3}\Lambda_{\kp,0}^{\hat{a}}(D)+\cR_1, \quad \text{ with } \, \|\cR_1\|_{B(L^2)}\lesssim (K \ep)^5 . 
\eeq
Note that $\Lambda_{\kp,0}^{\hat{a}}(D)$ is the  operator  obtained by linearizing the  KP-II equation  about $0$ in the moving frame. 
In view of the expression $L_{\kp}^{\hat{a}}= \Lambda_{\kp,0}^{\hat{a}}(D)-{C(V)} \,\p_x (\Psi_{\kdv}\cdot),$
we see that in order to prove
\eqref{id--KPapprox}, it suffices to show that 
\beq\label{nonlterm-approx}
 R_{11}^{{\ep}}\chi_{\kp}(D)=-\f{\ep^3}{2}\bigg(2\,{C(V)} \,\p_x (\Psi_{\kdv}\cdot)\bigg)+\cR_2\,, \quad \text{ with } \, \|\cR_2\|_{B(L^2)}\lesssim \ep^3\big(K^5\ep^2\big) .
\eeq
By introducing
\begin{align*}
   ( u_{1\ep}, n_{\ep}, \phi_{\ep} )(\cdot)=\ep^{-2}\,(\p_x\psi_c, n_c, \phi_c)(\ep^{-1}\cdot)\,, \quad d_{\ep}=c-\ep^2\, u_{1\ep}, \, \quad\rho_{\ep}=1+\ep^2 n_{\ep}\, ,
\end{align*}
and by using  the expression \eqref{defR}, it holds that 
\beq\label{expr-R11}
\begin{aligned}
     & R_{11}^{\ep}(x,D)=R_{11}\big(\ep^{-1}{x},\, \ep D_x,\, \ep^2 D_y\big)\\
 &=-\f{\ep^3}{2} \bigg(\sigma_{\ep}(D) (\p_x-\hat{a}) (u_{1\ep}\, \cdot)\sigma_{\ep}^{-1}(D)+ \sqrt{-\mu_{\ep}^2\,}(D)\,u_{1\ep}\, (\p_x-\hat{a})\, \sqrt{-\mu_{\ep}^2}^{-1}(D) \\
    & \qquad  \qquad \qquad\quad\, -\,\sigma_{\ep}(D)
    \big((\p_x-\hat{a})n_{\ep}(\p_x-\hat{a})+\ep^3 \,n_{\ep}\Delta_y \big)
    \, \ep\sqrt{-\mu_{\ep}^2}^{-1}(D)\\
    &  \qquad \qquad \qquad \qquad\,-\ep^{-3} \sqrt{-\mu_{\ep}^2}(D)\big( h'(\rho_{\ep})-h'(1)+I_{\hat{a},\ep^2\phi_{\ep}}-I_{\hat{a}}\big) \, \sigma_{\ep}^{-1}(D) \bigg)\, ,
\end{aligned}
\eeq
with $\mu_{\ep},\sigma_{\ep}$ defined in \eqref{defmuep} and 
\beqs 
I_{\hat{a},\ep^2\phi_{\ep}}=\big(e^{\ep^2\phi_{\ep}}-
\mu_{\ep}^2(D)\big)^{-1}, \quad I_{\hat{a}}=\big(\Id-\mu_{\ep}^2(D)\big)^{-1}\, .
\eeqs
Thanks to the expansion \eqref{expand-musigmaep}, it is direct to check that the sum of the first three terms inside of the bracket of \eqref{expr-R11} has the form 
\beqs\label{R11first3}
 (\p_x-\hat{a}) \big(2 u_{1\ep}+V n_{\ep}\big)+\cO_{B(L^2)}(K^5\ep^2) .
\eeqs
For the last term, we use the Taylor expansion
$h'(\rho_{\ep})-h'(1)=\ep^2 (\big(h''(1) n_{\ep} +\cO(\ep^2)\big)$
and the identity
\beqs 
I_{\hat{a},\ep^2\phi_{\ep}}- I_{\hat{a}}=I_{\hat{a},\ep^2\phi_{\ep}} (\Id-e^{\ep^2\phi_{\ep}})  I_{\hat{a}}=\ep^2\bigg(- \phi_{\ep}\,\Id + \cO_{B(L^2)}(K^2\ep^2)\bigg) 
\eeqs
to expand it as 
\beqs\label{R11last}
V^{-1}(\p_x-\hat{a}) \big(h''(1) n_{\ep}-\phi_{\ep}\big)+\cO_{B(L^2)}(K^5\ep^2)\,.
\eeqs
Moreover, thanks to \eqref{solitarywave-0} in  Theorem \ref{thm-existence}, it holds that for any $k\in \mathbb{N},$
\beqs 
(n_{\ep}, \phi_{\ep}, V^{-1} u_{1\ep})=\Psi_{\kdv}+\cO_{W_x^{k,\infty}}(\ep^2).
\eeqs
Consequently, by  combining the above three expansions, we obtain \eqref{nonlterm-approx} with
$$C(V)=\f{3V^2+h''(1)-1}{2V} =\f{3V^2+P''(1)-P'(1)-1}{2V}=V+\f{P''(1)}{2V}. $$

\appendix

\section{Study of the resonant modes}

In this appendix, we prove Theorem \ref{thm-resolmodes}. To be more specific, we will establish that for sufficiently small values of $\eta_0$, there exists a spectral curve $\lambda_c(\eta)\in C^{\infty}([-{\eta_0}, {\eta_0}])$
 and corresponding eigenmodes $U_c(\cdot, \eta)\in Y_{a,\eta}=L_a^2(\mR)\times \dot{H}^{1}_{a,\eta}(\mR),$ such that
\beq \label{resolnanteq-1}
L_c(\eta)\,U_c(\cdot, \eta)=\lambda_c(\eta)\,U_c(\cdot, \eta).
\eeq
Given that $L_c(\eta)=e^{-iy\cdot\zeta}\, L_c\, e^{iy\cdot\zeta} \,\mathbb{I}_{\{|\zeta|=|\eta|\}}$, this yields a set of continuous resonant modes $$\big(\lambda_c(|\zeta|), e^{i y\cdot\zeta}U_c(\cdot,|\zeta|)\big)_{|\zeta|\leq \eta_0}$$ for the linearized operator $L_{c}$.

 \subsection{KP-II approximation}
As mentioned in the introduction, it is known that the Korteweg-de Vries (KdV) equation and the Kadomtsev-Petviashvili II (KP-II) equation are effective approximate models for describing the ion dynamics in the long-wave region, 
in one and higher dimensions respectively. 
We thus expect that the  three dimensional Euler-Poisson system linearized  about  the one dimensional solitary water waves $(n_c, \psi_c, \phi_c)$ can be approximated by  the linearized  KP-II equation about   the line KdV soliton.  
This can be  shown by implementing the following scaling:
\beq \label{scaling}
\begin{aligned}
    \hat{t}=\ep^3 t, \quad \, \hat{x}=\ep x, \, \quad \hat{y}=\ep^2 y, \qquad (\hat{n}, \hat{\psi}, \hat{\phi})(\hat{t}, \hat{x})=(n, \ep \psi, \phi)(t, x), 
\end{aligned}
\eeq
Indeed, if $(n, \psi)^t$ solves the linearized equation \eqref{Linear-1}, then $(\hat{n}, \hat{\psi})^t$ solves:
\beq\label{scaled-eq}
\left\{ \begin{array}{l}
      \ep^3 \p_{\hat{t}}   \hat{n} - 
      \ep \p_{\hat{x}}\big(d_{\ep}\hat{n}\big)+\ep\p_{\hat{x}}\big(\rho_{\ep}\,\p_{\hat{x}}\hat{\psi}\big)\,+\ep^3\rho_{\ep}\,\Delta_{\hat{y}}\hat{\psi}=0, \\[3pt]
   \ep^3 \p_{\hat{t}} \hat{\psi} - \ep \hat{d}_{\ep} \p_{\hat{x}} \hat{\psi}+ \ep h'(\rho_{\ep})\hat{n}+\ep \hat{\phi} =0, \\[3pt]
\ep^2\big(\p_{\hat{x}}^2+\ep^2\p_{\hat{y}}^2\big)\hat{\phi}=e^{\ep^2\hat{\phi}_{\ep}}\hat{\phi}-\hat{n},
\end{array}\right.
\eeq
where 
\beq \label{def-dwep}
(\hat{n}_{\ep}, \hat{\psi}_{\ep}, \hat{\phi}_{\ep})( \hat{x})=\ep^{-2}(n_c, \ep \psi_c, \phi_c)(x) , \quad \hat{d}_{\ep}= V+\ep^2(1- \p_{\hat{x}}\hat{\psi}_{\ep}), \, \quad 
\hat{\rho}_{\ep}=1+\ep^2 \hat{n}_{\ep}, \quad V=\sqrt{P'(1)+1}.
\eeq
Let us now skip the subscript $\hat{}$ for the sake of notational convenience. We will perform  expansions with respect to $\ep$ in the above system. The expansion is usually performed by  assuming that a high enough Sobolev norm of the solution is bounded
on the appropriate time scale. Here we shall assume that the problem  is localized  in a bounded  frequency region with respect
to both variables, so that we  assume that all  the differential operators 
with positive order are bounded operators in $L_{\hat{a}}^2(\mR^3).$ We shall indeed perform this localization when we
will use this approximation argument.

Matching the $\ep^1$ order of  both sides of the equation \eqref{scaled-eq}, we obtain 
\beq\label{firsteq}
 \p_{x} \psi = V n=V\phi\,.
\eeq
By using the notation  $$I_{\ep, \phi_{\ep}}=\bigg(e^{\ep^2\hat{\phi}_{\ep}}-\ep^2\big(\p_{\hat{x}}^2+\ep^2\p_{\hat{y}}^2\big)\bigg)^{-1},\, \quad I_{\ep,0}=\bigg(1-\ep^2\big(\p_{\hat{x}}^2+\ep^2\p_{\hat{y}}^2\big)\bigg)^{-1},$$ it holds that 
\begin{align*}
    \phi= I_{\ep, \phi_{\ep}} n&= I_{\ep,0} \, n+ I_{\ep, \phi_{\ep}} (1-e^{\ep^2\hat{\phi}_{\ep}})\, I_{\ep, 0} \\
    & = n+\ep^2 \p_x^2 n-\ep^2\phi_{\ep} n +\cO_{L_{\hat{a}}^2(\mR^3)}(\ep^4).
\end{align*}
Plugging this formula into the second equation of \eqref{scaled-eq} and matching the $\ep^3$ order of the equation for $\p_x\psi+V n,$ we obtain by using \eqref{firsteq} the following equation
\beqs 
\p_t f+\p_x \bigg(-1+ \big( \p_x\psi_{\ep}+ \f{V}{2} n_{\ep}+\f{h''(1)n_{\ep}-\phi_{\ep}}{2V}\big)\big|_{\ep=0} +\f{1}{2V}\p_x^2\bigg)\, f +V\Delta_y \p_x^{-1} f=0.
\eeqs
Since  it follows from the expansion
\eqref{solitarywave-0} in  Theorem \ref{thm-existence}  that for any $k\in \mathbb{N},$
\beq\label{expan-waves}
n_{\ep}, \phi_{\ep}, V^{-1} u_{1\ep}=\Psi_{\kdv}+\cO_{W_x^{k,\infty}}(\ep^2).
\eeq
 we find the linear operator arising in   the  KP-II equation linearized about the line soliton $\Psi_{\kdv}:$
\beq \label{linear-kpII}
L_{\kp}= \p_x\bigg(1-C(V)\Psi_{\kdv}\cdot-\f{1}{2V} \p_x^2\bigg)
-V \Delta_y\p_x^{-1},
\eeq
where $C(V)=\f{3V^2+h''(1)-1}{2V}=V+\f{P''(1)}{2V}.$

\subsection{Existence of resonant modes--proof of Theorem \ref{thm-resolmodes} } \label{sec-resolnantmode}
We set  $U(\cdot, \eta)=(n, \psi)^t(\cdot, \eta),\,$  
then $\eqref{resolnanteq-1}$ can  be written as
\beqs\label{spectralpb}
\left\{ \begin{array}{l}
 \lambda  {n} - \p_{{x}} ({d}_{c} \, n) +\p_x(\rho_c\,\p_x\psi)+\eta^2  \rho_c \,\psi =0\, ,   \\[3pt]
   \lambda \psi- d_c \p_x \psi + h'(\rho_c) n +\phi=0\,, \\[3pt]
   (\p_x^2-\eta^2) \phi =e^{\phi_c} \phi -n\, . 
\end{array}\right.
\eeqs
To reduce the above system into a single equation, we use the last two equations to get that
\beq\label{eq-npsi}
\sigma_{\eta}(x, D_x) n=-(\lambda-d_c\p_x)\psi
\eeq
where $\sigma_{\eta}(x, D_x)=h'(\rho_c)+(e^{\phi_c}+\eta^2-\p_x^2)^{-1}.$ Note that as an operator acting on $L_a^2(\mR),$ this operator can be identified with  a pseudo-differential operator acting on the unweighted $L^2(\mR)$ space with the principal  symbol $h'(\rho_c)+(e^{\phi_c}+\eta^2+(\xi+ia)^2)^{-1},$ whose modulus  has positive lower bound, as long as $\ep$ is small enough. Remember that $\rho_c=1+n_c=1+\cO_{W_x^{k, \infty}}(\ep^2), \phi_c=\cO_{W_x^{k, \infty}}(\ep^2),\, \forall\, k\geq 0.$
Consequently, it follows from Theorem 4.29 \cite{Zworski-book} (or Lemma A.1, \cite{RS-WW}) that $\sigma_{\eta}$ is invertible on $L_a^2(\mR).$ Plugging \eqref{eq-npsi} into the first equation of \eqref{spectralpb}, we find
\beqs 
\big(\lambda-\p_x(d_c\cdot)\big)\sigma_{\eta}^{-1}(\lambda-d_c\p_x)\psi-\p_x (\rho_c\p_x\psi)+\eta^2\rho_c\psi=0.
\eeqs
Let  us now set 
\beqs 
\Lambda =\ep^{-3} \lambda, \qquad \, \hat{x}=\ep x, \qquad \hat{\eta}=\ep^{-2} \eta, \, \, \quad \sigma_{\ep, \hat \eta}^{-1}(\hat{x}, D_{\hat{x}})=\sigma_{\eta}^{-1}(x, D_x),\, \quad  \hat{\varphi}(\hat{x}, \hat{\eta})= (\p_x\psi)  (x,y)\, .
\eeqs
After the change of variable, we skip  the
subscript $\hat{}\,$ for clarity and we rewrite the above equation as 
\beq\label{eigenpb-real}
\cT(\ep,\eta, \Lambda) \vp:= \big[L_0(\ep,\eta)+\Lambda (\eta) L_1(\ep,\eta)
+ \Lambda^2(\eta)  L_2(\ep,\eta)\big]\vp=0\, ,
\eeq
where 
\beq\label{Ljepeta}
\begin{aligned}
   &\quad L_0(\ep,\eta)= \ep^{-2}\p_x \bigg(d_{\ep} \, \sigma_{\ep,\eta}^{-1}\,d_{\ep}-\rho_{\ep}\bigg)+\eta^2\rho_{\ep}\p_x^{-1}, \\
  &   L_1(\ep,\eta)=-\big(\sigma_{\ep,\eta}^{-1}\,d_{\ep}+\p_x(d_{\ep}\, \sigma_{\ep,\eta}^{-1}\p_x^{-1}) \big), \qquad L_2(\ep,\eta)=\ep^2 \sigma_{\ep,\eta}^{-1}\,\p_x^{-1}.
\end{aligned}
\eeq
We refer to \eqref{def-dwep} for the definition of $ d_{\ep}$ and $\rho_{\ep}.$ Thanks to \eqref{expan-waves} and the expansion for $\sigma_{\ep,\eta}^{-1},$ 
\beq\label{exp-sigmainverse}
\begin{aligned}
\sigma_{\ep,\eta}^{-1}&=\bigg( V^2+\ep^2 \big(h''(1)n_{\ep}-\phi_{\ep}+\p_x^2\big)\bigg)^{-1} \big(1+\cO_{B(H_{\hat{a}}^2(\mR);L_{\hat{a}}^2(\mR))}(\ep^4) \big) \\
&=V^{-2}-\ep^2 V^{-4} \big(h''(1)n_{0}-\phi_{0}+\p_x^2\big)+\cO_{B(L_{\hat{a}}^2(\mR))}(\ep^4)\, ,
\end{aligned}
\eeq
we obtain  that when $\ep=0,$ 
\beqs 
\cT(0, \eta, \Lambda)=-\f{2}{V} \big(\Lambda-L_{\kp}(\eta)\big),
\eeqs
where 
$$L_{\kp}(\eta)= \p_x\big(1-C(V)\Psi_{\kdv}\cdot\big)-\f{1}{2V} \p_x^3+\f{V}{2}\eta^2\p_x^{-1}.$$ 
In the following,  we use as in  \cite{Mizumachi-BL-linear}  the Lyapunov-Schmidt method to prove the existence of $(\Lambda, \vp)$
 and  an  expansion in terms of $\eta$ under the form
\beq \label{expan-modes}
\begin{aligned}
&\Lambda(\eta)=i\eta\, \Lambda_{1\ep}-\eta^2\Lambda_{2\ep}+\cO(\eta^3),\\
&\vp(\cdot,\eta)=\vp_{{0\ep}}(\cdot)+ i\eta\, \vp_{{1\ep}}(\cdot)+\eta^2\vp_{2\ep}(\cdot, \eta).
\end{aligned}
\eeq
Plugging  the above expansion into the equation \eqref{eigenpb-real} and sorting out in terms of powers of  $\eta,$ we find:
\begin{align}
   & L_0({\ep,0}) \, \vp_{0\ep}=0 \,, \notag \\
   &  L_0({\ep,0}) \, \vp_{1\ep}+ \Lambda_{1\ep} 
 L_1({\ep},0) \, \vp_{0\ep}=0\,, \label{first-expand} \\
  & L_0({\ep,0}) \, \vp_{2\ep}-\Lambda_{1\ep}  L_1({\ep},0) \, \vp_{1\ep}-\big(\Lambda_{1\ep}^2 L_2(\ep,0)+\Lambda_{2\ep} L_1(\ep,0)-L_0^1(\ep)\big)\vp_{0\ep}=0\,, \label{sec-expand}
\end{align}
where we have set  $L_0^1(\ep)=\f{L_0(\ep,\eta)-L_0(\ep, 0)}{\eta^{2}}|_{\eta=0}\,.$

Since $L_0(\ep,0)$ and $L_1(\ep,0)$ are related to the linearized operator arising in  the one dimensional Euler-Poisson system linearized about  the line soliton, it is shown in Appendix \ref{approx-resonants}
that 
\beq \label{firsttwomodes}
\vp_{0\ep}(\cdot)=\ep^{-1} \psi_c'' \bigg(\f{\cdot}{\ep}\bigg), \qquad \vp_{1\ep}(\cdot)=-\Lambda_{1\ep} \big(\p_c \psi_c'\big)\bigg(\f{\cdot}{\ep}\bigg),
\eeq
 Denote by  $L_{0}(\ep,0)^{*}\in B(L_{-\hat{a}}^2(\mR))$ the adjoint operator of   $L_{0}(\ep,0)\in B(L_{\hat{a}}^2(\mR))$ for  the duality  $\langle \cdot, \cdot\rangle_{L_{-\hat{a}}^2(\mR)\times L_{\hat{a}}^2(\mR)}$. By using the explicit expression in \eqref{Ljepeta},
it is direct to check that the operator $L_{0}(\ep, 0)$ has the structure 
\beq 
L_{0}(\ep,0)^{*}=-\p_x^{-1}L_{0}(\ep,0)\p_x\,,
\eeq
from which we obtain that $L_{0}(\ep,0)^{*}v_{\ep}=0$, with $$v_{\ep}(\cdot):=\p_x^{-1}\vp_{0\ep}(\cdot)= \psi_c'\big({\cdot}/{\ep}\big).$$ Consequently, it holds that
\beqs 
\langle L_{0}(\ep,0)\vp_{0\ep}, v_{\ep} \rangle=0\,,
\eeqs
 where throughout this proof,  we set $\langle \cdot,\cdot\rangle=\langle \cdot,\cdot\rangle_{(L_{-\hat{a}}^2(\mR), L_{\hat{a}}^2(\mR))}.$
Moreover, 
 by \eqref{first-expand} and the fact that  $L_{0}(\ep,0)^{*}v_{\ep}=0$, we get  that
\beqs 
\langle L_{1}(\ep,0)\vp_{0\ep}, v_{\ep} \rangle=0\, .
\eeqs
Therefore, it follows from \eqref{sec-expand}  that
\beq\label{eq-labda-1}
\Lambda_{1\ep}^2 \langle-L_1(\ep,0) v_{1\ep}  
+L_2(\ep,0) \psi_{0\ep}, v_{\ep} \rangle= \langle  L_0^1(\ep)\p_xv_{\ep}
, v_{\ep} \rangle ,
\eeq
where $v_{1\ep}(\cdot) =  \p_c\psi_c'(\cdot/\ep).$ 
We first note that when restricting to $\ep=\eta=0,$ 
\beq\label{L10-L01}
L_{1}(0,0)=-\f{2}{V}\,\Id,  \qquad L_0^1(0)\,\p_x=\Id.
\eeq
Moreover, thanks to \eqref{solitarywave-def}, \eqref{solitarywave-0} and the relation $\ep^2=c-V,$ it holds that:
\begin{align*}
&\qquad \qquad \qquad v_0(x)=V \Psi_{\kdv}(x)=\f{3V}{C(V)}\sech^2\big(\sqrt{{V}/{2}}\, \hat{x}\big) , \\
&v_{10}\,(x)=\p_c \psi_c'\big(\f{x}{\ep}\big)|_{\ep=0}=\f{\p(\ep^2)}{\p c}\p_{\ep^2}\psi_c'\big(\f{x}{\ep}\big)|_{\ep=0}=V\big(\Id+\f{x}{2}\p_x\big)\Psi_{\kdv}(x).
\end{align*}
Consequently, since $L_2(\ep,0)=\cO_{B(L_{\hat{a}}^2(\mR),H^1_{\hat{a}}(\mR))}(\ep^2)\,,$ one has
\begin{align*}
    \langle-L_1(\ep,0) v_{1\ep} +L_2(\ep,0) \psi_{0\ep}, v_{\ep} \rangle&=\f{2}{V}\langle v_{10}, v_0 \rangle+\cO(\ep^2)\\
   & =\f{3}{2V}\| v_{0}\|_{L^2}^2+\cO(\ep^2)>0\, ,
\end{align*}
provided that $\ep$ is small enough. 
It then follows from \eqref{eq-labda-1}, \eqref{L10-L01} that  
\beq \label{Lbda1ep}
\Lambda_{1\ep}^2=\f{\langle  L_0^1(\ep)\p_xv_{\ep}, v_{\ep} \rangle } {\langle-L_1(\ep, 0) \p_x\vp_{1\ep} +L_2(\ep,0) \psi_{0\ep}, v_{\ep} \rangle}
=\f{2V}{3}+\cO(\ep^2)>0\,.
\eeq
It now remains to show the existence of  the number $\Lambda_{2\ep}<0$ and the  function $\vp_{2\ep}(\cdot, \eta)\in L_{\hat{a}}^2(\mR),$ such that \eqref{expan-modes} holds and 
$$ \overline{\Lambda(\eta)}= {\Lambda(-\eta)},\qquad \overline{\vp_{2\ep}(\cdot, \eta)}=\vp_{2\ep}(\cdot, -\eta).\, $$
To use  the Lyapunov-Schmidt method,
it is more convenient to set  $\Lambda(\eta)=i\eta\,\tilde{\Lambda}(\eta)$ and to  look for $\tilde{\Lambda}(\eta),\, \kappa(\eta), \tilde{\vp}_{2\ep}(\cdot, \eta)$ such that 
\beq \label{realexp}
\vp(\cdot,\eta)=v_{\ep}'-\big(i\eta \tilde{\Lambda}(\eta)-\kappa(\eta)\eta^2 \big)v_{1\ep}+\eta^2\tilde{\vp}_{2\ep}(\cdot, \eta)
\eeq
and 
\beq\label{ortho-vp2}
\tilde{\vp}_{2\ep}(\cdot, \eta) \perp \mathrm{Span}\bigg\{v_{\ep},\, \int_{-\infty}^x \tilde{v}_{1\ep}(x') \,\d x'\bigg\}=\ker_g L_{0}^{*}(\ep, 0),
\eeq
where $\tilde{v}_{1\ep}\in \ker_g L_{0}(\ep, 0)$ satisfies $L_{0}(\ep, 0)\tilde{v}_{1\ep}=-\f{2}{V}v_{\ep}', \, \tilde{v}_{1\ep}|_{\ep=0}= v_{1\ep}|_{\ep=0}.$
It is  worth noting that the determination of $\tilde{\psi}_{1\ep}$ can be achieved through the Lyapunov-Schmidt approach. That is, one starts from the generalized kernel of $L_0(0, 0),$ does expansions in $\ep,$ and determines the remainders of order $\ep^2$ by the Implicit Function Theorem.
A very similar method is  also used  in the subsequent discussion, we thus omit the specific details  for brevity. 

Plugging the expansion \eqref{realexp} into the equation \eqref{eigenpb-real} and looking at the  $\cO(\eta^2)$ term, we find the equation satisfied by $\tilde{\vp}_{2\ep}:$ 
\beq \label{eqpsi2ep}
\cT_0(\ep,\eta, \tilde{\Lambda}) \,\tilde{\vp}_{2\ep}+  \cT_1(\ep, \eta,  \tilde{\Lambda}, \kappa ) + i\eta\, \cT_2(\ep, \eta,  \tilde{\Lambda}, \kappa ) =0
\eeq
where 
\begin{align*}
&\cT_0(\ep,\eta, \tilde{\Lambda}) =  L_0(\ep,\eta)+ i\eta\,\tilde{\Lambda} L_1(\ep,\eta)
-\eta^2 \tilde{\Lambda}^2  L_2(\ep,\eta)\, ,\\
&\cT_1(\ep, \eta,  \tilde{\Lambda}, \kappa ) =\tilde{\Lambda}^2 \big( L_{1}(\ep,\eta)v_{1\ep}-L_2(\ep,\eta)v_{\ep}'\big)+\kappa(\eta)L_0({\ep}, \eta)v_{1\ep}+ \f{L_{0}(\ep,\eta)-L_{0}(\ep, 0)}{\eta^2}v_{\ep}'\,, \\
&\cT_2(\ep, \eta,  \tilde{\Lambda}, \kappa ) =\tilde{\Lambda}\bigg(\kappa(\eta) L_1(\ep,\eta)-\f{L_{0}(\ep,\eta)-L_{0}(\ep, 0)}{\eta^2}+\tilde{\Lambda}^2L_2(\ep,\eta) \bigg)v_{1\ep}\, .
\end{align*}
Let $\mathbb{P}_{\ep}$ be the spectral projection onto $\ker_g L_{0}(\ep,0)$  in $L_{\hat{\alpha}}^2(\mR)$ and $ \mathbb{Q}_{\ep}=\Id-\mathbb{P}_{\ep},$ then the equation \eqref{eqpsi2ep} implies that
\beqs 
 \mathbb{Q}_{\ep} \cT_0(\ep,\eta, \tilde{\Lambda})  \mathbb{Q}_{\ep}\,\tilde{\vp}_{2\ep}+  \mathbb{Q}_{\ep}  \cT_1(\ep, \eta,  \tilde{\Lambda}, \kappa ) + i\eta\,  \mathbb{Q}_{\ep}  \cT_2(\ep, \eta,  \tilde{\Lambda}, \kappa ) =0.
\eeqs
On the one hand, it holds by  the definition of $\mathbb{Q}_{\ep}$ that 
\beqs 
\| \mathbb{Q}_{\ep} \cT_0(\ep, 0, \tilde{\Lambda}) ^{-1}\mathbb{Q}_{\ep} \|_{B(L_{\hat{\alpha}}^2(\mR), H_{\hat{\alpha}}^1(\mR))}<+\infty.
\eeqs
On the other hand, in view of the expressions of 
$L_0(\ep,\eta), L_1(\ep,\eta), L_2(\ep,\eta)$ in \eqref{Ljepeta}, one gets that
\begin{align*}
    \|\cT_0(\ep, \eta, \tilde{\Lambda})-\cT_0(\ep, 0, \tilde{\Lambda})\|_{B(L_{\hat{a}}^2)}&\lesssim \hat{a}^{-1}\eta^2(1+|\tilde{\Lambda}(\eta)|^2)+\eta |\tilde{\Lambda}(\eta)|\\
    & \lesssim \eta_0^2(1+M^2
    )+\eta_0 M,
\end{align*}
where  $M=\sup_{\eta\in[-\eta_0, \eta_0]}| \tilde{\Lambda}(\eta)|.$ We thus conclude that, upon choosing $\eta_0$ small enough, we have
\beqs 
\| \mathbb{Q}_{\ep} \cT_0(\ep, \eta, \tilde{\Lambda}) ^{-1}\mathbb{Q}_{\ep} \|_{B(L_{\hat{\alpha}}^2(\mR), H_{\hat{\alpha}}^1(\mR))}<+\infty.
\eeqs
 The existence of $\tilde{\vp}_{2\ep}$ is thus guaranteed upon establishing the existence of the  smooth curves $\tilde{\Lambda}(\cdot), \kappa(\cdot): [-\eta_0,\, \eta_0]\rightarrow \mathbb{C},$ which is the task of the following. 

Since $\tilde{\vp}_{2\ep} \perp \ker_g L_{0}^{*}(\ep, 0)$ (see \eqref{ortho-vp2}), we get from \eqref{eqpsi2ep} that 
\beq \label{Implicitthm}
\begin{aligned}
   & \cH_1 (\ep, \eta,  \tilde{\Lambda}, \kappa  ):=
 \big\langle 
     \cT_1(\ep, \eta,  \tilde{\Lambda}, \kappa ) + i\eta\, \cT_2(\ep, \eta,  \tilde{\Lambda}, \kappa ) +\big(\cT_0(\ep, \eta)-\cT_0(\ep, 0)\big) \tilde{\vp}_{2\ep}, \, v_{\ep}\big\rangle=0\,,\\
    & \cH_2 (\ep, \eta,  \tilde{\Lambda}, \kappa  ):=
\big \langle 
     \cT_1(\ep, \eta,  \tilde{\Lambda}, \kappa ) + i\eta\, \cT_2(\ep, \eta,  \tilde{\Lambda}, \kappa ) \\
     & \qquad \qquad\qquad\qquad+\big(\cT_0(\ep, \eta)-\cT_0(\ep, 0)\big) \tilde{\vp}_{2\ep}, \, \int_{-\infty}^x \tilde{v}_{1\ep}(x')  \,\d x' \big\rangle=0\,.
\end{aligned}
\eeq
Looking at these two identities at $\eta=0$, we obtain the following relations 
\beq\label{sec-relation}
 \begin{aligned}
&0=\cH_1|_{\eta=0}=  \big\langle  \cT_1|_{\eta=0}, \, v_{\ep}\big\rangle=  \big\langle \tilde{\Lambda}^2(0) \big( L_{1}(\ep,0)\vp_{1\ep}'-L_2(\ep,0)v_{\ep}'\big)+ L_0^1(\ep) v'_{\ep},  \, v_{\ep}\big\rangle, \\
& 0= \cH_2|_{\eta=0}=  \big\langle  \cT_1|_{\eta=0}, \, \int_{-\infty}^x \tilde{v}_{1\ep}(x') \,\d x'\big\rangle\\
& \qquad\qquad =  \big\langle \tilde{\Lambda}^2(0) \big( L_{1}(\ep,0)v_{1\ep}-L_2(\ep,0)v_{\ep}'\big)+ L_0^1(\ep) v'_{\ep}+\kappa(0)L_0(\ep, 0)v_{1\ep},  \,\int_{-\infty}^x \tilde{v}_{1\ep}(x') \,\d x'\big\rangle . 
 \end{aligned} 
 \eeq
The first relation yields $$\tilde{\Lambda}^2(0) = \f{\langle  L_0^1(\ep)v'_{\ep}, v_{\ep} \rangle } {\langle-L_1(\ep,0) \vp_{1\ep}' +L_2(\ep,0) \psi_{0\ep}, v_{\ep} \rangle}=\f{2V}{3}+\cO(\ep)>0, $$ which is consistent with \eqref{Lbda1ep}. Next, since 
\beq\label{rel-v10-v0}
\tilde{v}_{10}=v_{10}= v_0+\f{x}{2}\p_x v_0,\, \quad
L_0^*(\ep, 0)\int_{-\infty}^x \tilde{v}_{1\ep}(x') \,\d x'=\f{2}{V}v_{\ep},
\eeq 
it holds that
\beqs 
\big\langle L_0(\ep, 0)v_{1\ep}'\,,  \,\int_{-\infty}^x \tilde{v}_{1\ep}(x') \,\d x'\big\rangle
=\f{2}{V}\big\langle v_{1\ep},  \, v_{\ep} \big\rangle=\f{2}{V}\big\langle v_{10},  \, v_{0} \big\rangle+\cO(\ep^2)=\f{3}{2V}\,\|v_0\|_{L^2(\mR)}^2+\cO(\ep^2)\neq 0.
\eeqs
Moreover, in view of \eqref{L10-L01}, we have that
\begin{align*}
&\big\langle \tilde{\Lambda}^2(0) \big( L_{1}(\ep,0)v_{1\ep}-L_2(\ep,0)v_{\ep}'\big)+ L_0^1(\ep) v'_{\ep},  \,\int_{-\infty}^x \tilde{v}_{1\ep}(x') \,\d x'\big\rangle\\
&=\big\langle v_0-\f43\,v_{10},  \,\int_{-\infty}^x \tilde{v}_{10}(x') \,\d x'\big\rangle+\cO(\ep^2)=\f{1}{12}\, \|v_0\|_{L^1(\mR)}^2+\cO(\ep^2).
\end{align*} 
It thus follows from the second relation of \eqref{sec-relation} that
\beq\label{kappa0}
\kappa(0)=-\f{V}{18} \f{\|v_0\|_{L^1(\mR)}^2}{\|v_0\|_{L^2(\mR)}^2}+\cO(\ep^2)\neq 0.
\eeq 
To determine the curves $\tilde{\Lambda}(\eta),\, \kappa(\eta)$ around $\eta=0,$ we compute  by using \eqref{rel-v10-v0} that
\begin{align}
& \p_{\kappa} \cH_1|_{\eta=0}=\big\langle L_0(\ep, 0)v_{1\ep}, v_{\ep} \big\rangle=0, \notag \\
 &   \p_{\kappa} \cH_2|_{\eta=0}=\big\langle L_0(\ep, 0)v_{1\ep},\, \int_{-\infty}^{x} \tilde{v}_{1\ep}(x') \,\d x'  \big\rangle=\f{3}{2V}\,\|v_0\|_{L^2}^2+\cO(\ep^2)\neq 0, \notag\\
  &    \p_{\tilde{\Lambda}} \cH_1|_{\eta=0}= 2 \tilde{\Lambda}(0)\big\langle L_{1}(\ep,0)v_{1\ep}-L_2(\ep,0)v_{\ep}',\, v_{\ep}  \big\rangle=-\f{3}{V} \tilde{\Lambda}(0) \|v_0\|_{L^2}^2+\cO(\ep^2)\neq 0, \label{ptLbda}
\end{align}
which lead to
\beqs
\det\big(D_{\kappa,  \tilde{\Lambda} }(\cH_1, \cH_2)\big)|_{\eta=0} \neq 0. 
\eeqs
Consequently, by the Implicit Function Theorem, there exists $\eta_0$ and two $C^1$ curves $\tilde{\Lambda}(\eta),\, \kappa(\eta): [-\eta_0,\, \eta_0]\rightarrow \mathbb{C},$ such that \eqref{Implicitthm} holds true.
Moreover, it holds that
\beqs
i\,\Lambda_{2\ep}:= \p_{\eta} \tilde{\Lambda}\big|_{\eta=0}=-\f{\p_{\eta}\cH_1}{\p_{\tilde{\Lambda}} \cH_1}\big|_{\eta=0}\,.
\eeqs
In light of the expression of $\cH_1$ in \eqref{Implicitthm}, we compute
\begin{align*}
{i}\,\p_{\eta}\cH_1\big|_{\eta=0}&=-\tilde{\Lambda}(0)\big\langle 
\big(\kappa(0) L_1(\ep,0)-
L_0^1(\ep)+\tilde{\Lambda}^2(0)L_2(\ep,0)\big) \,v_{1\ep}+
L_1(\ep,0)\,\tilde{\psi}_{2\ep},\, v_{\ep}
\big\rangle  \\
&=\tilde{\Lambda}(0)\big\langle \big(\f{2}{V}\kappa(0)+\p_x^{-1}\big)\,v_{10},\, v_{0}
\big\rangle+\cO(\ep^2) \\
&=\tilde{\Lambda}(0) \bigg( \f{3}{2V}\kappa(0)\|v_0\|_{L^2(\mR)}^2-\f14\|v_0\|_{L^1(\mR)}^2 \bigg)+\cO(\ep^2).
\end{align*} 
Therefore, it follows from \eqref{kappa0} and \eqref{ptLbda} that for $\ep$ small enough, 
\beqs 
\Lambda_{2\ep}=\f{V}{9} \f{\|v_0\|_{L^1(\mR)}^2}{\|v_0\|_{L^2(\mR)}^2}+\cO(\ep)>0\, .
\eeqs
Finally, since we have  $\overline{L_0(\ep, \eta )}= {L_0(\ep, -\eta )},\,$ it holds that 
\beqs 
\overline{\cH_k(\ep,\eta, \kappa, \tilde{\Lambda})}=\cH_k\big(\ep, -\eta, \overline{\kappa}, \overline{\tilde{\Lambda}}\big), \qquad k=1,2. 
\eeqs
We thus get from the uniqueness of the curves that:
\beqs
\overline{\tilde{\Lambda}(\eta)}=\tilde{\Lambda}(-\eta), \qquad \overline{\kappa(\eta)}=\kappa(-\eta),
\eeqs
which further implies that
\beqs 
\overline{{\Lambda}(\eta)}={\Lambda}(-\eta), \qquad \overline{\tilde{\vp}_{2\ep}(\cdot, \eta)}=\tilde{\vp}_{2\ep}(\cdot, -\eta).
\eeqs

Let us summarize what we have obtained and some further  consequences which lead to Theorem \ref{thm-resolmodes}.
\begin{lem}
There exists $\ep_0, \hat{\eta}_0>0,$  such that for any $\ep\leq \ep_0,$ 

(1) For any $\hat{\eta}\leq \hat{\eta}_0, $ one can find  
$\Lambda(\hat{\eta}), 
\, {\vp}(\cdot, \hat{\eta}) \in L^2_{\hat{a}}$  which  solve \eqref{eigenpb-real} and enjoy the expansion: 
\beq \label{expan-modes-1}
\begin{aligned}
&\Lambda(\hat{\eta})=i\hat{\eta}\, \Lambda_{1\ep}-\hat{\eta}^2\Lambda_{2\ep}+\cO(\hat{\eta}^3),  \\
&\vp(\cdot,\hat{\eta})=
v_{\ep}'(\cdot)-\big(i\,\Lambda_{1\ep} \hat{\eta} - \tilde{\Lambda}_{2\ep}\hat{\eta}^2\big)v_{1\ep}(\cdot)+\hat{\eta}^2 \tilde{\vp}_{2\ep}(\cdot, \hat{\eta}) 
\end{aligned}
\eeq
where 
\begin{align*}
& \Lambda_{1\ep}=\sqrt{\f{2V}{3}}+\cO(\ep^2)\,, \qquad
\Lambda_{2\ep}=\f{V}{9} \f{\|v_0\|_{L^1(\mR)}^2}{\|v_0\|_{L^2(\mR)}^2}+\cO(\ep^2), \,    \quad \tilde{\Lambda}_{2\ep}=\f{V}{18} \f{\|v_0\|_{L^1(\mR)}^2}{\|v_0\|_{L^2(\mR)}^2}+\cO(\ep^2),\\
& v_{\ep}(\cdot)=\ep^{-2}(\p_x\psi_c) (\ep^{-1} \,\cdot)\, , \, 
\qquad v_{1\ep}(\cdot)=(\p_c\p_x\psi_c)(\ep^{-1} \,\cdot)\,, \qquad \tilde{\vp}_{2\ep}(\cdot, \hat{\eta}) \perp v_{\ep}\,.
\end{align*}

(2) For any $\eta\in [-\ep^2\hat{\eta}_0, \ep^2 \hat{\eta}_0],$ the operator $L_c(\eta)$ has  two pairs of eigenmodes $\big(\lambda_c(\eta), U_c(\cdot, \eta)\big) $ and $\big(\lambda_c(-\eta), U_c(\cdot, -\eta)\big)$ in the space $Y_{a, \eta}$
where 
\begin{align}\label{defU-base}
   & \lambda_c(\eta)=\ep^3 \Lambda(\ep^{-2}\eta)=i \ep\Lambda_{1\ep}\eta-\ep^{-1}\Lambda_{2\ep}\eta^2+\cO(\eta^3)\, , \notag\\ 
     U_c(\cdot, \eta)&= \left( \begin{array}{c}
-\sigma_{\ep,\eta}^{-1}(\hat{x}, D_{\hat{x}})\big(\lambda(\eta)-\ep \,d_c\,\p_{\hat{x}}\big)(\p_{\hat{x}}^{-1} \vp)    \\[3pt]
        -\int_{\hat{x}}^{+\infty}\vp 
    \end{array} 
    \right)(\ep\, \cdot, \, \ep^{-2}\eta ):= \left( \begin{array}{c}
U_{1c} (\cdot, \, \eta )    \\[3pt]
U_{2c} (\cdot, \, \eta )
    \end{array} 
    \right)  \, ,
\end{align}
with
\beqs 
\sigma_{\ep,\eta}^{-1}(\hat{x}, D_{\hat{x}})= \bigg(h'(\hat{\rho}_{\ep})+(e^{\ep^2\hat{\phi}_{\ep}}+\eta^2-\ep^2\p_{\hat{x}})^{-1}\bigg)^{-1}.
\eeqs
More precisely, one has the expansion
\beq \label{exp-basis}
U_c(\cdot,\eta)=U_c^1+ i \ep\Lambda_{1\ep} \eta\, U_c^2+\cO(\eta^2)
\eeq
    with $U_c^1, U_c^2$ defined in \eqref{def-firsttwo}.
    
(3) The eigenmode $(\lambda_c^{*}(\pm \eta), U_c^{*}(\cdot, \pm\eta))$ of $L^{*}_c(\eta)$ in the space $Y_{a, \eta}^{*}$ are characterized by 
\begin{align}\label{dualmodes}
    \lambda_c^{*}(\eta)=\lambda_c(-\eta), \qquad  U_c^{*}(\cdot, \eta)=(U_{2,c}, U_{1,c})^t(-\,\cdot, -\eta):=U_c^{*,1}- i \ep\Lambda_{1\ep} \eta\, U_c^{*,2}+\cO(\eta^2)\,.
\end{align}
Moreover, it holds that 
\beqs
\overline{\lambda_c(\eta)}=\lambda(-\eta), \quad \overline{U_c(\cdot, \eta)}=U_c(\cdot, -\eta), \quad \big\langle U_c(\cdot, \eta),\, U_c^{*}(\cdot, -\eta)  \big\rangle_{_{Y_{a,\eta}\times Y_{a,\eta}^{*}}}=0 ,
\eeqs
where the inner product is defined as follows:
 for $f=(f_1,f_2)^t\in Y_{a,\eta}, \, f^{*}=(f_1^*, f_2^*)^t\in  Y_{a,\eta}^{*},$
\beq \label{def-bracketYa}
\big\langle f, f^{*}\rangle_{_{Y_{a,\eta}\times Y_{a,\eta}^{*}}} =: \langle f_1, f_1^{*} \rangle_{L_a^2(\mR)\times L_{-a}^2(\mR)}+\left  \langle  \p_x f_2, \p_x f^{*}_2\right \rangle_{L_a^2(\mR)\times L_{-a}^2(\mR)}.
\eeq

(4). Write $L_c(\eta)=L_c(0)+\eta^2\,L_c^1(0)+\cO_{B(X_a)}(\eta^4).$
It holds that
\beq\label{Lbda1}
\ep\Lambda_{1\ep}=\lambda_{1c}=\sqrt{\f{\langle L_c^1(\eta)U_c^1, U_c^{*,1} \rangle}{-\langle U_c^2, U_c^{*,1}\rangle}}>0, 
\eeq
\beq\label{Lbda2}
\ep^{-1}\Lambda_{2\ep}=\lambda_{2c}=\f{\beta_{22}(c)\langle L_c^1(\eta)U_c^1, U_c^{*,1} \rangle-\beta_{12}(c)\big(\langle L_c^1(\eta)U_c^1, U_c^{*,2} \rangle+\langle L_c^1(\eta)U_c^2, U_c^{*,1} \rangle  \big)}{2\beta^2_{12}(c)}<0
\eeq
where $\beta_{jk}(c)=\int U_c^j \cdot \overline{U_c^{*,k}} \d x$ and the bracket $\langle\cdot,\cdot \rangle=\langle\cdot,\cdot \rangle_{(L_a^2(\mR))^2\times (L_{-a}^2(\mR))^2}.$ 

\end{lem}

\begin{proof}
The first point and the existence of the  eigenmode $\big(\lambda(\eta), U_c(\cdot, \eta)\big) $ for $L_c(\eta)$ in the second point have been proven. The precise expressions of 
$U_c^1$ and $U_c^2$ in the expansion \eqref{exp-basis} follows from \eqref{defU-base} and the equations  \eqref{profile-px}, \eqref{profile-pc}
satisfied by $U_c^1$ and $U_c^2.$
The existence of the other eigenmode follows from the evenness of $L_{c}(\eta)$ in $\eta.$ 
Now only the last point requires some clarification. In fact, as an operator taken on the space $Y_{a}^{*}(\eta),$ the dual operator $L_c^{\star}(\eta)$ takes the form
\beqs 
L_c^{*}(\eta)=-\left(\begin{array}{cc}
  d_c \p_x   & h'(\rho_c)+(e^{\phi_c}+\eta^2-\p_x^2)^{-1}  \\[3pt]
  \p_x (\rho_c\p_x)-\eta^2 \rho_c   & \p_x(d_c \cdot) 
\end{array}
\right).
\eeqs
Note that we have used the fact that as an operator acting on $L^2(\mR),$ $(\p_x-a)^{*}=(\p_x+a).$  
The identities in \eqref{dualmodes} then follow from the evenness of $\rho_c, \, d_c $ in $x$ and $L_c^{*}(\eta)$ in $\eta.$
\end{proof}

\begin{rmk}
    By the definition of $\beta_{jk}$ the expansions \eqref{exp-basis}, \eqref{dualmodes}, as well as the parity of the profiles, it holds that 
    \begin{align}\label{perop-betajk}
        \beta_{11}=\cO(\eta^2), \qquad \beta_{12}=\cO(1)<0\, .
    \end{align}
\end{rmk}

We are now in position to define precisely  the basis of $L_c(\eta).$ 
We already know that 
 $U_c(\cdot, \eta)$ and $\overline{U_c(\cdot, \eta)}$ are two eigenfunctions. 
 Nevertheless, their imaginary parts  vanish when $\eta=0$ and thus they do not provide a basis of the two-dimensional algebraic kernel
  of $L_c(0)$. To resolve this degeneracy, inspired by
\cite{Mizumachi-BL-linear}, 
we first introduce a pair of functions with  purely real inner product
\begin{align*}
     g_c(\cdot,\eta)=(\alpha_c(\eta)-{i}) \, U_c(\cdot, \eta), \quad g_c^{*}(\cdot, \eta)=C_{*}\, U_c^{*} (\cdot, \eta)
\end{align*}
where $$\alpha_c(\eta)=\f{\Re \langle U_c(\cdot, \eta), U_c^{*} (\cdot, \eta)\rangle}{\Im \langle U_c(\cdot, \eta), U_c^{*} (\cdot, \eta)\rangle}=\cO(\eta^2), \,\,\qquad  C_{*}=-\f{\sqrt{3}\,C(V)^2}{36 V}.$$
Note that 
\beq \label{prop-geta}
\begin{aligned}
&\quad \overline{ g_c(\cdot,\eta)}=  -g_c(\cdot, -\eta), \qquad \overline{ g_c^{*}(\cdot,\eta)}=  g_c^{*}(\cdot, -\eta), \\
 & \langle g_c(\cdot, \eta), g_c^{*}(\cdot, -\eta) \rangle =0, \qquad \Im \langle g_c(\cdot, \eta), g_c^{*}(\cdot, \eta) \rangle =0.
\end{aligned}
\eeq
We then define the basis and dual basis:
\beq \label{defbasis-dual}
\begin{aligned}
   & g_1(\cdot, \eta, c)= -\f{1}{2i
}\big( g_c(\cdot, \eta)+ g_c(\cdot, -\eta)\big) =\big(\Re U_c-\alpha_c(\eta)\,\Im U_c\big)(\cdot, \eta), \\  & g_2(\cdot,\eta,c)=\f{1}{2
\kappa_c(\eta)}\big( g_c(\cdot, \eta)-g_c(\cdot, -\eta)\big)=\f{1}{\kappa(\eta)}\big(\Im U_c+\alpha_c(\eta)\Re U_c\big)(\cdot, \eta), \\
  &  g^{*}_1(\cdot, \eta,c)=-\f{1}{2i\kappa_c(\eta)} (g_c^{*}(\cdot, \eta)-g_c^{*}(\cdot, -\eta))=
    -\f{C_{*}}{\,\kappa(\eta)}\Im U_c^{*} (\cdot, \eta),\\  & g_2^{*}(\cdot, \eta, c)=  \f{1}{2} (g_c^{*}(\cdot, \eta)+g_c^{*}(\cdot, -\eta))=C_{*}\, \Re  U_c^{*} (\cdot, \eta),
\end{aligned}
\eeq
where 
\begin{align*}
    \kappa_c(\eta)=\f12 \,\Re \langle g_c(\cdot, \eta), g_c^{*}(\cdot,\eta)\rangle_{_{Y_{a,\eta}\times Y_{a,\eta}^{*}}} =\f{C_{*}}{2} \bigg( \Im \langle U_c, U_c^{*} \rangle+\f{(\Re \langle U_c, U_c^{*} \rangle)^2}{\Im \langle U_c, U_c^{*} \rangle } \bigg)(\eta).
\end{align*}
Note that by \eqref{prop-geta}, it holds that 
\beqs 
\langle g_{j}(\cdot, \eta, c), g_k^{*}(\cdot, \eta, c)\rangle_{_{Y_{a,\eta}\times Y_{a,\eta}^{*}}} =\delta_{j k}, \quad \forall\, 1\leq j, k \leq 2; \, \eta\in [-\ep^2 \hat{\eta}_0, \ep^2 \hat{\eta}_0 ] \, .
\eeqs
It follows from straightforward computations that 
\begin{align}
  & g_1(\cdot,\eta, c)=U_c^{1}+\cO(\eta^2), \qquad 
    g_2(\cdot,\eta, c)=\f{1}{C_{*}\beta_{12}(c)} U_c^2+ \f{\alpha_c(\eta)}{\kappa_c(\eta)} U_c^1+\cO(\eta^2)
   \quad \text{ in } Y_{a,\eta}\,, \notag\\
 & g_1^{*}(\cdot,\eta, c)=\f{U_c^{*,2}}{\beta_{12}(c)} +\cO(\eta^4), \qquad \, g_2^{*}(\cdot,\eta,c)=C_{*}U_c^{*,1}+\cO(\eta^2) \quad \text{ in } Y_{-a,\eta}\,. \label{exp-g12dual}
\end{align}
where $\beta_{12}(c)=\int U_c^2 \cdot \overline{U_c^{*,1}} \d x=\int U_c^1 \cdot \overline{U_c^{*,2}} \,\d x\neq 0.$ 

\section{First two modes in the expansion of the continuous eigenfunction }
In this section, we aim to show that the first two modes in the expansion of the continuous eigenfunction $\vp_1(\cdot,\eta)$ (see \eqref{expan-modes}) is given by \eqref{firsttwomodes}. This can be achieved by differentiating with respect to the parameters  the profile equation satisfied by the background wave in the original variables. 

Recall that the wave $(n_c,\psi_c)^t$ solves the equations
\beq \label{profile eq}
    \left\{ \begin{array}{l}
       -c\,\p_x n_c+ \p_x \big((1+n_c)\p_x\psi_c\big)=0, \\[3pt]
       -c\,\p_x\psi_c+\f{|\p_x\psi_c|^2}{2}+h(1+n_c)+\phi_c=0, \,  \\[3pt]
       \p_x^2\phi_c=e^{\phi_c}-1-n_c\, .
    \end{array}
    \right. 
\eeq
Denote $u_{1c}=\p_x \psi_c, d_c=c-u_{1c}, \rho_c=1+n_c.$ Differentiating the last two equations, we find that
$(n_c, u_{1c})^t$ satisfies
\beq \label{profile-px}
    \left\{ \begin{array}{l}
       -d_c\p_x n_c+ \rho_c\,\p_x\, u_{1c}=0, \\[3pt]
      -d_c\p_x u_{1c}+h'(\rho_c)\p_x n_c+\p_x \phi_c =0 , \, \\[3pt]
    ( e^{\phi_c}-\p_x^2)\p_x\phi_c=\p_x n_c ,
    \end{array}
    \right. 
\eeq
from which we derive that 
\beq\label{id-1}
(d_c \, \sigma_0^{-1}\, d_c -\rho_c)\,u_{1c}=0
\eeq
where $\sigma_0^{-1}=\big(h'(\rho_c)+(e^{\phi_c}-\p_x^2)^{-1}\big)^{-1}.$

On the other hand, since the waves depend smoothly on the  wave speed $c,$ we can differentiate the equations \eqref{profile eq} in terms of $c$ to obtain
\beq\label{profile-pc}
    \left\{ \begin{array}{l}
       -\p_x(d_c \p_c n_c)+ \p_x (\rho_c\,\p_c\, u_{1c})=\p_x n_c, \\[3pt]
      -d_c \p_c u_{1c}+h'(\rho_c)\p_c n_c+\p_c \phi_c =u_{1c} , \, \\[3pt]
    ( e^{\phi_c}-\p_x^2)\p_c\phi_c=\p_c n_c ,
    \end{array}
    \right. 
\eeq
which enables us to find that 
\beq \label{id-2}
\p_x \big(d_c \sigma_0^{-1}\, d_c -\rho_c\big)\p_c u_{1c}=- \big( \sigma_0^{-1}\, d_c+\p_x(d_c\,\sigma_0^{-1}\p_x^{-1})\big)\p_x u_{1c}\,.
\eeq
Performing the scaling $x\rightarrow \hat{x}=\ep x,$ we derive \eqref{firsttwomodes} from \eqref{id-1} and \eqref{id-2}.

\section{The error  between the resonant modes
of Euler-Poisson and KP-II } \label{approx-resonants}
In this appendix,
we approximate the  resonant modes for the transformed linearized operator $L_{c,a}(\ep^2\hat{\eta})$
defined in \eqref{def-La} for the Euler-Poisson system by those of $L^{\hat{a}}_{\kp}(\hat{\eta})=e^{\hat{a}x} L_{\kp}(\hat{\eta}) e^{-\hat{a} x}.$ 
\subsection{Resonant modes for the linearized operator of KP-II equation} 
In this subsection, we recall some results proven in 
\cite{Mizumachi-KP-nonlinear},
pertaining to the resonant modes of   $L_{\kp}$,  the linearized operator for the  KP-II equation about  the line KdV soliton. 
This is  useful in approximating the Euler-Poisson equation by {KP-II} in the low frequency regime. 

Let  us set 
\beqs 
L_{\kp}(\eta)=e^{-iy\eta}\,L_{\kp}\,e^{iy\eta}=\p_x \big( 1-\f{1}{2V} \p_x^2-C(V)\Psi_{{\kdv}}\cdot \big)+\f{V}{2}\eta^2 \p_x^{-1},
\eeqs
where $V=\sqrt{P'(1)+1},\, C(V)=V+\f{P''(1)}{2V}$ and 
$$\Psi_{\kdv}(x)=\f{3}{C(V)}\sech^2\bigg(\sqrt{\f{V}{2}} x\bigg).$$
Performing the change of variable, 
\beqs 
\tilde{x}=\sqrt{\f{V}{2}} x, \,\quad  \tilde{\eta}=\f{2}{\sqrt{3}}\,\eta, \quad \tilde{f}(\tilde{x}, \tilde{\eta})=f(t, x)\,.
\eeqs
Then, we have
$$(L_{\kp}(\eta) f)(x,\eta)=-\sqrt{\f{V}{32}}\big( \p_{\tilde{x}}^3-4\p_{\tilde{x}}+6\p_{\tilde{x}}(2\,\sech^2 \cdot)-3\tilde{\eta}^2\p_{\tilde{x}}^{-1}\big)\tilde{f}:= \sqrt{\f{V}{32}} \tilde{L}_{\kp}(\tilde{\eta})\tilde{f}.$$
Based on the result in Lemma 2.1, \cite{Mizumachi-KP-nonlinear}, we know that the eigenmodes of $ \tilde{L}_{\kp}(\tilde{\eta}), \, \big(\tilde{L}_{\kp}(\tilde{\eta})\big)^{*}$ are 
\begin{align*}
\tilde{\lambda}_{\kp}(\tilde{\eta})=4i \tilde{\eta} \sqrt{1+i\tilde{\eta}}, \qquad \tilde{g}_0(\tilde{x}, \tilde{\eta})=\f{d_{*}}{\sqrt{1+i\tilde{\eta}}} \p_{\tilde{x}}^2\big( e^{-\sqrt{1+i\tilde{\eta}}\,\tilde{x}}\sech (\tilde{x})\big),\\
\overline{\tilde{\lambda}_{\kp}}(\tilde{\eta})=-4i \tilde{\eta} \sqrt{1-i\tilde{\eta}}, \qquad \tilde{g}^{*}_0(\tilde{x}, \tilde{\eta})=\f{i}{2d_{*}\tilde{\eta}} \p_{\tilde{x}}\big( e^{\sqrt{1-i\tilde{\eta}}\,\tilde{x}}\sech (\tilde{x})\big),
\end{align*}
where $d_{*}=-\f{3V}{C(V)}.$ Moreover, it holds by construction that for any $\tilde{\eta}\neq 0$, 
$\int \tilde{g}_0(\tilde{x}, \tilde{\eta}) \overline{\tilde{g}^{*}_0(\tilde{x}, \tilde{\eta})}\, \d \tilde{x}=1.$
Changing back to the variable $(x,\eta),$ we find that
\beq\label{resomode-kp}
\begin{aligned}
  &  {\lambda}_{\kp}(\tilde{\eta})=\sqrt{\f{2V }{3}} {i\eta}\sqrt{1+2i\eta/\sqrt{3}}\,, \qquad {g}_0(x, {\eta})=\f{\sqrt{\f{V}{2}}d_{*}}{\sqrt{1+2i{\eta}/\sqrt{3}}} \p_{\tilde{x}}^2\big( e^{-\sqrt{1+2i{\eta}/\sqrt{3}}\,\tilde{x}}\sech (\tilde{x})\big)\,,\\
&\overline{{\lambda}_{\kp}}({\eta})=-\sqrt{\f{2V }{3}} {i\eta}\sqrt{1-2i\eta/\sqrt{3}}\,, \qquad \qquad {g}^{*}_0(x\, \tilde{\eta})=\f{\sqrt{3}\,i}{4\,d_{*}{\eta}} \p_{\tilde{x}}\big( e^{\sqrt{1-2i{\eta}/\sqrt{3}}\,\tilde{x}}\sech (\tilde{x})\big)\,.
\end{aligned}
\eeq
Note that the extra factor $\sqrt{\f{V}{2}}$ in the definiton of $g_0(x,\eta)$ is added in order to ensure the condition
$\int {g}_0({x}, {\eta}) \overline{{g}^{*}_0({x}, {\eta})}\, \d {x}=1$  for any ${\eta}\neq 0.$

As in \cite{Mizumachi-KP-nonlinear}, to resolve the degeneracy of $g_0(x, \eta)$ and the singularity of $g_0^{*}(x, \eta)$ when $\eta=0,$ we define the basis and dual basis 
\begin{align*}
g_{_{01}}(\cdot, \eta)= \f12 \big(g_0(\cdot, \eta)+ g_0(\cdot, -\eta)\big), \qquad g_{_{02}}(\cdot, \eta)=\f{1}{ 2i \eta} \big(g_0(\cdot, \eta)- g_0(\cdot, -\eta)\big); \\
g_{_{01}}^{*}(\cdot, \eta)=g_0^{*}(\cdot, \eta)+ g_0^{*}(\cdot, -\eta), \qquad g_{_{02}}^{*}(\cdot, \eta)=-i\,\eta \big(g_0^{*}(\cdot, \eta)- g_0^{*}(\cdot, -\eta)\big)\,. 
\end{align*}
As is clear in the definitions, all these functions  are even in terms of $\eta,$ which leads to 
\begin{align*}
   \| g_{_{0k}}(\cdot, \eta)-  g_{_{0k}}(\cdot, 0)\|_{L^2}\lesssim \eta^2,\, \qquad  \| g^{*}_{_{0k}}(\cdot, \eta)-  g^{*}_{_{0k}}(\cdot, 0)\|_{L^2}\lesssim \eta^2 .
\end{align*}
Let $$v_0(x)=V \Psi_{\kdv}(x)=\f{3V}{C(V)}\sech^2\big( \f{\sqrt{V}}{2} x\big)\,, \quad v_{10}(x)=v_0 (x)+\f{x}{2}\,\p_x v_0 (x)\,,$$ 
it follows from direct calculations that
\beq\label{base-kp}
\begin{aligned}
&\qquad g_{_{01}}(\cdot, 0)=v_0'\, , \qquad\qquad  g_{_{02}}(\cdot, 0)=-\f{1}{\sqrt{3}}v_0'-\sqrt{\f{2V}{3}}v_{10}\,, \\
&g_{_{01}}^{*}(\cdot, 0)=-\f{C^2(V)}{9V\sqrt{2V}}
\int_{-\infty}^x v_{10}\,, \qquad g_{_{02}}^{*}(\cdot, 0)=-\f{\sqrt{3}\,C^2(V)}{18 V^2}v_0\, .
\end{aligned}
\eeq
\subsection{Expansions of the basis and dual basis} 
\label{approx-resonants-sec2}
We study the expansions of the basis in terms of $\ep$ and $\hat{\eta}_0.$
By the definition \eqref{defU-base}, expansions \eqref{expan-modes-1} and \eqref{exp-sigmainverse} as well as the fact that  $ d_c=c+\cO(\ep^2)=v+\cO(\ep^2),$ it holds that
\beq\label{U}
\begin{aligned}
\big(\diag (V, \p_x) U\big)(\ep^{-1}x, \ep^2\hat{\eta})&=\ep\, \vp(x, \hat{\eta})\left(\begin{array}{c}
   1  \\
    1 
\end{array} \right)+\cO\big(\ep(\ep^2+\hat{\eta}^2)\big)\\
&= \ep \bigg(v'_{\ep}- i\,\Lambda_{1\ep} \hat{\eta}\, v_{1\ep} \bigg) \left(\begin{array}{c}
   1  \\
    1 
\end{array} \right)+\cO\big(\ep(\ep^2+\hat{\eta}^2)\big) \, .
\end{aligned}
\eeq
Next, 
as $U^{*}(\cdot, \eta)=(U_2, U_1)^t(-\,\cdot, -\eta),$ we have by using the fact that $v_{\ep},\, v_{1\ep}$ are even that 
\beq \label{U*}
\big(\diag (V^{-1}, -\p_x^{-1}) U^{*}\big)(\ep^{-1}x, \ep^2\hat{\eta}) =V^{-1} \bigg(v_{\ep}-i\,\Lambda_{1\ep} \,\hat{\eta}\, 
\int_{-\infty}^x v_{1\ep}(x') \,\d x'  
\bigg) \left(\begin{array}{c}
     1  \\
    1 
\end{array} \right)+\cO(\ep^2+\hat{\eta}^2) \, .
\eeq
We now compute 
\begin{align*}
   & \langle U(\cdot, \eta), U^{*}(\cdot, \eta)  \rangle=2 \int U_1(x, \eta)\, \overline{U_2(-x, -\eta)}\, \d x \\
   & =2\ep^{-1}\int U_1(\ep^{-1}x, \ep^2\hat{\eta})\, {U_2(-\ep^{-1}x, \hat{\eta})}\, \d x =-2 V^{-1}\int \vp(x, \hat{\eta}) \int_{-\infty}^{x} {\vp(-x', \hat{\eta})} \,\d x' \d x+ \cO(\ep^2+\hat{\eta}^2).
\end{align*}  
 In view of the expansions of $\psi(\cdot, \hat{\eta})$ in \eqref{expan-modes-1},  we find after some calculations that
 \begin{align*}
    \Re \big\langle U(\cdot, \eta), U^{*}(\cdot, \eta) \big \rangle &=2
\, V^{-1}\hat{\eta}^2\bigg( \Lambda_{1\ep}^2 \int {v_{1\ep}} \int_{-\infty}^x {v_{1\ep}} \, \d x' \d x+2 \tilde{\Lambda}_{2\ep} \int v_{\ep} v_{1\ep}+\cO(\hat{\eta}^2)\bigg)\\
& =2V^{-1} \hat{\eta}^2 \bigg(\f18 \Lambda_{10}^2 \|v_{0}\|_{L^1(\mR)}^2+\f32 \tilde{\Lambda}_{20} \|v_0\|_{L^2(\mR)}^2+\cO(\ep^2+\hat{\eta}^2) \bigg) \\
&=\f13\, \hat{\eta}^2 \|v_{0}\|_{L^1(\mR)}^2+\cO\big(\hat{\eta}^2(\ep^2+\hat{\eta}^2)\big)
=\hat{\eta}^2 \bigg(\f{24 V}{C^2(V)}+\cO\big(\ep^2+\hat{\eta}^2\big)\bigg),
 \end{align*}
and 
\begin{align*}
\Im  \big\langle U(\cdot, \eta), U^{*}(\cdot, \eta)  \big\rangle &=\hat{\eta} \big(-4\,V^{-1}\Lambda_{1\ep}  \int v_{\ep} \, v_{1\ep} +\cO(\hat{\eta}^2) \big)\\
    &=-3V^{-1}\,\hat{\eta}\,\Lambda_{10}\, \|v_0\|_{L^2(\mR)}^2+\cO\big(\hat{\eta}(\ep^2+\hat{\eta}^2)\big) =\hat{\eta}\bigg(-\f{24\sqrt{3}\,V}{C^2(V)}+\cO\big(\ep^2+\hat{\eta}^2\big)\bigg)\, .
\end{align*} 
The last equality comes  from the fact that since  $v_0(\cdot)=\f{3V}{C(V)}\sech^2(\sqrt{{V}/{2}}\,\cdot)$, we have  that $\|v_0\|_{L^1}=\f{6\sqrt{2V}}{C(V)},\, \|v_0\|_{L^2}^2=\f{12\sqrt{2V}}{C^2(V)}.$
Consequently, we obtain
\beq \label{alpha-kpa}
\alpha(\eta)=\hat{\eta}\bigg( -\f{\sqrt{3}}{3} +\cO\big(\ep^2+\hat{\eta}^2\big)\bigg)\, , \quad \kappa(\eta)=\hat{\eta}\bigg( 1 +\cO\big(\ep^2+\hat{\eta}^2\big)\bigg)\, .
\eeq
Plugging \eqref{U}, \eqref{U*} and \eqref{alpha-kpa} into \eqref{defbasis-dual}, we obtain that:
\begin{align*}
&  ({2\ep})^{-1} (V, 1) \cdot \big(\diag (1, \, \p_x) g_1\big)(\ep^{-1}x, \ep^2\hat{\eta})= v_{\ep}' + \cO\big(\ep^2+\hat{\eta}^2\big), \\
 &  ({2\ep})^{-1} (V, 1) \cdot \big(\diag (1, \, \p_x) g_2\big)(\ep^{-1}x, \ep^2\hat{\eta})= -\f{1}{\sqrt{3}}v'_{\ep}-\sqrt{\f{2V}{3}}v_{1\ep} + \cO\big(\ep^2+\hat{\eta}^2\big), \\
&(1,1) \cdot \big(\diag (V^{-1}, \, -\p_x^{-1}) g_1^{*}\big)(\ep^{-1}x, \ep^2\hat{\eta})=-\f{C^2(V)}{9V\sqrt{2V}}
\int_{-\infty}^x v_{1\ep}+\cO\big(\ep^2+\hat{\eta}^2\big)\,,\\
&(1,1) \cdot \big(\diag (V^{-1}, \, -\p_x^{-1}) g_2^{*}\big)(\ep^{-1}x, \ep^2\hat{\eta})=-\f{\sqrt{3}\,C^2(V)}{18 V^2}v_{\ep}+\cO\big(\ep^2+\hat{\eta}^2\big).
\end{align*}
In view of the definitions \eqref{base-kp}, we finally obtain  that
\begin{align*}
& \big\|({2\ep})^{-1} (1, 1) \cdot \big(\diag (V, \, \p_x) g_k\big)(\ep^{-1}\cdot, \ep^2\hat{\eta}) -g_{0k}(\cdot, \hat{\eta})\big\|_{L_{\hat{a}}^2(\mR)}\lesssim \ep^2+\hat{\eta}^2, \, \\
&\big\|( (1, 1) \cdot \big(\diag (1, \, -\p_x^{-1}) g_k^{*}\big)(\ep^{-1}\cdot, \ep^2\hat{\eta}) -g_{0k}^{*}(\cdot, \hat{\eta})\big\|_{L_{-\hat{a}}^2(\mR)}\lesssim  \ep^2+\hat{\eta}^2\, .
\end{align*} 
\section{Some algebraic properties for the eigenvalues of the linearized operator}
In 
this section, we prove some algebraic properties for the eigenvalues of the linearized operator. To state the first property, we split the frequency space $R_{\zeta}^{I}$ (see \eqref{defregions} for the definition) into 
three subsets:
\begin{align*}
   & R_{\xi}^{H}=\{ (\xi, \zeta)\in \mR^3 \,\big||\xi|\geq K\ep, \, |\zeta|\leq 2 \}, \\
    &  R_{\zeta,1}^{I}=\{ (\xi, \zeta)\in \mR^3 \,\big| |\xi|\leq K\ep, \, \vartheta |\xi+i a| \leq |\zeta|\leq 2 \}, \\
    &  R_{\zeta,2}^{I}=\{ (\xi, \zeta)\in \mR^3 \,\big| |\xi|\leq K\ep, \, A\ep^2\leq |\zeta|\leq  \vartheta |\xi+i a| \}, 
\end{align*}
where $K(K+1)\in [{A}/{2}, A], \vartheta\in (0,\f12].$ 
\begin{lem}\label{lemmuasa}
     Let $$\mu_a(\xi,\zeta)=\sqrt{(\xi+ia)^2+|\zeta|^2\,}, \qquad \sigma_a(\xi,\eta)=\sqrt{h'(1)+(1+|\zeta|^2+(\xi+ia)^2)^{-1}}\, .$$
     Assume that $a=\hat{a}\ep$ and $\f{\hat{a}}{1-\hat{a}^2}\leq \sqrt{h'(1)+1}.$
 There exists some constant $C>0,$ such that
    \begin{align}\label{muasa}
 |\Im (\mu_a\, \sigma_a) | \leq   
    \left\{ \begin{array}{c}
        a(c-C\vartheta), \qquad \forall \, (\xi,\zeta)\in R_{\zeta}^{UH} \cup  R_{\zeta,1}^{I}\, , \\[5pt]
         a(c-C A \ep^{2}), \qquad \forall \,(\xi,\zeta)\in  R_{\xi}^{H}\cup R_{\zeta,2}^{I}\,.
     \end{array} \right.
   \end{align}
\end{lem}
\begin{proof}
We first claim that  $|\Im (\mu_a\, \sigma_a)|\leq |\Im (\mu_a)|\, \Re \sigma_a.$ 
Without loss of generality, we assume $\xi> 0.$ It is direct to compute that 
\beqs 
\Re \mu_a=\sqrt{\f{a_1+\sqrt{a_1^2+b_1^2\,}}{2}\,},\, \quad  \Im \mu_a= \f{a\xi}{\Re \mu_a},\, \quad \Re \sigma_a=\sqrt{\f{a_2+\sqrt{a_2^2+a_2^2\,}}{2}\,},\, \quad  \Im \sigma_a= \f{-a\xi}{J\Re \sigma_a}
\eeqs
where 
\beqs 
a_1=\xi^2-a^2+|\zeta|^2,\, 
\quad b_1=2a\xi, \quad a_2=h'(1)+\f{1+\xi^2-a^2+|\zeta|^2}{J},\,\quad  b_2=-\f{-2a\xi}{J}\, \,
\eeqs
and $J=(1+\xi^2-a^2+|\zeta|^2)^2+4a^2\xi^2.$ Therefore, it holds  after some calculations that for any $a\in(0, 1/2),$
\beqs
\Im (\mu_a\sigma_a)=\f{a\xi}{J \Re \mu_a \Re \sigma_a} \bigg( J(\Re \sigma_a)^2-(\Re \mu_a)^2\bigg)>0.
\eeqs
Since we also have  $\Im \mu_a \Im \sigma_a<0,$  this yields  $|\Im (\mu_a\, \sigma_a)|\leq |\Im (\mu_a)|\, \Re \sigma_a.$

We are now in position to prove \eqref{muasa}. First, on the one hand,
    it is shown in the proof of Lemma 3.5, \cite{RS-WW} that $|\Im \mu_a|\leq a$ for any $(\xi,\zeta)\in \mR^3$ and $|\,\Im \mu_a| \leq a \big(1-\kappa^2/e\big)$ whenever $|\zeta|\geq \kappa |\xi+ia|.$
    On the other hand, it holds that,  for any $(\xi,\zeta)\in \mR^3,$ 
    \beq\label{Resia}
\Re \sigma_a \leq |\sigma_a|\leq    \sqrt{h'(1)+\f{1}{J}\,}\leq \sqrt{h'(1)+\f{1}{1-a^2}}\leq \sqrt{h'(1)+1}+\ep^2= c\,
    \eeq
    as long as $a=\hat{a}\ep$ and $\f{\hat{a}}{1-\hat{a}^2}\leq \sqrt{h'(1)+1}.$
We thus get \eqref{muasa} for the regions $R_{\eta,1}^{I}$ and $R_{\eta,2}^{I}$ by setting $\kappa=\vartheta$ or $K \ep.$

Second, when $|\xi|\geq K\ep$ or $|\zeta|\geq 2,$ it holds that 
$J\geq 1+(K-1)^2\ep^2$ or $J\geq 16$ for any $a\in (0,1/2)$ and thus 
\beqs 
\Re\sigma_a\leq \sqrt{h'(1)+\f{1}{J}\,}\leq \left\{ \begin{array}{c}
\sqrt{h'(1)+1}-{C}(K\ep)^2, \quad \forall\, |\xi|\geq K\ep\, ,\\
\quad \sqrt{h'(1)+1}-\vartheta, \qquad\quad  \forall\, |\zeta|\geq 2\, ,
\end{array}\right.
\eeqs
for some positive constants $C$ and $\vartheta.$ We thus proved \eqref{muasa} for the regions $R_{\xi}^H $ and  $R_{\zeta}^{UH}$
by using that  $|\Im \mu_a| \leq a.$ 

Finally, it is also useful to note 
that thanks to \ref{Resia},  for any $(\xi,\zeta)\in \mR^3,$ and $a=\hat{a}\ep$ with $\f{\hat{a}}{1-\hat{a}^2}\leq \sqrt{h'(1)+1}, $ we have 
\beq \label{Relambdapm}
|\Im (\mu_a\sigma_a)|\leq ac\,.
\eeq
\end{proof}
Recall the definition of $\sigma_{a,\eta}(x,\xi)$ in \eqref{sigmaaeta}. By using that  $\rho_c=1+\cO(\ep^2),$ we can repeat the above arguments to show, upon choosing $\ep$ small enough,  similar results for $\sqrt{\rho_c}\,\Im \big(\mu_a\sigma_{a,\eta}\big)(x,\xi,\eta):$ 
\begin{cor}
 Assume $\ep$ to be small enough. We have,  by choosing $C>0$ found in Lemma \ref{lemmuasa} larger if necessary,  that
    \begin{align}\label{muasax}
\sup_{x\in\mR}\sqrt{\rho_c}(x) \big|\Im (\mu_a\, \sigma_{a,|\zeta|})(x,\xi) \big| \leq   
    \left\{ \begin{array}{c}
        a(c-C\vartheta), \qquad \forall \, (\xi,\zeta)\in R_{\zeta}^{UH} \cup  R_{\zeta,1}^{I}\, , \\[5pt]
         a(c-C A \ep^{2}), \qquad \forall \,(\xi,\zeta)\in  R_{\xi}^{UH}\cup R_{\zeta,2}^{I}\,.
     \end{array} \right.
   \end{align}
\end{cor}
   In the following lemma, we summarize some properties for the operator $L_a^1$ defined in \ref{def-La-1} that are widely used in the proof of the resolvent estimates in Section 3.
\begin{lem} 
Let $\pi_r, \pi_s$ be characteristic functions associated with the sets $\cS_{K,\vartheta}$ and $\cS_{K,\vartheta}^c,$ (see \eqref{def-singset}, \eqref{def-reguset}). 
    It holds that
    \beq\label{La1pis}
\|\pi_s(D)\, L_a^1 \|_{B(X)}+\|\pi_r(D) L_a^1 \, \pi_s (D)\|_{B(X)}\lesssim (K\ep+1)\,\ep^2
    \eeq
    \beq\label{La1-ll}
\| L_a^1 \,{\bI}_{4K\ep}(D_x){\bI}_{A\ep^2}(D_y)\|_{B(X)}\lesssim (K\ep+A\ep^2)\,\ep^2
    \eeq
    Moreover, let $\chi, \chi_1$ be two cut-off functions defined in \eqref{def-smoothcutoff}, the following estimate holds true
    \beq\label{la1smooth}
\|[\chi,L_a^1]\,{\bI}_{A\ep^2}(D_y)\|_{B(X)}+\|[\chi_1, L_a^1]\,{\bI}_{A\ep^2}(D_y)\|_{B(X)}\lesssim (K\ep+1)\,\ep^3.
    \eeq
  If it holds that $K\ep^4\leq 1,$ then for any $U\in X,$ 
  \beq \label{la1smooth-1}
\|[\chi_1(D_x),\,L_a^1(x, D)]\,\bI_{A\ep^2}(D_y) U\|_{X}\lesssim \ep^3\big( K^{-1}\|\chi(D_x)\, U\|_{X} +\|(1-\chi_1)(D_x) U\|_{X} \big) .
  \eeq
\end{lem}
\begin{proof}
    The first two inequalities \eqref{La1pis}, \eqref{La1-ll} follow directly from the very definition of $L_a^1$  in \eqref{def-La-1}, the support of the Fourier multipliers and the facts that $\|(\p_x\psi_c, n_c, \phi_c)\|_{W_x^{1,\infty}}=\cO(\ep^2),$ we thus omit the details.
  To prove \eqref{la1smooth}, we apply the commutator estimate \eqref{es-commutator} with $s=1$ to obtain
  \beqs 
\|\na [\chi, f] \,{\bI}_{A\ep^2}(D_y) \|_{B(L^2(\mR^3))}\lesssim (K\ep+A\ep^2) (K\ep)^{-1} \|\cF_{x\rightarrow \xi}(\p_x f)\|_{L_{\xi}^1}\lesssim (K\ep+1)\ep^3,
  \eeqs
  where $f$ is a placeholder for $(\p_x\psi_c, n_c, \phi_c).$
Let us note that in the above computation,  we have used that $\Supp (\chi'(\xi))\subset \{ \xi \,|\, K\ep\leq |\xi|\leq 4K\ep\},\, \|\chi'\|_{L_{\xi}^{\infty}}\lesssim (K\ep)^{-1}.$ 
Finally, to see \eqref{la1smooth-1}, we use   the commutator estimate \eqref{es-commutator} with $s=2$ to find 
\begin{align*}
     \|\p_x [\chi_1(D_x), f] \chi(D_x)U\|_{L^2}\lesssim C_2  
    \, \|\cF(\p_x^2 f)\|_{L_{\xi}^1}\, \|\chi(D_x) U\|_{L^2}\lesssim K^{-1}\ep^3 \|\chi(D_x) U\|_{L^2} \,
\end{align*}
where 
\beqs 
C_2=\sup_{\xi\in \mR, \, |\xi'|\leq 2K\ep}
\bigg| \f{ \chi_1(\xi)-\chi_1(\xi')}{|\xi-\xi'|^2}  \xi\bigg|\lesssim (K\ep)^{-1} .
\eeqs

\end{proof}

\section{Some useful lemmas}
We first gather some classical but useful 
product and commutator estimates.  
\begin{lem}[Product estimates]
For any $ 0\leq m\in \mathbb{Z},$ $1\leq p<+\infty,$
it holds that 
\begin{align}\label{productineq}
     \|f g\|_{W^{m,p}}\lesssim \|f\|_{L^{p_1}}\|g\|_{W^{m,p_2}}+\|f\|_{W^{m,p_2}}\|g\|_{L^{p_1}},
\end{align}
where $p_1,p_2\geq 1, \, \f{1}{p_1}+\f{1}{p_2}=\f1p.$
\end{lem}
\begin{lem}[Commutator estimate]\label{lem-appendix-commutator}
    Let $\theta$ be a smooth function on $\mR^3,$ homogeneous of degree $m$ away from a neighbourhood of $0.$ It holds that 
    \begin{align}\label{commutator-clasical}
   \big\|[\theta(D), f]g\big\|_{L^2}\lesssim \|\na f\|_{L^{\infty}} \|g\|_{H^{m-1}} +\|\na f\|_{H^{m-1}}\|g\|_{L^{\infty}}  . 
    \end{align}
    Moreover, if $\na f$ belongs only to $L^{\infty},$ 
    \begin{align}\label{commutator-crude}
         \big\|[\theta(D), f]g\big\|_{L^2}\lesssim \|\na f\|_{W^{m-1,\infty}} \|g\|_{H^{m-1}}.
    \end{align}
\end{lem}
Such inequalities are classical, we refer to Section 2.10, book \cite{book-danchin} for the proof.

In the next lemma, we state an interpolation inequality which is used frequently in the proof of the time decay estimate.
\begin{lem}\label{lem-interpo-low}
Let $8_{\kappa}=\f{8}{1-\kappa}, 8_{\kappa,1}=\f{8}{1-\f{10}{9}\kappa}, d_{\kappa}=\f{15-10\kappa/3}{3+\kappa}.$ For any 
\beq\label{relation-theta-p}
2\leq p< +\infty, \,\, -\f12\leq \theta\leq \min\bigg\{ d_{\kappa}(\f12-\f1p)-\f12, \f{d_{\kappa}-1}{2}+\f3p-\f{d_{\kappa}+3}{8_{\kappa}}  \bigg\}, 
\eeq
it holds that
\begin{align}\label{interpolation-import}
    \|P_{\leq 1} f\|_{\dot{W}^{\theta,p}} \lesssim \big\|P_{\leq 1} f\big\|_{\dot{W}^{1+\f{3}{8_{\kappa,1}},8_{\kappa}}}^{\vartheta_{\kappa}}\|P_{\leq 1}f\|_{\dot{H}^{-1/2}}^{1-\vartheta_{\kappa}},
\end{align}
where $\vartheta_{\kappa}=\f{4}{(3+\kappa)(d_{\kappa}+3)}(4+2\beta-\f6q)-\kappa^2$. 
\end{lem}
\begin{proof}
   We first prove the critical case when $2\leq p\leq 8_{\kappa}, \theta=d_{\kappa}(\f12-\f1p)-\f12\geq 0.$ By interpolation, 
   \begin{align*}
      \|P_{\leq 1} f\|_{\dot{W}^{\theta,p}} \lesssim \sum_{k\leq 1} 2^{k\ell \kappa^2} 2^{k(\theta-\ell\kappa^2)} \|P_N f\|_{L^p}&\leq \sum_{k\leq 1}   2^{k\ell\kappa^2} \big(2^{k(1+\f{3}{8_{\kappa,1}})}\|P_N f\|_{L^{8_{\kappa}}}\big)^{\vartheta_{\kappa}}\big(2^{-\f{k}{2}}\|P_N f\|_{L^2}\big)^{1-\vartheta_{\kappa}}\\
      &\lesssim \big\|P_{\leq 1} f\big\|_{\dot{W}^{1+\f{3}{8_{\kappa,1}},8_{\kappa}}}^{\vartheta_{\kappa}}\|P_{\leq 1}f\|_{\dot{H}^{-1/2}}^{1-\vartheta_{\kappa}}
   \end{align*}
   where $\ell=\f{15-10\kappa/3}{8}.$
   Next, for any $2\leq p< +\infty$ with 
$-\f12 \leq \theta\leq d_{\kappa}(\f12-\f1p)-\f12,$ we use the Sobolev embedding 
\begin{align*}
   \|f\|_{\dot{W}^{\beta,p}}\lesssim \|f\|_{\dot{W}^{\beta+\beta_1,r}} \qquad \big(\beta_1=3\big(\f1r-\f1p\big)=d_{\kappa}\big(\f12-\f1p\big)-\f12-\theta\big)
\end{align*}
to reduce the matter to the critical case. Note that to ensure $r\leq 8_{\kappa},$ we need $\theta\leq \f{d_{\kappa}-1}{2}+\f3p-\f{d_{\kappa}+3}{8_{\kappa}}.$
\end{proof}

\section{Some auxiliary  estimates supporting the proof of time decay estimates}
\begin{lem}\label{lem-H1H2}
    Let $H_1, H_2$ defined in \eqref{def-H1H2}. It holds that 
    \begin{align*}
       \||\na|^{\f{3}{8_{\kappa,1}}}(H_1, H_2)\|_{W^{1,8_{\kappa}}}\lesssim 
        \||\na|^{\f{3}{8_{\kappa,1}}}(\pt, \na_y)(n, \tilde{v})\|_{W^{1,8_{\kappa}}}
        +(1+t)^{-\f{29}{24}}\big(\cM(0)+\cN_{\tilde{c},\gamma}(T)+\cN(T)^2\big).
    \end{align*}
\end{lem}
\begin{proof}
    The proof follows from  straightforward examination of each term that appears in the definitions of $H_1$ and $H_2$.   
    It is found that 
    \beqs
    \begin{aligned}
& \||\na|^{\f{3}{8_{\kappa,1}}}(H_1, H_2)\|_{W^{1,8_{\kappa}}}\lesssim 
\||\na|^{\f{3}{8_{\kappa,1}}}(\pt, \na_y)(n, \tilde{v})\|_{W^{1,8_{\kappa}}}+\|(n,\tilde{v},w)\|_{W^{1,\infty}}
    \|(n,\tilde{v},w)\|_{W^{\ell_{\kappa}+1,8_{\kappa}}}\\   &\qquad
    +\|r_1\|_{W^{\ell_{\kappa}+1,8_{\kappa}}}+ \||\na_y(\gamma, c)|^2, \na_y^2\gamma\|_{L^{8_{\kappa}}}+ \|(\underline{u_c} n, (n_c,\underline{u_c})(\tilde{v}+w) )\|_{W^{\ell_{\kappa}+1,8_{\kappa}}},
 \end{aligned}
 \eeqs
where  $\ell_{\kappa}=1+\f{3}{8_{\kappa,1}}.$ 
Thanks to 
\eqref{alphaLP-low}  and \eqref{highreg-interp}, it holds that
\begin{align*}
&\|(n, \tilde{v})\|_{W^{1,\infty}}\lesssim (1+t)^{-\f23+\cO(\kappa)}\cN(T), \quad \|P_{\leq 0}(n, \tilde{v})\|_{L^{8_{\kappa}}}\lesssim (1+t)^{-\f{13}{24}+\cO(\kappa)}\cN(T), \\
&\|P_{\geq 0}(n, \tilde{v})\|_{{W}^{\ell_{\kappa}+1,8_{\kappa}}}\lesssim \|P_{\geq 0}(n, \tilde{v})\|_{{W}^{\ell_{\kappa}+\f{11}{8},\f{8}{2-\kappa}}} 
\lesssim (1+t)^{-\f23+\cO(\kappa)}\cN(T),
\end{align*} 
for $M\geq \f{7}{2}+\cO(\kappa).$
Therefore, by applying \eqref{es-w}, we see that the second term can be bounded by $(1+t)^{-\f{29}{24}}\cN(T)^2.$
Moreover, it follows from \eqref{es-r1r2} and interpolation that 
the next two terms are bounded by 
$(1+t)^{-\f{11}{8}}(\cN_{\tilde{c},\gamma}(T)+\cN(T)^2).$
We are left to estimate the last term. By using the inequality $\|f\|_{L_{x,y}^{8_{\kappa}}}\lesssim \|f\|_{L_y^2H_x^1}^{\f{1-\kappa}{4}}\|\p_y f\|_{L_y^2H_x^1}^{\f{3+\kappa}{4}},$ we find that 
\beqs 
\|\underline{u}_c n\|_{W^{\beta_{\kappa}},8_{\kappa}}\lesssim \|\underline{u}_c n\|_{H^{\beta_{\kappa}+1}}^{\f{1-\kappa}{4}}
\|\p_y (\underline{u}_c n)\|_{H^{\beta_{\kappa}+1}}^{\f{3+\kappa}{4}}\lesssim (1+t)^{-\f{11+\kappa}{8}}\bigg((1+t)\|n\|_{L_a^2}+(1+t)^{\f32}\|\p_y n\|_{L_a^2}+\cN(T)^2\bigg),
\eeqs
Together with the Proposition \ref{prop-weightednorm}, this  yields $$\|n_c \tilde{v}\|_{W^{\beta_{\kappa},8_{\kappa}}}\lesssim (1+t)^{-\f{11+\kappa}{8}} \big(\cM(0)+\cN_{\tilde{c},\gamma}(T)+\cN(T)^2\big).$$ The estimate of  $(\underline{u_c}, n_c)(\tilde{v}+w)$ is similar and is thus omitted.
\end{proof}

In the next two lemmas, we state the dispersive and 
bilinear estimates which are of constant use throughout Section \ref{sec-decayes}, concerning the time decay estimates for $(\vr, v)^t.$ 
To start, we first recall the dispersive estimates of the linear semigroup 
$e^{itP(D)},$  which are essentially adapted from \cite{G-P-global,Guo-Peng-Wang}. 

\begin{lem}
    It holds that for any $p\in[2,+\infty], t\in\mR,$
    \begin{align}\label{disper-general}
       \| e^{it P(D)} f \|_{L^p}\lesssim (1+|t|)^{-\f{4}{3}(1-\f{2}{p})} \|f\|_{W^{3(1-\f2p), p'}}
    \end{align}
   In particular, 
    \begin{align}\label{dispersive-L8}
         \| e^{it P(D)} f \|_{L^{8_{\kappa}}}\lesssim (1+|t|)^{-(1+\f{\kappa}{3})} \|f\|_{W^{\f{3(3+\kappa)}{4},8_{\kappa}'}}.
    \end{align}
 Moreover, when focusing on low frequencies, 
    \begin{align} \label{disper-low}
     \| e^{it P(D)}P_{\leq -5} f \|_{L^p}\lesssim \big(1+|t|\big)^{-\f32(1-\f{2}{p})}  \big\||\na|^{\f{1}{2}(1-\f2p)}P_{\leq -5} f\big\|_{L^{p'}}
    \end{align}
\end{lem}
It's worth noting that when focusing on low frequencies, not only does the decay rate improve, but there is also an increase in the number of homogeneous  derivatives on the right-hand side. This property plays a significant role in the analysis of decay estimates. 
We refer to Proposition 3.1 and Lemma 3.2, \cite{G-P-global} for the proof of \eqref{disper-general}. Let us remark that in \cite{G-P-global}, the authors prove it  for the case $h'(1)=1.$ However, the important algebraic properties of the function
$P'(r)=r\sqrt{h'(1)+(1+r^2)^{-1}}$
needed in the proof are the same. For instance, there is only one positive root for $P''(r):$ $r_0=\sqrt{1+\sqrt{4+\f{3}{h'(1)}}\,}$ and $P'''(r)$ does not vanish in the vicinity of $r_0.$ Regarding the proof of \eqref{disper-low},  we apply [Inequality (6), Theorem 1, \cite{Guo-Peng-Wang}] to obtain that 
\begin{align*}
     \| e^{it P(D)}P_{\leq -5} f \|_{L^{\infty}}\lesssim \big(1+|t|\big)^{-\f32}  \big\||\na|^{\f{1}{2}}P_{\leq -5} f\big\|_{L^{1}}
\end{align*}
from which we derive \eqref{disper-low} via interpolation.
 Note that the only properties for the function $P(r)$ used in the proof are the following
 \begin{align*}
    | P'(r)| \sim 1, \,\qquad |P''(r)|\sim r, \qquad 
    |P^{k}(r)|\lesssim r^{-k}\, (k\geq 3)  \quad \, \forall\, {0<r< 1},
 \end{align*}
which can be  easily checked.
 

In the next lemma, we state some estimates for the bilinear operator 
$$B^{\mu\nu}(f,g)=\cF^{-1}\bigg(\int_{\mR^3}\f{m_{\mu\nu}} {i\,\phi^{\mu\nu}} (\varsigma, \varsigma') \,f(\varsigma-\varsigma') \,g (\varsigma')\, \d \varsigma' \bigg),$$
where $m_{\mu\nu}(\varsigma, \varsigma')= |\varsigma| \cR (\varsigma-\varsigma')\cR(\varsigma'),\,\,\, \phi^{\mu\nu}(\varsigma, \varsigma')=P(\varsigma)-\mu P(\varsigma-\varsigma')-\nu P(\varsigma').$ 
The estimates \eqref{BL-High2}, \eqref{BL-High-less2} follow directly from the results in \cite{G-P-global}  while the remaining, newly derived estimates are also essential for our analysis. 
\begin{lem}
(I). (One is high frequency). It holds that for any $m\geq 0,$ any $\ep>0$ small enough
\begin{align}\label{BL-High2}
    \big\|B^{\mu\nu}( f, P_{\geq -10} \,g)\big\|_{H^m}\lesssim \big\| \f{f}{|\na|}\big\|_{W^{2,p_1}}\big\|P_{\geq -10}\, \f{g}{|\na|}\big\|_{W^{m+\lambda, q_1}}+
    \big\|P_{\geq -10}\f{f}{|\na|}\big\|_{W^{m+\lambda,p_1}}\big\|P_{\geq -10}\, \f{g}{|\na|}\big\|_{W^{2, q_2}}
\end{align}
where $\lambda>\f52$ and $\f{1}{p_1}+\f{1}{q_1}=\f{1}{p_2}+\f{1}{q_2}=\f{1}{2}+\ep.$ 
Moreover, there holds that 
\beq\label{BL-High-less2}
\begin{aligned}
     \big\|B^{\mu\nu}( f, P_{\geq -10} \,g)\big\|_{W^{m,p'}}&\lesssim_{\ep} \min\big\{ \big\| \f{f}{|\na|}\big\|_{W^{2,q}} \big\|P_{\geq -10}\, \f{g}{|\na|}\big\|_{H^{m+\lambda}}, \big\| \f{f}{|\na|}\big\|_{H^2} \big\|P_{\geq -10}\, \f{g}{|\na|}\big\|_{W^{m+\lambda,q}}\big\}\\
 &  \,  +\min\big\{
    \big\|P_{\geq -10}\f{f}{|\na|}\big\|_{H^{m+\lambda}}\big\|P_{\geq -10}\, \f{g}{|\na|}\big\|_{W^{2, q}}, \big\|P_{\geq -10}\f{f}{|\na|}\big\|_{W^{m+\lambda,q}}\big\|P_{\geq -10}\, \f{g}{|\na|}\big\|_{H^{2}}\big\},
\end{aligned}
\eeq
where  $p'=\f{1}{1-1/p},\, 2\leq p, q < 12$ and 
$\f1p+\f1q=\f12+\ep.$ 

(II). (Low-Low estimate).
It holds that for any $r\in [2,3],$ any arbitrary small $\ep>0$ 
\begin{align}\label{bilinear-worst}
    \big\|B^{\mu\nu}(P_{\leq -5} f, P_{\leq -5} \,g)\big\|_{L^r}\lesssim_{\ep} \big\|P_{\leq -5}\f{f}{|\na|}\big\|_{L^p} \big\|P_{\leq -5}\f{g}{|\nabla|}\big\|_{L^q},
\end{align}
where $\f{1}{p}+\f{1}{q}=
     \f{1}{3}+\f{1}{2r}+\ep.$ 
Moreover, for any $\theta\in [0, \f{1}{2}],$ 
it holds that 
\begin{align}\label{BL-low-moreder}
    \big\||\na|^{\theta}\,B^{\mu\nu}(P_{\leq -5} f, P_{\leq -5} \,g)\big\|_{L^{p'}}\lesssim_{\ep}
    \min\big\{ \big\|P_{\leq -5}\f{f}{|\na|}\big\|_{L^2} \big\|P_{\leq -5}\f{g}{|\nabla|}\big\|_{L^q},  \big\|P_{\leq -5}\f{f}{|\na|}\big\|_{L^q} \big\|P_{\leq -5}\f{g}{|\nabla|}\big\|_{L^p}\big\}
\end{align}
where $p'=\f{1}{1-1/p},\, 2\leq p, q < \f{6}{1/2-\theta}$ and 
$\f1p+\f1q=\f{7-2\theta}{12}+\ep.$ 

(III) (High-low estimate).
It holds that for any $r\in (1,+\infty),$ any $ m\geq 0,$ 
any $\ep>0$ small enough
\begin{align}\label{BL-HL}
\|B^{\mu\nu}(P_{\geq -5}f, P_{\leq -10}\,g)\|_{W^{m,r}}\lesssim \|P_{\geq-5} f\|_{W^{m+1+\ep,p}} \big\|P_{\leq -10}\f{f}{|\na|^{1+\ep}}\big\|_{L^q},
\end{align}
where $\f{1}{p}+\f1q=\f{1}{r}.$
\end{lem}
\begin{proof}
The estimates \eqref{BL-High2}, \eqref{BL-High-less2} are the consequence of Proposition 6.1 and Lemma 6.2 of \cite{G-P-global}. We thus only focus on the proof of \eqref{bilinear-worst}-\eqref{BL-HL}.\\[4pt]
Proof of \eqref{bilinear-worst}: 
We need the following result that is proved in [Lemma 6.2, \cite{G-P-global}]: 

For any $0\leq s\leq \f32, 2\leq r, p, q\leq \f{3}{3-2s},$ it holds that 
\begin{align}\label{uesful-bi}
    \|B^{\mu\nu}[P_j f, P_k \, g]\|_{L^r}\lesssim 
    \|\mathfrak{M}_{jk}^{\mu\nu}\|_{L_{\varsigma}^b H_{\varsigma'}^{s}}\big\|\f{P_j f}{|\na|}\big\|_{L^p}\big\|\f{P_k g}{|\na|}\big\|_{L^q}
\end{align}
where $\f1b+\f1r=\f12, \, \f1p+\f1q=1-\f{s}{3}$ and 
\beq\label{def-fM}
\mathfrak{M}_{jk}^{\mu\nu}(\varsigma, \varsigma')=\f{m_{\mu\nu}(\varsigma, \varsigma')|\varsigma-\varsigma'||\varsigma'|}{i\,\phi^{\mu\nu}(\varsigma, \varsigma')}\tilde{\vp}\big((\varsigma-\varsigma')/{2^j}\big)\tilde{\vp}\big(\varsigma'/2^k\big),
\eeq
 with $\tilde{\vp}(|\cdot|)$ being some functions localizing on the 
 annulus $\{2\leq |\cdot|\leq 4\}.$ 

  To prove \eqref{bilinear-worst}, let us write 
\begin{align*}
B^{\mu\nu}(P_{\leq -5} f, P_{\leq -5} \,g)=\bigg(\sum_{k+5\leq j\leq -5}
+\sum_{j+5\leq k\leq -5}+\sum_{\substack{|k-j|\leq 4,\\ \max\{k,j\}\leq -5}}\bigg)B^{\mu\nu}(P_{j} f, P_{k} \,g)=\colon (1)+(2)+(3).
 \end{align*}
 Let us give the proof of $(1)$ and $(3),$ the estimate (2) can be obtained from $(1)$ by symmetry.
Following closely the computations carried out in [Proposition 6, \cite{G-P-global}], we can verify that, for any 
$\mu,\nu\in\{+,-\},\, k, j\leq -5,$ any $s>1,$ 
\beq\label{useful-calculs}
 \|\mathfrak{M}_{jk}^{\mu\nu}(\varsigma,\cdot)\|_{
 H_{\varsigma'}^{s}}\lesssim 2^{(\f52-2s)k} \min\{ 1,\, {2^k}/{|\varsigma|}\}^{s-1}
 \eeq
 which leads to, since $|\varsigma|\lesssim |\varsigma-\varsigma'|\sim 2^j,$ that 
 \begin{align*}
     \|\mathfrak{M}_{jk}^{\mu\nu}(\varsigma,\cdot)\|_{L_{\varsigma}^b 
 H_{\varsigma'}^{s}}\lesssim 2^{k(\f32-s)}2^{j(\f3b+1-s)}.
 \end{align*}
 For any $r\in [2,3], 0<\ep<<\f13,$ we take $s=2-\f{3}{2r}-3\ep$ and apply \eqref{uesful-bi} to find that
 \begin{align*}
    \| (1)\|_{L^r}\lesssim \big\|P_{\leq -5}\f{g}{|\na|}\big\|_{\dot{B}_{q,q}^0} \sum_{j\leq -5} 2^{6\ep j}\big\|\f{P_j f}{|\na|}\big\|_{L^p} \lesssim  \big\|P_{\leq -5}\f{g}{|\nabla|}\big\|_{L^q}\big\|P_{\leq -5}\f{f}{|\na|}\big\|_{L^p}.
 \end{align*}
 Note that we have used the fact $L^q \hookrightarrow  \dot{B}_{q,q}^0$ for $q\geq 2.$
 We thus finish the estimate of $(1).$ Thanks again to \eqref{uesful-bi} and \eqref{useful-calculs}, 
 we handle the term (3) in  the following way:
 \begin{align*}
     \|(3)\|_{L^r}\lesssim \sum_{l\leq k+5\leq 0} \|P_{l}B^{\mu\nu}\big(P_{[k-4,k+4]}f ,P_k g\big)\|_{L^r}&\lesssim \sum_{l\leq k+5\leq 0}2^{\f3b l} 2^{(\f{5}{2}-2s)k} \big\|P_{[k-4,k+4]}\f{f}{|\na|}\big\|_{L^p} \big\|P_{k\leq -5}\f{g}{|\na|}\big\|_{L^q}
 \end{align*}
 which, by taking again $s=2-\f{3}{2r}-3\ep,$ leads to that $ \|(3)\|_{L^r}\lesssim \big\|P_{\leq -5}\f{f}{|\na|}\big\|_{L^p} \big\|P_{\leq -5}\f{g}{|\nabla|}\big\|_{L^q}.$ \\[5pt]
Proof of \eqref{BL-low-moreder}. Taking $s=\f54+\f{\theta}{2}-3\ep,$ we get from\eqref{useful-calculs}  that $|\varsigma'|^{\theta}\mathfrak{M}_{jk}^{\mu\nu}(\varsigma, \varsigma')\in L_{\varsigma}^{\infty}
 H_{\varsigma'}^{s}.$
 The estimate \eqref{BL-low-moreder} then follows from \eqref{uesful-bi} and duality. \\[5pt]
Proof of \eqref{BL-HL}. Let us write 
 \begin{align*}
 B^{\mu\nu}[P_j f, P_k \, g]= \cF_{\varsigma\rightarrow\mathrm{x}}^{-1} \bigg(\int \f{m_{\mu\nu}} {i\,\phi^{\mu\nu}} (\varsigma, \varsigma')\widehat{P_j f}(\varsigma-\varsigma') \widehat{P_k g}(\varsigma')\, \d \varsigma'\bigg)=\int K_{jk}(\mathrm{x}-\mathrm{y}, \mathrm{y}-\mathrm{z}) \f{P_j{f}}{|\na|}(\mathrm{y}) \f{P_k g}{|\na|}(\mathrm{z}) \,\d \mathrm{y} \d \mathrm{z}
 \end{align*}
where $K_{jk}^{\mu\nu}(\mathrm{x}, \mathrm{y})=\cF_{\varsigma\rightarrow\mathrm{x}, \,\varsigma'\rightarrow\mathrm{y}}\big(\mathfrak{M}_{jk}^{\mu\nu}(\varsigma, \varsigma') \big),$ $\mathfrak{M}_{jk}^{\mu\nu} $ being defined in \eqref{def-fM}.
It thus holds that for any $j\geq -5\geq k+5,$
\beqs 
\|B^{\mu\nu}[P_j f, P_k \, g]\|_{W^{m,r}}\lesssim 2^{jm} \|K_{jk}^{\mu\nu}\|_{L_{\mathrm{x},\mathrm{y}}^1}
\|P_j f\|_{L^p} \|P_k g\|_{L^q}.
\eeqs
To achieve \eqref{BL-HL}, it suffices to prove the following fact
\begin{align}\label{L1kernel}
\|K_{jk}^{\mu\nu}\|_{L_{\mathrm{x},\mathrm{y}}^1}\lesssim 2^{2j} (1+2^k)^{4}\lesssim 2^{2j}. 
\end{align}
We will sketch the proof of $K_{jk}^{++}$, the case of $K_{jk}^{+-}$ is similar, the other two cases $K_{jk}^{-+}$ and  $K_{jk}^{--}$ are easier since $|\phi^{-+}|\geq 2^j , |\phi^{--}|\geq 2^j$ when $ 2^k\sim|\varsigma'|\leq \f{1}{32}|\varsigma-\varsigma'|\sim 2^j.$ Denote for simplicity 
$$\phi(\varsigma,\varsigma')=\phi^{++}(\varsigma,\varsigma')=P(\varsigma)- P(\varsigma-\varsigma')- P(\varsigma').$$
Following the calculations made in [Proposition 6.1, \cite{G-P-global}], we find that for any $|\varsigma'|\leq \f{1}{32}|\varsigma-\varsigma'|,$
\begin{align*}
    &|\na_{\varsigma'}\phi|\lesssim \f{|\varsigma|}{\langle \varsigma'\rangle^2\langle \varsigma\rangle}+\sin \theta 
    , \quad  |\na_{\varsigma}\phi|\lesssim \f{|\varsigma'|}{\langle \varsigma \rangle^3}+\f{|\varsigma'|}{|\varsigma|} \sin\theta, \quad |\Delta_{\varsigma'}\phi|\lesssim \f{1}{ |\varsigma'|}, \quad |\Delta_{\varsigma}\phi|\lesssim \f{|\varsigma'|}{\langle\varsigma\rangle^2}\\
&|\na_{\varsigma'}\na_{\varsigma}\phi|\lesssim \f{1}{|\varsigma|},\quad 
\big|(\na_{\varsigma'}\Delta_{\varsigma}\na_{\varsigma}\Delta_{\varsigma'})\phi\big|\lesssim \f{1}{|\varsigma'|^2},\quad |\Delta_{\varsigma'}\Delta_{\varsigma}\phi|\lesssim \f{1}{|\varsigma|^5},
\end{align*}
and 
\beqs 
|\phi|\geq |\varsigma'|(\theta^2+d^2),
\eeqs
 where $\theta$ is the angle between the vectors $\varsigma$ and $\varsigma'.$ From the above estimates, we readily compute that 
 \begin{align*}
    | \mathfrak{M}_{jk} |\lesssim 2^{2j} A_{jk}, \quad |\Delta_{\varsigma}\mathfrak{M}_{jk}|\lesssim A_{jk}^2, \quad |\Delta_{\varsigma'}\mathfrak{M}_{jk}|\lesssim 2^{2(j-k)}A_{jk}^2, \quad |\Delta_{\varsigma'}\Delta_{\varsigma}\mathfrak{M}_{jk}|\lesssim 2^{-3k} A_{jk}^3
 \end{align*}
 where $A_{jk}=\big(\theta^2+(\f{2^{j}}{\langle 2^j\rangle\langle 2^k\rangle})^2\big)^{-1}.$
 We thus obtain that 
 \begin{align*}
\|\mathfrak{M}_{jk}\|_{L^2_{\varsigma,\varsigma'}}\lesssim 2^{\f72 j} 2^{\f32 k} \langle 2^k\rangle, \quad \|\Delta_{\varsigma'}\mathfrak{M}_{jk}\|_{L^2_{\varsigma,\varsigma'}}\lesssim 2^{\f72 j} 2^{-\f12 k} \langle 2^k\rangle^3, \\
\|\Delta_{\varsigma}\mathfrak{M}_{jk}\|_{L^2_{\varsigma,\varsigma'}}\lesssim 2^{\f32 j} 2^{\f32 k} \langle 2^k\rangle^3, \quad \|\Delta_{\varsigma'}\Delta_{\varsigma}\mathfrak{M}_{jk}\|_{L^2_{\varsigma,\varsigma'}}\lesssim 2^{\f32 j} 2^{-\f12 k} \langle 2^k\rangle^5. 
 \end{align*}
 The estimate \eqref{L1kernel} then follows from the inequality 
 \begin{align*}  \|K_{jk}^{\mu\nu}\|_{L_{\mathrm{x},\mathrm{y}}^1}\lesssim\|\mathfrak{M}_{jk}\|_{L^2_{\varsigma,\varsigma'}}^{{1}/{16}}
\|\Delta_{\varsigma'}\mathfrak{M}_{jk}\|_{L^2_{\varsigma,\varsigma'}}^{3/16}\|\Delta_{\varsigma}\mathfrak{M}_{jk}\|_{L^2_{\varsigma,\varsigma'}}^{3/16}\|\Delta_{\varsigma'}\Delta_{\varsigma}\mathfrak{M}_{jk}\|_{L^2_{\varsigma,\varsigma'}}^{9/16}.
 \end{align*}
\end{proof}

\section*{Acknowledgement}
The research of the first author is supported by the ANR 
project  ANR-23-CE40-0014-01 
while the second author is supported by the
 ANR 
 project ANR-24-CE40-3260.
The second author would like to thank Luis Miguel Rodrigues for enlightening discussions on the `high frequency damping' technique in the study of the stability of travelling waves
in partially dissipative perturbations of hyperbolic systems and for pointing out the references \cite{Mascia-Zumbrum, Miguel-zumbrum,Miguel-Faye}.

\bibliographystyle{abbrv}

\nocite{*}
\bibliography{ref}
\end{document}